                      \def\version{1st November 2024}                          %

\documentclass[reqno,12pt]{amsart} 
  
\usepackage[refpage,noprefix]{nomencl}
\usepackage{verbatim}
\usepackage[mathscr]{eucal}
\usepackage{graphicx}
\usepackage{graphics}
\usepackage{color}
\usepackage{amsmath}
\usepackage{amssymb} 
\usepackage{mathtools} 
\usepackage[titletoc]{appendix}
\numberwithin{equation}{section}
\usepackage{srcltx} 
\usepackage{upgreek}
\usepackage{enumerate}   

\usepackage{mdframed}
\mdfsetup{%
frametitlealignment=\center}
\newmdenv[frametitle=Free Parameter Constraints]{ctracker}    
  
\usepackage{slashed}
    
\newcommand{\ProofEnde}{ \hfill $\square$} 
\DeclareMathAlphabet{\mathpzc}{OT1}{pzc}{m}{it}

\newcommand{\abs}[1]{{\lvert #1\rvert}}

\newcommand{\norm}[1]{{\lVert #1\rVert}}
\newcommand{\tnorm}[1]{{|\hspace{-0.35mm}\lVert #1\rVert\hspace{-0.35mm}|}}

\newcommand{\bA} {\boldsymbol{A}} 
\newcommand{\bB} {\boldsymbol{B}} 
\newcommand{\bC} {\boldsymbol{C}} 
\newcommand{\bE} {\boldsymbol{E}}
\newcommand{\bF} {\boldsymbol{F}}

\newcommand{\bI} {\boldsymbol{I}}
 
\newcommand{\bM} {\boldsymbol{M}} 
\newcommand{\bP} {\boldsymbol{P}} 
\newcommand{\bQ} {\boldsymbol{Q}} 
\newcommand{\bR} {\boldsymbol{R}} 
\newcommand{\bS} {\boldsymbol{S}}
\newcommand{\bT} {\boldsymbol{T}}

\newcommand{\cbX} {\boldsymbol{\mathcal X}}

\newcommand{\cbV} {\boldsymbol{\mathcal V}}

\newcommand{\cbO} {\boldsymbol{\Omega}}

\newcommand{\ba} {\boldsymbol{a}}

\newcommand{\bd} {\boldsymbol{d}}
\newcommand{\be} {\boldsymbol{e}}

\newcommand{\bq} {\boldsymbol{q}}

\newcommand{\one} {\boldsymbol{1}}

\def\emptyset{\varnothing} 

\def\a{\alpha} 
\def\b{\beta} 
\def\g{\gamma}
 
\def\d{\delta} 
 
\def\l{\lambda} 

\newcommand{\s}{\sigma}

\def\o{\omega} 
\def\O{\Omega}

\def\p{\partial} 
\def\ex{{\rm e}}

\newfam\Bbbfam 
\font\tenBbb=msbm10 
\font\sevenBbb=msbm7 
\font\fiveBbb=msbm5 
\textfont\Bbbfam=\tenBbb 
\scriptfont\Bbbfam=\sevenBbb 
\scriptscriptfont\Bbbfam=\fiveBbb

 \newcommand{\C}     {\mathbb{C}} 
\newcommand{\R}     {\mathbb{R}} 
\newcommand{\Z}     {\mathbb{Z}} 
\newcommand{\N}     {\mathbb{N}}

\newcommand{\floor}[1]{\left\lfloor #1 \right\rfloor}

\def\1{{\mathchoice {1\mskip-4mu\mathrm l}      
{1\mskip-4mu\mathrm l} 
{1\mskip-4.5mu\mathrm l} {1\mskip-5mu\mathrm l}}} 
\newcommand{\ssup}[1] {{\scriptscriptstyle{({#1}})}} 
\def\comment#1{} 

\newcommand{\bunderline}[1]{\underline{#1\mkern-4mu}\mkern4mu }

\renewcommand{\theequation}{\thesection.\arabic{equation}} 
 
\newtheorem{theorem}{Theorem}[section] 
\newtheorem{lemma}[theorem]{Lemma} 
\newtheorem{prop}[theorem] {Proposition} 
 
\newtheorem{remark}[theorem]  {Remark}

\theoremstyle{definition}

\newtheoremstyle{thm}{2ex}{2ex}{\itshape\rmfamily}{} 
{\bfseries\rmfamily}{}{1.7ex}{} 
 
\newtheoremstyle{rem}{1.3ex}{1.3ex}{\rmfamily}{} 
{\itshape\rmfamily}{}{1.5ex}{} 
 
\renewcommand{\d}{{\rm d}}

\newcommand{\eps}{\varepsilon} 
\newcommand{\La}{\Lambda}

\newcommand{\Exp}{\mathscr{E}\kern-0.2mm{\operatorname{xp}}}
\newcommand{\Log}{\mathscr{L}\kern-0.2mm{\operatorname{og}}}

\newcommand{\Acal}   {{\mathcal A }}
\newcommand{\Bcal}   {{\mathcal B }}
\newcommand{\Ccal}   {{\mathcal C }}

\newcommand{\Gcal}   {{\mathcal G }} 
\newcommand{\Hcal}   {{\mathcal H }} 
\newcommand{\Ical}   {{\mathcal I }} 
\newcommand{\Jcal}   {{\mathcal J }} 
\newcommand{\Kcal}   {{\mathcal K }} 
 
\newcommand{\Mcal}   {{\mathcal M }} 
 
\newcommand{\Ocal}   {{\mathcal O }} 
\newcommand{\Pcal}   {{\mathcal P }} 
\newcommand{\Qcal}   {{\mathcal Q }} 
 
\newcommand{\Scal}   {{\mathcal S }} 
\newcommand{\Tcal}   {{\mathcal T }} 
\newcommand{\Ucal}   {{\mathcal U }} 
\newcommand{\Vcal}   {{\mathcal V }} 
\newcommand{\Wcal}   {{\mathcal W }}

\newcommand{\Zcal}   {{\mathcal Z }}

\newcommand{\Hscr} {\mathscr{H}}

\usepackage{lipsum}

\makeatletter
\newcommand\ackname{Acknowledgements}
\if@titlepage
  \newenvironment{acknowledgements}{%
      \titlepage
      \null\vfil
      \@beginparpenalty\@lowpenalty
      \begin{center}%
        \bfseries \ackname
        \@endparpenalty\@M
      \end{center}}%
     {\par\vfil\null\endtitlepage}
\else
  
\fi
\makeatother

\begin{document}

\title[\hfill]{\large 
Characterising the infinite volume scaling limit of gradient models with non-convex energy}

\maketitle

\thispagestyle{empty} 
\vspace{0.2cm} 
 
\centerline {\sc By Stefan Adams (\textdagger) and Andreas Koller\footnote{Mathematics Institute, University of Warwick, Coventry CV4 7AL, United Kingdom, {\tt Andreas.Koller@warwick.ac.uk}} }

\vspace{0.4cm}

\centerline{\small(\version)} 
\vspace{.5cm} 

\vspace{2cm}

\begin{quote} 
{\small {\bf Abstract:}}
We study the scaling limit of statistical mechanics models with non-convex Hamiltonians that are gradient perturbations of Gaussian measures. Characterising features of our gradient models are the imposed boundary tilt and the surface tension (free energy) as a function of tilt. In the regime of low temperatures and bounded tilt, we prove the scaling limit with respect to infinite volume Gibbs states for macroscopic functions on the continuum, and we show that the limit is a \emph{continuum Gaussian Free Field} with covariance (diffusion) matrix given as the Hessian of surface tension. Our proof of this longstanding conjecture for non-convex energy complements recent studies in \cite{Hil16, ABKM}, as well as the proof for strictly convex Hamiltonians in \cite{AW}.
\end{quote}
\vfill

\vfill


 
\eject

\setcounter{section}{0} 
\setcounter{footnote}{0}


\section{Introduction}

Random fields of gradients are a class of model systems in the study of random interfaces, random geometry, field theory and  elasticity theory. These random objects pose challenging problems for probabilists, as even an \emph{a priori} distribution involves strong correlations. In this article, we study random fields of gradients on the integer lattice $\Z^d$, which associate to each $x\in\Z^d$ a random variable $\phi(x)$ taking values in a finite-dimensional Euclidean space. Motivated by equilibrium statistical mechanics, the law of these variables is formally given in terms of a \emph{Hamiltonian}, or energy function, of the general form
\begin{equation} \label{E:formalH}
H(\phi) = \frac{1}{2}\sum_{x\sim y} V\big(\phi(x)-\phi(y)\big),
\end{equation}
where the sum is over unordered nearest-neighbour pairs on $\Z^d$ and $\phi(x)\in \R$. The function denoted $V$ is called the \emph{potential} and regulates the strength of the interaction. In the general setting of \eqref{E:formalH}, $V$ is typically an even, measurable function, bounded from below and exhibiting sufficient growth for large values of $\abs{\phi(x)-\phi(y)}$. In fact, for our results we allow more general anisotropic interactions defined by a potential $V:\R^d\to\R$ that satisfies assumptions described in detail further below.

To construct a corresponding probability measure rigorously, we have to restrict \eqref{E:formalH} to finite subsets $\Lambda\subset\Z^d$. The associated Boltzmann weight then determines a measure on $\R^\Lambda$ as follows:
\begin{equation} \label{E:GGMfinvol}
\frac{1}{Z}\exp\Big(-\frac{\beta}{2}\sum_{\substack{x\sim y \\ \{x,y\}\cap\Lambda\neq\emptyset}} V\big(\phi(x)-\phi(y)\big)\Big)\prod_{x\in\Lambda}\d\phi(x).
\end{equation}
In this expression, $\d\phi(x)$ stands for the Lebesgue measure and $\beta>0$ is the \emph{inverse temperature}. We also need to fix the values of $\phi$ on the boundary of $\Lambda$. We let $\phi(x)=x \cdot u$ for some \emph{tilt} $u\in\R^d$ and $x \in\Lambda^{\rm c}$. The normalising constant $Z$ of the measure depends on its parameters and is written $Z=Z_{\Lambda}(\beta,u)$. In principle, a measure on $\R^{\Z^d}$ may then be obtained by taking limits along a sequence of subsets $\Lambda\nearrow\Z^d$. This is complicated by the fact that values of the field $\phi(x)$ may \emph{delocalise}, in which case the above measures are not tight and no such infinite-volume limit for $\phi$ exists. However, the gradients of the field, that is, the differences $\phi(y)-\phi(x)$ for nearest neighbours $x$ and $y$ in $\Lambda$, generally fluctuate on a lower order. By restricting to events measurable with respect to the gradients, we ensure that an infinite-volume limit always exists in the setting we study in this article. These measures are referred to as \emph{gradient Gibbs measures} (or \emph{GGMs}) in the literature, since they depend only on the distribution of the gradients.

An important object in this context is the \emph{surface tension\/} or \emph{free energy\/} defined by the limit
\begin{equation}
\sigma_\beta(u)=-\lim_{\Lambda\nearrow\mathbb{Z}^d}\frac{1}{\beta|\Lambda|}\log Z_\Lambda(\beta,u)\,,\quad u\in\R^d\,.
\end{equation}
The surface tension $ \sigma_\beta(u) $ can be interpreted as the price to pay for tilting a macroscopically flat interface in the direction normal to the vector $(1,-u)\in\R^{d+1}$. The principal aim of this article is to prove that the infinite-volume scaling limit for non-convex potentials is characterised by the Hessian of surface tension (as a function of tilt) and thereby settle a longstanding conjecture.

The simplest example of the class of models described by \eqref{E:GGMfinvol} is represented by quadratic potentials $V$. The corresponding random object, the \emph{massless free field}, is Gaussian and its properties in finite and infinite volume are generally well known. Models with anharmonic interaction potentials have been studied extensively in the last three decades and considerable progress has been made in extending our understanding beyond the Gaussian setting. Among the main questions of interest we highlight the following two: first, one would like to know when the surface tension is strictly convex as a function of the tilt and, secondly, whether the large-scale behaviour is that of the massless free field (\emph{Gaussian universality class}). 

The most comprehensive results are available for the class of strictly (and uniformly) convex potentials $V$. Naddaf and Spencer \cite{NS97} showed that the infinite volume scaling limit of such GGMs with zero tilt is a (continuum) Gaussian Free Field (\emph{GFF}). About the same time, Funaki and Spohn \cite{FS97} proved the strict convexity of the surface tension (see also \cite{DGI00}) and established a classification of the ergodic infinite-volume GGMs by their tilt. Naddaf and Spencer's result was then extended to general tilt by Giacomin, Olla and Spohn \cite{GOS}, who also conjectured that the covariance (or diffusion) matrix of the limiting GFF should correspond to the Hessian matrix of the surface tension. Miller \cite{Miller} proved a Gaussian scaling limit with respect to the finite volume gradient measures on a bounded, two-dimensional domain with boundary conditions that are a continuous perturbation of a macroscopic tilt. His result also exhibits a concrete representation of the limiting covariance in terms of the infinite volume GGM. The aforementioned conjecture was finally settled for strictly convex potentials by Armstrong and Wu \cite{AW}, who also improved Funaki and Spohn's result on the regularity of the surface tension to $C^2$.

The strict convexity of $V$ appears as the collective assumption underlying these results, as they all employ methods that do not admit a straightforward analogue in the case of non-convex potentials. They variously rely on the Brascamp-Lieb inequality, the Helffer-Sj\"ostrand random walk representation (\cite{DGI00,GOS}), and exploiting the associated Langevin dynamics and homogenisation (\cite{FS97,DGI00,AW}). Much more is known about these models, and we refer to reviews by Funaki \cite{Fun05}, Sheffield \cite{She05} and Velenik \cite{Velenik} for a detailed overview and further references.

The field for non-convex potential $V$ is less completely surveyed. Biskup and Koteck\'{y} \cite{BK07} and Biskup and Spohn \cite{BS11} studied a class of non-convex $V$ that are a log-mixture of centred Gaussians. They showed that for sufficiently non-convex potentials, the (translation-invariant) infinite-volume Gibbs states are not unique \cite{BK07}, but the corresponding scaling limit remains in the Gaussian universality class \cite{BS11}. Their results are limited to the specific case of zero tilt and (for \cite{BK07}) dimension two. For the high-temperature (i.e., small $ \beta $) regime, Cotar, Deuschel and M\"uller  \cite{CDM} proved strict convexity of the surface tension, and Cotar and Deuschel \cite{CD09} showed Gaussian decay of correlations and a Gaussian scaling limit for a class of non-convex potentials which are small perturbations of a Gaussian measure.

The low temperature (i.e., large $ \beta$) and small tilt regime has been extensively studied by Adams, Koteck\'{y} and M\"uller in \cite{AKM16} for scalar-valued fields and by Adams, Buchholz, Koteck\'{y} and M\"uller for vector-valued fields in \cite{ABKM}. Both of these works utilise an extension and adaptation of renormalisation group ideas introduced by Brydges \textit{et al.}, see \cite{BY90, Bry09,BBS18}. In \cite{AKM16}, the authors show that the surface tension as a function of the tilt is strictly convex for large enough inverse temperature $ \beta $ and small tilts $ u\in\R^d $. Their analysis and results are extended to vector-valued fields in \cite{ABKM}, along with showing that the scaling limit for macroscopic functions on the torus is Gaussian. In earlier work, based on \cite{AKM16}, Hilger \cite{Hil16} had already proved that the scaling limit with respect to the finite volume measures on the torus (referred to as the \emph{torus scaling limit} below) is indeed a continuum GFF with some strictly positive symmetric covariance matrix.

Our goal is to identify the scaling limit of the infinite volume GGMs in the regime covered by the results of \cite{ABKM} and, in particular, to show that it is a GFF whose covariance matrix is given by the Hessian
\begin{equation}
\Hscr_\sigma(\beta,u)=D^2_u\sigma_\beta(u)
\end{equation}
of surface tension. This offers the first verification of the behaviour conjectured in \cite{GOS} outside the realm of models governed by strictly convex potentials. Our identification of the limiting covariance also applies to the torus scaling limit considered previously in \cite{Hil16} and \cite{ABKM}, and so we include it in our results with the proviso that only the particular form of the covariance matrix is novel. The infinite volume scaling limit corresponds to recent work \cite{BPR22b} on the integer-valued Gaussian measure in $d=2$ with zero tilt. Indeed, our method is inspired by the approach to the renormalisation group method in that work. The torus scaling limit likewise corresponds to \cite{BPR22a} for the integer-valued model.

\section{Setting and results} \label{S:2}

We begin with a description of the models we propose to study. Our aim is to provide a self-contained overview of the relevant definitions and we introduce our own notation. However, since our work builds on the renormalisation group method developed in \cite{AKM16} and \cite{ABKM}, we clearly situate our approach in the context of those works. We take as our starting point the general setup considered in \cite{ABKM}, subject to two key restrictions. First, we specialise to the case of scalar-valued fields and, secondly, we restrict our attention to interaction potentials that depend solely on gradients, that is, nearest neighbour interactions.

We take advantage of these limitations to simplify our definitions and notation by omitting those concepts introduced in \cite{ABKM} that serve a non-trivial purpose only in the more general setting considered there. For the convenience of the reader also familiar with \cite{AKM16}, we adopt the conventions used in that paper so far as possible, while ensuring compatibility with the general formalism of \cite{ABKM}. A detailed translation of our setting into the notation of \cite{ABKM} is provided in Appendix \ref{A:A}.

\bigskip

\subsection{Random gradient fields}

Let $ L>3$ be an odd integer and fix the dimension $d\ge 2$. For any integer $N\in\N$, we consider the space
\begin{equation} \label{E:VN}
\cbV_N=\{\phi: \mathbb Z^d\to\mathbb R;\  \phi(x+k)=\phi(x)\  \forall k\in (L^N\mathbb Z)^d\}
\end{equation}
that can be identified with the set of real-valued functions on the $d$-dimensional discrete \emph{torus} $\mathbb T_N=\bigl(\mathbb Z/L^N\mathbb Z\bigr)^d$. We denote by $\abs{\cdot}$ and $\abs{\cdot}_\infty$ respectively the quotient distances induced by the usual Euclidean and maximum norms on $\R^d$, for which we use the symbols $\abs{\cdot}^{\R^d}$ and $\abs{\cdot}_\infty^{\R^d}$ where required. Throughout this article, we interchangeably regard the torus $\mathbb{T}_N$ as represented by the centred lattice cube $\La_N =\{x\in\Z^d\colon \abs{x}_\infty^{\R^d}\le \frac{1}{2}(L^N-1)\} $ of side length $ L^N $, equipped with the metric $\abs{x-y}_\infty=\inf\{\abs{x-y+k}_\infty^{\R^d}: k\in (L^N\mathbb Z)^d\}$. $\cbV_N$ is an $L^{dN}$-dimensional Hilbert space when endowed with the standard scalar product
\begin{equation} \label{E:XNscalar}
(\phi,\psi)=\sum_{x\in \mathbb T_N}\phi(x)\psi(x).
\end{equation}

We focus on the subspace
\begin{equation}
\label{E:XN}
\cbX_N=\{\phi\in \cbV_N: \sum_{x\in \mathbb T_N }\phi(x)=0\}
\end{equation}
of functions (or fields) whose sum over the torus is zero. We use the symbol $ \l_N $ to denote the $(L^{dN}-1)$-dimensional Hausdorff measure on $ \cbX_N $ and we consider $ \cbX_N $ as measure space equipped with the $\sigma$-algebra $\boldsymbol{\Bcal_{\cbX_N}}$ induced by the Borel $\sigma$-algebra with respect to the product topology. We write $\Mcal_1(\cbX_N)=\Mcal_1(\cbX_N,\boldsymbol{\Bcal_{\cbX_N}})$ for the set of probability measures on $\cbX_N $, referring to elements in $\Mcal_1(\cbX_N)$ as \emph{random gradient fields}.

In this article we study a class of random gradient fields given (as Gibbs measures) in terms of a Hamiltonian function reflecting nearest-neighbour (or gradient) interaction. For a precise definition, we first introduce the \emph{discrete derivatives}
\begin{equation}
\label{E:nablai}
\nabla_i\phi(x)=\phi(x+{\rm e}_i)-\phi(x),\ \nabla_i^*\phi(x)=\phi(x-{\rm e}_i)-\phi(x)
\end{equation}
on $\cbV_N$. Here, ${\rm e}_i, i=1,\dots,d$, are unit coordinate vectors in $\R^d$. We write $\nabla\phi(x)\in\R^d$ for the vector formed by the discrete derivatives of $\phi$ evaluated at the site $x\in\mathbb{T}_N$.

Next, we introduce the class of interactions that we want to consider in terms of a \emph{potential} function $V:\R^d\to\R$. To be able to discuss random fields with a tilt  $u=(u_1\dots,u_d)\in\R^d$,  we use the method proposed by Funaki and Spohn in \cite{FS97}, who enforce the tilt on a measure defined on the torus space $\cbX_N$ by replacing the gradient $\nabla\phi(x)$ defined above by $\nabla\phi(x)+u$ as argument of the potential $V$. Accordingly, we define a family of Gibbs distributions by
\begin{equation}
\label{E:muNu}
\g_{N,\beta}^{\ssup{u}}(\d\phi)=\frac{1}{Z_{N,\beta}(u)} \exp\bigl(-\beta H_{N}^{\ssup{u}}(\phi)\bigr)\l_N(\d\phi),
\end{equation}
where 
\begin{equation}
\label{E:HNu}
H_{N}^{\ssup{u}}(\phi) =  \sum_{x\in \mathbb T_N} V(\nabla\phi(x)+u)
\end{equation}
is the Hamiltonian at tilt $u\in\R^d$ and $Z_{N,\beta}(u)$ is the normalising \emph{partition function}
\begin{equation}
\label{E:ZNu}
Z_{N,\beta}(u)=\int_{\cbX_N} \exp\bigl(-\beta H_{N}^{\ssup{u}}(\phi)\bigr)\l_N(\d\phi).
\end{equation}

An important role is played by the free energy or surface tension of the model, which is defined as the infinite volume limit
\begin{equation} \label{sigmau}
\sigma_\beta(u)=\lim_{N\to\infty} \sigma_{N,\beta}(u)=\lim_{N\to\infty}-\frac{1}{\beta L^{dN}}\log Z_{N,\beta}(u).
\end{equation}
The main application of the technical results of \cite{AKM16} and \cite{ABKM} was to establish, for large $\beta$, small tilts $u$ and suitable potentials $V$, the regularity and strict convexity of the surface tension $\sigma_\beta(u)$. For any regime of the parameters in which $\sigma_\beta$ is (at least) $C^2$ as a function of tilt, we write
\begin{equation} \label{Hessian}
\mathscr{H}_{\sigma}(\beta,u)=D^2_u \sigma_\beta(u)
\end{equation}
for the Hessian matrix of surface tension at inverse temperature $\beta$ and tilt $u$. In view of our results, we extend this notation to the case of subsequential limits.

\subsection{Scaling limits}

The principal object of this article is to give a characterisation in terms of the infinite volume surface tension of the scaling limit of certain models of the type described above. We consider two different regimes: the scaling limit of the finite volume measures $\gamma_{N,\beta}^{\ssup{u}}$ introduced above and the scaling limit with respect to an infinite volume measure. The former is referred to as the \emph{torus scaling limit} throughout this article and corresponds to the scaling limit studied in \cite{ABKM}. The second object is called the \emph{infinite volume scaling limit} and described further below.

We identify the distribution of these scaling limits using the Laplace transform of the relevant probability measures defined with the help of suitable test functions. For the torus scaling limit the appropriate class of test function consists of smooth functions on the continuum torus, $\mathbb{T}^d=(\R/\Z)^d$, with zero mean, which we discretise at the scale of $\mathbb{T}_N$. Given a smooth function $f\in C^\infty(\mathbb{T}^d,\R)$ with $\int_{\mathbb{T}^d}f=0$, we thus introduce the scaled function $f_N(x)=L^{-N\frac{d+2}{2}}f(L^{-N}x)$ for $x\in\mathbb{T}_N$ and our goal is to study the expectation
\begin{equation}
\mathbb{E}_{\gamma_{N,\beta}^{\ssup{u}}}\big[\ex^{(f_N,\phi)}\big]
\end{equation}
in the $N\to\infty$ limit.

We show in Appendix \ref{A:B} that the sequence of measures $\big(\gamma_{N,\beta}^{\ssup{u}}\big)_{N\in\N}$ is tight and hence converges weakly along an appropriate subsequence (modulo constant fields or as a gradient vector field). The infinite volume scaling limit is taken with respect to any such weak limit, denoted by $\gamma_{\infty,\beta}^{\ssup{u}}$ in the following. The corresponding space of test functions comprises smooth, compactly supported functions on the continuum, $f\in C_c^\infty(\R^d)$, such that $\int_{\R^d} f=0$. It is convenient to express this scaling limit with respect to the gradients and so we choose a test function $f=\Delta_{\R^d} g$. We use here the symbol $\Delta_{\R^d}=\sum_{i=1}^d \partial_i^2$ for the Laplacian on $\R^d$ to distinguish it from discrete operators that appear in the sequel. The scaled function at lattice spacing $\eps>0$ is given as $g_\eps(x)=\eps^{\frac{d}{2}-1}g(\eps x)$ for $x\in\Z^d$, and we are interested in the limit
\begin{equation}
\lim_{\eps\to 0} \; \mathbb{E}_{\gamma_{\infty,\beta}^{\ssup{u}}}\Big[\exp\Big({\sum_{x\in\Z^d}\nabla\phi(x)\cdot \nabla g_\eps(x)}\Big)\Big].
\end{equation}
The dot product in this expression stands for the usual scalar product in $\R^d$ and we occasionally use this notation throughout. As we see in the following, it suffices to study the expectations in finite volume: our main efforts are therefore directed at understanding the Laplace transform
\begin{equation} \label{E:IVscalingfinvol}
\mathbb{E}_{\gamma_{N,\beta}^{\ssup{u}}}\Big[\exp\Big(\sum_{x\in\mathbb{T}_N}\nabla\phi(x)\cdot \nabla g_\eps(x)\Big)\Big]
\end{equation}
as $N\to\infty$, followed by $\eps\to 0$. We note that, unlike the torus scaling limit, the infinite volume scaling limit has not been previously studied in this setting. Accordingly, the existence of these limits itself is a novel part of our results.

\subsection{Main results}

To achieve our characterisation of the scaling limits introduced in the preceding subsection, we impose a number of requirements on the potential $V$ of our models. For any twice differentiable $V$, we define the associated quadratic form on $\R^d$ via
\begin{equation} \label{E:QV}
\mathcal{Q}_V(z) = D^2 V(0)(z,z).
\end{equation}
Given integers $r_0\ge 3$ and $r_1\ge 0$, as well as real numbers $\o_0\in (0,1)$ and $\o\in (0,\frac{\o_0}{8})$, we consider the following conditions:
\begin{equation} \label{E:Assumptions}
\begin{cases}
V\in C^{r_0+r_1}(\R^d); \\
\o_0\abs{z}^2\le \mathcal{Q}_V(z) \le \o_0^{-1}\abs{z}^2 \text{ for all }z\in \R^d; \\
V(z)-DV(0)(z)-V(0) \ge \o \abs{z}^2 \text{ for all }z\in \R^d;\text{ and } \\
\lim_{t\to\infty} t^{-2}\log\Big(\sup_{\abs{z}\le t}\sum_{3\le \abs{\a} \le r_0+r_1}\frac{1}{\a!}\abs{\nabla^\alpha V(z)}\Big)=0.
\end{cases}
\end{equation}
These assumptions reflect the class of models studied in \cite{ABKM}, specialised to our setting. The same requirements on the potential $V$ apply for both types of scaling limit we consider. In finite volume, we obtain the following result. The notation $B_{\delta_0}(0)=\{u\in\R^d:\abs{u}^{\R^d}<\delta_0\}$ is used here to identify the permitted range of the tilt $u$.
\begin{theorem} [Torus scaling limit] \label{T:scaling}
Suppose that $V$ satisfies the assumptions \eqref{E:Assumptions} with $r_0=r_1=3$. Then there exists an $L_0\in\N$ such that for every $L\ge L_0$ there are constants $\delta_0>0$ and $\beta_0>0$ for which the following holds. Given $\beta\ge \beta_0$ and $u\in B_{\delta_0}(0)$, there is a subsequence $(N_\ell)_{\ell\in\N}$ such that for every $f\in C^\infty(\mathbb{T}^d,\R)$ with zero mean we have
\begin{equation} \label{E:scaling}
\lim_{\ell\to\infty} \mathbb{E}_{\gamma_{N_\ell,\beta}^{\ssup{u}}}\big[\ex^{(f_{N_\ell},\phi)}\big]=\exp\Big(\frac{1}{2\beta}(f,\mathscr{C}_{\mathbb{T}^d}^{\ssup{u}}f)_{L^2(\mathbb{T}^d)}\Big),
\end{equation}
where $\mathscr{C}_{\mathbb{T}^d}^{\ssup{u}}$ stands for the inverse of the operator
\begin{equation}
\mathscr{A}_{\mathbb{T}^d}^{\ssup{u}}=-\sum_{i,j=1}^d\big(\mathscr{H}_{\sigma}(\beta,u)\big)_{i,j}\partial_i\partial_j
\end{equation}
acting on $H^1(\mathbb{T}^d,\R)$ functions with zero mean and $\mathscr{H}_{\sigma}(\beta,u)$ is understood as the Hessian of the $C^{2,1}(B_{\delta_0}(0))$ limit of the finite volume surface tension along the subsequence $(N_\ell)_{\ell\in\N}$.
\end{theorem} 

In infinite volume, we show that the scaling limit is analogous, subject to the important distinction that the domain of the covariance (and hence the limiting Gaussian free field) is functions on $\R^d$ rather than $\mathbb{T}^d$.
\begin{theorem} [Infinite volume scaling limit] \label{T:scalingIV}
Suppose that $V$ satisfies the assumptions \eqref{E:Assumptions} with $r_0=r_1=3$. Then there exists an $L_0\in\N$ such that for every $L\ge L_0$ there are constants $\delta_0>0$ and $\beta_0>0$ for which the following holds. Given $\beta\ge \beta_0$ and $u\in B_{\delta_0}(0)$, there is a subsequence $(N_\ell)_{\ell\in\N}$ such that for every $f=\Delta_{\R^d} g\in C^\infty(\R^d,\R)$ and any weak limit $\gamma_{\infty,\beta}^{\ssup{u}}$ (as gradient vector field) along a subsequence of $(N_\ell)_{\ell\in\N}$ we have
\begin{equation} \label{E:scalingIV}
\lim_{\eps \to 0} \; \mathbb{E}_{\gamma_{\infty,\beta}^{\ssup{u}}}\Big[\exp\Big({\sum_{x\in\Z^d}\nabla\phi(x)\cdot \nabla g_\eps(x)}\Big)\Big] = \exp\Big(\frac{1}{2\beta}(f,\mathscr{C}_{\R^d}^{\ssup{u}}f)_{L^2(\R^d)}\Big),
\end{equation}
where $\mathscr{C}_{\R^d}^{\ssup{u}}$ stands for the inverse of the operator
\begin{equation}
\mathscr{A}_{\R^d}^{\ssup{u}}=-\sum_{i,j=1}^d\big(\mathscr{H}_{\sigma}(\beta,u)\big)_{i,j}\partial_i\partial_j
\end{equation}
acting on $H^1(\R^d,\R)$ functions with zero mean and $\mathscr{H}_{\sigma}(\beta,u)$ is understood as the Hessian of the $C^{2,1}(B_{\delta_0}(0))$ limit of the finite volume surface tension along the subsequence $(N_\ell)_{\ell\in\N}$.
\end{theorem} 

\bigskip

\begin{remark} \label{R:Hilger}
\begin{enumerate}[(a)]

\item To keep our presentation as simple as possible we have chosen to concentrate on the case of scalar-valued fields with nearest-neighbour interaction. We are not aware of any reason why the method presented in this article could not be extended to the more general setting of \cite{ABKM}, but we have not attempted to do so. We note that, despite the restriction to nearest-neighbour potentials, expressions with higher discrete derivatives appear in our analysis: see the definition of relevant Hamiltonians in Section \ref{S4:func} and the norms introduced in Section \ref{S4:Norms}. This is an artefact of our method. The associated range constant $R$ should therefore not be confused with the range of the interaction, which is always one in this work.

\item Hilger \cite{Hil16} proved the \emph{torus scaling limit} as above for some positive matrix $\bq\in\R^{d\times d}_{\rm sym} $ as the limiting covariance matrix. 

\item The need for subsequences in our statement is due to the fact that we base our approach on the finite-volume renormalisation group flow in \cite{AKM16,ABKM}. In a similar way to Hilger in \cite{Hil20(2)}, we expect that the analysis in this article could be extended to the infinite volume flow, thereby obviating the requirement of taking limits along a subsequence.

\end{enumerate}
\end{remark}

\section{Key Ideas} \label{S:3}

As already mentioned, what we refer to as the torus scaling limit has been studied previously by the authors of \cite{ABKM} and Hilger (see Remark \ref{R:Hilger} above). Indeed, once one is in possession of the powerful renormalisation group representation of the partition function $Z_{N,\beta}(u)$ shown in \cite{ABKM}, it is a relatively simple computation to extract from it an expression for the scaling limit for macroscopic functions on the torus: see the proof of Theorem 2.10 of \cite{ABKM}. This is aided by the feature that the discretised function $f_N$ automatically belongs in the correct scale for any given $N$. When evaluating the corresponding integrals for a test function $f_\eps$ that exists at a fixed scale relative to $N$, no such simple estimate is available. Our main task is therefore to extend the renormalisation group analysis in \cite{ABKM} to deal with the presence of additional terms depending on $f_\eps$. Furthermore, we want to link the expression for the scaling limit obtained in this way to the Hessian of surface tension. To arrive at a tractable expression for the second derivatives of $\sigma_{N,\beta}$ some additional modifications to the approach in \cite{ABKM} are required. A brief outline of the main ideas behind these changes follows in the next paragraphs.

The starting point of \cite{ABKM} (and \cite{AKM16}) is the evaluation of the partition function $Z_{N,\beta}(u)$ as perturbation of a Gaussian integral. To make precise this notion, we define a function $U:\R^d\times\R^d\to\R$ by
\begin{equation} \label{E:Ufct}
U(z,u)=V(z+u)-V(u)-DV(u)(z)-\frac{1}{2}D^2 V(0)(z,z)
\end{equation}
and observe that the Hamiltonian in \eqref{E:HNu} may then be rewritten as
\begin{equation} \label{E:HNuU}
H_{N}^{\ssup{u}}(\phi)=L^{dN}V(u) + \sum_{x\in \mathbb T_N} \Big(U(\nabla\phi(x),u)+\frac{1}{2}\mathcal{Q}_V(\nabla\phi(x))\Big),
\end{equation}
where we used the fact that $\sum_{x\in\mathbb{T}_N}\nabla_i\phi(x)=0$ for all $\phi\in\cbV_N$ and all $i=1,...,d$. The quadratic form induced by $\beta\mathcal{Q}_V$ is positive-definite on $\cbX_N$ (thanks to our assumptions) and we define the associated Gaussian probability measure as
\begin{equation}
\mu_\beta(\d\phi)=\frac{1}{Z_{N,\beta}^{\mathcal{Q}_V}}\exp\Big(-\frac{\beta}{2}\sum_{x\in\mathbb{T}_N}\mathcal{Q}_V(\nabla\phi(x))\Big)\lambda_N(\d\phi)
\end{equation}
with an appropriate normalisation constant $Z_{N,\beta}^{\mathcal{Q}_V}$. Using \eqref{E:HNuU} we can thus express the partition function \eqref{E:ZNu} as
\begin{align} \label{E:ZNuU}
Z_{N,\beta}(u) &= \ex^{-\beta L^{dN}V(u)}Z_{N,\beta}^{\mathcal{Q}_V}\int_{\cbX_N}\;\ex^{-\beta\sum_{x\in\mathbb{T}_N}U(\nabla\phi,u)}\,\mu_\beta(\d\phi) \nonumber \\
&= \ex^{-\beta L^{dN}V(u)}Z_{N,\beta}^{\mathcal{Q}_V}\int_{\cbX_N}\;\ex^{-\beta\sum_{x\in\mathbb{T}_N}U(\beta^{-1/2}\nabla\phi,u)}\,\mu(\d\phi).
\end{align}
The second equality above is obtained by rescaling the $\phi$-field by $\sqrt{\beta}$ and denoting the Gaussian measure at unit temperature, $\mu_1$, simply by $\mu$.

The argument then proceeds by introducing an initial tuning parameter into the Gaussian reference measure $\mu$. Suppose that $ \bq \in\R^{d\times d}_{\rm sym} $ is a symmetric $d\times d$-matrix with bounded operator norm $\norm{\bq}<\o_0$ (more precise requirements are introduced later). We perturb the measure $\mu$ by the quadratic form induced by this small matrix $\bq$ and thereby define a family of Gaussian probability measures on $\cbX_N$: 
\begin{equation} \label{E:Muq}
\mu^{\ssup{\bq}}(\d\phi) = \frac{\exp\Big(-\frac{1}{2}\sum_{x\in\mathbb{T}_N}\nabla\phi(x)\cdot(\bQ_V-\bq)\nabla\phi(x)\Big)}{Z_{N,\beta}^{\ssup{\bq}}}\l_N(\d\phi),
\end{equation}
where $\bQ_V$ stands for the symmetric operator that generates the quadratic form $\Qcal_V$. Re-writing \eqref{E:ZNuU} in terms of these measures, we obtain
\begin{equation} \label{E:ZNuUq}
Z_{N,\beta}(u) = \ex^{-\beta L^{dN}V(u)}Z_{N,\beta}^{\ssup{\bq}}\int_{\cbX_N}\;\ex^{-\frac{1}{2}\sum_{x\in\mathbb{T}_N}\nabla\phi(x)\cdot\bq\nabla\phi(x)}\ex^{-\beta\sum_{x\in\mathbb{T}_N}U(\beta^{-1/2}\nabla\phi(x),u)}\,\mu^{\ssup{\bq}}(\d\phi).
\end{equation}
In \cite{ABKM} it is then shown, on the assumptions set out in \eqref{E:Assumptions}, that for small enough tilts $u$ a particular choice of the matrix $\bq=\bq(u)$ exists that allows the integral in \eqref{E:ZNuUq} to be expressed as one plus an error term that can be controlled uniformly in $N$ using certain norms, which we will define in detail later. An equivalent representation exists, in which the gradients of $\phi$ are regarded as a random gradient vector field on $\mathbb{T}_N$, and it is opportune to switch to it, not least because the infinite volume limit, strictly, may only exist at the level of the gradient measures.

This shift in perspective is of course not new. A defining feature of the type of model under consideration, which has been much exploited in the literature going back to \cite{FS97}, is that (up to a constant) they depend only on the distribution of the gradients. Using a suitable one-to-one correspondence, we may therefore regard $\nabla\phi(x)$ as the fundamental random variables, distributed according to a Gibbs measure that preserves the structure set out in \eqref{E:muNu} and \eqref{E:HNu}. Deferring precise definitions until later, we denote the space of gradient vector fields (i.e. vector fields on $\mathbb{T}_N$ that arise as gradients of scalar fields in $\cbV_N$) by $\cbO_N$. In terms of a suitable image measure $\nu^{\ssup{\bq}}$, the gradient representation of the partition function takes the form
\begin{equation} \label{E:ZNuUqgrad}
Z_{N,\beta}(u) = \ex^{-\beta L^{dN}V(u)}Z_{N,\beta}^{\ssup{\bq}}\int_{\cbX_N}\;\ex^{-\frac{1}{2}\sum_{x\in\mathbb{T}_N}\eta(x)\cdot\bq\eta(x)}\ex^{-\beta\sum_{x\in\mathbb{T}_N}U(\beta^{-1/2}\eta(x),u)}\,\nu^{\ssup{\bq}}(\d\phi).
\end{equation}
The initial step in evaluating the expectation \eqref{E:IVscalingfinvol} is then identical to the torus scaling limit. We complete the square and perform a change of integration variable to arrive at
\begin{equation} \label{E:IVscalecomp}
\mathbb{E}_{\gamma_{N,\beta}^{\ssup{u}}}\Big[\exp\Big(\sum_{x\in\mathbb{T}_N}\nabla\phi(x)\cdot \nabla g_\eps(x)\Big)\Big] = \ex^{\frac{1}{2\beta}(\boldsymbol{\nabla}_Ng_\eps,\mathscr{C}^{\ssup{\bq},\nabla}\boldsymbol{\nabla}_N g_\eps)}\frac{Z_{N,\beta}(u,\mathscr{C}^{\ssup{\bq},\nabla}\boldsymbol{\nabla}_N g_\eps)}{Z_{N,\beta}(u)},
\end{equation}
in which expression $\mathscr{C}^{\ssup{\bq},\nabla}$ is the covariance of the Gaussian measure $\nu^{\ssup{\bq}}$ and the exact form of the numerator $Z_{N,\beta}(u,\mathscr{C}^{\ssup{\bq},\nabla}g\boldsymbol{\nabla}_N g_\eps)$ is not material at this stage. Its essential feature is that the integrand now depends on the gradient vector field $\mathscr{C}^{\ssup{\bq},\nabla}\boldsymbol{\nabla}_N g_\eps$ through additive coupling with the integration variable $\eta$. The convergence of the discrete inner product appearing in the exponential to the appropriate continuous one is not particularly difficult. The main challenge is to control the ratio of partition functions in \eqref{E:IVscalecomp}. The corresponding task for the torus scaling limit involves the term $\mathscr{C}^{\ssup{\bq}}f_N$ (or $\boldsymbol{\nabla}_N\mathscr{C}^{\ssup{\bq}}f_N$ in the gradient vector field picture), which is appropriately small and can be evaluated as part of the final integration step in the renormalisation group approach to $Z_{N,\beta}(u)$ in \cite{ABKM}. A similar approach is precluded in the case of the infinite volume scaling limit. Instead, we adapt the programme developed in \cite{BPR22b} to address an analogous problem to our setting. 

A central feature of the renormalisation group method underlying \cite{ABKM} and the present article is the scale-by-scale evaluation of integrals such as that in \eqref{E:ZNuUqgrad}: see section \ref{S:4} below and, more generally, \cite{ABKM} and the further literature cited there. We extend the approach in \cite{ABKM} by a scale decomposition of the additive term $\mathscr{C}^{\ssup{\bq},\nabla}\boldsymbol{\nabla}_N g_\eps$, called \emph{external fields} below, and a sequence of auxiliary coordinates that track their impact on the partition function. Instead of performing a single error estimate at the final integration step, this expedient permits us to incorporate the components of $\mathscr{C}^{\ssup{\bq},\nabla}\boldsymbol{\nabla}_N g_\eps$ into the renormalisation group step corresponding to the relevant scale, and this affords us good control over the associated error. This procedure begins at an initial scale, called the \emph{smoothness scale} of $f_\eps$, determined by the parameter $\eps$. The total perturbation contributed to the numerator $Z_{N,\beta}(u,\mathscr{C}^{\ssup{\bq},\nabla}\boldsymbol{\nabla}_N g_\eps)$ in \eqref{E:IVscalecomp} by these external fields is then shown to be of the order of the first one at the smoothness scale of $f_\eps$, which itself vanishes in the $\eps\to 0$ limit. The ratio in \eqref{E:IVscalecomp} therefore converges to one and the infinite volume scaling limit follows. A detailed outline of the method follows in section \ref{S:6}.

The identification of the covariance of the scaling limit with the Hessian of surface tension has to overcome the difficulty that the dependence of the tuning parameter $\bq(u)$ and other elements in the renormalisation group representation of \eqref{E:ZNuUqgrad} on the tilt $u$ is very implicit. To obtain an expression for the second derivatives of the finite volume surface tension $\sigma_{N,\beta}(u)$, it is convenient to compute a second order expansion in the tilt dependence, 
\begin{equation} \label{E:D2s}
D^2_u \sigma_{N,\beta}(u)= D^2_{\dot{u}}\sigma_{N,\beta}(u+\dot{u})\rvert_{\dot{u}=0},
\end{equation}
and rely on the uniform bounds provided by the renormalisation analysis of the partition function to control the derivatives of the perturbation $\sigma_{N,\beta}(u+\dot{u})$. For $\dot{u}$ arbitrarily small, we expect heuristically that the Gaussian partition function $Z_{N,\beta}^{\ssup{\bq(u)}}$ tilted by $\dot{u}$ is a good second order approximation of $Z_{N,\beta}(u+\dot{u})$ and that the integral in \eqref{E:ZNuUqgrad} only contributes higher order corrections, at least in the large $N$ limit. We thus write the partition function at tilt $(u+\dot{u})$ as
\begin{align} \label{E:ZNuudot}
Z_{N,\beta} & (u+\dot{u}) = Z_{N,\beta}^{\ssup{\bq(u)}} \exp\Big( -\beta L^{dN}\big( V(u) +DV(u)(\dot{u}) + \frac{1}{2}\Qcal_V(\dot{u})-\frac{1}{2}\dot{u}\cdot\bq(u)\dot{u}\big)\Big)\times \nonumber \\
&\times \int_{\cbO_N}\;\ex^{-\frac{1}{2}\sum_{x\in\mathbb{T}_N}(\eta(x)+\beta^{1/2}\dot{u})\cdot \bq(u)(\eta(x)+\beta^{1/2}\dot{u})}\ex^{-\beta\sum_{x\in\mathbb{T}_N}U(\beta^{-1/2}(\eta+\beta^{1/2}\dot{u}),u)}\,\nu^{\ssup{\bq(u)}}(\d\eta),
\end{align}
where we emphasise that the matrix $\bq(u)$ (and hence the measure $\mu^{\ssup{\bq(u)}}$) depend on $u$ but not on $\dot{u}$. We want to regard the term $\beta^{1/2}\dot{u}$ that appears here as a translation to the gradients $\eta(x)$ as a constant vector field, to which $\eta$ is coupled similar to the argument sketched above for the scaling limit. This means that the integrand of the perturbation integral in \eqref{E:ZNuudot}, which we abbreviate to $\Zcal_{N,\beta,u}(\dot{u})$, cannot be taken as a functional acting only on elements of $\cbO_N$ due to the presence of the constant field $\dot{u}$. We therefore show that the arguments in \cite{ABKM} may be extended to functionals of arbitrary vector fields, and this is indeed the main technical obstacle to overcome in pursuing this strategy. The main substantive changes to the analysis are: (i) a number of technical adaptations due to the gradient field representation; (ii) confirmation that the integration map $\bR_k^{\ssup{\bq}}$ introduced in Section \ref{S4:FR} below has the required regularity properties despite the larger domain of the class of field functionals we allow; and (iii) extension of the large field regulators to arbitrary vector fields (see Appendix \ref{A:C}).

Following this reasoning, we obtain a representation of $\Zcal_{N,\beta,u}(\dot{u})$ in terms of objects that appear in the renormalisation group flow of the partition function:
\begin{equation} \label{E:Zudotrep}
\Zcal_{N,\beta,u}(\dot{u}) = \int_{\cbO_N}\;\Big(1+K_N(\Lambda_N,\eta+\beta^{1/2}\dot{u})\Big)\,\nu_{N+1}^{\ssup{\bq(u)}}(\d\eta).
\end{equation}
Here, the Gaussian probability measure $\nu_{N+1}^{\ssup{\bq(u)}}$ is the final component of a finite range decomposition of the measure $\nu^{\ssup{\bq(u)}}$ introduced above and the vector field functional $K_N$ is small for large $N$ in a suitable norm, which in particular affords control over its derivatives. Substituting this expression into \eqref{E:D2s} and taking advantage of the bounds provided by the renormalisation group analysis then yields that
\begin{equation} \label{E:Hesslim}
D^2_u \sigma_{N,\beta}(u)=\bQ_V-\bq_N(u)+O(\vartheta^N)
\end{equation}
for some $0<\vartheta<1$ as $N\to\infty$, where we make the $N$-dependence of $\bq(u)$ explicit. This leads to an expression for the Hessian of $\sigma_\beta(u)$ in infinite volume, having regard to the uniform bounds on the third derivatives of $\sigma_{N,\beta}(u)$ implied by the results of \cite{ABKM}.

The remaining sections of this article are organised as follows. Section \ref{S:4} contains a detailed description of the vector field setting we consider, along with precise definitions for the renormalisation group flow adapted to functionals on arbitrary vector fields. In section \ref{S:5} we then provide a step-by-step procedure by which the proofs in \cite{ABKM} can be adapted to yield analogous results in our setting. Section \ref{S:6} sets out our proof strategy for the infinite volume scaling limit and the relevant definitions for the extended renormalisation group flow. This strategy is then implemented across sections \ref{S:7} to \ref{S:9}. Finally, section \ref{S:10} concludes with a proof of the main results advertised in section \ref{S:2} above. For the benefit of readers familiar with the approach in \cite{ABKM}, we end this section with a summary of the main differences to the renormalisation analysis of the partition function $Z_{N,\beta}(u)$.

A first group of changes are necessitated by the switch to gradients as the fundamental variables. We use a one-to-one map between $\cbX_N$ and $\cbO_N$ to translate key objects, such as the Gaussian reference measures $\mu^{\ssup{\bq}}$ and their finite range decompositions, into the gradient setting. Concepts that depend on discrete derivatives (e.g., the field norms) are adjusted, broadly speaking, by losing one degree of derivatives. We show carefully that the construction of relevant Hamiltonians can be adapted consistently and that analogous bounds hold. Some of the definitions and notation arguably are simpler, as the need for taking quotients by constant fields falls away.

An important technical difference arises in connection with the scale-by-scale integration maps. In the $\phi$-field setting, the mapping
\begin{equation} \label{E:intmap}
\phi\mapsto\int_{\cbX_N}\; F(\phi+\xi)\,\mu_k^{\ssup{\bq}}(\d\xi)
\end{equation}
is a convolution with a (smooth) Gaussian density. This immediately guarantees a number of desirable analytic properties. In our setting, the mapping \eqref{E:intmap} is defined on a space strictly larger than the domain of integration. We thus have to verify the required properties (e.g., commutation with differentiation) explicitly. This turns out to be fairly straightforward, but it emphasises the need to control our functionals by integrable weights.

The weight functions just alluded to constitute the third main group of changes we introduce. These functions (also referred to as large field regulators) control the permitted growth of the field functionals in the renormalisation group flow. Very roughly speaking, these functions (called here $w_k$ - we ignore the dependence on the block decomposition of $\mathbb{T}_N$) need to achieve two aims. First, they should be consistent under the integration map \eqref{E:intmap}: $w_{k+1}\approx w_k * \mu_{k+1}^{\ssup{\bq}}$ approximately. Secondly, the weights also need to control the field norms and certain other field functionals. By adapting the functions constructed in \cite{ABKM} we effectively obtain regulators on $\cbO_N$, leaving us with the task of extending them to arbitrary vector fields while ensuring these two main aims are met. The second aim, in particular, means that regulators that act only on the curl-free component of vector fields are insufficient. We set out our approach - using the orthogonal complement of $\cbO_N$ - in Appendix \ref{A:C}.

\section{Renormalisation group analysis of the vector field setting} \label{S:4}

The reasoning outlined in the previous section guides us to the perspective in which fields of gradients are the fundamental random objects. We begin this section with a precise description of this setting. Thereafter, we introduce the main ingredients of the renormalisation group analysis adapted to our needs. This section covers what is later referred to as the \emph{bulk} renormalisation group flow: its purpose is to recover the representation of the partition function in \cite{ABKM} in terms of a sequence of functionals with suitably expanded domains. The \emph{extended} renormalisation group flow and other definitions required for the infinite volume scaling limit are introduced in section \ref{S:6}. Our aim here is give concise definitions; a detailed motivation for the method may be found in \cite{ABKM} and the works cited therein.

\subsection{Vector field representation}

Denote by $\mathfrak{V}_N=\mathfrak{V}(\mathbb{T}_N)\cong\big(\R^d\big)^{\mathbb{T}_N}$ the space of $d$-dimensional vector fields on the torus, to which we extend finite difference operations component-wise. Let $\boldsymbol{\nabla}_N:\cbV_N\to\mathfrak{V}_N$ be the mapping defined by $\boldsymbol{\nabla}_N(\phi)(x)=\nabla\phi(x)$ and designate by $\cbO_N$ the image of $\cbX_N$ under $\boldsymbol{\nabla}_N$. We equip $\cbO_N$ with the $\s$-algebra $\boldsymbol{\Bcal_{\cbO_N}}$ induced by the appropriate Borel $\s$-algebra on $\mathfrak{V}_N$. Note that $\boldsymbol{\nabla}_N$ is a homeomorphism from $\cbX_N$ onto $\cbO_N$ and hence measurable. We henceforth regard $\boldsymbol{\nabla}_N$ exclusively as a map between these spaces. It follows that we can consider the push-forward measures $\o_N= \l_N\circ\boldsymbol{\nabla}_N^{-1}$ and $\g_{N,\beta}^{\ssup{u},\nabla}=\g_{N,\beta}^{\ssup{u}}\circ\boldsymbol{\nabla}_N^{-1}$. Manifestly, $\g_{N,\beta}^{\ssup{u},\nabla}\in\Mcal_1(\cbO_N)=\Mcal_1(\cbO_N,\boldsymbol{\Bcal_{\cbO_N}})$ and we refer to probability measures on $\cbO_N$ as \emph{random gradient vector fields}.

It is convenient to re-use some of the notation introduced previously. To avoid any confusion, we shall use the symbols $\phi$ and $\xi$ to designate elements of $\cbX_N$ and $\eta$ and $\varsigma$ to mean vector fields in $\cbO_N$. We will reserve the symbol $v$ to refer to a generic vector field in $\mathfrak{V}_N$. Accordingly, we will write $H_{N}^{\ssup{u}}(\eta)$ for the expression
\begin{equation} \label{E:HNuUeta}
L^{Nd}V(u) + \sum_{x\in \mathbb T_N} \Big(U(\eta(x),u)+\frac{1}{2}\mathcal{Q}_V(\eta(x))\Big)
\end{equation}
and we again use round $\big(\cdot,\cdot\big)$ brackets for the inner product on $\cbO_N$ (inherited from $\mathfrak{V}_N$):
\begin{equation} \label{E:innerproduct}
\big(\eta,\varsigma\big)=\sum_{x\in\mathbb{T}_N}\eta(x)\cdot\varsigma(x)=\sum_{x\in\mathbb{T}_N}\sum_{i=1}^d \eta_i(x)\varsigma_i(x).
\end{equation}
Where it adds clarity, we emphasise the relevant inner product space in our notation, e.g., $\big(\cdot,\cdot\big)_{\cbO_N}$. With these definitions in place, we can thus represent the random gradient vector field $\g_{N,\beta}^{\ssup{u},\nabla}$ in a form that mirrors the structure of \eqref{E:muNu} as anticipated:
\begin{equation} \label{E:NuNu}
\g_{N,\beta}^{\ssup{u},\nabla}(\d\eta)=\frac{1}{Z_{N,\beta}(u)} \exp\bigl(-\beta H_{N}^{\ssup{u}}(\eta)\bigr)\o_N(\d\eta).
\end{equation}
From now on, we always consider the matrix $\bq$ as an element in the space
\begin{equation} \label{E:Bkappa}
B_\kappa(0)=\{\bq\in\R_{\text{sym}}^{d\times d}:\norm{\bq}<\kappa\}
\end{equation}
for some constant $\kappa$ such that $\kappa\le \frac{\o_0}{2}$, where $\norm{\bq}$ represents the standard operator norm. Recall that $\o_0$ appears in the bounds on the quadratic form $\Qcal_V$ by virtue of the assumptions \eqref{E:Assumptions} imposed on the potential $V$. The Gaussian probability measures $\nu^{\ssup{\bq}}$ on $\cbO_N$ are introduced as push-forwards of the corresponding measures $\mu^{\ssup{\bq}}$ on $\cbX_N$. The representation
\begin{equation} \label{E:Nuq}
\nu^{\ssup{\bq}}(\d\eta) = \frac{\exp\Big(-\frac{1}{2}\sum_{x\in\mathbb{T}_N}\eta(x)\cdot (\bQ_V-\bq)\eta(x)\Big)}{Z_{N,\beta}^{\ssup{\bq}}}\o_N(\d\eta)
\end{equation}
is immediate. We then expand the perturbative integrand with the help of the $\R^d\to\R$ function
\begin{equation} \label{E:KubV}
\mathcal{K}_{u,\beta,V}(z)=\exp\big(-\beta U(\beta^{-1/2}z,u)\big)-1
\end{equation}
to obtain the vector field form of the partition function:
\begin{equation} \label{E:ZNuKgrad}
Z_{N,\beta}(u) = \ex^{-\beta L^{dN}V(u)}Z_{N,\beta}^{\ssup{\bq}}\int_{\cbO_N}\;\ex^{-\frac{1}{2}\sum_{x\in\mathbb{T}_N}\eta(x)\cdot\bq\eta(x)}\sum_{X\subset\mathbb{T}_N}\prod_{x\in X}\mathcal{K}_{u,\beta,V}(\eta(x))\,\nu^{\ssup{\bq}}(\d\eta).
\end{equation}
Since the pre-factors are explicitly computable, the main object of interest is the integral in \eqref{E:ZNuKgrad}. The analysis in \cite{ABKM} considers abstract perturbation functionals and then makes the connection with the potential $V$ in a separate step. In the same spirit, our approach thus focuses on deriving a renormalisation group representation of integrals of the type
\begin{equation} \label{E:pertintgrad}
\int_{\cbO_N}\;\ex^{-\frac{1}{2}\sum_{x\in\mathbb{T}_N}\eta(x)\cdot\bq\eta(x)}\sum_{X\subset\mathbb{T}_N}\prod_{x\in X}\mathcal{K}(\eta(x))\,\nu^{\ssup{\bq}}(\d\eta)
\end{equation}
for $C^{r_0}(\R^d)$ functions $\Kcal$ such that the norm
\begin{equation} \label{E:KNorm}
\norm{\Kcal}_{\zeta,\Qcal_V}=\sup_{z\in\R^d}\sum_{\abs{\a}\le r_o}\frac{1}{\a!}\abs {\partial^\a \Kcal(z)}e^{-\frac{1}{2}(1-\zeta)\Qcal_V(z)}
\end{equation}
is finite for some parameter $\zeta\in (0,1)$.

\subsection{Finite range decomposition} \label{S4:FR}

We follow the approach in \cite{ABKM} that exploits the existence of a so-called finite range decomposition for the family of measures $\mu^{\ssup{\bq}}$ with $\bq\in B_\kappa(0)$. This guarantees that the Gaussian field $\phi$, distributed according to $\mu^{\ssup{\bq}}$, is equal in distribution to the sum of independent Gaussian fields $\xi_1,...,\xi_{N+1}$ (distributed according to $\mu_1^{\ssup{\bq}},...,\mu_{N+1}^{\ssup{\bq}}$) with the property that the covariance of $\xi_k$ is given by a kernel satisfying
\begin{equation} \label{E:FRprop}
\Ccal_{\bq,k}(x)=-c_k\text{ for }\abs{x}_\infty\ge\frac{L^k}{2},
\end{equation}
where the $c_k\ge 0$ are constants independent of $\bq$. Furthermore, various uniform-in-$N$ bounds are available for the kernels $\Ccal_{\bq,k}$ as well as their Fourier transforms and derivatives. See Theorem 6.1 of \cite{ABKM}, which references the finite range decomposition constructed in \cite{B18}, for a precise statement. By linearity of $\boldsymbol{\nabla}_N$, this decomposition induces a corresponding decomposition of the gradient vector fields:
\begin{equation}
\nabla\phi=\eta\overset{D}{=}\varsigma_1+\varsigma_2+\cdots+\varsigma_{N+1}
\end{equation}
for Gaussian random gradient vector fields $\varsigma_k\sim\nu_k^{\ssup{\bq}}=\mu_k^{\ssup{\bq}}\circ\boldsymbol{\nabla}_N^{-1}$. The features of the measures $\mu_k^{\ssup{\bq}}$ easily translate into the vector field setting. In particular, as a direct consequence of \eqref{E:FRprop}, sufficiently distant gradients are independent,
\begin{equation}
\mathbb{E}_{\nu_k^{\ssup{\bq}}}[\varsigma_{k,i}(x)\varsigma_{k,j}(y)]=0,\;\forall x,y\in\mathbb{T}_N\text{ s.t. }\abs{x-y}_\infty\ge\frac{L^k}{2}+1,\,i,j=1,\ldots,d;
\end{equation}
and any expectation with respect to $\nu^{\ssup{\bq}}$ (such as appears in \eqref{E:ZNuKgrad}) may be evaluated as successive convolutions with the measures $\nu_k^{\ssup{\bq}}$:
\begin{equation} \label{E:successint}
\int_{\cbO_N}\;F(\eta)\,\nu^{\ssup{\bq}}(\d\eta)=\int_{\cbO_N\times\cdots\times\cbO_N}\;F\Big(\sum_{k=1}^{N+1}\varsigma_k\Big)\,\nu_1^{\ssup{\bq}}(\d\varsigma_1)\cdots\nu_{N+1}^{\ssup{\bq}}(\d\varsigma_{N+1})
\end{equation}
for any integrable functional $F$. We adopt the notation $\bR_k^{\ssup{\bq}}$ for the mapping on (integrable) functionals given by
\begin{equation} \label{E:Rkq}
\bR_k^{\ssup{\bq}}[F](\eta)=\int_{\cbO_N}\;F(\eta+\varsigma)\,\nu_k^{\ssup{\bq}}(\d\varsigma).
\end{equation}

\subsection{Block partitioning of $\mathbb{T}_N$ and geometric definitions} \label{S4:geo}

Consider any subset $X$ of the torus $\mathbb{T}_N$ and a functional $F$ that depends only on the values of $\phi$ in $X$. If $F$ is a function of gradients only (cf. the concept of \emph{shift-invariant} in \cite{ABKM}), there is a functional $F^\nabla$ such that $F(\phi)=F^\nabla(\boldsymbol{\nabla}_N\phi)$ and $F^\nabla$ only depends on gradients in $Y\subset X$. We ensure to keep the dependence of the gradients exactly as for the field itself. This straightforward inclusion of domains means that we can adopt the partitioning of $\mathbb{T}_N$ into $k$-blocks and associated geometric concepts from Chapter 6.2 of \cite{ABKM} for our purposes. We are not interested in exploiting situations where $Y\subsetneq X$ and so we make no changes to the definitions. We state the key concepts here in the interest of completeness.

We define a dimension-dependent constant $M$ by $M=2\floor{d/2}+3$ and, in our setting, the range constant $R$ takes the same value $R=M$. For every $1\le k \le N$, we pave $\Lambda_N$ with disjoint translates of the block $[-(L^k-1)/2,(L^k-1)/2]^d$ and denote the collection of such $k$-blocks by $\Bcal_k$. We set $\Bcal_0=\Lambda_N$. The following definitions and notation apply throughout this article:
\begin{enumerate}
\item A union of $k$-blocks is a $k$\emph{-polymer} and the set of all $k$-polymers is $\Pcal_k$. We write $\Bcal_k(X)$ for the set of $k$-blocks contained in $X\in\Pcal_k$ and $\abs{X}_k$ for $\abs{\Bcal_k(X)}$. Similarly, for $X\in\Pcal_k$, the set of $k$-polymers $Y$ such that $Y\subset X$ is denoted by $\Pcal_k(X)$. 
\item A subset $X\subset\Lambda_N$ is \emph{connected} if it is graph connected in $(\Lambda_N,E_\infty)$, where $E_\infty=\{\{x,y\}\in\mathscr{P}(\Lambda_N):\abs{x-y}_\infty=1\}$.
\item Two non-empty sets $X,Y\subset\Lambda_N$ are \emph{strictly disjoint} if $\abs{x-y}_\infty>1$ holds for all $x\in X$ and $y\in Y$.
\item The set of connected $k$-polymers is denoted by $\Pcal_k^c$ and $\Ccal_k(X)$ is the set of connected components of a polymer $X\in\Pcal_k$.
\item $\Scal_k\subset\Pcal_k^c$ is the set of \emph{small} connected polymers $X$ with $\abs{X}_k\le 2^d$. Polymers in $\Pcal_k^c\setminus\Scal_k$ are called \emph{large}.
\item The \emph{closure} $\overline{X}$ of a $k$-polymer $X$ is the smallest $(k+1)$-polymer containing $X$.
\end{enumerate}
We further introduce two notions of neighbourhoods of elements of $\Bcal_k$ and $\Pcal_k$:
\begin{align}
X^* &= \begin{cases}
X+[-R,+R]^d & \text{ for } X\in\Pcal_0 \\
X+[-2^d-R,2^d+R]^d & \text{ for } X\in\Pcal_1 \\
X+[-2^d L^{k-1},2^d L^{k-1}]^d & \text{ for } X\in\Pcal_k,2\le k\le N-1 \end{cases} \\
X^+ &= \begin{cases}
X+[-R,R]^d & \text{ for } X\in\Pcal_0 \\
X+[-L^k,L^k]^d & \text{ for } X\in\Pcal_k,1\le k\le N-1. \end{cases} \nonumber
\end{align}
$X^*$ and $X^+$ are referred to as the \emph{small neighbourhood} and \emph{large neighbourhood} respectively.

There exist a series of translation-invariant maps $\tilde{\pi}_k:\Pcal_k^c\to\Pcal_{k+1}^c$ for $0\le k\le N-1$ with the property that
\begin{equation}
\tilde{\pi}_k(X) = \begin{cases}
\overline{X} & \text{ if } X\in\Pcal_k^c\setminus\Scal_k \\
B\in\Bcal_{k+1} & \text{ with } B\cap X\neq \emptyset \text{ for } X\in\Scal_k\setminus\{\emptyset\} \\
\emptyset & \text{ if } X=\emptyset. \end{cases}
\end{equation}
A construction - and hence proof of existence - is given in Chapter 6.3 of \cite{ABKM}. We define $\Pcal_k\to\Pcal_{k+1}$ maps
\begin{equation} \label{E:pimap}
\pi_k(X)=\bigcup_{Y\in\Ccal_k(X)}\tilde{\pi}_k(Y)
\end{equation}
and we suppress the subscript in the following where the scale of the $k$-polymer $X$ is clear from the context.

\subsection{Vector field functionals and relevant Hamiltonians} \label{S4:func}

As outlined in section \ref{S:3} above, the main departure from \cite{ABKM} involves the domain of the field functionals we want to consider. For our purposes, the fundamental class of functionals, denoted by $M(\Pcal_k,\mathfrak{V}_N)$, consists of maps $F: \Pcal_k\times\mathfrak{V}_N\to\R$ with the following properties:
\begin{enumerate}
\item $F(X,\cdot)$ is $C^{r_0}(\mathfrak{V}_N)$-differentiable for all $X\in\Pcal_k$;
\item $F(X,\cdot)$ is \emph{local} for all $X\in\Pcal_k$ in the sense that $F(X,v)=F(X,w)$ whenever $v(x)=w(x)$ for all $x\in X^*$; and
\item $F$ is translation invariant meaning that $F(X+a,v(\cdot-a))=F(X,v)$ for all $a\in (L^k\Z)^d, X\in\Pcal_k$ and $v\in\mathfrak{V}_N$.
\end{enumerate}
Observe that condition (1) implies, in particular, the $\boldsymbol{\Bcal_{\cbO_N}}$-measurability of the $\cbO_N\to\R$ map given by $\eta\mapsto F(X,v+\eta)$. The derived spaces $M(\Bcal_k,\mathfrak{V}_N)$, $M(\Pcal_k^c,\mathfrak{V}_N)$ and $M(\Scal_k,\mathfrak{V}_N)$ are obtained by suitably restricting the domain of the first (polymer) coordinate. We equip these spaces with the following operations, for which we introduce particular shorthand notation. Canonical inclusions $M(\Bcal_k,\mathfrak{V}_N) \to M(\Pcal_k^c,\mathfrak{V}_N)$ and $M(\Pcal_k^c,\mathfrak{V}_N) \to M(\Pcal_k,\mathfrak{V}_N)$ are given by
\begin{align}
F^X(v) &= \prod_{B\in\Bcal_k(X)}F(B,v)\text{ for }F\in M(\Bcal_k,\mathfrak{V}_N)\text{ and} \label{E:canincl1} \\
F(X,v) &= \prod_{Y\in\Ccal_k(X)}F(Y,v)\text{ for }F\in M(\Pcal_k^c,\mathfrak{V}_N); \label{E:canincl2}
\end{align}
and an associative and commutative \emph{circle product} is defined on $M(\Pcal_k,\mathfrak{V}_N)$ via the expression
\begin{equation}
F\circ G(X,v)=\sum_{Y\in\Pcal_k(X)}F(X\setminus Y,v)G(Y,v).
\end{equation}
We omit specifying the scale $k$ in these operations as it is always clear from the functionals to which they are applied. Note that the expression $\prod_{x\in X}\Kcal(v(x))$ appearing in \eqref{E:pertintgrad} (applied to an arbitrary vector field $v$) defines a functional in $M(\Pcal_k^c,\mathfrak{V}_N)$, as well as a functional in $M(\Pcal_k,\mathfrak{V}_N)$ that arises through the inclusion in \eqref{E:canincl2}. We say that an element in $M(\Pcal_k,\mathfrak{V}_N)$ \emph{factors} if $F(X\cup Y,v)=F(X,v)F(Y,v)$ whenever $X$ and $Y$ are strictly disjoint, and it is clear that a functional possesses this property if and only if it can be identified with a functional that arises in this way.

A special class of functionals is the set $M_0(\Bcal_k,\mathfrak{V}_N)\subset M(\Bcal_k,\mathfrak{V}_N)$ of \emph{relevant Hamiltonians} (or \emph{ideal Hamiltonians} in the language of \cite{AKM16}). They take the form
\begin{equation}
H(B,v)=\sum_{x\in B}\Hcal(\{x\},v),
\end{equation}
where $\Hcal(\{x\},v)$ is a linear combination of
\begin{enumerate}
\item the constant monomial $1$;
\item linear monomials $\nabla^\alpha v_i(x)$ for $1\le i\le d$, $0\le\abs{\a}\le\floor{d/2}$; and
\item quadratic monomials $v_i(x)v_j(x)$ for $1\le i,j\le d$.
\end{enumerate}
Let $\mathfrak{m}$ denote the set $\{1,\ldots,d\}\times\{\a\in\N_0^d:0\le\abs{\a}\le\floor{d/2}\}$, and for $m=(i,\a)\in\mathfrak{m}$ write $\nabla^m v(x)$ for the discrete derivative $\nabla^\a v_i(x)$. We index relevant Hamiltonians by parameters describing the space of all relevant terms $\Hcal(\{x\},v)$, that is, we identify $H\in M_0(\Bcal_k,\mathfrak{V}_N)$ with its coordinates $\l\in\R$, $\ba \in \R^{\mathfrak{m}}$ and $\bd\in\R_{\text{sym}}^{d\times d}$ given by the identity
\begin{equation} \label{E:Hparas}
\Hcal(\{x\},v)=\lambda + \sum_{m\in\mathfrak{m}} a_m \nabla^m v(x) + \frac{1}{2}\sum_{i,j=1}^d d_{i,j} v_i(x) v_j(x).
\end{equation}

\subsection{Norms} \label{S4:Norms}

A series of (semi-)norms on vector fields in $\mathfrak{V}_N$ is given by the expression
\begin{equation} \label{E:vnorm}
\abs{v}_{j,X}=\sup_{x\in X^*} \; \sup_{1\le i \le d} \; \sup_{0\le \abs{\a}\le \floor{d/2}+1}\mathfrak{w}_j(\a)^{-1}\abs{\nabla^\a v_i(x)},
\end{equation}
where $X$ is any polymer, $\mathfrak{w}_j(\a)=h_j L^{-j\frac{d+2\abs{\a}}{2}}$ and $h_j=2^j h$ for some constant $h$, to be chosen later. Given an element $g^{(r)}$ of the tensor product $\mathfrak{V}_N^{\otimes r}$ (for $r\ge 1$), the injective tensor norm induced by \eqref{E:vnorm} may be represented as
\begin{multline}
\abs{g^{(r)}}_{j,X}=\sup_{x_1,\ldots,x_r\in X^*}\;\sup_{1\le i_1,\ldots,i_r \le d}\;\sup_{0\le \abs{\a_1},\ldots,\abs{\a_r}\le \floor{d/2}+1} \\
\prod_{l=1}^r \mathfrak{w}_j(\a_l)^{-1}\abs{\nabla^{\a_1}\otimes \cdots\otimes\nabla^{\a_r} g^{(r)}\big((x_1,i_1),\ldots,(x_r,i_r)\big)}.
\end{multline}
We complement the above by taking the usual absolute value on $\mathfrak{V}_N^{\otimes 0}=\R$ and then define a (semi-)norm on the space of test functions
\begin{equation}
\Phi=\bigoplus_{r=0}^{r_0} \mathfrak{V}_N^{\otimes r}
\end{equation}
via $\abs{g}_{j,X}=\sup_{0\le r\le r_0}\abs{g^{(r)}}_{j,X}$. Let $F$ be a $C^{r_0}(\mathfrak{V}_N)$ functional and consider its $r_0$-th order Taylor polynomial at $v\in\mathfrak{V}_N$, which we denote by $\text{Tay}_v F$. Viewing the $r$-th order term, written $\text{Tay}_v^{(r)} F$, as a symmetric $r$-linear map uniquely identifies an element in the dual of $\mathfrak{V}_N^{\otimes r}$. This induces a dual pairing
\begin{equation} \label{E:polydual}
\big\langle \text{Tay}_v F,g\big\rangle_{r_0} = \sum_{r=0}^{r_0} \text{Tay}_v^{(r)} F(g^{(r)}),
\end{equation} 
where an explicit expression for the pairing at order $r$ is given by
\begin{multline} \label{E:dualform}
\text{Tay}_v^{(r)} F(g^{(r)})=\frac{1}{r!}\sum_{x_1,\ldots,x_r\in\mathbb{T}_N}\sum_{i_1,\ldots,i_r=1}^d D^r F(v)\big((x_1,i_1),\ldots,(x_r,i_r)\big) \\
\times g^{(r)}\big((x_1,i_1),\ldots,(x_r,i_r)\big).
\end{multline}
We denote the associated dual norm by $\abs{\text{Tay}_v F}_{j,X}=\sup\{\big\langle \text{Tay}_v F,g\big\rangle_{r_0}:g\in\Phi,\abs{g}_{j,X}\le 1\}$, and note that in practice we only consider functionals $F$ for which this expression is finite. We also introduce the shorthand $\abs{F}_{j,X,v}=\abs{\text{Tay}_v F}_{j,X}$.

Norms on functionals in $M(\Pcal_k,\mathfrak{V}_N)$ and $M(\Pcal_k^c,\mathfrak{V}_N)$ are constructed with the help of certain \emph{weight functions} $\bunderline{w}_{k}^X$, $\bunderline{w}_{k:k+1}^X$ and $\bunderline{W}_{k}^X$ for $X\in\Pcal_k$. Adapting the corresponding objects in \cite{ABKM} for arbitrary vector fields in $\mathfrak{V}_N$ involves some additional notation and we therefore postpone that task. We state the properties we require in section \ref{S:5} below. A detailed construction of the weight functions and proof that their relevant properties extend from those in \cite{ABKM} is set out in Appendix \ref{A:C}. Using the notation $\bunderline{W}_k^{-X}=(\bunderline{W}_k^X)^{-1}$ (and likewise for $\bunderline{w}$), we define the strong and respectively weak norms by
\begin{align}
\tnorm{F(X)}_{k,X}=\sup_{v\in\mathfrak{V}_N}\abs{F(X)}_{k,X,v}\bunderline{W}_k^{-X}(v); \\
\norm{F(X)}_{k,X}=\sup_{v\in\mathfrak{V}_N}\abs{F(X)}_{k,X,v}\bunderline{w}_k^{-X}(v); \text{ and}\\
\norm{F(X)}_{k:k+1,X}=\sup_{v\in\mathfrak{V}_N}\abs{F(X)}_{k,X,v}\bunderline{w}_{k:k+1}^{-X}(v)
\end{align}
for $F\in M(\Pcal_k,\mathfrak{V}_N)$. Two families of global weak norms, depending on a parameter $A\ge 1$, are defined for functionals $F\in M(\Pcal_k^c,\mathfrak{V}_N)$ by
\begin{equation} \label{E:globalweaknorm}
\norm{F}_k^{(A)}=\sup_{X\in\Pcal_k^c}\norm{F(X)}_{k,X}A^{\abs{X}_k}\,\text{ and }\,\norm{F}_{k:k+1}^{(A)}=\sup_{X\in\Pcal_k^c}\norm{F(X)}_{k:k+1,X}A^{\abs{X}_k}.
\end{equation}
Finally, we define a separate norm on relevant Hamiltonians. For $H\in M_0(\Bcal_k,\mathfrak{V}_N)$ given as in \eqref{E:Hparas}, we let
\begin{equation}
\norm{H}_{k,0}=L^{dk}\abs{\l}+\sum_{m=(i,\a)\in\mathfrak{m}}h_k L^{k\frac{d-2\abs{\a}}{2}}\abs{a_m}+\frac{1}{2}\sum_{i,j=1}^d h_k^2 \abs{d_{i,j}}.
\end{equation}

In the following, we are only concerned with functionals that are finite in the appropriate norm as defined above. We therefore introduce the Banach spaces
\begin{align}
&\bM(\Pcal_k^c,\mathfrak{V}_N)=\{F\in M(\Pcal_k^c,\mathfrak{V}_N):\norm{F}_k^{(A)}<\infty\}; \\
&\bM(\Bcal_k,\mathfrak{V}_N)=\{F\in M(\Bcal_k,\mathfrak{V}_N):\tnorm{F}_k<\infty\}
;\text{ and} \\
&\bM_0(\Bcal_k,\mathfrak{V}_N)=(M_0(\Bcal_k,\mathfrak{V}_N),\norm{\cdot}_{k,0}).
\end{align}
Note that there are only finitely many connected polymers, so that the space $\bM(\Pcal_k^c,\mathfrak{V}_N)$ does not depend on the choice of constant $A\ge 1$.

\subsection{Definition of the renormalisation map} \label{S4:RGmap}

We construct the renormalisation group flow as a sequence of pairs of functionals $(H_k,K_k)$ determined by a series of \emph{renormalisation maps}
\begin{equation}
\bT_k:\mathcal{U}_{k,\rho,\kappa}\subset \bM_0(\Bcal_k,\mathfrak{V}_N)\times \bM(\Pcal_k^c,\mathfrak{V}_N)\times\R_{\text{sym}}^{d\times d}\to \bM_0(\Bcal_{k+1},\mathfrak{V}_N)\times \bM(\Pcal_{k+1}^c,\mathfrak{V}_N)
\end{equation}
for $k=0,...,N-1$, depending on the initial tuning parameter $\bq\in \R_{\text{sym}}^{d\times d}$. We also write $\bT_k^{\ssup{\bq}}$ for this map where we consider $\bq$ fixed. The open subset
\begin{equation}
\mathcal{U}_{k,\rho,\kappa}=\{(H_k,K_k,\bq):\norm{H_k}_{k,0}<\rho,\norm{K_k}_k^{(A)}<\rho,\norm{\bq}<\kappa\}
\end{equation}
is determined by constraints on the norms of the arguments $H_k$ and $K_k$ that ensure the smoothness of $\bT_k$. We return to this in section \ref{S:5} below. The constraint on the matrix $\bq$ ensures that we continue to restrict our attention to small matrices contained in the space $B_\kappa(0)$.
 
Two separate mappings of field functionals appear in the definition of $\bT_k^{\ssup{\bq}}$. One is the integration map $\bR_{k+1}^{\ssup{\bq}}$ introduced in \eqref{E:Rkq} above. Observe that, since the domain of our functionals is strictly larger than $\cbO_N$, this map is not a convolution in the strict sense and we verify explicitly that it is well defined and has the analytic properties we require of it: see Lemma \ref{L:Rreg} below. The other mapping we need is the projection onto relevant Hamiltonians $\Pi_2$, defined in the following paragraphs.

Let $K\in \bM(\Pcal_k^c,\mathfrak{V}_N)$ and $B\in\Bcal_k$. Recall that the second-order Taylor polynomial of $K(B)$ at $0\in\mathfrak{V}_N$ may be written in our notation as
\begin{equation}
\text{Tay}_0^{(0)}K(B)+\text{Tay}_0^{(1)}K(B)+\text{Tay}_0^{(2)}K(B)
\end{equation}
and, as a second-order polynomial in $v_i(x)$ for $i=1,\ldots ,d$ and $x\in B^*$, uniquely identified with an element in the dual $\big(\R\oplus (\R^d)^{B^{++}}\oplus (\R^d)^{B^{++}}\otimes(\R^d)^{B^{++}}\big)'$. For this purpose, we may identify $B^{++}$ with a square block of side length $5L^k$ in $\Z^d$ since $B^{++}$ does not wrap around $\mathbb{T}_N$ for $L\ge 7$. We define subspaces of $(\R^d)^{B^{++}}$ and $(\R^d)^{B^{++}}\otimes(\R^d)^{B^{++}}$ with the help of the following basis. For $t,k\in\Z$ let the symbol $\binom{t}{k}$ stand for the usual binomial coefficient
\begin{equation}
\binom{t}{k}=\frac{t(t-1)...(t-k+1)}{k!}
\end{equation}
if $k\in\N$ and set $\binom{t}{k}=1$ for $k=0$, $\binom{t}{k}=0$ for $k\in\Z\setminus\N_0$. For $\a\in\N_0^d$ and $x\in B^{++}(\subset \Z^d)$, we define
\begin{equation}
b_{\a}(x)=\binom{x_1}{\a_1}\cdots\binom{x_d}{\a_d}
\end{equation}
and, for $m=(i,\a)\in\mathfrak{m}$,
\begin{equation}
b_m(j,x)=b_{i,\a}(j,x)=b_{\a}(x)\delta_{i,j}
\end{equation}
as an element of $(\R^d)^{B^{++}}$. We now take $\mathscr{P}_0=\R$, $\mathscr{P}_1=\text{span}\{b_{i,\a}:1\le i\le d,0\le \abs{\a}\le \floor{d/2}\}$ and $\mathscr{P}_2=\text{span}\{(b_{i,0}\otimes b_{j,0}+b_{j,0}\otimes b_{i,0}):1\le i,j\le d\}$; and let $\mathscr{P}=\mathscr{P}_0\oplus\mathscr{P}_1\oplus\mathscr{P}_2$.

Using the dual pairing introduced in \eqref{E:polydual}, the projection mapping $\Pi_2:\bM(\Pcal_k^c,\mathfrak{V}_N)\to \bM_0(\Bcal_k,\mathfrak{V}_N)$ is then defined by $\Pi_2 K(B)=H$, where $H$ is the unique relevant Hamiltonian satisfying
\begin{equation} \label{E:Pi2condition}
\big\langle H,g\big\rangle_2=\big\langle \text{Tay}_0K(B),g\big\rangle_2
\end{equation}
for all $g\in\mathscr{P}$. To verify that $H$ is well defined as an element of $M_0(\Bcal_k,\mathfrak{V}_N)$ (and hence as an element of $\bM_0(\Bcal_k,\mathfrak{V}_N)$), we state explicitly the constraints imposed by \eqref{E:Pi2condition} on the relevant constant, linear and quadratic monomials comprised in $H$. In the notation of \eqref{E:Hparas}, we clearly must have $\l=L^{-kd}K(B)(0)$ and this is translation-invariant by assumption on $K$. The quadratic term is similarly straightforward: simplifying $\frac{1}{2}(d_{i,j}+d_{j,i})=d_{i,j}$ as we insisted on $\bd$ being symmetric, we see that
\begin{equation}
d_{i,j}=L^{-dk}\sum_{x,y\in B^{++}}D^2K(B,0)\big((x,i),(y,j)\big),
\end{equation}
which is again translation-invariant as $K$ is. For the linear terms, we obtain from \eqref{E:Pi2condition} that
\begin{equation} \label{E:Pi2linear}
\sum_{0\le \abs{\a}\le \floor{d/2}} a_{i,\a}\sum_{x\in B} b_{\a'-\a}(x)=\sum_{x\in B^{++}} DK(B,0)(x,i) b_{\a'}(x)
\end{equation}
must hold for all $1\le i\le d$ and $0\le \abs{\a'}\le \floor{d/2}$. Writing $\b\le\a$ for the partial order on multi-indices (i.e. $\b_i\le \a_i, \forall i$), we note that $b_{\a'-\a}(x)=1$ for $\a=\a'$ and $b_{\a'-\a}(x)=0$ unless $\a\le\a'$. The coefficients of this system of equations can thus be arranged as a lower triangular matrix with entries $L^{dk}$ on the main diagonal. It follows that the parameters $a_m$ for $m\in\mathfrak{m}$ are uniquely determined. Moreover, we claim that the solution is invariant under translation by $y\in (L^k\Z)^d$. Indeed, with the help of the identity
\begin{equation}
b_\a(x+y)=\sum_{\b\le\a}b_\b(x)b_{\a-\b}(y),
\end{equation}
the translation by $y$ of \eqref{E:Pi2linear} may be written as
\begin{multline}
L^{dk}a_{i,\a'}+\sum_{\a<\a'}a_{i,\a}\sum_{x\in B}b_{\a'-\a}(x) + \sum_{0<\b\le \a'}b_\b(y)\Big[\sum_{a\le \a'-\b}a_{i,\a}\sum_{x\in B}b_{\a'-\b-\a}(x)\Big] = \\
= \sum_{x\in B^{++}}DK(B,0)(x,i)b_{\a'}(x) + \sum_{0<\b\le \a'}b_\b(y)\Big[\sum_{x\in B^{++}}DK(B,0)(x,i)b_{\a'-\b}(x)\Big],
\end{multline}
where we again used the translation-invariance of $K$. We recognise the terms in square brackets as lower-order equations of the untranslated system \eqref{E:Pi2linear} and so the claim follows. This completes our construction of the projection $\Pi_2$.

We now define the renormalisation map $\bT_k^{\ssup{\bq}}$ in terms of these mappings. Denote the image under this map by $(H_{k+1},K_{k+1})=\bT_k^{\ssup{\bq}}(H_k,K_k)$. First, we introduce $\widetilde{H}_k$, the integrated relevant Hamiltonian at the scale $k$, via
\begin{equation} \label{E:Htildedef}
\widetilde{H}_k(B,v)=\bR_{k+1}^{\ssup{\bq}}H_k(B,v)-\Pi_2\big(\bR_{k+1}^{\ssup{\bq}}K_k\big)(B,v).
\end{equation}
Note that $\bR_{k+1}^{\ssup{\bq}}H_k\in M_0(\Bcal_k,\mathfrak{V}_N)$ can be seen directly by integrating the monomials that generate the class of relevant Hamiltonians, since the $\nu_{k+1}^{\ssup{\bq}}$ are mean-zero Gaussian fields. Then, for $B'\in\Bcal_{k+1}$, we let
\begin{equation} \label{E:H'}
H_{k+1}(B',v)=\sum_{B\in\Bcal_k(B')}\widetilde{H}_k(B,v)
\end{equation}
be the first component of the image of $(H_k,K_k)$ under $\bT_k^{\ssup{\bq}}$. It will be convenient to designate the part of $H_
{k+1}$ stemming from the first term in \eqref{E:Htildedef} via the operator
\begin{equation} \label{E:A}
\bA_k^{\ssup{\bq}} H_k(B')=\sum_{B\in\Bcal_k(B')}\bR_{k+1}^{\ssup{\bq}}H_k(B)
\end{equation}
and the part stemming from the second term in \eqref{E:Htildedef} via the operator
\begin{equation} \label{E:B}
\bB_k^{\ssup{\bq}}K_k(B')=-\sum_{B\in\Bcal_k(B')}\Pi_2\big(\bR_{k+1}^{\ssup{\bq}}K_k\big)(B).
\end{equation}
To define the second component of $\bT_k^{\ssup{\bq}}$, we introduce additional shorthand notation. We write $I_k(B,v+\eta)=\exp(-H_k(B,v+\eta))$, $\widetilde{I}_k(B,v)=\exp(-\widetilde{H}_k(B,v))$ and define a functional $\widetilde{K}_k:\Pcal_k\times\mathfrak{V}_N\times\cbO_N$ by
\begin{equation}
\widetilde{K}_k(X,v,\eta)=\Big((1-\widetilde{I}_k(\cdot,v))\circ(I(\cdot,v+\eta)-1)\circ K_k(\cdot,v+\eta)\Big)(X).
\end{equation}
We recall the canonical inclusions \eqref{E:canincl1} and \eqref{E:canincl2}, which are used here to extend the domain of relevant functionals to arbitrary $k$-polymers. It is convenient to denote the indicator function associated with the polymer map $\pi_k$ defined in \eqref{E:pimap} by the symbol $\chi:\Pcal_k\times\Pcal_{k+1}\to\{0,1\}$:
\begin{equation}
\chi(X,U)=\1_{\pi_k(X)=U}.
\end{equation}
Given a $(k+1)$-polymer $U$, we now define the image functional $K_{k+1}$ by
\begin{equation} \label{E:K'}
K_{k+1}(U,v)=\sum_{X\in\Pcal_k}\chi(X,U)\widetilde{I}(U\setminus X,v)(\widetilde{I})^{-1}(X\setminus U,v)\int_{\cbO_N}\;\widetilde{K}(X,v,\eta)\,\nu_{k+1}^{\ssup{\bq}}(\d\eta)
\end{equation}
and we introduce the symbol $\boldsymbol{S}_k$ for this component of the renormalisation map (i.e. $K_{k+1}=\boldsymbol{S}_k(H_k,K_k,\bq)$). We show subsequently that this definition factors, hence permitting us to specialise the definition to $U\in\Pcal_{k+1}^c$, which accords with the stated co-domain of $\bT_k^{\ssup{\bq}}$. It also remains to check that this expression is well-defined as an element of $\bM(\Pcal_{k+1}^c,\mathfrak{V}_N)$. In the interest of succinctness, we state these claims as part of the properties of the renormalisation maps described and proved in the following sections. This completes the setup and required definitions for the renormalisation group flow on functionals of arbitrary vector fields.

\section{Adapting the analysis in \cite{ABKM}} \label{S:5}

In principle, the setting described in section \ref{S:4} above diverges sufficiently from that in \cite{ABKM} so that all steps in the argument require re-statement and separate proof. In practice, however, many of the changes and adaptations needed are minor and straightforward. It therefore appears to us that little of value would be contributed by reproducing the technical parts of \cite{ABKM} in full. Instead, this section is organised in a series of checks, following the logical order of the renormalisation group argument, by which our claim that essentially the same approach succeeds in our setting may be verified. We generally omit proofs that may be adapted from \cite{ABKM} without substantial changes or difficulty. Our statements in this section are accompanied by a reference to the corresponding statement of \cite{ABKM} that we are adapting to our setting and purpose.

We emphasise in this context the importance of the various constants appearing in our statements (cf. also Chapter 5 of \cite{ABKM}), which differ slightly from those in \cite{ABKM}. Although much of the following takes the form of (partially) independent claims about various properties of the objects we have defined above, an essential part of our proof is to establish that a consistent choice of constants exists so that all required properties hold simultaneously. We achieve this by keeping track of all relevant constraints as we proceed through the following sub-sections. We briefly list the \emph{fixed parameters}, which we regard as given at the outset:
\begin{enumerate}
\item The spatial dimension $d\ge 2$;
\item The inverse temperature $\b\ge 1$;
\item The number of discrete derivatives in the definition of our weight functions $M$: we take $M=2\floor{d/3}+3$;
\item The range parameter $R$: we take $R=M=2\floor{d/3}+3$;
\item The assumed smoothness parameters $r_0$ and $r_1$: we take $r_0=r_1=3$;
\item Two tuning parameters, $n$ and $\widetilde{n}$, relating to the finite range decomposition (see Theorem 6.1 of \cite{ABKM}): we take $n=4\floor{d/2}+6$ and $\widetilde{n}=5\floor{d/3}+7$;
\item The parameter $\o_0>0$ controlling the quadratic form $\Qcal_V$;
\item The parameter $\zeta\in (0,1)$ controlling the exponential weight in the norm $\norm{\cdot}_{\zeta,\Qcal_V}$; and
\item The parameter $\vartheta\in (0,2/3]$ controlling the convergence of the renormalisation flow.
\end{enumerate}
In addition to these numerical constants we of course also take the potential function $V$, satisfying the assumptions in \eqref{E:Assumptions}, and all derived objects introduced above as fixed for the purposes of our analysis. The main \emph{free parameters} are the size parameter $L\in\N$, the norm scaling parameter $h\ge 1$ and the global norm parameter $A\ge 1$. As already alluded to, the way we choose these parameters (and other constants that depend on them) diverges slightly from \cite{ABKM}. This leads to a possibly tighter constraint on the $\norm{\Kcal_{u,\beta,V}}_{\zeta,\Qcal_V}$ norm, which in turn translates into a requirement to choose $\beta$ larger and $\abs{u}^{\R^d}$ smaller than in the corresponding case in \cite{ABKM}.

A number of derived constants, depending on the fixed parameters and (sometimes) the free parameters, are used frequently and represented by a unique symbol, e.g., the integration constant $\Acal_\Pcal(L)$. Many constants of lesser importance also appear in the following statements; they are designated generically by the letter $C$ and any dependence on the free parameters is indicated explicitly when they are first introduced. Where such constants are referred to subsequently they are labelled by the equation number of their first appearance, e.g., $C_{\eqref{E:Rqbound}}$.

\subsection{Preliminaries: geometry of $\mathbb{T}_N$} \label{S5:geo}

We begin with the simple observation that all purely geometric features of the block partitioning of $\mathbb{T}_N$ in \cite{ABKM} are immediately available to us because we adopt all relevant definitions (see section \ref{S4:geo} above) without change. This means, in particular, that all counting arguments relating to the re-blocking step of the renormalisation map run identically in our setting. This in turn minimises our need to pick constants differently.

In order for the various geometric definitions to interact in the desired manner, a lower bound on the size parameter $L$ is required. This essentially stems from the existence of the range $R$ and the geometric constant $2^d$ (which counts the number of blocks that are at most $\abs{\cdot}_\infty$-distance $1$ from a corner vertex). The main relations that we require are the inclusions $X^*\subset X^+$ $\forall X\in\Pcal_k$, $X^*\subset X^+\subset Y^*$ $\forall X\in\Scal_k,Y\in\Pcal_{k+1}$ with $X\cap Y\neq\emptyset$ and $X^*\subset \pi_k(X)^*$ $\forall X\in\Pcal_k$, as well as the inequality
\begin{equation} \label{E:rangeineq}
\text{dist}_\infty(U_1^*,U_2^*)\ge \frac{L^{k+1}}{2}+1
\end{equation}
for $U_1,U_2\in\Pcal_{k+1}$ and strictly disjoint. For these properties we require $L\ge 2^{d+2}+4R$. The proof is exactly as in \cite{ABKM}.

\hspace{1cm}
\begin{ctracker}[userdefinedwidth=8cm,align=center]
\begin{equation}
L\ge 2^{d+2}+4R
\end{equation}
\end{ctracker}

\subsection{Preliminaries: finite range decomposition} \label{S5:fr}

As outlined above, we obtain a finite range decomposition of the Gaussian measure $\nu^{\ssup{\bq}}$ by pushing forward the finite range decomposition of $\mu^{\ssup{\bq}}$ used in \cite{ABKM}, the existence of which is proved in \cite{B18}. Let $\mathscr{C}:\cbX_N\to\cbX_N$ be a symmetric operator that satisfies the conditions for Theorem 6.1 of \cite{ABKM}. Writing $\Ccal_k$ ($k=1,\ldots,N+1$) for the kernels of the finite range decomposition of $\mathscr{C}$ given by that result, a simple calculation shows that the push-forward Gaussian measures correspond to kernels $\Ccal_k^{\nabla}:\cbO_N\to\R$ satisfying
\begin{equation}
\Ccal_{k,(i,j)}^{\nabla}(x)=\nabla_i\nabla_j^*\Ccal_k(x).
\end{equation} 
For the Fourier transform, this relation corresponds to the equality
\begin{equation}
\widehat{\Ccal}_{k,(i,j)}^{\nabla}(p)=q_i(p)q_j(-p)\widehat{\Ccal}_k(p),
\end{equation}
where $q_j(p)=e^{ip_j}-1$ for $1\le j\le d$. Together these identities permit analogous properties for the kernels $\Ccal_k^{\nabla}$ and hence the finite range decomposition of the associated random gradient vector field to be deduced from Theorem 6.1 of \cite{ABKM}. We de not state them here, because it is equivalent to evaluate any integral with respect to a measure in the push-forward finite range decomposition (i.e. one of the $\nu_k^{\ssup{\bq}}$) as an integral over $\cbX_N$ (i.e. with respect to $\mu_k^{\ssup{\bq}}$) and rely on the properties of the original measures directly. To illustrate this approach, we give an important application of the properties of the finite range decomposition to the derivatives of expectation with respect to the generating matrix of the covariance: see Theorem 6.2 of \cite{ABKM}. For the push-forward finite range decomposition the result (given our choice of fixed parameters) reads as follows.
\begin{theorem} [Theorem 6.2 of \cite{ABKM}] \label{T:Dqdotbound}
Let $\mathscr{C}_{\bQ,k+1}$ be a finite range decomposition as in Theorem 6.1 of \cite{ABKM} and $\mathscr{C}_{\bQ,k+1}^{\nabla}$ be the corresponding operator of the push-forward under $\boldsymbol{\nabla}_N$, denoting the associated measure on $\cbO_N$ by $\nu_{\bQ,k+1}$. Assume further that $X\subset\mathbb{T}_N$ has diameter $D=\textup{diam}_\infty(X)\ge L^k$ and $F:\mathfrak{V}_N\to\R$ is a measurable functional depending only on values in $X$. Then, for $\ell\ge 1$ and $p>1$ it holds that
\begin{equation} \label{E:Dqdotbouond}
\Big|\frac{\d^{\ell}}{\d t^{\ell}}\int_{\cbO_N}\;F(\eta)\,\nu_{\bQ+t\dot{\bQ},k+1}(\d\eta)\big\rvert_{t=0}\Big|\le C_{\ell,p}(L,w_0,d)\big((D+1)L^{-k}\big)^{\frac{d\ell}{2}}\norm{\dot{\bQ}}^{\ell}\norm{F}_{L^p(\nu_{\bQ,k+1})}.
\end{equation}
\end{theorem}
This is an immediate consequence of Theorem 6.2 of \cite{ABKM} and elementary properties of image measures. Finally, we observe that for $\bq\in B_{\kappa}(0)$ defined in \eqref{E:Bkappa} above, the family of covariances corresponding to the measures $\mu^{\ssup{\bq}}$ satisfies the assumptions of Theorem 6.1 of \cite{ABKM}. This has the consequences described above for the $\cbO_N$-measures $\nu^{\ssup{\bq}}$. 

\subsection{Step 1: Properties of the weight functions} \label{S5:wf}

The functional norms $\norm{\cdot}_{k,X}$, $\norm{\cdot}_{k:k+1,X}$ and $\tnorm{\cdot}_{k,X}$ defined in section \ref{S4:Norms} depend on certain weight functions that are in turn constructed from several families of symmetric operators on $\mathfrak{V}_N$. These weight functions are important both for various properties of the norms themselves (e.g. sub-multiplicativity) and to establish key bounds relating to the integration map $\bR_{k+1}^{\ssup{\bq}}$. A natural first step is thus to verify that the extended weight functions defined for the purposes of this article satisfy requirements analogous to those proved in \cite{ABKM}. The detailed construction of our weight functions is given in Appendix \ref{A:C}. We state in the following result all properties that are needed for the main renormalisation group analysis. This statement incorporates our assumptions on the fixed parameters mentioned previously, which we do not explicitly call out here.
\begin{theorem} [Theorem 7.1 of \cite{ABKM}] \label{T:wf}
For every $L\ge 2^{d+3}+16R$, there are constants
\[\l=\min\{\frac{\o_0\zeta'}{4},(\frac{\zeta'}{3\O}+\mathfrak{r}(1-\zeta')\o_0^{-1})^{-1},\frac{1}{4}\},\, \delta=\min\{\frac{\zeta'}{4\O\Xi_{\textup{max}}},\frac{\zeta'}{8\mu\Xi_{\textup{max}}}\},\, \rho=(1+\zeta'/4)^{1/3}-1\]
and
\[\kappa=\min\{\frac{\rho L^{-4(d+\widetilde{n})-2}}{K_1},\frac{\o_0}{2}\},\, h_0=\delta^{-1/2}\max\{\sqrt{8},\sqrt{2}c_d\},\]
where $\mu=\mu(L,\l)$, $c_d$, $\O$ and $K_1$ are constants depending (otherwise than stated) only on $V$ and the fixed parameters and determined as in Chapter 7 of \cite{ABKM}, such that the weight functions $\bunderline{w}_k^X$, $\bunderline{w}_{k:k+1}^X$ and $\bunderline{W}_k^X$ given as in Appendix \ref{A:C} are well-defined and satisfy:
\begin{enumerate}
\item For any $Y\subset X\in\Pcal_k$, $0\le k\le N$, and $v\in\mathfrak{V}_N$,
\begin{equation}
\bunderline{w}_k^Y(v)\le \bunderline{w}_k^X(v)\,\text{ and }\,\bunderline{w}_{k:k+1}^Y(v)\le \bunderline{w}_{k:k+1}^X(v);
\end{equation}
\item The estimate
\begin{equation}
\bunderline{w}_k^X(v)\le \exp\Big(\frac{(v,\bunderline{M}_k v)}{2\l}\Big)\,\text{ and }\,\bunderline{w}_{k:k+1}^X(v)\le \exp\Big(\frac{(v,\bunderline{M}_k v)}{2\l}\Big)
\end{equation}
holds for $0\le k\le N$, $X\in\Pcal_k$ and $v\in\mathfrak{V}_N$;
\item For any strictly disjoint polymers $X,Y\in\Pcal_k$, $0\le k\le N$ and $v\in\mathfrak{V}_N$,
\begin{equation}
\bunderline{w}_k^{X\cup Y}(v)=\bunderline{w}_k^X(v)\bunderline{w}_k^Y(v);
\end{equation}
\item For any polymers $X,Y\in\Pcal_k$ such that $\text{dist}_\infty(X,Y)\ge \frac{3}{4}L^{k+1}$, $0\le k\le N$ and $v\in\mathfrak{V}_N$,
\begin{equation}
\bunderline{w}_{k:k+1}^{X\cup Y}(v)=\bunderline{w}_{k:k+1}^X(v)\bunderline{w}_{k:k+1}^Y(v);
\end{equation}
\item For any disjoint polymers $X,Y\in\Pcal_k$, $0\le k\le N$ and $v\in\mathfrak{V}_N$,
\begin{equation}
\bunderline{W}_k^{X\cup Y}(v)=\bunderline{W}_k^X(v)\bunderline{W}_k^Y(v);
\end{equation}
\item For $h\ge h_0$, disjoint polymers $X,Y\in\Pcal_k$, $0\le k\le N$ and $v\in\mathfrak{V}_N$,
\begin{equation}
\bunderline{w}_k^{X\cup Y}(v)\ge \bunderline{w}_k^X(v)\bunderline{W}_k^Y(v);
\end{equation}
\item For $h\ge h_0$, $X\in\Pcal_k$, $U=\pi_k(X)\in\Pcal_{k+1}$, $0\le k\le N-1$ and $v\in\mathfrak{V}_N$,
\begin{equation}
\bunderline{w}_{k+1}^U(v)\ge \bunderline{w}_{k:k+1}^X(v)\Big(\bunderline{W}_k^{U^+}(v)\Big)^2;
\end{equation}
\item For $h\ge h_0$, $X\in\Pcal_{k+1}$, $0\le k\le N-1$ and $v\in\mathfrak{V}_N$,
\begin{equation}
\exp\Big(\frac{\abs{v}_{k+1,X}^2}{2}\Big)\bunderline{w}_{k:k+1}^X(v)\le \bunderline{w}_{k+1}^X(v);
\end{equation}
\item There is a constant $\Acal_{\Pcal}=\Acal_{\Pcal}(L)$ such that for $\bq\in B_{\kappa}(0)$, $\overline{\rho}\in [0,\rho]$, $X\in\Pcal_k$, $0\le k\le N$ and $v\in\mathfrak{V}_N$,
\begin{equation}
\Big(\int_{\cbO_N}\;\big(\bunderline{w}_k^X(v+\eta)\big)^{1+\overline{\rho}}\,\nu_{k+1}^{\ssup{\bq}}(\d\eta)\Big)^{\frac{1}{1+\overline{\rho}}}\le \Big(\frac{\Acal_{\Pcal}}{2}\Big)^{\abs{X}_k}\bunderline{w}_{k:k+1}^X(v);\text{ and}
\end{equation}
\item There is a constant $\Acal_{\Bcal}$ independent of $L$ such that for $\bq\in B_{\kappa}(0)$, $\overline{\rho}\in [0,\rho]$, $B\in\Bcal_k$, $0\le k\le N$ and $v\in\mathfrak{V}_N$,
\begin{equation}
\Big(\int_{\cbO_N}\;\big(\bunderline{w}_k^B(v+\eta)\big)^{1+\overline{\rho}}\,\nu_{k+1}^{\ssup{\bq}}(\d\eta)\Big)^{\frac{1}{1+\overline{\rho}}}\le \frac{\Acal_{\Bcal}}{2}\bunderline{w}_{k:k+1}^X(v).
\end{equation}
\end{enumerate}
\end{theorem}
A proof of this theorem is contained in Appendix \ref{A:C}. The role of the constants $\zeta'$ and $\mathfrak{r}$, as well as their relationship with the parameter $\zeta$ of the perturbation norm $\norm{\cdot}_{\zeta,\Qcal_V}$, is equally explained there. Before moving on, we update the requirements imposed by Theorem \ref{T:wf} on the free parameters.

\hspace{1cm}
\begin{ctracker}[userdefinedwidth=8cm,align=center]
\begin{align}
L &\ge 2^{d+3}+16R \\
h &\ge h_0=\delta^{-1/2}\max\{\sqrt{8},\sqrt{2}c_d\}
\end{align}
\end{ctracker}

\subsection{Step 2: Properties of the norms} \label{S5:norms}

The key property verified in this step is the sub-multiplicativity of the norms $\norm{\cdot}_{k,X}$ and $\norm{\cdot}_{k:k+1,X}$. We begin with the point-wise properties of $\abs{\cdot}_{k,X,v}$, which follow from general features of injective tensor norms and their dual norms.
\begin{lemma} [Lemma 8.1 of \cite{ABKM}]\label{L:factor}
Let $X\in\Pcal_k$, $v\in\mathfrak{V}_N$ and $F,G\in C^{r_0}(\mathfrak{V}_N)$ such that $F$ and $G$ only depend on the values of the field in $X^*$. Then it holds that
\begin{equation}
\abs{FG}_{k,X,v}\le \abs{F}_{k,X,v}\abs{G}_{k,X,v}
\end{equation}
and
\begin{equation} \label{E:Tayk+1est}
\abs{F}_{k+1,X,v}\le (1+\abs{v}_{k+1,X})^3\big(\abs{F}_{k+1,X,0}+16L^{-\frac{3d}{2}}\sup_{0\le t\le 1}\abs{F}_{k,X,tv}\big).
\end{equation}
\end{lemma} 
The proof is identical to that given in \cite{ABKM} once the shift by one in the degree of derivatives is accounted for. It is hence omitted.
\begin{lemma} [Lemma 8.2 of \cite{ABKM}] \label{L:abs}
For any $v\in\mathfrak{V}_N$, $F_1,F_2 \in M(\Pcal_k,\mathfrak{V}_N)$ and any $X_1,X_2 \in\Pcal_k$ (not necessarily disjoint), we have the bound
\begin{equation}
\abs{F_1(X_1)F_2(X_2)}_{k,X_1\cup X_2,v}\le \abs{F_1(X_1)}_{k,X_1,v}\abs{F_2(X_2)}_{k,X_2,v}.
\end{equation}
Furthermore, if $L\ge 2^d+R$, $X\in\Pcal_k$ and $F\in C^{r_0}(\mathfrak{V}_N)$ is a functional depending only on the values of the field in $\pi_k(X)^*$, then the inequality
\begin{equation}
\abs{F}_{k+1,\pi_k(X),v}\le\abs{F}_{k,X\cup \pi_k(X),v}
\end{equation}
is satisfied for all $v\in\mathfrak{V}_N$.
\end{lemma}
Again, the proof is as in \cite{ABKM} and therefore omitted. Combined with the properties of the weight functions established in Theorem \ref{T:wf}, these results permit the desired sub-multiplicativity of our functional norms to be deduced. The argument runs as in \cite{ABKM} and so we do not repeat it. The required statements are summarised in the following lemma.  Here, the symbol $\one\in M(\Bcal_k,\mathfrak{V}_N)$ stands for the constant functional $\one(B)(v)=1$ for all $B\in\Bcal_k,v\in\mathfrak{V}_N$.
\begin{lemma} [Lemma 8.3 of \cite{ABKM}] \label{L:norms}
Let $K\in M(\Pcal_k^c,\mathfrak{V}_N)$ for $0\le k\le N-1$ and consider it as an element $K\in M(\Pcal_k,\mathfrak{V}_N)$ via the inclusion \eqref{E:canincl2}. Furthermore, let $F,F_1,F_2\in M(\Bcal_k,\mathfrak{V}_N)$. On the assumptions of Theorem \ref{T:wf}, the following bounds hold:
\begin{enumerate}
\item For every $X\in\Pcal_k$,
\begin{equation}
\norm{K(X)}_{k,X}\le \prod_{Y\in\Ccal_k(X)}\norm{K(Y)}_{k,Y}\,\text{ and }\,\norm{K(X)}_{k:k+1,X}\le \prod_{Y\in\Ccal_k(X)}\norm{K(Y)}_{k:k+1,Y};
\end{equation}
\item For every $X,Y\in\Pcal_k$ disjoint,
\begin{equation}
\norm{K(X)F^Y}_{k,X\cup Y}\le \norm{K(X)}_{k,X}\tnorm{F}_k^{\abs{Y}_k};
\end{equation}
\item For any polymers $X,Y,Z_1,Z_2\in\Pcal_k$ such that $X\cap Y=\emptyset$, $Z_1\cap Z_2=\emptyset$ and $Z_1,Z_2\subset X\cup Y\cup \pi_k(X\cup Y)$,
\begin{equation}
\norm{K(X)F^Y F_1^{Z_1} F_2^{Z_2}}_{k+1,\pi_k(X\cup Y)}\le \norm{K(X)}_{k:k+1,X}\tnorm{F}_k^{\abs{Y}_k}\tnorm{F_1}_k^{\abs{Z_1}_k}\tnorm{F_2}_k^{\abs{Z_2}_k};
\end{equation} 
\item For any $B\in\Bcal_k$,
\begin{equation}
\tnorm{\one}_{k,B}=\tnorm{\one}_k=1.
\end{equation}
\end{enumerate}
\end{lemma}
No further constraints on the free parameters, beyond the assumptions needed for Theorem \ref{T:wf}, are required for this step.

\subsection{Step 3: Properties of the integration map $\bR^{\ssup{\bq}}$} \label{S5:Rk}

As alluded to in section \ref{S4:RGmap}, our setting requires us to treat the integration map $\bR_{k+1}^{\ssup{\bq}}$ as a mapping of functionals defined on $\mathfrak{V}_N$, a larger space than the domain of integration $\cbO_N$. We verify below that this transformation is well-defined for the class of functionals with which we are concerned. Moreover, since $\bR_{k+1}^{\ssup{\bq}}$ does not act as convolution, we show explicitly that it commutes with differentiation in the desired manner. This is essential in order to establish the boundedness of $\bR_{k+1}^{\ssup{\bq}}$ in the relevant norms we have defined.

A similar issue arises in relation to taking derivatives with respect to the matrix $\bq$ as a parameter in the Gaussian probability measures $\nu_{k+1}^{\ssup{\bq}}$. We therefore also check that the latter operation commutes with differentiation with respect to the field variable, which is needed to show the requisite regularity of $\bR_{k+1}^{\ssup{\bq}}$ as a function of its parameters. To that end, we denote the directional derivative with respect to $\bq$ in the direction $\dot{\bq}$ by
\begin{equation}
D_{\dot{\bq}}\bR_{k+1}^{\ssup{\bq}}[F](X,v)=\frac{\d}{\d t}\int_{\cbO_N}\;F(X,v+\eta)\,\nu_{k+1}^{\ssup{\bq+t\dot{\bq}}}(\d\eta)\Big\rvert_{t=0}.
\end{equation}
\begin{lemma} \label{L:Rreg}
Assume that the requirements of Theorem \ref{T:wf} are met and $\bq\in B_{\kappa}(0)$. For any $F\in M(\Pcal_k,\mathfrak{V}_N)$ and $X\in\Pcal_k$ ($0\le k\le N$) such that $\norm{F(X)}_{k,X}<\infty$:
\begin{enumerate}
\item The image $\bR_{k+1}^{\ssup{\bq}}[F](X)$ is well-defined as an element of $C^{r_0}(\mathfrak{V}_N)$ and $D^s\bR_{k+1}^{\ssup{\bq}}[F](X)=\bR_{k+1}^{\ssup{\bq}}[D^s F](X)$ for all $s\le r_0$. \\
\item Furthermore, the derivative $D_{\dot{\bq}}\bR_{k+1}^{\ssup{\bq}}[F](X)\in C^{r_0}(\mathfrak{V}_N)$ and $D^sD_{\dot{\bq}}\bR_{k+1}^{\ssup{\bq}}[F](X)=D_{\dot{\bq}}D^s\bR_{k+1}^{\ssup{\bq}}[F](X)=D_{\dot{\bq}}\bR_{k+1}^{\ssup{\bq}}[D^s F](X)$ for all $s\le r_0$.
\end{enumerate}
\end{lemma}
\begin{proof}
Fix $F\in M(\Pcal_k,\mathfrak{V}_N)$ and $X\in\Pcal_k$. The argument is elementary once it has been shown that our norms imply suitable bounds on the partial derivatives of $F(X)$. Let $\p^{\alpha}$ stand for the partial derivative represented by the multi-index $\alpha\in\N_0^{\dim\mathfrak{V}_N}$, with $\abs{\alpha}\le r_0$, acting on $C^{r_0}(\mathfrak{V}_N)$ functions. We infer from the definition of the Taylor polynomial norm $\abs{\cdot}_{k,X,v}$ and \eqref{E:dualform} that there is a constant $C$, depending on $d$, $L$, $h$ and $k$, such that
\begin{equation}
\abs{\p^{\alpha}F(X,v)}\le C(d,L,h,k)\abs{F(X)}_{k,X,v}
\end{equation}
holds for all $\abs{\alpha}\le r_0$ and $v\in\mathfrak{V}_N$. We now use the definition of the norm $\norm{\cdot}_{k,X}$ to deduce the bound $\abs{\p^{\alpha}F(X,v)}\le C \bunderline{w}_k^X(v) \norm{F(X)}_{k,X}$. It then follows from Theorem \ref{T:wf}(9) that
\begin{equation}
\int_{\cbO_N}\;\abs{\p^{\alpha}F(X,v+\eta)}\,\nu_{k+1}^{\ssup{\bq}}(\d\eta)\le C \Big(\frac{\Acal_{\Pcal}}{2}\Big)^{\abs{X}_k}\norm{F(X)}_{k,X}\overline{w}_{k:k+1}^X(v)<\infty
\end{equation}
and hence that $\bR_{k+1}^{\ssup{\bq}}[\p^{\alpha}F](X)$ is well-defined point-wise. Moreover, we can use that the weight function $\bunderline{w}_k^X$ is constructed from a non-negative quadratic form on $\mathfrak{V}_N$ (see Appendix \ref{A:C}) to split the integrand as
\begin{equation} \label{E:wfsplit}
\bunderline{w}_k^X(v+\eta)\le \big(\bunderline{w}_k^X(v)\big)^{1+\frac{1}{\rho}}\big(\bunderline{w}_k^X(\eta)\big)^{1+\rho}.
\end{equation}
By continuity of $\bunderline{w}_k^X$, the factor in $v$ is bounded on any bounded neighbourhood of $v$, and, again as a result of Theorem \ref{T:wf}(9), the factor in $\eta$ is an $L^{1+\rho}(\cbO_N,\nu_{k+1}^{\ssup{\bq}})$-integrable function. It follows by dominated convergence that $\bR_{k+1}^{\ssup{\bq}}[\p^{\alpha}F](X)$ is continuous. We consider the derivatives of $\bR_{k+1}^{\ssup{\bq}}[F](X)$ next. Let $v_0$ stand for any unit basis vector in $\mathfrak{V}_N$ and $\p_{v_0}$ for the relevant partial derivative. We form the difference quotient
\begin{multline}
\int_{\cbO_N}\;\frac{F(X,v+tv_0+\eta)-F(X,v+\eta)}{t}\,\nu_{k+1}^{\ssup{\bq}}(\d\eta) \\
= \int_{\cbO_N}\;\int_0^1\;\p_{v_0}F(X,v+stv_0+\eta)\,\d s \, \nu_{k+1}^{\ssup{\bq}}(\d\eta)
\end{multline}
for small $t$, say, $\abs{t}<1$. Relying again on \eqref{E:wfsplit} and that
\begin{equation}
\sup_{s\in [0,1],\abs{t}<1}\big(\bunderline{w}_k^X(v+stv_0)\big)^{1+\frac{1}{\rho}}<\infty,
\end{equation}
we obtain an integrable bound and the desired interchangeability of $\p_{v_0}$ and $\bR_{k+1}^{\ssup{\bq}}$ is deduced by dominated convergence from the continuity of $\p_{v_0}F(X)$. The full claim then follows by induction on the order of derivatives and the continuity of all partial derivatives already established.

The second part of the result is shown similarly. It suffices to consider a single derivative. As before, we consider the appropriate difference quotient
\begin{align}
D(h,t) &= \frac{1}{h}\bigg[\frac{1}{t}\Big(\int_{\cbO_N}\;F(X,v+hv_0+\eta)\,\nu_{k+1}^{\ssup{\bq+t\hat{\bq}}}(\d\eta)-\int_{\cbO_N}\;F(X,v+hv_0+\eta)\,\nu_{k+1}^{\ssup{\bq}}(\d\eta)\Big) \nonumber \\
&\quad\quad\quad\quad -\frac{1}{t}\Big(\int_{\cbO_N}\;F(X,v+\eta)\,\nu_{k+1}^{\ssup{\bq+t\hat{\bq}}}(\d\eta)-\int_{\cbO_N}\;F(X,v+\eta)\,\nu_{k+1}^{\ssup{\bq}}(\d\eta)\Big)\bigg] \nonumber \\
&= \frac{1}{t}\Big(\int_{\cbO_N}\;\int_0^1\;\p_{v_0}F(X,v+shv_0+\eta)\,\d s\,\nu_{k+1}^{\ssup{\bq+t\hat{\bq}}}(\d\eta) \nonumber \\
&\quad\quad\quad\quad - \int_{\cbO_N}\;\int_0^1\;\p_{v_0}F(X,v+shv_0+\eta)\,\d s\,\nu_{k+1}^{\ssup{\bq}}(\d\eta)\Big)
\end{align}
in the limit $t\to 0$ followed by $h\to 0$. We obtain the desired expression for the derivative $\p_{v_0}D_{\dot{\bq}}\bR_{k+1}^{\ssup{\bq}}[F](X,v)$, provided that the $h\to 0$ limit of the expression
\begin{equation}
I(h)=\frac{\d}{\d t}\int_{\cbO_N}\;\int_0^1\;\p_{v_0}F(X,v+shv_0+\eta)-\p_{v_0}F(X,v+\eta)\,\d s\,\nu_{k+1}^{\ssup{\bq+t\hat{\bq}}}(\d\eta)\Big\rvert_{t=0}
\end{equation}
vanishes. Using the bound from Theorem \ref{T:Dqdotbound} and Jensen's inequality, we get the estimate
\begin{multline}
\abs{I(h)}\le C_{\eqref{E:Dqdotbouond}}(p,d,w_0,L,k,X,\dot{\bq}) \times\\
\times\Big(\int_{\cbO_N}\;\int_0^1\; \abs{\p_{v_0}F(X,v+shv_0+\eta)-\p_{v_0}F(X,v+\eta)}^p\,\d s\,\nu_{k+1}^{\ssup{\bq}}(\d\eta)\Big)^{1/p}
\end{multline}
for any $p>1$. Arguing as above, setting $p=1+\frac{\rho}{2}$ and $\gamma=\frac{\rho}{2+\rho}$, we bound this integrand as
\begin{align}
|\p_{v_0}F(X,v+shv_0 &+ \eta)-\p_{v_0}F(X,v+\eta)|^p \nonumber \\
&\le \big(C\norm{F(X)}_{k,X}\big)^p\,\abs{\bunderline{w}_k^X(v+shv_0+\eta)+\bunderline{w}_k^X(v+\eta)}^p  \nonumber \\
&\le C'\,\abs{\big(\bunderline{w}_k^X(v+shv_0)\big)^{1+\frac{1}{\gamma}}+\big(\bunderline{w}_k^X(v)\big)^{1+\frac{1}{\gamma}}}^p\big(\bunderline{w}_k^X(\eta)\big)^{1+\rho},
\end{align}
which is integrable on $\cbO_N\times [0,1]$. It then follows again by dominated convergence and the continuity of $\p_{v_0}F(X)$ that $I(h)\to 0$ and hence that $\p_{v_0}D_{\dot{\bq}}\bR_{k+1}^{\ssup{\bq}}[F](X)=D_{\dot{\bq}}\bR_{k+1}^{\ssup{\bq}}[\p_{v_0}F](X)$. A similar reasoning shows the continuity of $D_{\dot{\bq}}\bR_{k+1}^{\ssup{\bq}}[\p_{v_0}F](X)$ and this completes the proof.
\end{proof}

This result covers the technical gap opened by our need to consider functionals on a larger underlying space. The required bounds and regularity of $\bR_{k+1}^{\ssup{\bq}}$ can now be established as in \cite{ABKM} and hence we omit a detailed proof. We summarise the required features in the following statement.
\begin{lemma} [Lemma 8.4 of \cite{ABKM}] \label{L:Rbound}
Assume that the requirements of Theorem \ref{T:wf} are met, $\bq\in B_{\kappa}(0)$ and $\Acal_{\Pcal}$ is the constant given by Theorem \ref{T:wf}(9). Given $X\in \Pcal_k$ ($0\le k\le N$), for any $F\in M(\Pcal_k,\mathfrak{V}_N)$ with $\norm{F(X)}_{k,X}<\infty$ it holds that
\begin{equation} \label{E:Rbound}
\norm{\bR_{k+1}^{\ssup{\bq}}[F](X)}_{k:k+1,X}\le \Acal_{\Pcal}^{\abs{X}_k}\norm{F(X)}_{k,X}.
\end{equation}
Given $X\in\Pcal_k$ such that $\pi_k(X)\in\Pcal_{k+1}^c$, for $\ell\ge 1$ it further holds that
\begin{equation} \label{E:Rqbound}
\sup_{\norm{\dot{\bq}}\le 1}\norm{D_{\bq}^{\ell}\bR_{k+1}^{\ssup{\bq}}[F](X)(\dot{\bq},\ldots,\dot{\bq})}_{k:k+1,X}\le C_{\ell}(L)\Acal_{\Pcal}^{\abs{X}_k}\norm{F(X)}_{k,X}.
\end{equation}
If $X\in \Bcal_k$, the above bounds hold with $\Acal_{\Pcal}$ replaced by $\Acal_{\Bcal}$.
\end{lemma}
These properties of $\bR_{k+1}^{\ssup{\bq}}$ are a stepping stone towards verifying that our definitions in section \ref{S4:RGmap} are consistent and ensure that key properties of our functionals are preserved by the renormalisation group flow. We consider the other component of $\bT_k^{\ssup{\bq}}$, the projection $\Pi_2$, next.

\subsection{Step 4: Estimates for the projection $\Pi_2$} \label{S5:Pi2}

Two features of the projection map $\Pi_2$ defined in section \ref{S4:RGmap} are essential for our subsequent analysis. Since our conventions for relevant Hamiltonians in the vector field setting differ slightly from \cite{ABKM}, we include a proof for clarity. First, we verify that $\Pi_2$ is bounded in the appropriate norms. This is needed, in particular, to show that the image under $\bT_k^{\ssup{\bq}}$ is well-defined as an element of $\bM_0(\Bcal_{k+1},\mathfrak{V}_N)\times\bM(\Pcal_{k+1}^c,\mathfrak{V}_N)$ as asserted previously.
\begin{lemma} [Lemma 8.7 of \cite{ABKM}] \label{L:Pi2bound}
For $L\ge 2^d+R$, there is a constant $C$ (depending only on $d$) such that
\begin{equation} \label{E:Pi2bound}
\norm{\Pi_2(K)(B)}_{k,0}\le C(d)\abs{K(B)}_{k,B,0}
\end{equation}
holds for all $0\le k\le N-1$, $K\in M(\Pcal_k^c,\mathfrak{V}_N)$ and $B\in\Bcal_k$. 
\end{lemma}
\begin{proof}
We consider in turn each of the components of $\Pi_2(K(B))$ using the notation introduced in \eqref{E:Hparas}. The constant term is straightforward. By definition, we have $\lambda=L^{-dk}K(B,0)$ and so the bound $L^{dk}\abs{\lambda}\le \abs{K(B,0)}$ certainly holds. For the quadratic terms, recall that the coefficient $d_{i,j}$ is determined by reference to the dual pairing
\begin{equation}
\text{Tay}_0^{(2)}K(B)\big(b_{i,0}\otimes b_{j,0}+b_{j,0}\otimes b_{i,0}\big).
\end{equation}
Using the shorthand notation
\begin{equation}
\abs{K(B)}_{k,B,0}^{(2)}=\sup\{\text{Tay}_0^{(2)}K(B)(g^{(2)}):g^{(2)}\in\mathfrak{V}_N^{\otimes 2},\abs{g^{(2)}}_{k,B}\le 1\},
\end{equation}
we therefore find an upper bound
\begin{equation}
\abs{d_{i,j}}\le \abs{b_{i,0}\otimes b_{j,0}+b_{j,0}\otimes b_{i,0}}_{k,B}\abs{K(B)}_{k,B,0}^{(2)}
\end{equation}
in terms of the second-order component of $\abs{K(B)}_{k,B,0}$. We find that $\abs{b_{i,0}\otimes b_{j,0}+b_{j,0}\otimes b_{i,0}}_{k,B}=(1+\delta_{i,j})L^{dk}h_k^{-2}$ by inserting the relevant definitions and so we obtain the estimate
\begin{equation}
\frac{1}{2}\sum_{i,j=1}^dh_k^2\abs{d_{i,j}}\le \frac{d(d+1)}{2}\abs{K(B)}_{k,B,0}^{(2)}.
\end{equation}
Finally, we address the linear terms. The coefficients $a_m$ are given as the solution of the system \eqref{E:Pi2linear} for each direction $0\le i\le d$. Its dimension, denoted by $D$ in the following, is a combinatorial constant depending only on $d$. After rescaling, this system can be represented as
\begin{equation}
\sum_{0\le\abs{\a}\le\floor{d/2}}\bA_{\a',\a}L^{-k\abs{\a}}a_{i,\a}=\text{Tay}_0^{(1)}K(B)(b_{i,\a'})L^{-k\frac{d+2\abs{\a'}}{2}},
\end{equation}
where the matrix $\bA$ has entries $\bA_{\a',\a}=L^{-k\frac{d+2\abs{\a'}-\abs{\a}}{2}}\sum_{x\in B}b_{\a'-\a}(x)$. As observed in the construction of $\Pi_2$, $\bA$ is triangular with (rescaled) eigenvalue $L^{\frac{dk}{2}}$. It follows that we can bound the $a_{i,\a}$ coefficients (for each $1\le i\le d$) as
\begin{align}
L^{-k\abs{\a}}\abs{a_{i,\a}} &\le \norm{\bA^{-1}} D^{1/2} \max_{0\le \abs{\a'}\le \floor{d/2}} L^{-k\frac{d+2\abs{\a'}}{2}}\abs{\text{Tay}_0^{(1)}K(B)(b_{i,\a'})} \nonumber \\
&\le L^{-\frac{dk}{2}}D^{1/2}\abs{K(B)}_{k,B,0}^{(1)}\max_{0\le \abs{\a'}\le \floor{d/2}} L^{-k\frac{d+2\abs{\a'}}{2}} \abs{b_{i,\a'}}_{k,B} \nonumber \\
&\le h_k^{-1}L^{-\frac{dk}{2}}D^{1/2}\abs{K(B)}_{k,B,0}^{(1)} \max_{0\le \abs{\a'}\le \floor{d/2}} \max_{\a''\le\a'}\sup_{x\in B^*} L^{k(\abs{\a''}-\abs{\a'}}\abs{b_{\a'-\a''}(x)}.
\end{align}
By translation invariance, we may insist that $0\in B\subset \Lambda_N$. We then see that $\sup_{x\in B^*} L^{k(\abs{\a''}-\abs{\a'}}\abs{b_{\a'-\a''}(x)}$ is bounded (uniformly in $k$) depending only on $d$ and $R$. Accordingly, we have that
\begin{equation}
\sum_{m=(i,\a)\in\mathfrak{m}}h_k L^{k\frac{d-2\abs{\a}}{2}}\abs{a_m}\le C \abs{K(B)}_{k,B,0}^{(1)}  
\end{equation}
for some (dimension dependent) constant $C$. By summing these three bounds we then obtain the claim in \eqref{E:Pi2bound}.
\end{proof}

The second property of $\Pi_2$ we check is crucial to the proof that the linearisation of $\bT_k^{\ssup{\bq}}$ contracts in the second coordinate (i.e. $K_{k+1}$). For this, we want that $(1-\Pi_2)(K_k)$ is small when measured at the next scale $(k+1)$ relative to the size of $K_k$ at scale $k$.
\begin{lemma} [Lemma 8.9 of \cite{ABKM}] \label{L:Pi2contract}
For $L\ge 2^d+R$, there is a (different) constant $C$ (depending only on $d$) such that
\begin{equation} \label{E:Pi2contr}
\abs{(1-\Pi_2)(K)(B)}_{k+1,B,0}\le C(d)L^{-(d/2+\floor{d/2}+1)}\abs{K(B)}_{k,B,0}
\end{equation}
holds for all $0\le k\le N-1$, $K\in M(\Pcal_k^c,\mathfrak{V}_N)$ and $B\in\Bcal_k$. 
\end{lemma}
\begin{proof}
The proof of Lemma 8.9 of \cite{ABKM} uses discrete Taylor remainder estimates to bound the difference between a $\Z^d$ function and its $r$-th Taylor polynomial $\abs{f(z)-\text{Tay}_a^rf(z)}$ (and derivatives of this difference up to order $r$) in terms of the discrete derivatives of $f$ of $(r+1)$-st order. By passing to field norms of the next scale $(k+1)$ this yields a tightening in the bounds of the order of $L^{-(d+1)}$. We state the precise estimates, corresponding to Lemma 8.10 of \cite{ABKM}, below. There exists a constant $C$ such that, for all $L\ge 2^d+R$, the bounds
\begin{align}
\min_{P\in\Pcal_1}\abs{v-P}_{k,B} &\le CL^{-(d/2+\floor{d/2}+1)}\abs{v}_{k+1,B};\text{ and} \label{E:P1bound}\\
\min_{P\in\Pcal_2}\abs{Sg-P}_{k,B} &\le CL^{-(d+1)}\abs{g}_{k+1,B} \label{E:P2bound}
\end{align}
hold for all $v\in\mathfrak{V}_N$ and $g\in\mathfrak{V}_N^{\otimes 2}$ respectively. The symbol $S$ stands here for the symmetrisation operator $S v\otimes w=\frac{1}{2}(v\otimes w + w\otimes v)$. These inequalities are derived exactly as in \cite{ABKM}, once the difference in the highest order of derivatives in $\abs{\cdot}_{k,B}$ is accounted for. We therefore do not repeat the proof here.

By construction, $\Pi_2(K)(B)$ is a quadratic polynomial and so $\text{Tay}_0^{(r)}\big[(1-\Pi_2)(K)(B)\big]=\text{Tay}_0^{(r)}K(B)$ for all $r\ge 3$. Observing further that, for such orders $r\ge 3$, the straightforward estimate
\begin{equation}
\abs{g}_{k,B}\le 8L^{-\frac{3}{2}d}\abs{g}_{k+1,B}
\end{equation}
is available for all $g\in\mathfrak{V}_N^{\otimes r}$, we deduce that
\begin{equation}
\abs{\text{Tay}_0^{(r)}\big[(1-\Pi_2)(K)(B)\big](g)}\le \abs{K(B)}_{k,B,0}\abs{g}_{k,B}\le 8L^{-\frac{3}{2}d}\abs{K(B)}_{k,B,0}\abs{g}_{k+1,B}
\end{equation}
holds for all $g\in\mathfrak{V}_N^{\otimes r}$. The constant term is trivial and so it remains to consider the linear and quadratic terms. By definition of $\Pi_2$, we have the equality
\begin{equation}
\text{Tay}_0^{(1)}\big[(1-\Pi_2)(K)(B)\big](v)=\text{Tay}_0^{(1)}\big[(1-\Pi_2)(K)(B)\big](v-P)
\end{equation}
for all $v\in\mathfrak{V}_N$ and $P\in\Pcal_1$. The bound $\abs{H(B)}_{k,B,0}\le \norm{H}_{k,0}$, where $H$ is any relevant Hamiltonian, follows from the definitions of the two norms. With its help, we obtain the estimate
\begin{multline}
\abs{\text{Tay}_0^{(1)}\big[(1-\Pi_2)(K)(B)\big](v)}\le \abs{(1-\Pi_2)(K)(B)}_{k,B,0}\min_{P\in\Pcal_1}\abs{v-P}_{k,B} \\ \le C L^{-(d/2+\floor{d/2}+1)} \abs{K(B)}_{k,B,0}\abs{v}_{k+1,B},
\end{multline}
from Lemma \ref{L:Pi2bound} and \eqref{E:P1bound} above, where $C$ depends on the constants appearing in those results. Since the dual pairing $\text{Tay}_0^{(2)}\big[(1-\Pi_2)(K)(B)\big](g)$ for $g\in\mathfrak{V}_N^{\otimes 2}$ only depends on the symmetric part $Sg$, we can argue similarly for the quadratic terms to see that
\begin{multline}
\abs{\text{Tay}_0^{(2)}\big[(1-\Pi_2)(K)(B)\big](g)}\le \abs{(1-\Pi_2)(K)(B)}_{k,B,0}\min_{P\in\Pcal_2}\abs{Sg-P}_{k,B} \\ \le C L^{-(d+1)} \abs{K(B)}_{k,B,0}\abs{g}_{k+1,B},
\end{multline}
for all $g\in\mathfrak{V}_N^{\otimes 2}$. The claim then follows from the definition of $\abs{K(B)}_{k+1,B,0}$ by taking the widest of the individual bounds shown above. 
\end{proof}
We are now in a position to show that our definition of the renormalisation map $\bT_k$ in section \ref{S4:RGmap} is consistent and achieves what we require of it.

\subsection{Step 5: Definition and properties of the renormalisation map} \label{S5:Tk}

We already showed that the relevant Hamiltonian component of the image under $\bT_k^{\ssup{\bq}}$ is well-defined as a functional in $\bM_0(\Bcal_{k+1},\mathfrak{V}_N)$ pursuant to its construction in section \ref{S4:RGmap}. It remains to verify the definition of the other component. Since this relies on the boundedness of various operators that appear in the definition of $K_{k+1}$, it is convenient to treat this matter as part of the main statement on the smoothness of $\bT_k$, which is required for our overall strategy.

\begin{theorem} [Theorem 6.7 of \cite{ABKM}] \label{T:smooth}
Let $L_0=\max\{2^{d+3}+16R,4d(2^d+R)\}$. For every $L\ge L_0$, there are $h_0$, $A_0$ and $\kappa$, depending only on $L$ and the fixed parameters, with the property that, for any $h\ge h_0$ and $A\ge A_0$, there exists a $\rho(A)$ such that the map $\boldsymbol{S}_k:\Ucal_{k,\rho,\kappa}\to \bM(\Pcal_{k+1}^c,\mathfrak{V}_N)$ is $C^\infty$ and there are constants $C_{i_1,i_2,i_3}(A,L)$ such that
\begin{equation}
\norm{D_1^{i_1}D_2^{i_2}D_3^{i_3}\boldsymbol{S}_k(H_k,K_k,\bq)(\dot{H}^{i_1},\dot{K}^{i_2},\dot{\bq}^{i_3})}_{k+1}^{(A)}\le C_{i_1,i_2,i_3}(A,L)\norm{\dot{H}}_{k,0}^{i_1}\big(\norm{\dot{K}}_k^{(A)}\big)^{i_2}\norm{\dot{\bq}}^{i_3}
\end{equation}
for any $(H_k,K_k,\bq)\in \Ucal_{k,\rho,\kappa}$ and $i_1,i_2,i_3\ge 0$. 
\end{theorem}

The corresponding result occupies the whole of Chapter 9 of \cite{ABKM}. However, the proof mainly relies on the properties of the various norms and of the mappings $\bR_{k+1}^{\ssup{\bq}}$ and $\Pi_2$. Since we have shown analogous properties in our setting in the preceding steps, adapting the proof from \cite{ABKM} is straightforward. Minor differences from \cite{ABKM} arise due to the way we extend the weight function $\bunderline{W}_k^B$ to arbitrary vector fields, which are reflected in our constants. A detailed proof of Theorem \ref{T:smooth} is omitted, so as to avoid repeating substantial parts of \cite{ABKM}. We record here simply the appropriate constraints on our choice of the free parameters ($\kappa$ takes the same value as in Theorem \ref{T:wf}).

\hspace{1cm}
\begin{ctracker}[userdefinedwidth=10cm,align=center]
\begin{align}
L &\ge L_0=\max\{2^{d+3}+16R,4d(2^d+R)\} \\
h &\ge h_0=\delta^{-1/2}\max\{\sqrt{8},\sqrt{2}c_d\} \\
A &\ge A_0=\big(48\Acal_{\Pcal}(L)\big)^{L^d/\alpha(d)}
\end{align}
\end{ctracker}
With a choice of parameters consistent with the above, we obtain a domain $\Ucal_{k,\rho,\kappa}$, on which $\boldsymbol{S}_k$ is smooth, determined by $\kappa$ and
\begin{equation}
\rho=\rho(A)=\frac{1}{8A^2}\min\Big\{\frac{1}{8},\Big(1+c_{2,0}+\frac{C_{\eqref{E:Pi2bound}}\Acal_{\Bcal}}{A}\Big)^{-1}\Big\},
\end{equation}
where $\a(d)$ is a dimension-dependent geometric constant and $c_{2,0}$ is a constant (only depending on the fixed parameters) characterising the finite range decomposition and given by Theorem 2.5 of \cite{B18}.

Armed with this result, we may now collect the key properties of the irrelevant component of the renormalisation map, that is, $K_{k+1}=\boldsymbol{S}_k(H_k,K_k,\bq)$, in the next statement.

\begin{lemma} [Lemma 6.4 of \cite{ABKM}] \label{L:K'}
On the assumptions of Theorem \ref{T:smooth} and for any choice of $L$,$h$,$A$ and $\rho$ consistent with it, the functional $K_{k+1}$ given as in \eqref{E:K'} satisfies the following:
\begin{enumerate}
\item $K_{k+1}$ is translation-invariant at the scale $(k+1)$;
\item $K_{k+1}$ is local;
\item $K_{k+1}(U)\in C^{r_0}(\mathfrak{V}_N)$ for all $U\in\Pcal_{k+1}$; and
\item If $K_k$ factors at the scale $k$ then $K_{k+1}$ factors at the scale $(k+1)$.
\end{enumerate}
\end{lemma}
\begin{proof}
Theorem \ref{T:smooth} implies that, for $(H_k,K_k,\bq)\in\Ucal_{k,\rho,\kappa}$, all the constituent functionals of the definition of $K_{k+1}$ in \eqref{E:K'} are finite in the relevant norms. This in turn entails that $K_{k+1}$ is well-defined pointwise and our subsequent manipulations of these expressions are justified rather than merely formal. We now take each of the claims in order.

We have already shown that the projection $\Pi_2$ preserves translation-invariance and the same straightforwardly holds for the integration mapping $\bR_{k+1}^{\ssup{\bq}}$. Furthermore, $\pi_k$ is translation-invariant by construction and claim (1) is thus a consequence of the translation-invariance of $H_k$ and $K_k$. To see (2), we observe that the functionals $\widetilde{I}$, $I$ and $\widetilde{K}$ (the latter two in their first vector field argument) inherit the local property from $H_k$ and $K_k$. $\chi(X,U)$ is non-zero only if $X^*\subset U^*$ (see our comments in section \ref{S5:geo} above) and so it follows that $K_{k+1}$ only depends on the values of $v$ in $U^*$. Next, the smoothness of $I$ and $\widetilde{I}$ (as vector field functionals) is clear from their definitions. The $C^{r_0}(\mathfrak{V}_N)$-differentiability of $K_{k+1}$ is then a consequence of Lemma \ref{L:Rreg}, in light of our observation above that all relevant functionals have finite norms, and the chain rule. This confirms claim (3).

Finally, to prove the factorisation property (4), take $U=U_1\cup U_2\in\Pcal_{k+1}$ with $U_1$ and $U_2$ strictly disjoint and let $X\in\Pcal_k$ be any polymer such that $\pi_k(X)=U$. Consider the (strictly disjoint at scale $k$) polymers $X_1=U_1^*\cap X$ and $X_2=U_2^*\cap X$. Since $X\subset X^*\subset \pi_k(X)^*=(U_1\cup U_2)^*=U_1^*\cup U_2^*$, we see that $X_1\cup X_2=X$. Moreover, the identity $\pi_k(X_1\cup X_2)=\pi_k(X_1)\cup\pi_k(X_2)=U$ and the strict disjointedness of $U_1$ and $U_2$ then imply that $\pi_k(X_1)=U_1$ and $\pi_k(X_2)=U_2$. Suppose that $X'\subset X\in\Pcal_k$ also satisfies $\pi_k(X')=U_1$. Then it holds that $X'\subset U_1^*$ and hence $X'\subset X_1$. It follows that the decomposition of $X$ into disjoint $X_1$ and $X_2$ with $\pi_k(X_1)=U_1$ and $\pi_k(X_2)=U_2$ is unique. Accordingly, we may re-arrange the expression of $K_{k+1}(U)$ as follows:
\begin{align}
K_{k+1} & (U_1\cup U_2,v) \nonumber \\
&= \sum_{\substack{X_1,X_2\in \Pcal_k: \\ \pi_k(X_1)=U_1; \\ \pi_k(X_2)=U_2}} \widetilde{I}^{(U_1\cup U_2)\setminus (X_1\cup X_2)}(v)\widetilde{I}^{-(X_1\cup X_2)\setminus (U_1\cup U_2)}(v)\int_{\cbO_N}\;\widetilde{K}(X_1\cup X_2,v,\eta)\,\nu_{k+1}^{\ssup{\bq}}(\d\eta) \nonumber \\
&= \sum_{\substack{X_1,X_2\in \Pcal_k: \\ \pi_k(X_1)=U_1; \\ \pi_k(X_2)=U_2}} \frac{\widetilde{I}^{U_1\setminus X_1}(v)\widetilde{I}^{U_2\setminus X_2}(v)}{\widetilde{I}^{X_1\setminus U_1}(v)\widetilde{I}^{X_2\setminus U_2}(v)}   \int_{\cbO_N}\;\widetilde{K}(X_1,v,\eta)\,\nu_{k+1}^{\ssup{\bq}}(\d\eta)\int_{\cbO_N}\;\widetilde{K}(X_2,v,\eta)\,\nu_{k+1}^{\ssup{\bq}}(\d\eta) \nonumber \\
&= K_{k+1}(U_1,v)K_{k+1}(U_2,v),
\end{align}
where we used that the circle product preserves factorisation and locality in combination with the finite-range property of $\nu_{k+1}^{\ssup{\bq}}$ in view of the estimate
\begin{equation}
\text{dist}_{\infty}(X_1^*,X_2^*)\ge \text{dist}_{\infty}(U_1^*,U_2^*)\ge \frac{L^{k+1}}{2}+1
\end{equation}
provided by \eqref{E:rangeineq}.
\end{proof}
This confirms the features of the definition of $K_{k+1}$ advertised in section \ref{S4:RGmap}. A central aspect of the definition of $\bT_k^{\ssup{\bq}}$ is of course that the sequence $(H_k,K_k)$ corresponds to successive iterations of the integration maps $\bR_{k+1}^{\ssup{\bq}}$. We record the main result on the construction of the renormalisation maps $\bT_k^{\ssup{\bq}}$ in the following proposition.

\begin{prop} [Proposition 6.6 of \cite{ABKM}] \label{P:Tk}
On the assumptions of Theorem \ref{T:smooth} and for any choice of $L$,$h$,$A$ and $\rho$ consistent with it, the renormalisation maps $T_k(H_k,K_k,\bq)=(H_{k+1},K_{k+1})$ for $0\le k\le N-1$, where $H_{k+1}$ and $K_{k+1}$ are given by \eqref{E:H'} and \eqref{E:K'} respectively, are well-defined as $\mathcal{U}_{k,\rho,\kappa}\to \bM_0(\Bcal_{k+1},\mathfrak{V}_N)\times \bM(\Pcal_{k+1}^c,\mathfrak{V}_N)$ mappings and satisfy the identity
\begin{equation}
\bR_{k+1}^{\ssup{\bq}}\big(\ex^{-H_k}\circ K_k\big)(\Lambda_N,v)=\big(\ex^{-H_{k+1}}\circ K_{k+1}\big)(\Lambda_N,v).
\end{equation}
\end{prop}

The principal difference between our setting and \cite{ABKM}, so far as the impact on these constructions is concerned, has already been addressed in the preceding steps. Proposition \ref{P:Tk} follows from Theorem \ref{T:smooth} and Lemma \ref{L:K'} in the same way as in \cite{ABKM} and a proof is therefore omitted.

\subsection{Step 6: Contraction estimates} \label{S5:contract}

Having shown that $\bT_k$ is smooth on a suitably restricted domain, the second important feature of the renormalisation map is its hyperbolicity near the trivial fixed point $(0,0)$. We adopt the following notation for the derivative of $\bT_k^{\ssup{\bq}}$ at that point:
\begin{equation}
D\bT_k(0,0,\bq) \begin{pmatrix}
\dot{H} \\
\dot{K}
\end{pmatrix} = \begin{pmatrix}
\bA_k^{\ssup{\bq}} & \bB_k^{\ssup{\bq}} \\
0 & \bC_k^{\ssup{\bq}}
\end{pmatrix} \begin{pmatrix}
\dot{H} \\
\dot{K}
\end{pmatrix},
\end{equation}
where $\bA_k^{\ssup{\bq}}$ and $\bB_k^{\ssup{\bq}}$ are the linear operators introduced in \eqref{E:A} and \eqref{E:B} above. A straightforward computation in \cite{ABKM}, which we do not repeat here, produces an expression for $\bC_k^{\ssup{\bq}}$ in terms of the mappings $\Pi_2$ and $\bR_{k+1}^{\ssup{\bq}}$: 
\begin{multline}
\bC_k^{\ssup{\bq}}\dot{K}(U,v)=\sum_{B\in\Bcal_k:\overline{B}=U}\big(1-\Pi_2\big)\int_{\cbO_N}\;\dot{K}(B,v+\eta)\,\nu_{k+1}^{\ssup{\bq}}(\d\eta) + \\
\sum_{X\in\Pcal_k^c\setminus\Bcal_k:\pi_k(X)=U}\,\int_{\cbO_N}\;\dot{K}(X,v+\eta)\,\nu_{k+1}^{\ssup{\bq}}(\d\eta)
\end{multline}
for any $U\in\Pcal_{k+1}^c$. In the following, we consider the standard operator norms of $\bA_k^{\ssup{\bq}}$,$\bB_k^{\ssup{\bq}}$ and $\bC_k^{\ssup{\bq}}$ regarded as $\bM_0(\Bcal_k,\mathfrak{V}_N)\to\bM_0(\Bcal_{k+1},\mathfrak{V}_N)$, $\bM(\Pcal_k^c,\mathfrak{V}_N)\to\bM_0(\Bcal_{k+1},\mathfrak{V}_N)$ and $\bM(\Pcal_k^c,\mathfrak{V}_N)\to\bM(\Pcal_{k+1}^c,\mathfrak{V}_N)$ Banach space maps respectively. We now state the key result of this step.
\begin{theorem} [Theorem 6.8 of \cite{ABKM}] \label{T:linearisation}
There exists a constant $L_0$ and constants $h_0$, $A_0$ and $\kappa$, depending on $L$, such that for any $L\ge L_0$, $h\ge h_0$, $A\ge A_0$ and $\bq\in B_{\kappa}(0)$ the following bounds hold uniformly in $k$ and $N$:
\begin{equation}
\norm{\bC_k^{\ssup{\bq}}}\le \frac{3}{4}\vartheta,\;\norm{(\bA_k^{\ssup{\bq}})^{-1}}\le \frac{3}{4}\text{ and }\norm{\bB_k^{\ssup{\bq}}}\le \frac{1}{3}.
\end{equation}
Moreover, the derivatives of these operators with respect to $\bq$ are bounded as
\begin{multline}
\norm{D_{\bq}^l\bA_k^{\ssup{\bq}}(\dot{\bq}^l)\dot{H}}_{k+1,0}\le C \norm{\dot{H}}_{k,0}\norm{\dot{\bq}}^l,\; \norm{D_{\bq}^l\bB_k^{\ssup{\bq}}(\dot{\bq}^l)\dot{K}}_{k+1,0}\le C \norm{\dot{K}}_k^{(A)}\norm{\dot{\bq}}^l \\
\text{and }\norm{D_{\bq}^l\bC_k^{\ssup{\bq}}(\dot{\bq}^l)\dot{K}}_{k+1}^{(A)}\le C \norm{\dot{K}}_k^{(A)}\norm{\dot{\bq}}^l
\end{multline}
for some constant $C$ depending on $L$, $A$ and $l$.
\end{theorem}
The bounds on the derivatives follow straightforwardly from Theorem \ref{T:smooth} or (in the case of $\bA_k^{\ssup{\bq}}$ and $\bB_k^{\ssup{\bq}}$) estimates that form part of its proof. The bounds on the operator norms may be derived as in \cite{ABKM} using properties for which we have already shown analogous results in our setting. We therefore skip a detailed proof to avoid repetition. Finally, we update our constraints on the free parameters to reflect the requirements of Theorem \ref{T:linearisation}. As before, $\kappa$ is chosen in accordance with Theorem \ref{T:wf}. The following restrictions on $L$, $h$ and $A$ thus correspond to a choice of parameters such that the conclusions of Theorems \ref{T:smooth} and \ref{T:linearisation} hold simultaneously.

\hspace{1cm}
\begin{ctracker}[userdefinedwidth=14cm,align=center]
\begin{align}
L \ge L_0 &=\max\{(2\Acal_{\Bcal}C_{\eqref{E:Pi2contr}}C'\vartheta^{-1})^{\frac{1}{1+\floor{d/2}-d/2}},(16\Acal_{\Bcal}C'(C_{\eqref{E:Pi2bound}}+2)\vartheta^{-1})^{2/d}, \nonumber \\
& \quad\quad\quad\quad 2^{d+3}+16R,4d(2^d+R)\} \\
h \ge h_0 &=\max\{\sqrt{8/\delta},c_d\sqrt{2/\delta},\sqrt{c_{2,0}}\} \\
A \ge A_0 &=\max\{\big(48\Acal_{\Pcal}(L)\big)^{L^d/\alpha(d)},3C_{\eqref{E:Pi2bound}}\Acal_{\Bcal}L^d,\big(\frac{2^{L^d+3}\Acal_{\Pcal}(L)^{1+2\a}}{\vartheta}\big)^{-2\a}, \nonumber \\
& \quad\quad\quad\quad 8\vartheta^{-1}\Acal_{\Pcal}(L)^2 L^d(2^{d+1}-1)^{d2^d}\}
\end{align}
\end{ctracker}
In addition to the constants that appeared previously, the expression for $L_0$ features the analytic constant $C'= \sup_{t\ge 0}(1+t)^5e^{-t/2}$, which is accordingly independent of all others.  

\subsection{Step 7: Stable manifold theorem for the renormalisation group flow}

The remaining steps are an implementation of the general strategy described in \cite{Bry09} (see Theorem 2.16). Since \cite{ABKM} contains a detailed proof, which may readily be adapted to our setting in light of the preceding steps, we do not propose to give a detailed account here of what are largely standard renormalisation group arguments. We simply introduce the necessary definitions and notation to be able to state the main result.

The central idea is to conceive of the renormalisation group flow as a dynamical system on sequences of functionals $(H_k,K_k)$ such that a fixed point under the dynamic corresponds precisely to a sequence that satisfies Proposition \ref{P:Tk} with suitable constraints on the initial and final pairs $(H_0,K_0)$ and $(H_N,K_N)$. In the following, we designate by $\bE_{\zeta,\Qcal_V}$ the space of $C^{r_0}(\R^d)$ functions that are finite in the $\norm{\cdot}_{\zeta,\Qcal_V}$ norm defined in \eqref{E:KNorm}. We consider the Banach space of sequences
\begin{multline}
\Zcal=\big\{Z=(H_0,H_1,\ldots,H_{N-1},K_1,K_2,\ldots,K_N):H_k\in\bM_0(\Bcal_k,\mathfrak{V}_N), K_{k+1}\in\bM(\Pcal_{k+1}^c,\mathfrak{V}_N) \\
\forall k=0,\ldots,N-1\big\}
\end{multline}
equipped with the norm
\begin{equation}
\norm{Z}_{\Zcal}=\max_{0\le k\le N-1}\Big\{\vartheta^{-k}\norm{H_k}_{k,0}\vee\vartheta^{-(k+1)}\norm{K_{k+1}}_{k+1}^{(A)}\Big\}.
\end{equation}
Next, we define a dynamic on $\Zcal$, which depends on the choice of a seed relevant Hamiltonian and perturbation. Given a relevant Hamiltonian $\Hcal\in M_0(\Bcal_k,\mathfrak{V}_N)$, we designate by $\bq(\Hcal)$ the coefficients of the quadratic terms. Accordingly, in the notation of \eqref{E:Hparas}, $\bq(\Hcal)=\boldsymbol{d}\in\R_{\text{sym}}^{d\times d}$. Given further a perturbation $\Kcal\in \bE_{\zeta,\Qcal_V}$, we write the irrelevant component of the expression that first appeared as the integrand in \eqref{E:pertintgrad} as
\begin{equation}
\widehat{K}_0(\Kcal,\Hcal)(X,v)=\ex^{-\Hcal(X,v)}\prod_{x\in X}\Kcal(v(x)).
\end{equation}
With this notation in place, the dynamic $\Tcal$ is defined as a map
\begin{align}
\Tcal: &\quad \bE_{\zeta,\Qcal_V}\times M_0(B_0,\mathfrak{V}_N)\times \Zcal\to \Zcal \nonumber \\
&\quad (\Kcal,\Hcal,Z)\mapsto Z',
\end{align}
where the components of $Z'$ are determined as follows:
\begin{align}
H_0' &= \big(\bA_0^{\ssup{\bq(\Hcal)}}\big)^{-1}\Big(H_1-\bB_0^{\ssup{\bq(\Hcal)}}\widehat{K}_0(\Kcal,\Hcal)\Big); \\
H_k' &= \big(\bA_k^{\ssup{\bq(\Hcal)}}\big)^{-1}\Big(H_{k+1}-\bB_k^{\ssup{\bq(\Hcal)}}K_k\Big),\text{ for } 1\le k\le N-2; \\
H_{N-1}' &= - \big(\bA_{N-1}^{\ssup{\bq(\Hcal)}}\big)^{-1} \bB_{N-1}^{\ssup{\bq(\Hcal)}}K_{N-1}; \\
K_1' & =\boldsymbol{S}_0(H_0,\widehat{K}_0(\Kcal,\Hcal),\bq(\Hcal));\text{ and} \\
K_{k+1}' &= \boldsymbol{S}_k(H_k,K_k,\bq(\Hcal)),\text{ for } 1\le k\le N-1.
\end{align}
The operators $\bA_k^{\ssup{\bq}}$, $\bB_k^{\ssup{\bq}}$ and $\boldsymbol{S}_k$ are as defined in section \ref{S4:RGmap} above. This definition ensures that a sequence $Z$ is a fixed point of $\Tcal$ if, and only if, $(H_{k+1},K_{k+1})=\bT_k(H_k,K_k,\bq(\Hcal))$ for all $0\le k\le N-1$ with $H_N=0$ and $K_0=\widehat{K}_0(\Kcal,\Hcal)$. By virtue of Proposition \ref{P:Tk} and \eqref{E:successint}, this means that a fixed point of $\Tcal$ satisfies
\begin{equation} \label{E:fixedpointidentity}
\int_{\cbO_N}\;\big(\ex^{-H_0}\circ\widehat{K}_0(\Kcal,\Hcal)\big)(\Lambda_N,v+\eta)\,\nu^{\ssup{\bq(\Hcal)}}(\d\eta)=\int_{\cbO_N}\;\Big(1+K_N(\Lambda_N,v+\eta)\Big)\,\nu_{N+1}^{\ssup{\bq(\Hcal)}}(\d\eta).
\end{equation}
The main result of this step guarantees the existence of a unique fixed point $Z(\Hcal,\Kcal)$, which depends smoothly on its parameters, for all sufficiently 
small $\Hcal$ and $\Kcal$.
\begin{theorem} [Theorem 12.1 of \cite{ABKM}] \label{T:Tfp}
Let $\kappa(L)$ and $\delta(L)$ be as in Theorem \ref{T:wf} and choose $L_0$, $h_0(L)\ge \delta(L)^{-1/2}\vee 1$ and $A_0$ such that the conclusions of Theorems \ref{T:smooth} and \ref{T:linearisation} hold for every triplet $(L,h,A)$ with $L\ge L_0$, $h\ge h_0(L)$ and $A\ge A_0(L)$. Then, for every such triplet, there exist constants $\rho_1(h,A)>0$, $\rho_2(L)>0$ and $\overline{C}_{i_1,i_2,i_3}(L,A)$ such that $\Tcal$ is smooth on $B_{\rho_1}(0)\times B_{\rho_2}(0)\times B_{\rho(A)}(0)$, where $\rho(A)$ is given by Theorem \ref{T:smooth}, $\bq(\Hcal)\in B_{\kappa}(0)$ for all $\Hcal\in B_{\rho_2}(0)$ and its derivatives are bounded as
\begin{equation} \label{E:Tderivbound}
\frac{1}{i_1!i_2!i_3!}\norm{D_{\Kcal}^{i_1}D_{\Hcal}^{i_2}D_Z^{i_3}\Tcal(\Kcal,\Hcal,Z)(\dot{\Kcal}^{i_1},\dot{\Hcal}^{i_2},\dot{Z}^{i_3})}_{\Zcal}\le \overline{C}_{i_1,i_2,i_3}(L,A)\norm{\dot{\Kcal}}_{\zeta,\Qcal_V}^{i_1}\norm{\dot{\Hcal}}_{0,0}^{i_2}\norm{\dot{Z}}_{\Zcal}^{i_3}
\end{equation}
for all $(\Kcal,\Hcal,Z)\in B_{\rho_1}(0)\times B_{\rho_2}(0)\times B_{\rho(A)}(0)$. Moreover, there exist constants $\eps(L,h,A)>0$, $\eps_1(L,h,A)>0$, $\eps_2(L,h,A)>0$ and $\widehat{C}_{i_1,i_2}(L,A)$ such that, for each $(\Kcal,\Hcal)\in B_{\eps_1}(0)\times B_{\eps_2}(0)$, there exists a unique $Z(\Kcal,\Hcal)\in B_{\eps}(0)$ with $\Tcal(\Kcal,\Hcal,Z(\Kcal,\Hcal))=Z(\Kcal,\Hcal)$. The map $(\Kcal,\Hcal)\mapsto Z(\Kcal,\Hcal)$ is smooth on $B_{\eps_1}(0)\times B_{\eps_2}(0)$ and its derivatives satisfy the bounds
\begin{equation}
\frac{1}{i_1!i_2!}\norm{D_{\Kcal}^{i_1}D_{\Hcal}^{i_2}Z(\Kcal,\Hcal)(\dot{\Kcal}^{i_1},\dot{\Hcal}^{i_2})}_{\Zcal}\le \widehat{C}_{i_1,i_2}(L,A)\norm{\dot{\Kcal}}_{\zeta,\Qcal_V}^{i_1}\norm{\dot{\Hcal}}_{0,0}^{i_2}
\end{equation}
for all $(\Kcal,\Hcal)\in B_{\eps_1}(0)\times B_{\eps_2}(0)$.
\end{theorem}
As alluded to previously, a proof of this result is obtained by adapting the reasoning in Chapter 12 of \cite{ABKM} to the present setting and hence omitted from this work. For completeness, we set out below a choice of parameters for which the result of Theorem \ref{T:Tfp} holds. For the domain of the dynamic $\Tcal$, we may take the relevant constants as
\begin{equation}
\rho_1(h,A)=\frac{\rho(A)}{8A h^{r_0}}\,\text{ and }\,\rho_2(L)=\min\Big\{\frac{1}{16},\frac{h^2\kappa(L)}{2}\Big\}.
\end{equation}
For the triplet $(\eps_1,\eps_2,\eps)$ that characterises the fixed point map $Z(\Kcal,\Hcal)$, any choice with $\eps_1\le\rho_1$, $\eps_2\le\rho_2$ and $\eps\le\rho(A)/2$ is valid provided that it satisfies the additional constraints
\begin{equation}
2\overline{C}_{0,0,2}\,\eps+\overline{C}_{1,0,1}\,\eps_1+\overline{C}_{0,1,1}\,\eps_2\le\frac{1}{8}\,\text{ and }\,\overline{C}_{1,0,0}\,\eps_1<\frac{\eps}{8},
\end{equation}
where $\overline{C}_{i_1,i_2,i_3}$ are the constants in \eqref{E:Tderivbound}.

\subsection{Step 8: Main representation theorem}

By comparing \eqref{E:fixedpointidentity} to \eqref{E:pertintgrad}, we see that the fixed point $Z(\Kcal,\Hcal)$ provided by Theorem \ref{T:Tfp} achieves the desired representation of the perturbative component of the partition function $Z_{N,\beta}(u)$, but for the fact that $\Hcal$ need not equal $H_0$. This is resolved by a second application of the implicit function theorem to deduce that, on a suitable domain, there is a smooth map $\Hcal=\widehat{\Hcal}(\Kcal)$ such that $Z_{H_0}(\Kcal,\widehat{\Hcal}(\Kcal))=\widehat{\Hcal}(\Kcal)$, where the subscript $H_0$ indicates a projection onto the appropriate coordinate of the image $Z$ under the fixed point map. Using the same notation as before, we obtain the following.
\begin{lemma} [Lemma 12.6 of \cite{ABKM}] \label{L:Hfp}
Under the assumptions of Theorem \ref{T:Tfp}, there is a constant $\tilde{\varrho}>0$, independent of $N$, and a map $\widehat{\Hcal}:B_{\tilde{\varrho}}(0)\subset\bE_{\zeta,\Qcal_V}\to B_{\eps_2}$ such that
\begin{equation}
Z_{H_0}(\Kcal,\widehat{\Hcal}(\Kcal))=\widehat{\Hcal}(\Kcal)\text{ and }\bq(\widehat{\Hcal}(\Kcal))\in B_{\kappa}(0)
\end{equation}
for all $\Kcal\in B_{\tilde{\varrho}}(0)$. The map $\widehat{\Hcal}$ is smooth and its derivatives are bounded uniformly in $N$.
\end{lemma}
A valid choice for the parameter $\tilde{\varrho}$, which determines the domain of the map $\widehat{\Hcal}$, is given by
\begin{equation}
\tilde{\varrho}=\min\Big\{\frac{\rho_0}{2\widehat{C}_{1,0}},\frac{1-4\widehat{C}_{0,2}\rho_0}{2\widehat{C}_{1,1}},\eps_1\Big\},
\end{equation}
where $\rho_0=\min\{\eps_2/2,(8\widehat{C}_{0,2})^{-1}\}$ and the constants $\eps_1$, $\eps_2$ and $\widehat{C}_{i_1,i_2}$ are as in Theorem \ref{T:Tfp}. Combined with the result of step 7, this permits us to recover the main renormalisation group representation of the partition function from \cite{ABKM} in our setting.
\begin{theorem} [Theorem 11.1 of \cite{ABKM}] \label{T:mainrep}
On the assumptions of Theorem \ref{T:Tfp} and for a choice of constants consistent with it, for every triplet $(L,h,A)$ with $L\ge L_0$, $h\ge h_0(L)$ and $A\ge A_0(L)$, there exists a $\varrho=\varrho(L,h,A)>0$ such that for each $N\ge 1$ there are smooth maps $\lambda_N: B_{\varrho}(0)\subset \bE_{\zeta,\Qcal_V}\to \R$, $\bq_N: B_{\varrho}(0)\to B_{\kappa}(0)$ and $K_N: B_{\varrho}(0)\to \bM(\Pcal_N^c,\mathfrak{V}_N)$, whose derivatives are bounded uniformly in $N$, with the following properties. The equality
\begin{equation}
\int_{\cbO_N}\;\sum_{X\subset\mathbb{T}_N}\prod_{x\in X}\Kcal(\eta(x))\,\nu^{\ssup{0}}(\d\eta)=\frac{Z_N^{\ssup{\bq_N(\Kcal)}}\ex^{L^{Nd}\lambda_N(\Kcal)}}{Z_N^{\Qcal_V}}\int_{\cbO_N}\;\Big(1+K_N(\Kcal)(\Lambda_N,\eta)\Big)\,\nu_{N+1}^{\ssup{\bq_N(\Kcal)}}(\d\eta)
\end{equation}
holds for all $\Kcal\in B_{\varrho}(0)$. There are constants $\widetilde{C}_\ell(L,h,A)>0$ such that
\begin{equation} \label{E:KNbound}
\norm{D_{\Kcal}^{\ell} K_N(\Kcal)(\dot{\Kcal},...,\dot{\Kcal})}_N^{(A)}\le \widetilde{C}_\ell(L,h,A) \vartheta^N\norm{\dot{\Kcal}}_{\zeta,\Qcal_V}^{\ell}
\end{equation}
for every $\ell\ge 0$. Finally, we have the bound
\begin{equation} \label{E:KNbound1}
\int_{\cbO_N}\;\abs{K_N(\Kcal)(\Lambda_N,\eta)}\,\nu_{N+1}^{\ssup{\bq_N(\Kcal)}}(\d\eta)\le \frac{1}{2}.
\end{equation}
\end{theorem}
To obtain \eqref{E:KNbound1} for all sufficiently large $N$ it is enough to take $\varrho=\tilde{\varrho}$, the constant from Lemma \ref{L:Hfp}. Otherwise, a sufficient constraint for any given $N$ is
\begin{equation}
\varrho\le\frac{A}{\Acal_{\Bcal}\widetilde{C}_1}\vartheta^{-N}.
\end{equation}
These statements may be derived from our previous steps above by following closely the arguments in Chapters 11 and 12 of \cite{ABKM}. A proof is therefore omitted.

Theorem \ref{T:mainrep} is the main technical statement of \cite{ABKM}, adapted to our setting and notation: the results on the surface tension and scaling limit in \cite{ABKM}, which represent the starting point for this work, are a direct consequence of this theorem. It is also the main tool for our argument to relate the tuning parameter $\bq_N(\Kcal)$ to the Hessian of surface tension. Theorem \ref{T:mainrep} is given in terms of an abstract perturbation function $\Kcal\in\bE_{\zeta,\Qcal_V}$. It therefore remains to connect this statement to the concrete function $\Kcal_{u,\beta,V}$ introduced in \eqref{E:KubV} and make the dependence on the tilt $u$ explicit.

\subsection{Step 9: Relating the abstract perturbation to $V$ and $\sigma_\beta(u)$}

We first make a connection with surface tension as a function of the tilt. We pass from \eqref{E:ZNuKgrad} to the surface tension to obtain
\begin{equation} \label{E:sigmaNexp}
\sigma_{N,\beta}(u)=V(u)-\frac{1}{\beta L^{Nd}}\log Z_{N,\beta}^{\ssup{0}} + \frac{\widetilde{\sigma}_N(\Kcal_{u,\beta,V})}{\beta},
\end{equation}
in which expression the perturbative component,
\begin{equation}
\widetilde{\sigma}_N(\Kcal_{u,\beta,V})=-\frac{1}{L^{Nd}}\log\int_{\cbO_N}\;\sum_{X\subset\mathbb{T}_N}\prod_{x\in X}\Kcal_{u,\beta,V}(\eta(x))\,\nu^{\ssup{0}}(\d\eta),
\end{equation}
corresponds to the integral in \eqref{E:pertintgrad} after inserting the particular perturbation function $\Kcal_{u,\beta,V}$. The corresponding object in \cite{ABKM} is denoted by $\Wcal_N$; we relabel it to be consistent with our notation for the surface tension. We recall here the constant $r_1=3$ listed among the fixed parameters at the outset of this analysis. As the next result makes clear, it controls the number of derivatives of $\sigma_\beta$ we eventually achieve.
\begin{theorem} [Theorem 2.9 of \cite{ABKM}] \label{T:sigmatilde}
Let $L\ge L_0$ and $\varrho$ be as in Theorem \ref{T:mainrep} and assume that $\Qcal_V$ satisfying \eqref{E:Assumptions} is given. For any integer $N\ge 1$, any open set $\Ocal\subset\R^d$ and any map
\begin{align}
\Ocal &\to \bE_{\zeta,\Qcal_V} \nonumber \\
u &\mapsto \Kcal_u
\end{align}
that belongs in the $C^{r_1}(\Ocal)$ class and satisfies the bounds
\begin{align}
\sup_{u\in\Ocal} \norm{\Kcal_u}_{\zeta,\Qcal_V} &< \varrho \label{E:Kucond1}\\
\Theta = \sup_{u\in\Ocal}\sum_{\abs{\gamma}\le r_1}\frac{1}{\gamma !}\norm{\partial_u^{\gamma}\Kcal_u}_{\zeta,\Qcal_V} &< \infty, \label{E:Kucond2}
\end{align}
the function $u\mapsto \widetilde{\sigma}_N(\Kcal_u)$ is $C^{r_1}(\Ocal)$ and its partial derivatives $\abs{\partial_u^\alpha\widetilde{\sigma}_N(\Kcal_u)}$ for $\abs{\alpha}\le r_1$ are bounded in terms of $\Theta$. Consequently, there exists a function $\widetilde{\sigma}\in C^{r_1-1,1}(\Ocal)$ and a sub-sequence $(N_\ell)$ such that $\widetilde{\sigma}_{N_\ell}\to\widetilde{\sigma}$ in $C^{r_1-1}(\Ocal)$ as $\ell\to\infty$. The derivatives of $\widetilde{\sigma}$ of order up to $r_1-1$, and the Lipschitz constant of the $(r_1-1)$-st derivatives, are bounded in terms of $L$ and $\Theta$.
\end{theorem}

The final link relates the preceding statement to our assumptions on $V$ in \eqref{E:Assumptions}. To state the next proposition in a convenient form, we designate by the letter $\Psi$ the function
\begin{equation}
\Psi(t) = \sup_{z\in\R^d:\abs{z}\le t}\;\sum_{3\le\abs{\alpha}\le r_0+r_1}\frac{1}{\alpha !}\abs{\partial^\alpha V(z)}
\end{equation}
that features in the fourth equality in \eqref{E:Assumptions}.
\begin{prop} [Proposition 2.11 of \cite{ABKM}] \label{P:Vass}
Let $V$ satisfy the assumptions in \eqref{E:Assumptions} with $r_0=r_1=3$ and $\varrho>0$ be given. There exist positive reals $\zeta=\zeta(\omega,\omega_0)$, $\delta_0=\delta_0(\omega,\omega_0,r_0,\Psi,\varrho)$ and $\Theta=\Theta(\omega,r_0,r_1,\Psi)$, as well as a constant $\beta_0=\beta_0(\omega,r_0,\Psi,\varrho)\ge 1$ such that, for all $\beta\ge\beta_0$, the following hold:
\begin{enumerate}
\item The $B_{\delta_0}(0)\subset\R^d\to \bE_{\zeta,\Qcal_V}$ map $u\mapsto\Kcal_{u,\beta,V}$ is $C^{r_1}$; \\
\item  
$\displaystyle
\norm{\Kcal_{u,\beta,V}}_{\zeta,\Qcal_V} < \varrho\; \text{; and} $
\hfill
\refstepcounter{equation}\textup{(\theequation)}
\vspace{0.5cm}
\item 
$\displaystyle
\sum_{\abs{\gamma}\le r_1}\frac{1}{\gamma !}\norm{\partial_u^{\gamma}\Kcal_{u,\beta,V}}_{\zeta,\Qcal_V} \le \Theta, $
\hfill
\refstepcounter{equation}\textup{(\theequation)}
\vspace{0.1cm}
\end{enumerate}
in each case, for all $u\in B_{\delta_0}(0)$.
\end{prop}
This completes the transposition of the results of \cite{ABKM} into our setting and thereby establishes the starting point for the remaining sections of this article.

\section{Extended renormalisation group flow} \label{S:6}

\subsection{External fields} \label{S6:1}

As outlined in section \ref{S:3}, the principal objective of our extended renormalisation group analysis is to control the integral
\begin{equation} \label{E:extfieldint}
\int_{\cbO_N}\;\Big(\ex^{-H_0}\circ K_0\Big)(\mathbb{T}_N,\eta+\mathscr{C}^{\ssup{\bq},\nabla}\boldsymbol{\nabla}_N g_\eps)\,\nu^{\ssup{\bq}}(\d\eta)
\end{equation}
as a function of the scale parameter $\eps$, uniformly in $N$. The above expression anticipates that the representation of the partition function in terms of the functionals $(H_k,K_K)$ remains an important tool. Although the support of $g_\eps$ is small for $N$ sufficiently large, the covariance $\mathscr{C}^{\ssup{\bq},\nabla}$ does not have finite range. Without more, the coupling to the gradient vector field $\mathscr{C}^{\ssup{\bq},\nabla}\boldsymbol{\nabla}_N g_\eps$ therefore amounts to a global perturbation of the renormalisation flow in \eqref{E:extfieldint}. We avail of the finite range decomposition of $\mathscr{C}^{\ssup{\bq},\nabla}$ to split $\mathscr{C}^{\ssup{\bq},\nabla}\boldsymbol{\nabla}_N g_\eps$ into a sequence of fields, the \emph{external fields} alluded to in section \ref{S:3}, with suitably bounded support. At each scale $k$, the corresponding external field - denoted by $\psi_k$ - is taken into account as a local perturbation that consequently affects only a subset of polymers. This gives rise to an extended renormalisation flow that keeps track of the impact of the $\psi_k$ scale-by-scale. It does not displace the sequence $(H_k,K_k)$ derived above, which we refer to as the bulk renormalisation group flow, since not all (indeed most) polymers are unaffected by the external fields.

As the external fields depend on the test function, it is convenient to settle some notation for the discretised fields $f_\eps$ and $g_\eps$. We recall that $g_\eps(x)=\eps^{\frac{d-2}{2}}g(\eps x)$ for some smooth function $g\in C_c^\infty(\R^d)$. Since the gradients of $g_\eps$ are tested against a translation-invariant vector field, we may assume without of loss of generality that the support of $g_\eps$ is centred in $\Z^d$ and, for all $N$ large enough, regard it as a field in $\boldsymbol{\Vcal}_N$. We then let
\begin{equation}
f_\eps(x)=\Delta g_\eps(x)=\sum_{i=1}^d \nabla_i^* \nabla_i g_\eps(x)
\end{equation}
and observe that $f_\eps\in\cbX_N$. Since only the gradients of $g_\eps$ are of concern, we may take $g_\eps\in\cbX_N$ as well. The scaling and compact support of the function $g$ imply that there is a constant $C_f>0$ (independent of $\eps$) such that the following assumptions hold:
\begin{equation} \label{E:fassump}
\begin{cases}
\abs{\nabla^\alpha f_\eps}_\infty \le C_f \eps^{\frac{d+2}{2}+\abs{\a}}\quad\text{ for all }0\le \abs{\a}\le M; \\
\abs{\nabla^\alpha g_\eps}_\infty \le C_f \eps^{\frac{d-2}{2}+\abs{\a}}\quad\text{ for all }0\le \abs{\a}\le M; \\
\bigcup_{0\le \abs{\a}\le M}\text{supp}\,(\nabla^\a f_e)\subset (-C_f \eps^{-1},C_f \eps^{-1})^d\bigcap\Z^d\subset\Lambda_N.
\end{cases}
\end{equation}
It follows that, for $k$ (and $N$) large enough, the support of $f_\eps$ and any of its discrete derivatives that appear in our analysis is contained in the unique $k$-block $B_k^0\in \Bcal_k$ such that $0\in B_k^0$. We refer to the smallest scale $\bar{k}=\bar{k}(\eps)$ such that this holds as the \emph{smoothness scale} of $f_\eps$. More precisely, we require the support of $f_\eps$ to be sufficiently distant from the boundary of $B_{\bar{k}}^0$ and so we impose the following requirement:
\begin{equation} \label{E:kfdef}
\bar{k}=\min\big\{k:(-C_f \eps^{-1},C_f \eps^{-1})^d\cap\Z^d\subset \big[-\floor{L^k/4},\floor{L^k/4}\big]^d\cap\Lambda_N\big\}.
\end{equation}

Most of the detailed calculations involving the external fields are carried out at the level of scalar fields in $\cbX_N$, and so it is opportune to introduce them as such. Accordingly, we give the following definition of the external fields $\psi_k$ for $k=1$ to $k=N$:
\begin{equation}
\psi_k= \begin{cases}
0 & \text{if}\quad k<\bar{k} \\
\sum_{j=1}^{\bar{k}-1}\mathscr{C}_j^{\ssup{\bq}}f_\eps & \text{if}\quad k=\bar{k} \\
\mathscr{C}_{k-1}^{\ssup{\bq}}f_\eps & \text{if}\quad \bar{k}<k<N \\
\big(\mathscr{C}_{N-1}^{\ssup{\bq}}+\mathscr{C}_N^{\ssup{\bq}}+\mathscr{C}_{N+1}^{\ssup{\bq}}\big)f_\eps & \text{if}\quad k=N.
\end{cases}
\end{equation}
Manifestly, we have $\sum_{k=1}^N \psi_k=\mathscr{C}^{\ssup{\bq}}f_\eps$. These objects readily translate into the vector field setting. We have $\boldsymbol{\nabla}_N \mathscr{C}_j^{\ssup{\bq}}f_\eps=\mathscr{C}_j^{\ssup{\bq},\nabla}\boldsymbol{\nabla}_N g_\eps$, from which an analogous decomposition of the gradient vector field $\mathscr{C}^{\ssup{\bq},\nabla}\boldsymbol{\nabla}_N g_\eps$ in \eqref{E:extfieldint} derives.

\subsection{Coordinates of the extended flow} \label{S6:2}

For now and until the completion of our argument in section \ref{S:10}, we work with a fixed $\eps$ and hence a fixed $\bar{k}$. For any $N\gg\bar{k}$, we assume that the tuning parameter $\bq=\bq(N,u)$, as well as the sequence of functionals $(H_k,K_k)$, have been constructed as described in section \ref{S:5} such that Theorem \ref{T:mainrep} holds. Using the bulk renormalisation group flow as reference background, we compute the impact of the external fields $\psi_k$ on the successive integration steps from $k=\bar{k}$ to $k=N$. To do this, we introduce a sequence of functionals $K_k(<k):\Pcal_k^c \times\mathfrak{V}_N\to\R$ that capture the decaying part and a sequence $e_k$ (of, essentially, real numbers) that capture the residual error at scale $k$. For consistency, we relabel the irrelevant components of the bulk flow $K_k(0)$.

To simplify some of our expressions, we begin by introducing the shorthand notation
\begin{multline} \label{E:Psidef}
\Psi_k(X,v)=-K_k(X,v,<k)+ \\
\sum_{Y\in\Pcal_k(X)}\Big(e^{-H_k(\cdot,v+\boldsymbol{\nabla}_N\psi_k)}-e^{-H_k(\cdot,v)}\Big)^{X\setminus Y}K_k(Y,v+\boldsymbol{\nabla}_N\psi_k,<k)
\end{multline}
for $X\in\Pcal_k^c$, and we extend $\Psi_k$ and $K_k(<k)$ to arbitrary polymers in $\Pcal_k$ via the inclusion in \eqref{E:canincl2}. Note that, for scales $k<N$, our constraints on the fields $\psi_k$ ensure that $\psi_k\equiv 0$ on any polymer neighbourhood $X^*$ whenever $B_k^0\cap X=\emptyset$. This entails that $\Psi_k(X,v)\equiv 0$ for any such $X$ and, if $B_k^0\in\Pcal_k(X)$, the only non-trivial terms in \eqref{E:Psidef} are
\begin{multline}
\Psi_k(X,v)=K_k(X,v+\boldsymbol{\nabla}_N\psi_k,<k)-K_k(X,v,<k) \\
+\Big(e^{-H_k(B_k^0,v+\boldsymbol{\nabla}_N\psi_k)}-e^{-H_k(B_k^0,v)}\Big)K_k(X\setminus B_k^0,v+\boldsymbol{\nabla}_N\psi_k,<k).
\end{multline}
The virtue of this construction is that we can write the integrand in the presence of the external field at scale $k$ (for $k\ge\bar{k}$) as follows:
\begin{equation}
\Big(e^{-H_k}\circ K_k(<k)\Big)(\Lambda_N,v+\boldsymbol{\nabla}_N\psi_k)=\Big(e^{-H_k}\circ (K_k(<k)+\Psi_k)\Big)(\Lambda_N,v).
\end{equation}
To complete this initial setup, we introduce $e_k:\Bcal_k \times \mathfrak{V}_N\to\R$ as a constant functional,
\begin{equation}
e_k(B,v)=\1_{B=B_k^0}e_k(B_k^0,0),
\end{equation}
and treat it as a real number where the context permits no ambiguity.

The functionals $K_k(<k)$ are defined iteratively, similar to the bulk renormalisation group flow. For a connected polymer $U\in\Pcal_{k+1}^c$, we let
\begin{multline} \label{E:K<kdef}
K_{k+1}(U,v,<k+1)=\sum_{X\in\Pcal_k}\chi(X,U)\Big(e^{-(e_{k+1}-e_k)}\Big)^X\Big(e^{-\widetilde{H}_k}\Big)^{U\setminus X}\Big(e^{\widetilde{H}_k}\Big)^{X\setminus U}\times \\
\int_{\cbO_N}\;\Big(1-e^{-\widetilde{H}_k+e_{k+1}-e_k}\Big)\circ\Big(e^{-H_k}-1\Big)\circ\big(K_k(<k)+\Psi_k\big)(X,v+\eta)\,\nu_{k+1}^{\ssup{\bq}}(\d\eta),
\end{multline}
where $H_k$ and $\widetilde{H}_k$ are given exactly as in section \ref{S:4} for the bulk flow. We designate this mapping as $K_{k+1}(U,v,<k+1)=\overline{\bS}_k(H_k,K_k(0),K_k(<k))$, by analogy with the irrelevant component of the renormalisation map $\bS_k$. Note that we tacitly assume here that all relevant functionals are integrable. This is formally established in section \ref{S:8}, where the boundedness of the corresponding integration maps is shown pursuant to the proof that $\overline{\bS}_k$ is smooth with respect to the norms introduced further below. Subject to that caveat, this definition results in the desired replication of the perturbed integrand scale-by-scale (see Lemma \ref{L:K<kprop}):
\begin{equation}
\bR_{k+1}^{\ssup{\bq}}\Big(e^{e_k}\big(e^{-H_k}\circ K_k(<k)\big)\Big)(\Lambda_N,v+\boldsymbol{\nabla}_N \psi_k)=e^{e_{k+1}}\Big(e^{-H_{k+1}}\circ K_{k+1}(<k+1)\Big)(\Lambda_N,v).
\end{equation}
The choice of error constants $e_k$ for $k\ge\bar{k}$ is free in the above expressions. We determine them through the recursion
\begin{equation}
e_{k+1}-e_k=\int_{\cbO_N}\; K_k(B_k^0,\eta,<k)+\Psi_k(B_k^0,\eta)-K_k(B_k^0,\eta,0)\,\nu_{k+1}^{\ssup{\bq}}(\d\eta),
\end{equation}
and show by a brief computation in the next subsection that this is indeed the correct choice. The departure point for this construction is the smoothness scale, that is, we set $e_{\bar{k}}=0$ and $K_{\bar{k}}(<\bar{k})=K_{\bar{k}}(0)$.

We observe that the functionals we have introduced above are not translation-invariant. In particular, it holds that $K_k(X,<k)=K_k(X,0)$ if $0\notin X$, but $K_k(X,<k)\neq K_k(X,0)$ in general. We verify at the end of this section that our norms (and various inequalities and other results shown in relation to them) extend naturally to the objects defined for the purposes of the extended renormalisation group flow.

\subsection{Motivation and the key estimate for $K_k(<k)$} \label{S6:3}

Before proceeding to the more technical parts of our analysis, we aim to give a brief explanation why the definitions in the preceding two subsections can be expected to achieve their purposes. By analogy with Lemma \ref{L:K'}, we first collect in one statement some basic features of the extended renormalisation group flow introduced above. We take the following lemma subject to the assumption (as highlighted previously) that all objects exist and are well-defined.
\begin{lemma} \label{L:K<kprop}
On the assumptions of Theorem \ref{T:mainrep} and for any choice of constants consistent with it, the functionals $K_k(<k)$ defined in subsection \ref{S6:2} satisfy the following:
\begin{enumerate}
\item $K_k(X,<k)$ is of class $C^{r_0}(\mathfrak{V}_N)$ and local for all $X\in\Pcal_k$ and for all $k=\bar{k},\ldots,N$; \\
\item $K_k(X,<k)=K_k(X,0)$ whenever $B_k^0\cap X=\emptyset$ for all $k=\bar{k},\ldots,N-1$; \\
\item $K_k(<k)$ factors at scale $k$ for all $k=\bar{k},\ldots,N$; and
\item For all $k=\bar{k},\ldots,N-1$, it holds that
\begin{multline}
\quad\quad\quad\bR_{k+1}^{\ssup{\bq}}\Big(e^{e_k}\big(e^{-H_k}\circ K_k(<k)\big)\Big)(\Lambda_N,v+\boldsymbol{\nabla}_N \psi_k) \\
=e^{e_{k+1}}\Big(e^{-H_{k+1}}\circ K_{k+1}(<k+1)\Big)(\Lambda_N,v).
\end{multline}
\end{enumerate}
\end{lemma}
\begin{proof}
The proof of (1)-(3) is by induction. At the smoothness scale, $K_{\bar{k}}(<\bar{k})=K_{\bar{k}}(0)$ by definition and $K_{\bar{k}}(0)$ has the claimed properties according to Lemma \ref{L:K'}. For the induction step, we see that $\Psi_k$ inherits locality and $C^{r_0}(\mathfrak{V}_N)$-differentiability from $K_k(<k)$ and $H_k$. The same therefore holds for $K_{k+1}(<k+1)$, having regard to Lemma \ref{L:Rreg}, following the same reasoning as in Lemma \ref{L:K'}.

Suppose next, that $U\in\Pcal_{k+1}$ is such that $B_{k+1}^0\notin\Bcal_{k+1}(U)$. The range of $\mathscr{C}_j^{\ssup{\bq}}$ for $j\le k-1$ is bounded by $L^{k-1}/2$; $\psi_k$ (for $k<N$) is therefore supported in a cube of side length no greater than $L^{\bar{k}}/2+L^{k-1}<L^k$ centred on $0\in\Lambda_N$. Accordingly, $\psi_k$ is supported in $B_k^0$ as claimed previously and $U^*\cap B_k^0=\emptyset$ as $L$ is chosen large enough so that $L>2^{d+1}+1$ by assumption. Since $\chi(X,U)=0$ unless $X^*\subset U^*$, it follows that only polymers strictly disjoint from $B_k^0$ contribute to the expression defining $K_{k+1}(U,<k+1)$. This means that, in the circle product appearing in \eqref{E:K<kdef} for such polymers $X$, the relevant terms satisfy
\begin{equation}
(K_k(<k)+\Psi_k)(Y)=K_k(Y,0)
\end{equation}
for any $Y\subset X$, by the induction hypothesis. The expression in \eqref{E:K<kdef} thus collapses to the definition of $K_{k+1}(U,v)$ in \eqref{E:K'}. This confirms property (2) at scale $k+1$. The factorisation asserted in (3) is shown as for the bulk renormalisation group flow in Lemma \ref{L:K'}.

Finally, the computation to verify the identity in (4), which is purely formal at this stage, runs very similar to the general case (see Chapter 6.3 of \cite{ABKM}) and hence omitted.
\end{proof}

Our main technical goal is to show that the functionals $K_k(<k)$ decay in a norm that fulfils an analogous role to the $\norm{\cdot}_k^{(A)}$-norm in relation to the $K_k(0)$ functionals. This can only hold if the error term $e_{k+1}-e_k$ extracts "enough" in an appropriate sense at each scale. As before, the principal tool to establish this desired contraction is the linearisation of $\overline{\boldsymbol{S}}_k$ near zero, which is appropriate since we expect all arguments to be small in their respective norms. A fairly straightforward computation from \eqref{E:K<kdef} yields an expression
\begin{align} \label{E:DSzero}
&D\overline{\boldsymbol{S}}_k  (0,0,0)(H_k,K_k(0),K_k(<k))(U,v)=\sum_{\substack{X\in\Pcal_k^c\setminus\Bcal_k \\ \pi_k(X)=U}}\int_{\cbO_N}\;K_k(X,<k,v+\boldsymbol{\nabla}_N\psi_k+\eta)\,\nu_{k+1}^{\ssup{\bq}}(\d\eta) \nonumber \\
&\;+\sum_{\substack{B\in\Bcal_k \\ \overline{B}=U}}\big(1-\Pi_2\big)\int_{\cbO_N}\;K_k(X,0,v+\eta)\,\nu_{k+1}^{\ssup{\bq}}(\d\eta) + \1_{U=B_{k+1}^0}\bI^0(H_k,K_k(0),K_k(<k))(v),
\end{align}
where the symbol $\bI^0$ stands for the integral
\begin{align} \label{E:Izero}
\bI^0(H_k, & K_k(0),K_k(<k))(v) = \int_{\cbO_N}\;\big(K_k(B_k^0,<k,v+\boldsymbol{\nabla}_N\psi_k+\eta)-K_k(B_k^0,<k,\psi_k+\eta)\big) \nonumber \\
&-\big(K_k(B_k^0,0,v+\eta)-K_k(B_k^0,0,\eta)\big)-\big(H_k(B_k^0,v+\boldsymbol{\nabla}_N\psi_k+\eta)-H_k(B_k^0,\boldsymbol{\nabla}_N\psi_k+\eta)\big) \nonumber \\
&+\big(H_k(B_k^0,v+\eta)-H_k(B_k^0,\eta)\big)\,\nu_{k+1}^{\ssup{\bq}}(\d\eta).
\end{align}
The first two sums in \eqref{E:DSzero} are familiar from the bulk flow and can be handled similarly. The main difference is that an estimate of the integral
\begin{equation}
\int_{\cbO_N}\; \bunderline{w}_k^X(v+\boldsymbol{\nabla}_N\psi_k+\eta)\,\nu_{k+1}^{\ssup{\bq}}(\d\eta)
\end{equation}
is needed to play the role of the bounds in Theorem \ref{T:wf}(9) and (10). The task of proving such an estimate occupies the largest part of section \ref{S:7} below. The additional terms for the special block $B_k^0$ in \eqref{E:Izero} are amenable to a first order Taylor estimate using similar arguments as in the bulk flow. This can only be expected to result in a contraction by a factor of $L^{-d/2}$, which is however sufficient in this case since these terms only appear for a single $k$-block rather than $L^d$ $k$-blocks contained in $U\in\Bcal_{k+1}$. Since our bounds for the integrated weight functions in the presence of the external fields $\psi_k$ are initially quite poor, we actually only show that $\overline{\bS}_k$ contracts from some scale $k>\bar{k}$. Nothing turns on this as far as our main results are concerned.

\subsection{Extended norms} \label{S6:4}

Pursuant to the approach outlined above, we aim to show that the map $\overline{\bS}_k$ has uniformly bounded derivatives on a suitable domain. To identify such a domain we first need to adapt the norms introduced in section \ref{S4:Norms} so that they can accommodate functionals that differ depending on whether or not their polymer argument contains $B_k^0$.

The most straightforward change concerns the strong norm. We simply insist on an explicit supremum over all $k$-blocks,
\begin{equation}
\tnorm{F}_k=\sup_{B\in\Bcal_k}\tnorm{F}_{k,B}=\sup_{B\in\Bcal_k}\,\sup_{v\in\mathfrak{V}_N}\abs{F(B)}_{k,B,v}\bunderline{W}_k^{-B}(v),
\end{equation}
and note that this does not affect any previous definition or statement for translation-invariant $F$. For the weak norms, we observe that the definition of $\norm{\cdot}_k^{(A)}$ in \eqref{E:globalweaknorm} does not rely on translation-invariance, so we may consider the functionals $K_k(<k)$ in respect of this norm. We do, however, also require an additional weak norm that accounts for the potential impact of the external field $\psi_k$ through the weight functions and we signal this difference with an asterisk. We therefore define a set of local and global weak norms by letting
\begin{equation}
\norm{F(X)}_{k,X}^* = \sup_{v\in\mathfrak{V}_N}\abs{F(X)}_{k,X,v}\big(\bunderline{w}_k(v)\vee\bunderline{w}_k(v+\boldsymbol{\nabla}_N\psi_k)\big)^{-X} 
\end{equation}
and
\begin{equation} \label{E:Astarnorm}
\norm{F}_k^{(A),*} = \sup_{X\in\Pcal_k^c}A^{\abs{X}_k}\norm{F(X)}_{k,X}^*,
\end{equation}
where the exponent notation involving the weight functions is interpreted as follows:
\begin{equation}
\big(\bunderline{w}_k(v)\vee\bunderline{w}_k(v+\boldsymbol{\nabla}_N\psi_k)\big)^{-X}=\Big(\max\big\{\bunderline{w}_k^X(v),\bunderline{w}_k^X(v+\boldsymbol{\nabla}_N\psi_k)\big\}\Big)^{-1}.
\end{equation} 
Clearly, it holds that $\norm{\cdot}_k^{(A),*}\le\norm{\cdot}_k^{(A)}$. In rough outline, the $K_k(<K)$ functionals are measured in our usual $\norm{\cdot}_k^{(A)}$-norm; we then switch to the starred norms for $K_k(<k)+\Psi_k$ and revert to the unstarred norms after integrating out the $(k+1)$-fluctuation field.

We further need a series of norms that measure functionals at scales $k$ (and the intermediate scale ${k:k+1}$) that only factor at scale ${k+1}$. This is achieved in \cite{ABKM} through the concept of \emph{clusters}, which we have elided so far since it features only in the proof of Theorem 6.7 of \cite{ABKM} (Theorem \ref{T:smooth} in our numbering). We borrow this idea in section \ref{S:8} below and so we recall the definition and a number of key facts next.

We say that $X\in\Pcal_k\setminus\{\emptyset\}$ is a \emph{cluster} if $\pi_k(X)\in\Pcal_{k+1}^c$, and we denote the collection of all clusters at scale $k$ by the symbol $\Pcal_k^{\text{cl}}$. A polymer $Y\subset X\in\Pcal_k$ is a cluster of $X$, written $Y\in\Ccal_k^{\text{cl}}(X)$, if there exists a connected polymer $U\in\Ccal_{k+1}(\pi_k(X))$ such that
\begin{equation}
Y=\bigcup_{Z\in\Ccal_k(X):\pi_k(Z)\subset U} Z.
\end{equation}
In fact, the relationship between clusters and connected components of $\pi_k(X)$ is one-to-one and any (non-empty) polymer is the union of its clusters. The salient features of this notion for our purposes are summarised in the next statement, which is proved in Chapter 9 of \cite{ABKM}.
\begin{lemma} [Lemma 9.2 of \cite{ABKM}] \label{L:clusters}
Assume that $X\in\Pcal_k\setminus\{\emptyset\}$ for some $k=0,\ldots,N-1$. Then it follows that:
\begin{enumerate}
\item Given any $U\in\Ccal_{k+1}(\pi_k(X))$, there is a cluster $Y\in\Ccal_k^{\text{cl}}(X)$ such that $\pi_k(Y)=U$; \\
\item $X=\bigcup_{Y\in\Ccal_k^{\text{cl}}(X)}Y$; \\
\item For any $Y_1,Y_2\in\Ccal_k^{\text{cl}}(X)$, $Y_1\neq Y_2$ implies that $Y_1$ and $Y_2$ are strictly disjoint (at scale $k$). If, in addition, $L\ge 2^{d+3}$ then different clusters of $X$ are separated by at least $\textup{dist}_{\infty}(Y_1,Y_2)\ge \frac{3}{4}L^{k+1}$; and \\
\item $\sum_{Y\in\Ccal_k^{\text{cl}}(X)} \abs{\Ccal_k(Y)}=\abs{\Ccal_k(X)}$.
\end{enumerate}
\end{lemma}

It is clear from Lemma \ref{L:clusters} that a functional defined on clusters may be extended to arbitrary polymers by analogy with the inclusion for functionals defined on $\Pcal_k^c$ in \eqref{E:canincl2}. We define several norms on such functionals that feature an additional weight depending on the number of connected components. For clarity, we include the standard norms that do not explicitly depend on the external fields $\psi_k$. Given weights $A,B\ge 1$, we let
\begin{align}
\norm{K}_k^{(A,B)} &= \sup_{X\in\Pcal_k^{\textup{cl}}}A^{\abs{X}_k}B^{\Ccal_k(X)}\norm{K(X)}_{k,X} \text{; and} \nonumber \\
\norm{K}_{k:k+1}^{(A,B)} &= \sup_{X\in\Pcal_k^{\textup{cl}}}A^{\abs{X}_k}B^{\Ccal_k(X)}\norm{K(X)}_{k:k+1,X}
\end{align}
for the standard case, and
\begin{equation}
\norm{K}_k^{(A,B),*} = \sup_{X\in\Pcal_k^{\text{cl}}}A^{\abs{X}_k}B^{\abs{\Ccal_k(X)}}\norm{K(X)}_{k,X}^*
\end{equation}
for functionals affected by the external fields in the extended renormalisation group flow. In practice, several versions of these norms feature in our analysis corresponding to particular choices for the weights $A$ and $B$. In the following we write "$M(\text{polymer set},\text{norm})$" for the Banach spaces of (not necessarily translation-invariant) functionals determined by the various norms introduced above. For example, the restriction of $M(\Pcal_k^c,\norm{\cdot}_k^{(A)})$ to translation-invariant functionals corresponds to what we have defined as the space $\bM(\Pcal_k^c,\mathfrak{V}_N)$ in section \ref{S4:Norms}.

It remains to check that equivalent factorisation properties to Lemma \ref{L:norms} hold for the $\norm{\cdot}_{k,X}^*$-norm, which is at the centre of our extended definitions. This is mostly straightforward, but to address one difficulty (and for convenience) we restate the result in its adapted form below.
\begin{lemma} \label{L:*norms}
Let $K\in M(\Pcal_k^c,\norm{\cdot}_k^{(A)})$, $K'\in M(\Pcal_k^c,\norm{\cdot}_k^{(A),*})$ and $F\in M(\Bcal_k,\tnorm{\cdot}_k)$ for $0\le k\le N-1$, and extend them to arbitrary polymers via the inclusions \eqref{E:canincl1} and \eqref{E:canincl2}. On the assumptions of Theorem \ref{T:wf}, the following bounds hold:
\begin{enumerate}
\item For every $B_k^0\subset X\in\Pcal_k$,
\begin{equation}
\norm{K'(X)}_{k,X}^*\le \prod_{Y\in\Ccal_k(X)}\norm{K'(Y)}_{k,Y}^*;
\end{equation}
\item For every $X,Y\in\Pcal_k$ disjoint,
\begin{equation}
\norm{K(X)F^Y}_{k,X\cup Y}^*\le \norm{K(X)}_{k,X}\tnorm{F}_k^{\abs{Y}_k};
\end{equation}
\item For every $X,Y\in\Pcal_k$ disjoint such that $B_k^0\subset X$,
\begin{equation}
\norm{K'(X)F^Y}_{k,X\cup Y}^*\le \norm{K'(X)}_{k,X}^*\tnorm{F}_k^{\abs{Y}_k}.
\end{equation}
\end{enumerate}
\end{lemma}
\begin{proof}
To see (1), we expand the definition of the $\norm{\cdot}_{k,X}^*$-norm using Lemma \ref{L:abs} to obtain
\begin{align} \label{E:*normfactor}
\abs{K'(X)}_{k,X,v}&\big(\bunderline{w}_k(v)\vee\bunderline{w}_k(v+\boldsymbol{\nabla}_N\psi_k)\big)^{-X} \le \frac{\prod_{Y\in\Ccal_k(X)}\abs{K'(Y)}_{k,Y,v}}{\bunderline{w}_k^X(v)\vee\bunderline{w}_k^X(v+\boldsymbol{\nabla}_N\psi_k)} \nonumber \\
&\le \prod_{Y\in\Ccal_k(X)}\norm{K'(Y)}_{k,Y}^*\frac{\prod_{Y\in\Ccal_k(X)}\bunderline{w}_k^Y(v)\vee\bunderline{w}_k^Y(v+\boldsymbol{\nabla}_N\psi_k)}{\bunderline{w}_k^X(v)\vee\bunderline{w}_k^X(v+\boldsymbol{\nabla}_N\psi_k)}.
\end{align}
We now want to rely on Theorem \ref{T:wf}(3) and the fact that $\psi_k$ is not identically zero on exactly one connected component of the polymer $X$ to deduce that the fraction appearing on the second line above equals one. A careful look at the construction of weight functions in \cite{ABKM} suggests that this may not be good enough. Indeed, Lemma 7.5 of \cite{ABKM} asserts that the operator we call $\bunderline{A}_{k,\nabla}^Z$ (see Appendix \ref{A:C}) depends on gradients in $Z^{++}$, which includes $k$-blocks strictly disjoint from $Z$. This is the relevant part of our operator $\bunderline{A}_k^Z$ since $\boldsymbol{\nabla}_N\psi_k$ is evidently a gradient vector field. However, a tighter range bound actually holds at all relevant scales.

For ease of reference, we record the argument in the notation of Chapter 7 of \cite{ABKM}. For $k\ge 2$, we have $A_k^Z=A_{k-1:k}^{Z^*}+\delta_k M_k^Z$ by construction. The operator $M_k^Z$ depends on gradients in $Z^+ + [-M,M]^d$ and, by Lemma 7.5(3) of \cite{ABKM}, $A_{k-1:k}^{Z^*}$ depends on gradients in $(Z^*)^{++}$. For $L\ge 2^d+2$, the latter polymer is contained in $Z^+$ and so the range of $A_k^Z$ is that of $M_k^Z$ given above. Provided that $L>M$, $Z^+ + [-M,M]^d\subset (Z^+)^*$ and, if $Z$ and $B_k^0$ are strictly disjoint, the support of $\psi_k$ does not intersect $(Z^+)^*$ so long as $L>2^{d+2}+1$. In light of our assumptions, $L$ (and $k$) are large enough so that all of the preceding requirements are met. Returning to the product of weight functions in \eqref{E:*normfactor}, this means that only one of the factors $\bunderline{w}_k^Y$ for $Y\in\Ccal_k(X)$ depends on a range that includes the support of $\boldsymbol{\nabla}_N\psi_k$, and so we may conclude as desired.

The inequality in (2) is an immediate consequence of the corresponding statement in Lemma \ref{L:norms}. For the final claim in (3), we argue similarly:
\begin{align}
\abs{K'(X)F^Y}_{k,X\cup Y,v}&\big(\bunderline{w}_k(v)\vee\bunderline{w}_k(v+\boldsymbol{\nabla}_N\psi_k)\big)^{-X\cup Y} \le \frac{\abs{K'(X)}_{k,X,v}\prod_{B\in\Bcal_k(Y)}\abs{F(B)}_{k,B,v}}{\bunderline{w}_k^{X\cup Y}(v)\vee\bunderline{w}_k^{X\cup Y}(v+\boldsymbol{\nabla}_N\psi_k)} \nonumber \\
&\le \norm{K'(X)}_{k,X}^*\tnorm{F}_k^{\abs{Y}_k} \frac{\bunderline{w}_k^{X\cup Y}(v)\vee\big(\bunderline{w}_k^X(v+\boldsymbol{\nabla}_N\psi_k)\bunderline{W}_k^Y(v)\big)}{\bunderline{w}_k^{X\cup Y}(v)\vee\bunderline{w}_k^{X\cup Y}(v+\boldsymbol{\nabla}_N\psi_k)},
\end{align}
using Theorem \ref{T:wf}(5) and (6). The strong weight function $\bunderline{W}_k^Y$ depends on values of its argument in $Y^*$ (as a generous estimate, see Appendix \ref{A:C}). Since $B_k^0$ and $Y$ are disjoint, it follows from the foregoing discussion that
\begin{equation}
\bunderline{w}_k^X(v+\boldsymbol{\nabla}_N\psi_k)\bunderline{W}_k^Y(v)=\bunderline{w}_k^X(v+\boldsymbol{\nabla}_N\psi_k)\bunderline{W}_k^Y(v++\boldsymbol{\nabla}_N\psi_k)\le \bunderline{w}_k^{X\cup Y}(v+\boldsymbol{\nabla}_N\psi_k),
\end{equation}
relying again on Theorem \ref{T:wf}(6), and this completes the proof. 
\end{proof}
We end this section by noting that Lemma \ref{L:norms} remains available to us in a setting where translation-invariance does not hold for all functionals since its proof does not rely on it.

\section{Estimates for the external fields} \label{S:7}

In this section, we show a number of bounds relating to the external fields $\psi_k$ that are needed to implement our overall strategy in sections \ref{S:8} and \ref{S:9} below. These estimates can broadly be understood as precise versions of the general idea that the external fields are suitably small at their respective scales. The main result proved here is the consistency of the weight functions under integration in the presence of the $\psi_k$.

We rely in the following on various properties of the finite range decomposition proved in \cite{B18}. We generally use the notation introduced in Chapter 6.1 of \cite{ABKM}. In particular, the kernels of the covariances $\mathscr{C}_k^{\ssup{\bq}}$ are denoted by $\Ccal_k^\ssup{\bq}\in\cbX_N$ and we write $\widehat{\Ccal}_k^{\ssup{\bq}}$ for their Fourier transforms. The latter is given by the expression
\begin{equation} \label{E:FTdef}
\widehat{\phi}(p)=\sum_{x\in\mathbb{T}_N}\ex^{-i p\cdot x}\phi(x),
\end{equation}
for discrete vectors $p$ in the dual torus
\begin{equation}
\widehat{\mathbb{T}}_N = \Big\{-\frac{(L^N-1)\pi}{L^N},-\frac{(L^N-3)\pi}{L^N},\ldots,\frac{(L^N-1)\pi}{L^N}\Big\}^d.
\end{equation}
To avoid lengthy repetitions we refer to \cite{ABKM} and \cite{B18} for an overview of the discrete Fourier transform so far as relevant to those works.

\subsection{Basic estimates} \label{S7:1}

We begin with a few simpler estimates. The bound
\begin{align}
\mathfrak{w}_{\bar{k}}(\alpha)^{-1}\abs{\nabla^\alpha (\nabla_i\mathscr{C}_j^{\ssup{\bq}} f_\eps(x))}&=h_{\bar{k}}^{-1}L^{\bar{k}\frac{d+2\abs{\a}}{2}}\abs{\mathscr{C}_j^{\ssup{\bq}} \nabla^\alpha\nabla_i f_\eps(x)} \nonumber \\
&\le h_{\bar{k}}^{-1}L^{\bar{k}\frac{d+2\abs{\a}}{2}}\sum_{y\in\mathbb{T}_N}\abs{\Ccal_{\bq,j}(x-y)}\abs{\nabla^\alpha\nabla_i f_\eps(y)} \nonumber \\
&\le h_{\bar{k}}^{-1}L^{\bar{k}\frac{d+2\abs{\a}}{2}} C_{FR} C_f L^{2j+d-2} \eps^{\frac{d+2\abs{\a}+4}{2}},
\end{align}
valid for $j<\bar{k}$ and $d>2$, is immediate from \eqref{E:fassump}. $C_{FR}$ is a constant (depending only on the fixed parameters) relating to the finite range decomposition as given in Theorem 6.1 of \cite{ABKM}. Using that $\eps\le 4C_f L^{-\bar{k}+1}$, implied by the third line of \eqref{E:fassump}, we therefore obtain
\begin{equation}
\abs{\boldsymbol{\nabla}_N\psi_{\bar{k}}}_{\bar{k},X}\le h_{\bar{k}}^{-1}C_{FR}(4C_f)^{d+4}L^{-2\bar{k}+2d+1}\sum_{j=1}^{\bar{k}-1}L^{2j}\le h_{\bar{k}}^{-1}C(d,f)L^{2d-1},
\end{equation}
for all polymers $X$ and $d>2$. Since $h_{\bar{k}}=2^{\bar{k}}h$, we may assume that $\abs{\boldsymbol{\nabla}_N\psi_{\bar{k}}}_{\bar{k},X}<1/2$ for all $\eps$ small enough (depending on $L$ and $h$). In dimension two, the estimate is slightly weaker due to a logarithmic correction to the bounds on the finite range decomposition kernels: $\abs{\boldsymbol{\nabla}_N\psi_{\bar{k}}}_{\bar{k},X}\le h_{\bar{k}}^{-1}C(d,f)(\log L)L^3$, but this does not materially affect our preceding comment.

For scales above $\bar{k}$, we rely on the bounds on the finite range decomposition kernels. It is useful to take advantage here of the particular form of our test functions. Summing by parts, we may choose to write the external field as
\begin{align} \label{E:partint}
\psi_k(x)&=\mathscr{C}_{k-1}^{\ssup{\bq}}f_{\eps}(x)=\sum_{y\in\mathbb{T}_N} \Ccal_{k-1}^{\ssup{\bq}}(x-y)\sum_{i=1}^d \nabla_i^* \nabla_i g_\eps(y)=\sum_{y\in\mathbb{T}_N}\sum_{i=1}^d \nabla_i^* \nabla_i \Ccal_{k-1}^{\ssup{\bq}}(x-y)(x-y)g_\eps(y) \nonumber \\
&=\Delta \mathscr{C}_{k-1}^{\ssup{\bq}}g_{\eps}(x).
\end{align}
Roughly speaking, by doing this we trade a factor of $L^{-2\bar{k}}$ for a factor of $L^{-2(k-2)}$, which is advantageous for $k>\bar{k}+1$. Given $\bar{k}<k< N$, we have
\begin{align}
\mathfrak{w}_{k}(\alpha)^{-1}\abs{\nabla^\alpha (\nabla_i\mathscr{C}_{k-1}^{\ssup{\bq}} f_\eps(x))}&\le h_{k}^{-1}L^{k\frac{d+2\abs{\a}}{2}}\sum_{y\in\mathbb{T}_N}\abs{\nabla^\alpha\nabla_i\Delta\Ccal_{k-1}^{\ssup{\bq}}(x-y)}\abs{g_\eps(y)} \nonumber \\
&\le h_{k}^{-1} C(d,f) L^{-(k-\bar{k})\frac{d+2}{2}}L^{\frac{7d+6}{2}};
\end{align}
and a similar (actually slightly tighter) bound holds for $k=N$. This expression is certainly small whenever $k\gg \bar{k}$, but even at scales near $\bar{k}$ we can rely on the smallness of $h_k^{-1}$ as above for $\eps$ small enough depending on $L$ and $h$.

A further useful estimate concerns the operator $\bunderline{M}_k$, which features in the upper bound on the weight functions $\bunderline{w}_k^X$ and $\bunderline{w}_{k:k+1}^X$ (see Theorem 5.2(2)). We only use it at the final scale $N$ and it suffices to consider the curl-free component $\bunderline{M}_{N,\nabla}$. We find that
\begin{align} \label{E:psiNest}
(\boldsymbol{\nabla}_N\psi_N,\bunderline{M}_{N,\nabla}\boldsymbol{\nabla}_N\psi_N)_{\cbO_N} &\le \sum_{1\le\abs{\a}\le M}L^{2N(\abs{\a}-1)}\sum_{x\in\mathbb{T}_N}\Big|\nabla^{\a}\big(\Delta\Ccal_{N-1}^{\ssup{\bq}}+\Delta\Ccal_N^{\ssup{\bq}}+\Delta\Ccal_{N+1}^{\ssup{\bq}}\big) * g_\eps(x)\Big|^2 \nonumber \\
&\le \sum_{1\le\abs{\a}\le M}L^{2N(\abs{\a}-1)} C_f^2 \eps^{d-2}L^{2d\bar{k}} L^{dN} C_{FR}^2 L^{-2(N-2)(d+\abs{\a})} \nonumber \\
&\le C(d,f)L^{4M+5d-2}L^{-(N-\bar{k})(d+2)},
\end{align}
for which we rely on the fact that the choice of $n=4\floor{d/2}+6$ ensures that we have bounds on $M+2=2\floor{d/3}+5$ derivatives of the covariance kernels (see Theorem 6.1 of \cite{ABKM}). This expression is bounded uniformly in $N$ (depending on $L$) and converges to zero as $N\to\infty$ for $\eps$ (and hence $\bar{k}$) fixed.

\subsection{Consistency of the extended weight functions}

The $\norm{\cdot}_{k,X}^*$-norm defined in section \ref{S6:4} features weight functions shifted by an external field: $\bunderline{w}_k^X(v+\boldsymbol{\nabla}_N\psi_k)$. It is clear that the integration map $\bR_{k+1}^{\ssup{\bq}}$ can only be bounded with respect to the starred norms if these weight functions remain integrable. We thus require an analogue of Theorem \ref{T:wf}(9) and (10) in the presence of the $\psi_k$, up to no more than a multiplicative constant. Moreover, this constant should not depend on $L$ to avoid circularity in our argument that $\overline{\bS}_k$ contracts; the independence of $\Acal_{\Bcal}$ and $L$ plays a similar role in Theorem \ref{T:linearisation}. The required statement follows.

\begin{prop} \label{P:wfpsi}
On the assumptions of Theorem \ref{T:wf} and with a choice of constants consistent with it, there exist a real number $C(L)>0$, depending on $L$, and a natural number $\mathfrak{M}$, independent of $L$, $N$ and $\eps$, such that
\begin{equation} \label{E:wfpsi}
\int_{\cbO_N}\;\bunderline{w}_k^X(v+\eta+\boldsymbol{\nabla}_N\psi_k)\,\nu_{k+1}^{\ssup{\bq}}(\d\eta) \le C_k\Big(\frac{\Acal_\Pcal}{2}\Big)^{\abs{X}_k}\bunderline{w}_{k:k+1}(v)
\end{equation}
holds for all $v\in\mathfrak{V}_N$, $X\in\Pcal_k$ and $k=\bar{k},\ldots,N-1$, where the constants $C_k$ satisfy
\begin{equation} \label{E:Ckdef}
C_k \le \begin{cases}
C(L) & \quad\text{if }\; \bar{k}\le k < \bar{k}+\mathfrak{M} \\
\frac{4}{3} & \quad\text{if }\; k \ge \bar{k}+\mathfrak{M}.
\end{cases}
\end{equation}
If $X$ is a single $k$-block, then the constant $\Acal_\Pcal$ may be replaced by $\Acal_\Bcal$ in the above inequality.
\end{prop}

We handle two main cases separately. A difference arises because, at the smoothness scale $k=\bar{k}$, we have to rely on $f_\eps$ for our upper bounds, whereas at subsequent scales $k>\bar{k}$ we may have recourse to stronger estimates for the finite range covariances $\mathscr{C}_{k-1}^{\ssup{\bq}}$. However, the starting point is identical. Since $\boldsymbol{\nabla}_N\psi_k\in\cbO_N$, we concentrate on the curl-free part of our operators and evaluate the relevant integral by a change of variables:
\begin{align} \label{E:wfpsistart}
\int_{\cbO_N}\; \exp\Big(\frac{1}{2}\big(v^\Omega+\eta+ & \boldsymbol{\nabla}_N\psi_k,\bunderline{A}_{k,\nabla}^X(v^\Omega+\eta+\boldsymbol{\nabla}_N\psi_k)\big)_{\cbO_N}\Big)\;\nu_{k+1}^{\ssup{\bq}}(\d\eta) \nonumber \\
& = \int_{\cbX_N}\; \exp\Big(\frac{1}{2}\big(\phi+\xi,A_k^X(\phi+\xi)\big)_{\cbX_N} + \big(\xi,(\mathscr{C}_{k+1}^{\ssup{\bq}})^{-1}\psi_k\big)_{\cbX_N} \nonumber \\
&\quad\quad\quad\quad - \frac{1}{2}\big(\psi_k,(\mathscr{C}_{k+1}^{\ssup{\bq}})^{-1}\psi_k\big)_{\cbX_N}\Big)\;\mu_{k+1}^{\ssup{\bq}}(\d\xi),
\end{align}
in which expression $\phi=\boldsymbol{\nabla}_N^{-1}v^\Omega$ and $v^\Omega$ stands for the curl-free part (orthogonal projection) of $v$. Our goal is to bound the mixed term in $\xi$ and $\psi_k$. In each case, this takes the form
\begin{equation}
\big(\xi,(\mathscr{C}_{k+1}^{\ssup{\bq}})^{-1}\psi_k\big)_{\cbX_N} \le \frac{1}{2}\big(\xi,\Ocal(k)\xi\big)_{\cbX_N} + \,\textup{Remainder}(k,f_\eps),
\end{equation}
for a suitable symmetric operator $\Ocal(k)$ and a remainder term that depends solely on $f_\eps$ (or $g_\eps$).

We preface the detailed implementation of this approach with several remarks. First, we recall that all translation-invariant operators (including therefore the covariance operators $\mathscr{C}_k^{\ssup{\bq}}$) are simultaneously diagonalised in the Fourier basis corresponding to \eqref{E:FTdef} and hence commute. Strictly, these covariance operators are defined on $\cbX_N$, but we may consider them as operators on $\cbV_N$ with kernel given as the orthogonal complement of $\cbX_N$, that is, the constant fields. We assume below that they are so extended wherever required. Finally, inequalities involving symmetric operators are understood in the sense of quadratic forms. Accordingly, we say that $A\le B$ if, and only if, $\big(\phi,A\phi\big)\le \big(\phi,B\phi\big)$ for all $\phi$ in the relevant space. 

We now begin with the easier case $k>\bar{k}$. In the following, we use $\Ical_f$ to denote the identity operator on $B_{\bar{k}}^0$, extended by zero elsewhere. We further write $D_{FR}$ for the positive real number $D_{FR}=C^2c^{-1}$, where $C$ and $c$ are the constants appearing in the bounds on the covariance kernels in Fourier space given in Theorem 6.1 of \cite{ABKM}. The shorthand $\mathfrak{n}(d)$ stands for $\mathfrak{n}(d)=9d+6\tilde{n}+3n$. With this notation, we want to focus on the symmetric operator
\begin{equation} \label{E:Odefk>kbar}
\Ocal(k) = \frac{\rho L^{2k-\mathfrak{n}(d)}}{(1+\rho)D_{FR}} (\mathscr{C}_{k+1}^{\ssup{\bq}})^{-1} \, \Delta\mathscr{C}_{k-1}^{\ssup{\bq}} \, \Ical_0 \, \Delta\mathscr{C}_{k-1}^{\ssup{\bq}} \, (\mathscr{C}_{k+1}^{\ssup{\bq}})^{-1},
\end{equation}
and the corresponding remainder term is the inner product
\begin{equation} \label{E:Rdefk>kbar}
\textup{Remainder}(k,f_\eps) = \frac{1+\rho}{2\rho} D_{FR}L^{\mathfrak{n}(d)-2k} \big(g_\eps,g_\eps\big)_{\cbX_N}.
\end{equation}
The crucial estimates for this case are summarised in the next statement.
\begin{lemma} \label{L:k>kbarcase}
On the assumptions of Theorem \ref{T:wf} and with a choice of constants consistent with it, the following inequalities hold at all scales $\bar{k}<k<N$ for the operator given in \eqref{E:Odefk>kbar} and the expression in \eqref{E:Rdefk>kbar}:
\begin{enumerate}
\item $\big((\mathscr{C}_{k+1}^{\ssup{\bq}})^{-1}-\Ocal(k)\big)^{-1} \le (1+\rho)\mathscr{C}_{k+1}^{\ssup{\bq}}$;\\
\item $\textup{Tr}\big((\mathscr{C}_{k+1}^{\ssup{\bq}})^{1/2} \Ocal(k) (\mathscr{C}_{k+1}^{\ssup{\bq}})^{1/2}\big) \le C_1(d) L^{-d((k-\bar{k})}$; and \\
\item $\textup{Remainder}(k,f_\eps) \le C_2(d,f) L^{\mathfrak{n}(d)+d-2}L^{-2(k-\bar{k})}$,
\end{enumerate}
where $C_1,C_2>0$ depend on the fixed parameters and $C_2$ additionally depends on $C_f$.
\end{lemma}
\begin{proof}
We derive (1) from a suitable upper bound on the operator $\Ocal_k$. We have
\begin{equation}
(\mathscr{C}_{k+1}^{\ssup{\bq}})^{-1} \, \Delta\mathscr{C}_{k-1}^{\ssup{\bq}} \, \Ical_0 \, \Delta\mathscr{C}_{k-1}^{\ssup{\bq}} \, (\mathscr{C}_{k+1}^{\ssup{\bq}})^{-1} \le \Big[(\mathscr{C}_{k+1}^{\ssup{\bq}})^{-1} \, \big(\Delta\mathscr{C}_{k-1}^{\ssup{\bq}}\big)^2 \Big] \, (\mathscr{C}_{k+1}^{\ssup{\bq}})^{-1},
\end{equation}
and so to compare the corresponding quadratic form to that generated by $(\mathscr{C}_{k+1}^{\ssup{\bq}})^{-1}$ it suffices to estimate the eigenvalues of the operator $(\mathscr{C}_{k+1}^{\ssup{\bq}})^{-1} \, \big(\Delta\mathscr{C}_{k-1}^{\ssup{\bq}}\big)^2$. The Fourier coefficients of the kernel $\Delta\Ccal_{k-1}^{\ssup{\bq}}$ are easily computed (cf. the notation $q_j(p)=\ex^{ip_j}-1$ in \cite{ABKM}):
\begin{equation}
\widehat{\Delta\Ccal_{k-1}^{\ssup{\bq}}}(p)=
\sum_{j=1}^d (e^{ip_j}-1)(e^{-ip_j}-1)\widehat{\Ccal_{k-1}^{\ssup{\bq}}}(p)\le \abs{p}^2\widehat{\Ccal_{k-1}^{\ssup{\bq}}}(p).
\end{equation}
Relying on the bounds provided by Theorem 6.1 of \cite{ABKM}, we obtain
\begin{equation} \label{E:pbounds}
\frac{\widehat{\Ccal_{k+1}^{\ssup{\bq}}}(p)^{-1}\widehat{\Delta\Ccal_{k-1}^{\ssup{\bq}}}(p)^2}{D_{FR}} \le \begin{cases}
L^{9d+6\tilde{n}+3n}L^{-2k}L^{-(k-j)(d+n-3)} &\text{ for } L^{-(j+1)}<\abs{p}\le L^{-j};\,j<k-1 \\
L^{4d+2\tilde{n}+2n+1}L^{-2k} &\text{ for } L^{-k}<\abs{p}\le L^{-k+1} \\
L^{3d+2\tilde{n}+n-4}L^{-2k} &\text{ for } L^{-(k+1)}<\abs{p}\le L^{-k} \\
L^{2d+2\tilde{n}-5}L^{2k-4j} &\text{ for } L^{-(j+1)}<\abs{p}\le L^{-j};\,j\ge k+1,
\end{cases}
\end{equation}
from which we deduce that $\Ocal(k)\le\frac{\rho}{1+\rho}(\mathscr{C}_{k+1}^{\ssup{\bq}})^{-1}$, and this implies the inequality claimed in (1).

The same estimate on the Fourier kernels also helps to compute the trace in (2). We work with the standard basis of $\cbV_N$ ($e_x(y)=\delta_{x,y}\,\forall x,y\in\mathbb{T}_N$) and note that considering our operators on $\cbV_N$ does not change their trace. We obtain
\begin{align}
\text{Tr} \big((\mathscr{C}_{k+1}^{\ssup{\bq}})^{-1/2} \, \Delta\mathscr{C}_{k-1}^{\ssup{\bq}} \, \Ical_0 &\, \Delta\mathscr{C}_{k-1}^{\ssup{\bq}} \, (\mathscr{C}_{k+1}^{\ssup{\bq}})^{-1/2} \big) = \text{Tr}\big(\Ical_0 (\mathscr{C}_{k+1}^{\ssup{\bq}})^{-1} \, \big(\Delta\mathscr{C}_{k-1}^{\ssup{\bq}}\big)^2 \big) \nonumber \\
&= \sum_{x \in\mathbb{T}_N}\big(\Ical_0 e_x,(\mathscr{C}_{k+1}^{\ssup{\bq}})^{-1} \, \big(\Delta\mathscr{C}_{k-1}^{\ssup{\bq}}\big)^2 e_x\big)_{\cbV_N} \nonumber \\
&=\sum_{x\in B_{\bar{k}}^0}L^{-dN}\sum_{p\in\widehat{\mathbb{T}}_N}\widehat{\Ccal_{k+1}^{\ssup{\bq}}}(p)^{-1}\widehat{\Delta\Ccal_{k-1}^{\ssup{\bq}}}(p)^2\abs{\widehat{e_x}(p)}^2,
\end{align}
by Plancherel's identity. Direct computation shows that $\abs{\widehat{e_x}(p)}=1$ and, up to a geometric constant, we can estimate the size of the dyadic partition of the dual torus as follows:
\begin{equation}
L^{-dN}\abs{p\in\widehat{\mathbb{T}}_N:L^{-(j+1)}<\abs{p}\le L^{-j}}\le C(d) L^{-jd} \;\text{ for } 0\le j\le N.
\end{equation}
Strictly, the $j=0$ range goes up to $\abs{p}^2\le d\pi^2$, but this merely affects the constant $C(d)$. In combination with the bounds in \eqref{E:pbounds}, we thus show that the above sum is bounded as
\begin{equation} \label{E:Trbound}
\text{Tr} \big((\mathscr{C}_{k+1}^{\ssup{\bq}})^{-1/2} \, \Delta\mathscr{C}_{k-1}^{\ssup{\bq}} \, \Ical_0 \, \Delta\mathscr{C}_{k-1}^{\ssup{\bq}} \, (\mathscr{C}_{k+1}^{\ssup{\bq}})^{-1/2} \big) \le 2C(d)D_{FR}L^{9d+6\tilde{n}+n+6}L^{-2k}L^{-d(k-\bar{k})}.
\end{equation}
Taking account of the pre-factor in the definition of $\Ocal(k)$, we conclude with the bound asserted in (2). Finally, the assumptions on $g_\eps$ in \eqref{E:fassump} together with the definition of $\bar{k}$ imply that
\begin{equation}
\big(g_\eps,g_\eps\big)_{\cbX_N} \le L^{d\bar{k}}C_f^2\eps^{d-2}\le 4^{d-2}C_f^d L^{d-2}L^{2\bar{k}},
\end{equation} 
from which (3) follows.
\end{proof}

We obtain comparable bounds at the smoothness scale $k=\bar{k}$ next. The shorthand $\overline{\mathscr{C}}_{\bar{k}-1}^{\ssup{\bq}}=\sum_{j=1}^{\bar{k}-1}\mathscr{C}_{\bar{k}-1}^{\ssup{\bq}}$ is convenient, and we again use the operator $\Ical_0$ introduced above. The term we have to contend with thus reads
\begin{equation} \label{E:kbarmixedterm}
\big(\xi,(\mathscr{C}_{\bar{k}+1}^{\ssup{\bq}})^{-1}\overline{\mathscr{C}}_{\bar{k}-1}^{\ssup{\bq}}f_\eps\big)_{\cbX_N}
\end{equation}
in this notation. The presence of the small-scale covariances $\mathscr{C}_1^{\ssup{\bq}},\mathscr{C}_2^{\ssup{\bq}},\ldots$ means that the operator $\overline{\mathscr{C}}_{\bar{k}-1}^{\ssup{\bq}}$ cannot help to control the large ($\abs{p}\sim 1$) Fourier coefficients of $(\mathscr{C}_{\bar{k}+1}^{\ssup{\bq}})^{-1}$. Instead, we need to borrow some of the decay of the large Fourier modes of $f_\eps$ as well as its bounded support for our estimates. This is achieved through a Fourier multiplier, called $\Mcal$ in the following, that satisfies appropriate lower bounds on its eigenvalues and has finite range. Since it performs a purely auxiliary function, we relegate its construction to Appendix \ref{A:D} and rely on the features shown in Lemma \ref{L:Fmult}. These allow us to re-write the inner product given in \eqref{E:kbarmixedterm} as
\begin{equation}
\big(\xi,(\mathscr{C}_{\bar{k}+1}^{\ssup{\bq}})^{-1}\overline{\mathscr{C}}_{\bar{k}-1}^{\ssup{\bq}}\Mcal^{-1}\Ical_0\Mcal f_\eps\big)_{\cbX_N}=\big(\Ical_0\Mcal^{-1}\overline{\mathscr{C}}_{\bar{k}-1}^{\ssup{\bq}}(\mathscr{C}_{\bar{k}+1}^{\ssup{\bq}})^{-1}\xi,\Mcal f_\eps\big)_{\cbV_N},
\end{equation}
in terms of $\Mcal$ and $\Ical_0$. Our estimates at the smoothness scale therefore involve the operator
\begin{equation} \label{E:Odefk=kbar}
\Ocal(\bar{k}) = \frac{\rho \,C_a^2 L^{-2\bar{k}-\mathfrak{n}'(d)}}{16(1+\rho)D_{FR}}(\mathscr{C}_{\bar{k}+1}^{\ssup{\bq}})^{-1} \, \overline{\mathscr{C}}_{\bar{k}-1}^{\ssup{\bq}} \, \Mcal^{-1} \, \Ical_0 \, \Mcal^{-1} \, \overline{\mathscr{C}}_{\bar{k}-1}^{\ssup{\bq}} \, (\mathscr{C}_{\bar{k}+1}^{\ssup{\bq}})^{-1},
\end{equation}
and the remainder term now features the inner product
\begin{equation} \label{E:Rdefk=kbar}
\textup{Remainder}(\bar{k},f_\eps) = \frac{8(1+\rho)}{\rho\,C_a^2} D_{FR}L^{\mathfrak{n}'(d)+2\bar{k}} \big(\Mcal f_\eps,\Mcal f_\eps\big)_{\cbX_N}.
\end{equation}
In these expressions, $C_a$ is the constant from Lemma \ref{L:Fmult} and $\mathfrak{n}'(d)=7d+6\tilde{n}+n-2$. The relevant technical statement now reads as follows.
\begin{lemma} \label{L:k=kbarcase}
On the assumptions of Theorem \ref{T:wf} and with a choice of constants consistent with it, given $a\in\N$ such that $4a>2d+n-3$, the following inequalities hold for the operator in \eqref{E:Odefk=kbar} and the expression in \eqref{E:Rdefk=kbar}:
\begin{enumerate}
\item $\big((\mathscr{C}_{\bar{k}+1}^{\ssup{\bq}})^{-1}-\Ocal(\bar{k})\big)^{-1} \le (1+\rho)\mathscr{C}_{\bar{k}+1}^{\ssup{\bq}}$;\\
\item $\textup{Tr}\big((\mathscr{C}_{\bar{k}+1}^{\ssup{\bq}})^{1/2} \Ocal(\bar{k}) (\mathscr{C}_{\bar{k}+1}^{\ssup{\bq}})^{1/2}\big) \le C_3(d) $; and \\
\item $\textup{Remainder}(\bar{k},f_\eps) \le C_4(d,a,f) L^{\mathfrak{n}'(d)+4a+d+2}$,
\end{enumerate}
where $C_3,C_4>0$ depend on the fixed parameters and $C_4$ additionally depends on $C_f$ and $a$.
\end{lemma}
\begin{proof}
We argue as in the proof of Lemma \ref{L:k>kbarcase} above and thus begin by estimating the Fourier coefficients of the operator $(\mathscr{C}_{\bar{k}+1}^{\ssup{\bq}})^{-1} (\overline{\mathscr{C}}_{\bar{k}-1}^{\ssup{\bq}})^2 \Mcal^{-2}$. Combining the estimates from Theorem 6.1 of \cite{ABKM} and Lemma \ref{L:Fmult}, we arrive at the following:
\begin{equation} \label{E:pboundsinitial}
\frac{\widehat{\Ccal_{\bar{k}+1}^{\ssup{\bq}}}(p)^{-1}  \widehat{\overline{\Ccal}_{\bar{k}-1}^{\ssup{\bq}}}(p)^2 \widehat{\Mcal}(p)^{-2}}{16D_{FR}C_a^{-2}L^{2\bar{k}}} \le \begin{cases}
L^{\mathfrak{n}'(d)-4a} L^{(j+2-\bar{k})(4a-d-n+3)} &\text{ for } L^{-(j+1)}<\abs{p}\le L^{-j};j<\bar{k}-1 \\
L^{4d+2\tilde{n}+2n-1} &\text{ for } L^{-\bar{k}}<\abs{p}\le L^{-\bar{k}+1} \\
L^{3d+2\tilde{n}+n-4} &\text{ for } L^{-(\bar{k}+1)}<\abs{p}\le L^{-\bar{k}} \\
L^{2d+2\tilde{n}-5} &\text{ for } \abs{p}\le L^{-(\bar{k}+1)}.
\end{cases}
\end{equation}
Strictly, the symbol $\Mcal$ here should be read as a reference to the kernel of the corresponding operator (as in Lemma \ref{L:Fmult}). We see that this implies the inequality $\Ocal(\bar{k}) \le \frac{\rho}{1+\rho}(\mathscr{C}_{\bar{k}+1}^{\ssup{\bq}})^{-1}$, provided that $4a>d+n-3$, and (1) then follows as before.

The trace in (2) is also determined in the same manner as above. We thus evaluate the trace
\begin{align}
\text{Tr} \big( (\mathscr{C}_{\bar{k}+1}^{\ssup{\bq}})^{-1/2} \, \overline{\mathscr{C}}_{\bar{k}-1}^{\ssup{\bq}} \, \Mcal^{-1} &\, \Ical_0 \, \Mcal^{-1} \, \overline{\mathscr{C}}_{\bar{k}-1}^{\ssup{\bq}} \, (\mathscr{C}_{\bar{k}+1}^{\ssup{\bq}})^{-1/2} \big) \nonumber \\
&= \sum_{x\in B_{\bar{k}}^0}L^{-dN}\sum_{p\in\widehat{\mathbb{T}}_N}\Ccal_{\bar{k}+1}^{\ssup{\bq}}(p)^{-1}  \widehat{\overline{\Ccal}_{\bar{k}-1}^{\ssup{\bq}}}(p)^2 \widehat{\Mcal}(p)^{-2}\abs{\widehat{e_x}(p)}^2 \nonumber \\
&\le 32C(d)D_{FR}C_a^{-2}L^{\mathfrak{n}'(d)+2d-4a}L^{2\bar{k}}, 
\end{align}
for which we require the stronger assumption on $a$ to get the desired $L^{-d\bar{k}}$ factor for $\abs{p}\sim 1$. This confirms the claim in (2), and the inequality in (3) is an immediate consequence of the definition in \eqref{E:Rdefk=kbar} and the estimate on the relevant inner product in Lemma \ref{L:Fmult}.
\end{proof}

This completes our preparatory steps and so we end this section with a proof of its main result.
\begin{proof} [Proof of Proposition \ref{P:wfpsi}]
Recall the starting point of our argument in \eqref{E:wfpsistart}. Following our strategy, we have to integrate (with respect to $\lambda_N$) the exponential of
\begin{equation}
\frac{1}{2}\big(\phi+\xi,A_k^X(\phi+\xi)\big)_{\cbX_N} + \frac{1}{2}\big(\xi,\Ocal(k)\xi\big)_{\cbX_N} -\frac{1}{2}\big(\xi,(\mathscr{C}_{k+1}^{\ssup{\bq}})^{-1}\xi\big)_{\cbX_N} +\textup{Remainder}(k,f_\eps).
\end{equation}
The remainder term does not depend on the integration variable and thus contributes simply a multiplicative constant. The residual integrand is evaluated by standard Gaussian calculus as in Chapter 7 of \cite{ABKM}. We have the identity (see, e.g., (7.2.2) in \cite{ABKM})
\begin{multline} \label{E:Gausscalc}
\det(\mathscr{C}_{k+1}^{\ssup{\bq}})^{-1/2}\int_{\cbX_N}\;\ex^{\frac{1}{2}\big(\phi+\xi,A_k^X(\phi+\xi)\big)-\frac{1}{2}\big(\xi,((\mathscr{C}_{k+1}^{\ssup{\bq}})^{-1}-\Ocal(k))\xi\big)}\,\lambda_N(\d\xi) = \\
\det\Big((\mathscr{C}_{k+1}^{\ssup{\bq}})^{1/2}((\mathscr{C}_{k+1}^{\ssup{\bq}})^{-1}-\Ocal(k)-A_k^X)(\mathscr{C}_{k+1}^{\ssup{\bq}})^{1/2}\Big)^{-1/2} \ex^{\frac{1}{2}\big(\phi,\big((A_k^X)^{-1}-((\mathscr{C}_{k+1}^{\ssup{\bq}})^{-1}-\Ocal(k))^{-1}\big)^{-1}\phi\big)},
\end{multline}
provided the determinant does not vanish. Lemmas 7.5 and 7.7 of \cite{ABKM} provide the bounds
\begin{equation}
A_k^X\le \frac{1}{1+\zeta'/4}(\mathscr{C}_{k+1}^{\ssup{0}})^{-1}\,\text{ and }\,\mathscr{C}_{k+1}^{\ssup{\bq}} \le (1+\rho)\mathscr{C}_{k+1}^{\ssup{0}},
\end{equation}
which, together with Lemmas \ref{L:k>kbarcase}(1) and \ref{L:k=kbarcase}(1) above, imply that
\begin{equation}
(\mathscr{C}_{k+1}^{\ssup{\bq}})^{1/2}((\mathscr{C}_{k+1}^{\ssup{\bq}})^{-1}-\Ocal(k))(\mathscr{C}_{k+1}^{\ssup{\bq}})^{1/2}\ge \frac{1}{1+\rho}\,\text{ and }\,(\mathscr{C}_{k+1}^{\ssup{\bq}})^{1/2}A_k^X(\mathscr{C}_{k+1}^{\ssup{\bq}})^{1/2} \le \frac{1}{(1+\rho)^2}.
\end{equation}
This justifies \eqref{E:Gausscalc}, and we proceed with the chain of inequalities
\begin{equation}
\big((\mathscr{C}_{k+1}^{\ssup{\bq}})^{-1}-\Ocal(k)\big)^{-1} \le (1+\rho)\mathscr{C}_{k+1}^{\ssup{\bq}} \le (1+\rho)^2 \mathscr{C}_{k+1}^{\ssup{0}} \le (1+\zeta'/4) \mathscr{C}_{k+1}^{\ssup{0}}.
\end{equation}
Having regard to Lemma 7.1 of \cite{ABKM}, they imply that
\begin{equation}
\Big((A_k^X)^{-1}-\big((\mathscr{C}_{k+1}^{\ssup{\bq}})^{-1}-\Ocal(k)\big)^{-1}\Big)^{-1} \le \big((A_k^X)^{-1}-(1+\zeta'/4)\mathscr{C}_{k+1}^{\ssup{0}}\big)^{-1}=A_{k:k+1}^X,
\end{equation}
which means we recover the curl-free part of $\bunderline{w}_{k:k+1}^X(v)$ as desired. It remains to estimate the determinant in \eqref{E:Gausscalc}. We argue as in the proof of Lemma 7.7 of \cite{ABKM} using the trace bounds we have shown in Lemmas \ref{L:k>kbarcase}(2) and \ref{L:k=kbarcase}(2). We obtain the inequality
\begin{align} \label{E:detest}
-\frac{1}{2}\log\det\big(\1- & (\mathscr{C}_{k+1}^{\ssup{\bq}})^{1/2}(A_k^X+\Ocal(k))(\mathscr{C}_{k+1}^{\ssup{\bq}})^{1/2}\big) \le\frac{\log(2/\rho)}{2-\rho}\text{Tr}\big((\mathscr{C}_{k+1}^{\ssup{\bq}})^{1/2}(A_k^X+\Ocal(k))(\mathscr{C}_{k+1}^{\ssup{\bq}})^{1/2}\big) \nonumber \\
&\le \frac{1}{1+\rho}\abs{X}_k\log(\Acal_\Pcal/2)+\frac{\log(2/\rho)}{2-\rho}\text{Tr}\big((\mathscr{C}_{k+1}^{\ssup{\bq}})^{1/2} \Ocal(k) (\mathscr{C}_{k+1}^{\ssup{\bq}})^{1/2}\big),
\end{align}
which means that the additional factor contributed by the determinant is also bounded uniformly in $k$ and $N$.

Combining our various estimates, we conclude that the upper bound in \eqref{E:wfpsi} holds with a multiplicative constant $C_k$ that satisfies, in each case,
\begin{equation}
\log C_{\bar{k}} \le C_3(d)+\frac{C_4(d,a,f)\log(2/\rho)}{2-\rho}L^{\mathfrak{n}'(d)+4a+d+2}
\end{equation}
at the smoothness scale $k=\bar{k}$, and
\begin{equation} \label{E:Ckk>kbar}
\log C_k \le C_1(d)L^{-d(k-\bar{k})}+\frac{C_2(d,f)\log(2/\rho)}{2-\rho}L^{\mathfrak{n}(d)+d-2}L^{-2(k-\bar{k})}
\end{equation}
at subsequent scales up to (and including) $N-1$. It follows that $C_k$ is bounded uniformly in $N$ (depending on $L$) and, since \eqref{E:Ckk>kbar} is monotone in $L$, the choice of
\begin{equation}
\mathfrak{M}=\textup{ceil}\,\max\Big\{\frac{\log\frac{\log 4/3}{2C_1(d)}}{2\log 81},\frac{\log\frac{(2-\rho)\log 4/3}{2C_2(d,f)\log(2/\rho)}}{2\log 81}+\frac{\mathfrak{n}(d)+d-2}{2}\Big\},
\end{equation}
corresponding to $L\ge 2^{d+3}+16R$ in $d=2$, ensures the asserted behaviour for $k\ge \bar{k}+\mathfrak{M}$. Finally, we observe that the part of the determinant in \eqref{E:detest} that depends on the polymer $X$ is unaffected by our manipulations. It follows that $\Acal_\Pcal$ may be replaced by $\Acal_\Bcal$ for single $k$-blocks as in Lemma 7.7 of \cite{ABKM}. This completes the proof.
\end{proof}

\section{Smoothness of the map $\overline{\bS}_k$} \label{S:8}

The aim of this section is to show that the map $\overline{\bS}_k$ defined in section \ref{S6:2} above is a smooth map ($C^\infty$ in the Fr\'ech\'et sense) on a suitable subspace
\begin{equation}
\mathfrak{U}_k\subset M(\Bcal_k,\norm{\cdot}_{k,0})\times\ M(\Pcal_k^c,\norm{\cdot}_k^{(A)})\times M(\Pcal_k^c,\norm{\cdot}_k^{(A)}) \to M(\Pcal_{k+1}^c,\norm{\cdot}_{k+1}^{(A)}),
\end{equation}
with derivatives uniformly bounded on $\mathfrak{U}_k$ and such that
\begin{equation}
\big(H_{k+1},K_{k+1}(0),\overline{\bS}_k(H_k,K_k(0),K_k(<k))\big)\in\mathfrak{U}_{k+1}
\end{equation}
for all scales $k\ge \bar{k}$. The proof largely follows Chapter 9 of \cite{ABKM} with some modifications to accommodate the differences for polymers containing $0\in\Lambda_N$ due to the presence of the external fields. Let us point out that smoothness is more than we strictly require, but the existing analysis regarding the bulk renormalisation group flow means the work towards that more ambitious goal is already half-way done.

In light of Lemma \ref{L:K<kprop}(2), we may restrict the domain of $\overline{\bS}_k$ in its second and third arguments to pairs of functionals $(K_k(0),K_k(<k))$ that satisfy $K_k(X,<k)=K_k(X,0)$ whenever $X\cap B_k^0=\emptyset$, and the construction of $\mathfrak{U}_k$ in the following is to be understood in this way. For convenience, we nevertheless track the $K_k(0)$ and $K_k(<k)$ arguments separately. With this restriction in place, we observe immediately that, for $U\in\Pcal_{k+1}^c$ with $U\cap B_{k+1}^0=\emptyset$, the $\overline{\bS}_k$ map coincides with the ordinary (irrelevant component of the) renormalisation map, that is, we have the identity
\begin{equation}
\overline{\bS}_k(H_k,K_k(0),K_k(<k))(U)=\bS_k(H_k,K_k(0))(U).
\end{equation}
It then follows from Theorem \ref{T:smooth} that $\overline{\bS}_k(\cdot)(U)$, for such $U$, satisfies suitable uniform bounds on its derivatives on the set $\{(H_k,K_k(0)):\norm{H_k}_{k,0},\norm{K_k(0)}_k^{(A)}<\rho(A)\}$ with the constant $\rho(A)$ determined as described in section \ref{S:5} above. To be slightly more precise, Theorem \ref{T:smooth} implies upper bounds on the ratios
\begin{equation}
\frac{A^{\abs{U}_{k+1}}\norm{D_1^{i_1}D_2^{i_2}\overline{\bS}(H_k,K_k(0),K_k(<k))(\dot{H}^{i_1},\dot{K}(0)^{i_2})(U)}_{k+1,U}}{\norm{\dot{H}}_{k,0}^{i_1}(\norm{\dot{K}(0)}_k^{(A)})^{i_2}},
\end{equation}
for $U\in\Pcal_{k+1}^c$ with $U\cap B_{k+1}^0=\emptyset$, which need to be complemented by similar estimates for polymers containing $B_{k+1}^0$. The main focus of this section is therefore on the case $0\in U$. We show the following result.
\begin{theorem} \label{T:EFsmooth}
On the assumptions of Theorem \ref{T:mainrep} and for any choice of constants consistent with it, for all $\bar{k}$ large enough, there exist constants $\mathfrak{R}(L,A)$, $\bunderline{C}_{i_1,i_2,i_3}(A,L)$ and a domain
\begin{equation}
\mathfrak{U}_k=\big\{(H_k,K_k(0),K_k(<k)): \norm{H_k}_{k,0},\norm{K_k(0)}_k^{(A)},\norm{K_k(<k)}_k^{(A)}<\mathfrak{R}\big\}
\end{equation}
such that the map $\overline{\bS}_k: \mathfrak{U}_k \to M(\Pcal_{k+1}^c,\norm{\cdot}_{k+1}^{(A)})$ is smooth ($C^\infty$), and the inequality
\begin{multline}
\norm{D_1^{i_1}D_2^{i_2}D_3^{i_3}\overline{\bS}_k(H_k,K_k(0),K_k(<k))(\dot{H}^{i_1},\dot{K}(0)^{i_2},\dot{K}(<k)^{i_3})}_{k+1}^{(A)}\\
\le \bunderline{C}_{i_1,i_2,i_3}(A,L)\norm{\dot{H}}_{k,0}^{i_1}\big(\norm{\dot{K}(0)}_k^{(A)}\big)^{i_2}\big(\norm{\dot{K}(<k)}_k^{(A)}\big)^{i_3}
\end{multline}
holds for any $(H_k,K_k(0),K_k(<k))\in \mathfrak{U}_k$ and $i_1,i_2,i_3\ge 0$. 
\end{theorem}

Our approach involves writing $\overline{\bS}_k$ as the composition of several maps, which we individually show to be smooth with appropriate bounds on their derivatives. We recall the various norms introduced or extended in section \ref{S6:4}, which feature in the following definitions. 

We begin with four maps that do not have any direct analogue in the smoothness proof of $\bS_k$ and therefore represent the most significant departure. The \emph{error map} is given as
\begin{align}
\boldsymbol{e^*} :& \,M(\Pcal_k^c,\norm{\cdot}_k^{(A)}) \times M(\Pcal_k^c,\norm{\cdot}_k^{(A),*})\to\R \nonumber \\
&(K_1,K_2) \mapsto \exp\big(-\bR_{k+1}^{\ssup{\bq}}\big(K_2-K_1\big)(B_k^0,0)\big);
\end{align}
the \emph{alternate exponential map}, which differs from the usual one on $B_k^0$, is given as
\begin{align}
\boldsymbol{E^*} :& \,M(\Bcal_k,\norm{\cdot}_{k,0})\to M(\Bcal_k,\tnorm{\cdot}_k) \nonumber \\
&H \mapsto \boldsymbol{E^*}(H)(B,v)= \exp(H(B,v+\psi_k));
\end{align}
and the key polynomial \emph{perturbation map} is defined via the expression
\begin{align}
\boldsymbol{P^*} :& M(\Bcal_k,\tnorm{\cdot}_k)\times M(\Bcal_k,\tnorm{\cdot}_k)\times M(\Pcal_k^c,\norm{\cdot}_k^{(A)}) \times M(\Pcal_k^c,\norm{\cdot}_k^{(A)}) \to M(\Pcal_k^c,\norm{\cdot}_k^{(A),*}) \nonumber \\
&(I_1,I_2,K_1,K_2) \mapsto \boldsymbol{P^*}(I_1,I_2,K_1,K_2)(X,v) \nonumber \\
&\qquad\qquad\qquad\qquad\qquad\qquad=\big(I_1-I_2\big)^{B_k^0}K_1(X\setminus B_k^0,v) + K_2(X,v+\psi_k),
\end{align}
where the first term above is regarded as identically zero for any polymer $X$ not intersecting $B_k^0$. In the event that removing $B_k^0$ disconnects $X$, the inclusion in \eqref{E:canincl2} is used to give meaning to the resulting expression involving $K_1(X\setminus B_k^0)$. We also introduce a \emph{multiplication map}, which simplifies our notation, defined as follows:
\begin{align}
\boldsymbol{M^*} :& \,M(\Bcal_k,\tnorm{\cdot}_k)\times \R\setminus\{0\} \to M(\Bcal_k,\tnorm{\cdot}_k) \nonumber \\
&(F,t) \mapsto \boldsymbol{M^*}(F,t)(B)=t^{-\1_{B=B_k^0}}F(B).
\end{align}

The remaining maps are analogous to the decomposition in Chapter 9 of \cite{ABKM}, with some minor modifications; note, in particular, the differences in the exact domains and co-domains. For convenience, we adopt the same nomenclature. We define the (ordinary) \emph{exponential map}:
\begin{align}
\boldsymbol{E} :& \,M(\Bcal_k,\norm{\cdot}_{k,0})\to M(\Bcal_k,\tnorm{\cdot}_k) \nonumber \\
&H \mapsto \exp(H);
\end{align}
two \emph{integration maps}:
\begin{align}
\boldsymbol{R_1} :& \,M(\Pcal_k^{\text{cl}},\norm{\cdot}_k^{(A/2,B),*})\to M(\Pcal_k^{\text{cl}},\norm{\cdot}_{k:k+1}^{(A/(2\Acal_{\Pcal}),B)}) \nonumber \\
&K \mapsto \boldsymbol{R_1}(K)(X,v)=\int_{\cbO_N}\;K(X,v+\eta)\,\nu_{k+1}^{\ssup{\bq}}(\d\eta);
\end{align}
and
\begin{align}
\boldsymbol{R_2} :& \,M(\Bcal_k,\norm{\cdot}_{k,0})\times M(\Pcal_k^c,\norm{\cdot}_k^{(A)})\to M(\Bcal_k,\norm{\cdot}_{k,0}) \nonumber \\
&(H,K) \mapsto \boldsymbol{R_2}(H,K)(B,v)=\bR_{k+1}^{\ssup{\bq}}H(B,v)-\Pi_2\big(\bR_{k+1}^{\ssup{\bq}}K(B,v)\big);
\end{align}
and finally three \emph{polynomial maps}:
\begin{align}
\boldsymbol{P_1} :& \,\R \times M(\Bcal_k,\tnorm{\cdot}_k) ^3 \times M(\Pcal_k^{\text{cl}},\norm{\cdot}_{k:k+1}^{(A/(2\Acal_{\Pcal}),B)}) \to M(\Pcal_{k+1}^c,\norm{\cdot}_{k+1}^{(A)}) \nonumber \\
&(t,I_1,I_2,J,K) \mapsto \boldsymbol{P_1}(t,I_1,I_2,J,K)(U,v) = \sum_{\substack{X_1,X_2\in\Pcal_k \\ X_1\cap X_2=\emptyset}} \chi(X_1\cup X_2,U) \,t^{\1_{B_k^0\subset X_1\cup X_2}}\nonumber \\
& \qquad\qquad\qquad\qquad\qquad\times I_1^{U\setminus (X_1\cup X_2)}(v) I_2^{(X_1\cup X_2)\setminus U}(v) J^{X_1}(v) \prod_{Y\in\Ccal_k^{\text{cl}}(X_2)} K(Y,v);
\end{align}
\begin{align}
\boldsymbol{P_2} :& \,M(\Bcal_k,\tnorm{\cdot}_k) \times M(\Pcal_k^c,\norm{\cdot}_k^{(A),*}) \to M(\Pcal_k^c,\norm{\cdot}_k^{(A/2),*}) \nonumber \\
&(I,K) \mapsto (I-1)\circ K;
\end{align}
and
\begin{align}
\boldsymbol{P_3} :& \,M(\Pcal_k^c,\norm{\cdot}_k^{(A/2),*}) \to M(\Pcal_k^{\text{cl}},\norm{\cdot}_k^{(A/2,B),*}) \nonumber \\
&K \mapsto \boldsymbol{P_3}(K)(X,v)=\prod_{Y\in\Ccal_k(X)}K(Y,v).
\end{align}

With these definitions in place, the map $\overline{\bS}_k$ can be expressed as
\begin{multline}
\overline{\bS}_k(H,K(0),K(<k))= \boldsymbol{P_1}\bigg(\boldsymbol{e^*}\Big(K(0),\boldsymbol{P^*}(\cdot)\Big),\boldsymbol{E}\Big(-\boldsymbol{R_2}\big(H,K(0)\big)\Big),\boldsymbol{E}\Big(\boldsymbol{R_2}\big(H,K(0)\big)\Big), \\
1-\boldsymbol{M^*}\Big(\boldsymbol{E}\big(-\boldsymbol{R_2}(H,K(0))\big),\boldsymbol{e^*}\big(K(0),\boldsymbol{P^*}(\cdot)\big)\Big),\boldsymbol{R_1}\Big(\boldsymbol{P_3}\Big(\boldsymbol{P_2}\big(\boldsymbol{E}(-H),\boldsymbol{P^*}(\cdot)\big)\Big)\Big)\bigg),
\end{multline}
where we have suppressed the subscripts for readability and the symbol $\boldsymbol{P^*}(\cdot)$ is shorthand for
\begin{equation}
\boldsymbol{P^*}\Big(\boldsymbol{E}(-H),\boldsymbol{E^*}(-H),K(0),K(<k)\Big).
\end{equation}
Before considering each of these maps individually, it is worth emphasising that $\boldsymbol{R_1}$ and $\boldsymbol{R_2}$ are linear maps, $\boldsymbol{e^*}$
is the composition of the (real) exponential function with a linear map, and $\boldsymbol{P_1}$, $\boldsymbol{P_2}$, $\boldsymbol{P_3}$ and $\boldsymbol{P^*}$ are polynomials in their respective arguments. Higher order derivatives are therefore not as cumbersome to compute as the above expressions might suggest at first glance.

\subsubsection{The error map $\boldsymbol{e^*}$}

Combining the definition of the star-norm in \eqref{E:Astarnorm} with Theorem \ref{T:wf}(10) and Proposition \ref{P:wfpsi}, we obtain the straightforward estimates
\begin{equation}
\Big|\int_{\cbO_N}\;K_1(B_k^0,\eta)\,\nu_{k+1}^{\ssup{\bq}}(\d\eta)\Big| \le \frac{\Acal_{\Bcal}}{2A}\norm{K_1}_k^{(A)}
\end{equation}
and
\begin{align}
\Big|-\int_{\cbO_N}\;K_2(B_k^0,\eta)\,\nu_{k+1}^{\ssup{\bq}}(\d\eta)\Big| &\le \int_{\cbO_N}\;\norm{K_2}_{k,B_k^0}^*\big(\bunderline{w}_k^{B_k^0}(\eta)+\bunderline{w}_k^{B_k^0}(\eta+\boldsymbol{\nabla}_N\psi_k)\big)\,\nu_{k+1}^{\ssup{\bq}}(\d\eta) \nonumber\\
&\le (1+C_k)\frac{\Acal_{\Bcal}}{2A}\norm{K_2}_k^{(A),*},
\end{align}
where $C_k$ is bounded uniformly in $N$ (depending on $L$). The linear map in the exponent is thus bounded and so it follows that $\boldsymbol{e^*}$ is smooth on its domain. We further obtain the inequality
\begin{multline}
\big|D_1^{j_1}D_2^{j_2}\boldsymbol{e^*}(K_1,K_2)(\dot{K}_1^{j_1},\dot{K}_2^{j_2})\big|\le (1+C_k)^{j_2}\left(\frac{\Acal_{\Bcal}}{2A}\right)^{j_1+j_2} \\
\times\exp\Big(\frac{\Acal_{\Bcal}}{2A}\norm{K_1}_k^{(A)} +  (1+C_k)\frac{\Acal_{\Bcal}}{2A}\norm{K_2}_k^{(A),*}\Big) \left(\norm{\dot{K}_1}_k^{(A)}\right)^{j_1}\left(\norm{\dot{K}_2}_k^{(A),*}\right)^{j_2},
\end{multline}
and so the derivatives of $\boldsymbol{e^*}$ are uniformly bounded (depending on $j_1$, $j_2$, $A$ and $L$) on any bounded domain $\{\norm{K_1}_k^{(A)},\norm{K_2}_k^{(A),*}<\textup{const}\}$.

\subsubsection{The alternate exponential map $\boldsymbol{E^*}$} \label{S8:E*}

Note that $H(B_k^0,\cdot+\psi_k)$ is not a relevant Hamiltonian and so this map must be considered separately from the exponential map $\bE$. Instead, we rely on the identity $\abs{H(B_k^0,\cdot+\boldsymbol{\nabla}_N\psi_k)}_{k,B_k^0,v}=\abs{H(B_k^0,\cdot)}_{k,B_k^0,v+\boldsymbol{\nabla}_N\psi_k}$ in concert with the estimate
\begin{equation} \label{E:TayvHbound}
\abs{H(B)}_{k,B,v}\le \big(1+\abs{v}_{k,l^2(B)}\big)^2\norm{H}_{k,0}\le 2(1+\abs{v}_{k,l^2(B)}^2)\norm{H}_{k,0}, 
\end{equation}
where
\begin{equation} \label{E:l2norm}
\abs{v}_{k,\ell^2(B)}^2=\frac{1}{h_k^2}\sup_{1\le i\le d}\;\sup_{0\le \abs{\a}\le \floor{d/2}}\; \sum_{x\in B}L^{2k\abs{\a}}\abs{\nabla^{\a}v_i(x)}^2
\end{equation}
is a convenient (semi-)norm that is weaker than $\abs{\cdot}_{k,B}$. Indeed, using the orthogonal decomposition $v=v^{\Omega}+v^{\perp}$ from Appendix \ref{A:C} and the fact that $\nabla_i v^{\Omega}_j=\nabla_j v^{\Omega}_i$, we see that
\begin{align} \label{E:Gest}
2\big(v,\bunderline{G}_k^B v\big) &= \frac{1}{h_k^2}\sum_{i=1}^d \sum_{0\le \a\le \floor{d/2}} \sum_{x\in B} L^{2k\abs{\a}}\bigg(\frac{2\abs{\nabla^{\a}v_i^{\Omega}}^2}{\abs{\{j:\a_j+\delta_{i,j}>0\}}}+2\abs{\nabla^{\a}v_i^{\perp}}^2\bigg) \nonumber \\
&\ge \frac{1}{h_k^2}\sum_{x\in B} L^{2k\abs{\a}}\big(2\abs{\nabla^{\a}v_i^{\Omega}}^2+2\abs{\nabla^{\a}v_i^{\perp}}^2\big),\text{ for all }1\le i\le d,0\le \abs{\a}\le \floor{d/2} \nonumber \\
&\ge \abs{v}_{k,\ell^2(B)}^2.
\end{align}
With the help of these relations, we now deduce from \eqref{E:TayvHbound} the inequality
\begin{align} \label{E:E*est}
\abs{H(B_k^0,\cdot+\boldsymbol{\nabla}_N\psi_k)}_{k,B_k^0,v} &\le 2(1+2\abs{v}_{k,l^2(B_k^0)}^2+2\abs{\boldsymbol{\nabla}_N\psi_k}_{k,B_k^0}^2)\norm{H}_{k,0}  \nonumber \\
&\le 4(1+\abs{v}_{k,l^2(B_k^0)}^2)\norm{H}_{k,0}, 
\end{align}
where we rely on the assumption that $\bar{k}$ is sufficiently large to avail of the upper bound $\abs{\boldsymbol{\nabla}_N\psi_k}_{k,B_k^0}\le 1/2$ in section \ref{S7:1}. With the help of the inequality $(1+t)\le 2 e^{t/4}$ for $t\ge 0$ and the bound in \eqref{E:Gest}, we thus infer that
\begin{equation}
\abs{H(B_k^0,\cdot+\boldsymbol{\nabla}_N\psi_k)}_{k,B_k^0,v} \le 4(1+\abs{v}_{k,\ell^2(B)}^2)\norm{H}_{k,0}\le 8e^{\frac{1}{4}\abs{v}_{k,\ell^2(B)}^2}\norm{H}_{k,0}\le 8\bunderline{W}_k^B(v)\norm{H}_{k,0},
\end{equation}
from which we obtain the estimate $\tnorm{H(B_k^0,\cdot+\boldsymbol{\nabla}_N\psi_k)}_{k,B_k^0}\le 8 \norm{H}_{k,0}$. The proof of \eqref{E:TayvHbound} is very similar to the analogous statement in Lemma 8.9 of \cite{ABKM} and therefore omitted.

We now proceed as for the exponential map $\bE$, save that we repeatedly use the estimate in \eqref{E:E*est} in place of the (slightly stronger) estimate that holds without the external field $\psi_k$. Using the shorthand notation $H^*(B_k^0,v)=H(B_k^0,v+\boldsymbol{\nabla}_N\psi_k)$, we consider the functional
\begin{equation}
(e^{(H+\dot{H})^*}-e^{H^*}-e^{H^*}\dot{H}^*)\dot{H}_1^*\cdots \dot{H}_r^*,
\end{equation}
for $H$, $\dot{H}$ and $\dot{H}_j$ elements of $M(\Bcal_k,\norm{\cdot}_{k,0})$. Arguing along the lines of the proof of Lemma 9.3 of \cite{ABKM}, we obtain the inequality
\begin{align}
|(e^{(H+\dot{H})^*}-e^{H^*}- & e^{H^*}\dot{H}^*)\dot{H}_1^*\cdots \dot{H}_r^*|_{k,B_k^0,v} \le e^{\abs{H^*}_{k,B_k^0,v}}\abs{\dot{H}^*}_{k,B_k^0,v}^2 e^{\abs{\dot{H}^*}_{k,B_k^0,v}}\abs{\dot{H}_1^*}_{k,B_k^0,v}\cdots \abs{\dot{H}_r^*}_{k,B_k^0,v} \nonumber \\
&\le e^{4(1+\abs{v}_{k,l^2(B_k^0)}^2)(\norm{H}_{k,0}+\norm{\dot{H}}_{k,0})}\big(4(1+\abs{v}_{k,l^2(B_k^0)}^2)\big)^{r+2} \norm{\dot{H}}_{k,0}^2 \prod_{j=1}^r \norm{\dot{H}_j}_{k,0},
\end{align}
from which we deduce that requiring $\norm{H}_{k,0}<1/32$ (and, without loss of generality, $\norm{\dot{H}}_{k,0}<1/64$) yields a suitable upper bound:
\begin{multline}
|(e^{(H+\dot{H})^*}-e^{H^*}-e^{H^*}\dot{H}^*)\dot{H}_1^*\cdots \dot{H}_r^*|_{k,B_k^0,v} \le \\
e^{3/16}\Big(\sup_{t\ge 0}\big(4(1+t)\big)^{r+2}e^{-\frac{t}{16}}\Big)\bunderline{W}_k^{B_k^0}(v) \norm{\dot{H}}_{k,0}^2 \prod_{j=1}^r \norm{\dot{H}_j}_{k,0}.
\end{multline}
This implies the limit
\begin{equation}
\lim_{\dot{H}\to 0}\frac{1}{\norm{\dot{H}}_{k,0}}\sup_{\substack{\norm{\dot{H}_j}_{k,0}\le 1 \\ j=1,...,r}}\tnorm{(e^{(H+\dot{H})^*}-e^{H^*}-e^{H^*}\dot{H}^*)\dot{H}_1^*\cdots \dot{H}_r^*}_{k,B_k^0}=0,
\end{equation}
which in turn shows the infinite differentiability of the map $\bE^*$ as claimed (by induction on the order of derivatives). Moreover, reading off the above an expression for the $r$-th derivative $D^r\bE^*$, we compute similarly
\begin{align}
\abs{D^r\bE^*(H)(\dot{H}_1,...,\dot{H}_r)}_{k,B_k^0,v} & \le e^{\abs{H^*}_{k,B_k^0,v}}\abs{\dot{H}_1^*}_{k,B_k^0,v}\cdots \abs{\dot{H}_r^*}_{k,B_k^0,v} \nonumber \\
&\le e^{\frac{1}{8}(1+\abs{v}_{k,l^2(B_k^0)}^2)}\big(4(1+\abs{v}_{k,l^2(B_k^0)}^2)\big)^r\norm{\dot{H}_1}_{k,0}\cdots \norm{\dot{H}_r}_{k,0} \nonumber \\
&\le C_r \bunderline{W}_k^{B_k^0}(v)\norm{\dot{H}_1}_{k,0}\cdots \norm{\dot{H}_r}_{k,0},
\end{align}
where the constant $C_r$ is given by the expression
\begin{equation}
C_r=e^{\frac{1}{8}}4^r\sup_{t\ge 0}(1+t)^re^{-\frac{t}{8}}.
\end{equation}
This establishes the uniform boundedness of the derivatives $\norm{D^r\bE^*(H)(B_k^0)}\le C_r$ on the set $\{\norm{H}_{k,0}<1/32\}$. By integrating the first derivative, we obtain the inequality
\begin{equation}
\tnorm{\boldsymbol{E^*}(H)-1}_{k,B_k^0}\le C_1 \norm{H}_{k,0} \le 16 \norm{H}_{k,0},
\end{equation}
which is useful later. Finally, we observe that on $k$-blocks other than $B_k^0$ the external field $\psi_k$ vanishes identically and so the mapping coincides with the (regular) exponential map $\bE$, for which similar (indeed, stronger) bounds hold.

\subsubsection{The perturbation map $\bP^*$}

Addressing first the trivial case $X\cap B_k^0=\emptyset$, we see that the only non-zero derivative is $D_4\bP^*(I_1,I_2,K_1,K_2)(\dot{K}_2)(X,v)=\dot{K}_2(X,v)$ (recall that $\psi_k\equiv 0$ on $X^*$ for such $X$). This is a bounded linear map since, by construction, the $\norm{\cdot}_k^{(A),*}$-norm is weaker than the $\norm{\cdot}_k^{(A)}$-norm. Now assume that $B_k^0\subset X\in\Pcal_k^c$. We treat separately the case $X=B_k^0$, in which all terms are linear. For the derivative with respect to the first variable, $D_1\bP^*(I_1,I_2,K_1,K_2)(\dot{I}_1)(B_k^0)=\dot{I}_1(B_k^0)$, we obtain a straightforward bound
\begin{equation}
\frac{A\abs{\dot{I}_1(B_k^0)}_{k,B_k^0,v}}{\bunderline{w}_k^{B_k^0}(v)\vee\bunderline{w}_k^{B_k^0}(v+\boldsymbol{\nabla}_N\psi_k)} \le A \tnorm{\dot{I}_1(B_k^0)}_{k,B_k^0} \frac{\bunderline{W}_k^{B_k^0}(v)}{\bunderline{w}_k^{B_k^0}(v)\vee\bunderline{w}_k^{B_k^0}(v+\psi_k)} \le  A \tnorm{\dot{I}_1}_k,
\end{equation}
since it generally holds that $\bunderline{W}_k\le\bunderline{w}_k$ in view of Theorem \ref{T:wf}(6), and likewise for the derivative with respect to the $I_2$-variable. The term in $K_2$ is always linear, so it suffices to address it in the general case - $B_k^0\subsetneq X$ - which we consider next.

There are two situations to analyse. We first assume that $X\setminus B_k^0 \in\Pcal_k^c$, in which case only first order and mixed second order partial derivatives are of relevance. By way of illustration, we obtain the following bound for $D_1\boldsymbol{P^*}(I_1,I_2,K_1,K_2)(\dot{I}_1)$,
\begin{equation}
\norm{\dot{I}_1(B_k^0)K_1(X\setminus B_k^0)}_{k,X}^* \le \tnorm{\dot{I}_1}_k\norm{K_1(X\setminus B_k^0)}_{k,X\setminus B_k^0} \le A^{-\abs{X}_k} A \tnorm{\dot{I}_1}_k \norm{K_1}_k^{(A)},
\end{equation}
using Lemma \ref{L:*norms}(2), and similar estimates apply for the first derivatives with respect to the second and third variable of $\bP^*$. The second derivatives (i.e. $D_1D_3\bP^*$ and $D_2D_3\bP^*$) can be handled likewise, save that the norm of $K_1$ is replaced with the norm of $\dot{K_1}$.

If $X\setminus B_k^0$ is disconnected, the estimates change and higher order derivatives with respect to the third variable appear. However, we rely on the crude upper bounds
\begin{equation}
\abs{\Ccal_k(X\setminus B_k^0)}\le 3^d\quad \text{ and }\quad \binom{\abs{\Ccal_k(X\setminus B_k^0)}}{\ell}\le 2^{\abs{\Ccal_k(X\setminus B_k^0)}}
\end{equation}
to see that the partial derivatives of $\boldsymbol{P^*}$ with respect to the first three variables satisfy
\begin{align}
\norm{D_1\boldsymbol{P^*}(I_1,I_2,K_1,K_2)(\dot{I}_1)}_{k,X}^* &\le A^{-\abs{X}_k} A \tnorm{\dot{I}_1}_k \big(\norm{K_1}_k^{(A)}\big)^{3^d}; \\
\norm{D_2\boldsymbol{P^*}(I_1,I_2,K_1,K_2)(\dot{I}_2)}_{k,X}^* &\le A^{-\abs{X}_k} A \tnorm{\dot{I}_2}_k \big(\norm{K_1}_k^{(A)}\big)^{3^d};\text{ and} \\
\norm{D_3^{\ell}\boldsymbol{P^*}(I_1,I_2,K_1,K_2)(\dot{K}_1^\ell)}_{k,X}^* &\le A^{-\abs{X}_k} \ell !\, 2^{3^d} A \tnorm{I_1-I_2}_k \big(\norm{\dot{K}_1}_k^{(A)}\big)^\ell \big(\norm{K_1}_k^{(A)}\big)^{3^d-\ell},
\end{align}
using Lemmas \ref{L:norms}(1) and \ref{L:*norms}(2) repeatedly. Finally, for the derivative with respect to $K_2$, $D_4\boldsymbol{P^*}(I_1,I_2,K_1,K_2)(\dot{K}_2)(X,v)=\dot{K}_2(X,v+\boldsymbol{\nabla}_N\psi_k)$, we obtain the straightforward inequality
\begin{equation}
\frac{\abs{\dot{K}_2(X,\cdot+\boldsymbol{\nabla}_N\psi_k)}_{k,X,v}}{\bunderline{w}_k^X(v)\vee\bunderline{w}_k^X(v+\boldsymbol{\nabla}_N\psi_k)} \le \frac{\norm{\dot{K}_2(X)}_{k,X}\,\bunderline{w}_k^X(v+\boldsymbol{\nabla}_N\psi_k)}{\bunderline{w}_k^X(v)\vee\bunderline{w}_k^X(v+\boldsymbol{\nabla}_N\psi_k)} \le A^{-\abs{X}_k}\norm{\dot{K}_2}_k^{(A)},
\end{equation}
which the definition of the $\norm{\cdot}_{k,X}^*$-norm is designed to facilitate.

We conclude that $\bP^*$ is smooth everywhere and has uniformly bounded derivatives on any domain $\{\tnorm{I_1}_k,\tnorm{I_2}_k,\norm{K_1}_k^{(A)}<C\}$, depending on $A$ and $C$. We also record an upper bound on the norm of $\bP^*$. For $X=B_k^0$, arguing as above, we see that
\begin{multline}
A\norm{\bP^*(I_1,I_2,K_1,K_2)(B_k^0)}_{k,B_k^0}^* \le A\tnorm{I_1(B_k^0)-1-(I_2(B_k^0)-1)}_{k,B_k^0}+A\norm{K_2(B_k^0)}_{k,B_k^0} \\
\le A \tnorm{I_1-1}_k + A \tnorm{I_2-1}_k + \norm{K_2}_k^{(A)}
\end{multline}
holds. An expression for the remaining cases can be derived similarly (or by integrating the derivatives computed above). Together they imply the inequality
\begin{multline} \label{E:P*est}
\norm{\bP^*(I_1,I_2,K_1,K_2)}_k^{(A),*} \le A \max\Big\{\big(\tnorm{I_1}_k+\tnorm{I_2}_k\big)\max\big\{\norm{K_1}_k^{(A)},\big(\norm{K_1}_k^{(A)}\big)^{3^d}\big\}, \\
\tnorm{I_1-1}_k + \tnorm{I_2-1}_k\Big\} + \norm{K_2}_k^{(A)}.
\end{multline}

\subsubsection{The multiplication map $\bM^*$}

Evidently, on any $k$-block other than $B_k^0$, $\bM^*$ acts as the identity in the first variable and is constant in the second variable; we therefore concentrate on the case $B=B_k^0$. We compute the partial derivatives as follows:
\begin{equation}
D_1\bM^*(F,t)(B_k^0,v)(\dot{F})=\frac{\dot{F}(B_k^0,v)}{t};
\end{equation}
and
\begin{equation}
D_2^r\bM^*(F,t)(B_k^0,v)(\dot{t}^r)=\frac{(-1)^r r! F(B_k^0,v)}{t^{r+1}}\dot{t}^r.
\end{equation}
Accordingly, we want to consider a domain $\{(F,t)\in M(\Bcal_k,\tnorm{\cdot}_k)\times\R\setminus \{0\}:\tnorm{F}_k<C,\abs{t}>C^{-1}\}$, and on such domain the map $\bM^*$ is smooth with derivatives uniformly bounded as
\begin{equation}
\norm{D_1D_2^r\bM^*(F,t)}\le r!\, C^{r+2}.
\end{equation}

\subsubsection{The exponential map $\bE$}

This is the same mapping, adapted to the vector field formalism, as the mapping that appears under the same name in Chapter 9 of \cite{ABKM} and features in the proof of Theorem \ref{T:smooth} (in our numbering).  A slight difference arises, essentially due to the extra factor of two in \eqref{E:Gest}, in that we take the domain to be $\{\norm{H}_{k,0}<1/16\}$ (instead of an upper bound involving $1/8$). Otherwise, Lemma 9.3 of \cite{ABKM} translates without difficulty into our setting, and so we may rely on the fact that $\bE$ is smooth with uniformly bounded derivatives on the above domain. We note in passing that insisting on an explicit supremum over all $k$-blocks $B$ in the definition of the norm $\tnorm{\cdot}_k$ does not affect any results or estimates for translation-invariant functionals (since they behave identically on all $k$-blocks).

\subsubsection{The integration map $\bR_1$}

This map is essentially the integration map in Lemma \ref{L:Rbound}, save that we consider it with respect to different norms. Since $\bq$ is fixed in this analysis, $\bR_1$ is a linear map and so it suffices to prove boundedness. Using Lemma \ref{L:Rreg} to interchange the order of differentiation and integration, as well as Theorem \ref{T:wf}(9) and Proposition \ref{P:wfpsi}, we have
\begin{align}
\abs{\bR_1(K)(X)}_{k,X,v} &\le \int_{\cbO_N}\;\abs{K(X)}_{k,X,v+\eta}\,\nu_{k+1}^{\ssup{\bq}}(\d\eta) \nonumber \\
&\le \norm{K(X)}_{k,X}^* \int_{\cbO_N}\;\bunderline{w}_k^X(v+\eta) + \bunderline{w}_k^X(v+\boldsymbol{\nabla}_N\psi_k+\eta)\,\nu_{k+1}^{\ssup{\bq}}(\d\eta) \nonumber \\
&\le \left(\frac{A}{2}\right)^{-\abs{X}_k}B^{-\abs{\Ccal_k(X)}}\norm{K}_k^{(A/2,B),*}\left(\frac{\Acal_{\Pcal}}{2}\right)^{\abs{X}_k}(1+C_k)\bunderline{w}_{k:k+1}^X(v),
\end{align}
from which we deduce that $\norm{\bR_1(K)}_{k:k+1}^{((A/2\Acal_{\Pcal}),B)}\le (1+C_k)\norm{K}_k^{(A/2,B),*}$.

\subsubsection{The integration map $\bR_2$}

This map, which is linear in its arguments for $\bq$ fixed, also plays a role in the proof of Theorem \ref{T:smooth} and with respect to the same norms. As outlined in section \ref{S:5} above, the corresponding statement in \cite{ABKM} (Lemma 9.8) translates into our setting without difficulty and so we simply recall here the relevant bounds. There is a constant $C_0$, not depending on $L$, $A$ and $h$, such that
\begin{equation} \label{E:R2est}
\norm{\bR_2(H,K)}_{k,0}\le C_0\big(\norm{H}_{k,0}+\norm{K}_k^{(A)}\big);
\end{equation}
and, for $j_1$,$j_2=0$ or $1$,
\begin{equation}
\norm{D_1^{j_1}D_2^{j_2}\bR_2(H,K)(\dot{H}^{j_1},\dot{K}^{j_2})}_{k,0}\le C_0j_1(1-j_2)\norm{\dot{H}}_{k,0}+C_0j_2(1-j_1)\norm{\dot{K}}_k^{(A)}.
\end{equation}

\subsubsection{The polynomial map $\bP_2$}

It is convenient to treat the polynomial map $\bP_1$ last and so we begin by considering $\bP_2$. The argument closely follows that in Chapter 9 of \cite{ABKM} for the equivalent map, save that we rely on the properties of the starred norms instead. We require that $\tnorm{I-1}_k<(2A)^{-1}$ and $\norm{K}_k^{(A),*}< 1/2$. On those assumptions, we first bound the norm
\begin{equation}
\norm{\bP_2(I,K)(X)}_{k,X}^* \le \sum_{Y\in\Pcal_k(X)}\tnorm{I-1}_k^{\abs{X\setminus Y}_k}\big(\norm{K}_k^{(A),*}\big)^{\abs{\Ccal_k(Y)}}A^{-\abs{Y}_k},
\end{equation}
using Lemma \ref{L:*norms}(1) and (3), from which we obtain the estimate 
\begin{align}
\norm{\bP_2(I,K)(X)}_{k,X}^* &\le \sum_{\ell=0}^{\abs{X}_k} \binom{\abs{X}_k}{\ell} A^{-\ell}\tnorm{I-1}_k^{\abs{X}_k-\ell} \nonumber \\
&= \big( \tnorm{I-1}_k + A^{-1} \big)^{\abs{X}_k} \le \left(\frac{1}{2A}+\frac{1}{A}\right)^{\abs{X}_k}\le 1,
\end{align}
uniformly in $N$. For the derivatives, we compute as follows:
\begin{align}
 & \frac{1}{j_1!j_2!}  \norm{D_1^{j_1}D_2^{j_2}\bP_2(I,K)(\dot{I}^{j_1},\dot{K}^{j_2})(X)}_{k,X}^* \nonumber \\
&\le \sum_{Y\in\Pcal_k(X)} \binom{\abs{X\setminus Y}_k}{j_1} \tnorm{I-1}_k^{\abs{X\setminus Y}_k-j_1} \tnorm{\dot{I}}_k^{j_1} \binom{\abs{\Ccal_k(Y)}}{j_2} \big(\norm{K}_k^{(A),*}\big)^{\abs{\Ccal_k(Y)}-j_2} \big(\norm{\dot{K}}_k^{(A),*}\big)^{j_2} A^{-\abs{Y}_k} \nonumber \\
&\le \sum_{Y\in\Pcal_k(X)} 2^{\abs{X\setminus Y}_k} (2A)^{j_1-\abs{X\setminus Y}_k} \tnorm{\dot{I}}_k^{j_1} 2^{\abs{\Ccal_k(Y)}} 2^{j_2-\abs{\Ccal_k(Y)}} \left(\norm{\dot{K}}_k^{(A),*}\right)^{j_2} A^{-\abs{Y}_k} \nonumber \\
&\le \left(\frac{A}{2}\right)^{-\abs{X}_k}\left(2A\tnorm{\dot{I}}_k\right)^{j_1} \left(2\norm{\dot{K}}_k^{(A),*}\right)^{j_2}.
\end{align}
This confirms that the derivatives of $\bP_2$ are uniformly bounded as
\begin{equation}
\norm{D_1^{j_1}D_2^{j_2}\bP_2(I,K)(\dot{I}^{j_1},\dot{K}^{j_2})}_k^{(A/2),*} \le j_1!j_2!\left(2A\tnorm{\dot{I}}_k\right)^{j_1} \left(2\norm{\dot{K}}_k^{(A),*}\right)^{j_2}
\end{equation}
on the domain $\{(I,K):\tnorm{I-1}_k<(2A)^{-1},\norm{K}_k^{(A),*}<1/2\}$. We further deduce the bound
\begin{equation} \label{E:P2est}
\norm{\bP_2(I,K)}_k^{(A/2),*} \le 2A\tnorm{I-1}_k+2\norm{K}_k^{(A),*}
\end{equation}
by integrating the first derivatives as in \cite{ABKM}.

\subsubsection{The polynomial map $\bP_3$}

Given a cluster $X\in\Pcal_k^{\text{cl}}$, we rely again on the sub-multiplicativity of the $\norm{\cdot}_{k,X}^*$-norm shown in Lemma \ref{L:*norms}(1) to see the estimate
\begin{align}
\frac{1}{j!}  ||D^j\bP_3(K)(\dot{K}^j) & (X)||_{k,X}^* \le \sum_{\substack{\Jcal\subset\Ccal_k(X) \\ \abs{\Jcal}=j}} \prod_{Z\in\Ccal_k(X)\setminus\Jcal}\norm{K(Z)}_{k,Z}^* \prod_{Z'\in\Jcal}\norm{\dot{K}(Z')}_{k,Z'}^* \nonumber \\
&\le \binom{\abs{\Ccal_k(X)}}{j} \left(\frac{A}{2}\right)^{-\abs{X}_k} \left(\norm{K}_k^{(A/2),*}\right)^{\abs{\Ccal_k(X)}-j} \left(\norm{\dot{K}}_k^{(A/2),*}\right)^j \nonumber \\
&\le \left(\frac{A}{2}\right)^{-\abs{X}_k} B^{-\abs{X}_k} \left(2B\norm{K}_k^{(A/2),*}\right)^{\abs{\Ccal_k(X)}-j} \left(2B\norm{\dot{K}}_k^{(A/2),*}\right)^j,
\end{align}
from which we deduce the bound
\begin{equation}
\norm{D^j\bP_3(K)(\dot{K}^j)}_k^{(A/2,B),*} \le j! \left(2B\norm{\dot{K}}_k^{(A/2),*}\right)^j,
\end{equation}
uniformly on the domain $\{K\in M(\Pcal_k^c,\norm{\cdot}_k^{(A/2),*}):\norm{K}_k^{(A/2),*}<(2B)^{-1}\}$.

\subsubsection{The polynomial map $\bP_1$}

The outermost map $\bP_1$ is identical to that considered in Lemma 9.6 of cite{ABKM} (subject to translation the latter into our vector field setting), save for the appearance of an additional scalar factor (labelled $t$) that appears for polymers $X_1\cup X_2$ containing $B_k^0$. Provided that this scalar is bounded (i.e. $\abs{t}<C$), it thus follows that the norm of $\bP_1(t,I_1,I_2,J,K)$ is bounded (uniformly in $N$) as
\begin{equation}
\norm{\bP_1(t,I_1,I_2,J,K)}_{k+1}^{(A)}\le C,
\end{equation}
for a choice of constants compatible with Theorem 5.10 above (including setting $B=A$) and restricted to the domain determined by the constraints
\begin{equation}
\Big\{\abs{t}<C,\tnorm{I_1-1}_k,\tnorm{I_2-1}_k<1/2,\tnorm{J}_k<A^{-2},\norm{K}_{k:k+1}^{(A/(2\Acal_{\Pcal}),B)}<1\Big\}.
\end{equation}

Regarding the derivatives, $\bP_1$ is linear in $t$ and so all estimates from \cite{ABKM} carry across with the presence of an additional factor scalar factor. Since we are only concerned with showing bounds on the derivatives, we can ignore the fact that this additional factor is only applicable for certain polymers in the sum (assuming $C\ge 1$). In summary, we arrive at the following inequalities:
\begin{multline}
\norm{D_2^{j_2}D_3^{j_3}D_4^{j_4}D_5^{j_5}(t,I_1,I_2,J,K)(\dot{I_1}^{j_2},\dot{I_2}^{j_3},\dot{J}^{j_4},\dot{K}^{j_5})}_{k+1}^{(A)} \le \\
j_2!j_3!j_4!j_5! \;C \tnorm{\dot{I_1}}_k^{j_2}\tnorm{\dot{I_2}}_k^{j_3}\left(A^2\tnorm{\dot{J}}_k^{j_2}\right)^{j_4}\left(\norm{\dot{K}}_{k:k+1}^{(A/(2\Acal_{\Pcal}),B)}\right)^{j_5};
\end{multline}
and
\begin{multline}
\norm{D_1D_2^{j_2}D_3^{j_3}D_4^{j_4}D_5^{j_5}(t,I_1,I_2,J,K)(\dot{t},\dot{I_1}^{j_2},\dot{I_2}^{j_3},\dot{J}^{j_4},\dot{K}^{j_5})}_{k+1}^{(A)} \le \\
j_2!j_3!j_4!j_5! \;\abs{\dot{t}} \times \tnorm{\dot{I_1}}_k^{j_2}\tnorm{\dot{I_2}}_k^{j_3}\left(A^2\tnorm{\dot{J}}_k^{j_2}\right)^{j_4}\left(\norm{\dot{K}}_{k:k+1}^{(A/(2\Acal_{\Pcal}),B)}\right)^{j_5}.
\end{multline}

\subsubsection{Conclusions}

In the preceding sub-sections we have identified individual domains, on which each of the maps in the composition of $\overline{\bS}_k$ is smooth with uniformly bounded derivatives. It remains to put everything together.

We begin with the requirements of the outer map, the polynomial $\bP_1$, and work our way through the inner maps in the composition. The inequality
\begin{equation} \label{E:P1A}
\Big|\Big|\bR_1\Big(\bP_3\big(\bP_2(\bE(-H),\bP^*(\bE(-H),\bE^*(-H),K(0),K(<k)))\big)\Big)\Big|\Big|_{k:k+1}^{(A/2\Acal_{\Pcal},B)}<1
\end{equation}
is satisfied if
\begin{equation}
\Big|\Big|\bP_3\Big(\bP_2\big(\bE(-H),\bP^*(\bE(-H),\bE^*(-H),K(0),K(<k))\big)\Big)\Big|\Big|_k^{(A/2,B),*}<\frac{1}{1+C_k}
\end{equation}
holds, and thus if
\begin{equation}
\norm{\bP_2(\bE(-H),\bP^*(\bE(-H),\bE^*(-H),K(0),K(<k)))}_k^{(A/2),*}<\frac{1}{2B(1+C_k)}.
\end{equation}
Since we require, in any event, $\norm{H}_{k,0}<1/32$ for the smoothness of $\bE$ and $\bE^*$, we obtain from \eqref{E:P*est} in combination with the bounds $\tnorm{\bE(-H)-1}_k\le 8\norm{H}_{k,0}$ and $\tnorm{\bE^*(-H)-1}_k\le 16\norm{H}_{k,0}$, the following estimate for the perturbation map:
\begin{multline} \label{E:P*calc}
\norm{\bP^*\big(\bE(-H),\bE^*(-H),K(0),K(<k)\big)}_k^{(A),*}\le \max\big\{3A\norm{K(0)}_k^{(A)},24A\norm{H}_{k,0}\big\} \\
+\norm{K(<k)}_k^{(A)}.
\end{multline}
Together with \eqref{E:P2est}, this confirms that the desired bound in \eqref{E:P1A} holds provided that
\begin{equation}
\max\big\{\norm{H}_{k,0},\norm{K(0)}_k^{(A)},\norm{K(<k)}_k^{(A)}\big\}<\frac{1}{140AB(1+C_k)}.
\end{equation}
These constraints simultaneously ensure the smoothness of the individual maps $\bE^*$, $\bP^*$, $\bE$, $\bR_1$, $\bP_2$ and $\bP_3$ in accordance with the conditions identified previously.

For the constraint on the fourth argument of $\bP_1$, we employ the straightforward identity
\begin{equation}
\bM^*(\bE(\overline{H}),s)-1=\int_0^1\; \frac{\d}{\d t}\bM^*(\bE(t\overline{H}),1+t(s-1))\,\d t,
\end{equation}
where $\overline{H}$ is short for $-\bR_2(H,K(0))$. Given the bounds on the derivatives of $\bM^*$ and $\bE$ computed above, this expression yields the following inequality:
\begin{align}
\tnorm{\bM^*(\bE(\overline{H}),s)-1}_{k,B_k^0} &\le \int_0^1\; \frac{8\norm{\overline{H}}_{k,0}}{1+t(s-1)}+(1+8\norm{\overline{H}}_{k,0})\frac{\abs{s-1}}{(1+t(s-1))^2}\,\d t \nonumber \\
&\le 8\norm{\overline{H}}_{k,0}\frac{\log s}{s-1}+(1+8\norm{\overline{H}}_{k,0})\Big|\frac{s-1}{s}\Big|.
\end{align}
A simple calculation then shows that the required upper bound,
\begin{equation}
\tnorm{\bM^*(\bE(\overline{H}),s)-1}_k<\frac{1}{A^2},
\end{equation}
holds if $\norm{\overline{H}}_{k,0}<(32A^2)^{-1}$ and $\abs{s-1}<(5A^2)^{-1}$. We recall here that, for $B\neq B_k^0$, the simpler estimate $\tnorm{\bM^*(\bE(\overline{H}),s)-1}_{k,B}\le 8 \norm{\overline{H}}_{k,0}$ is available. The condition on $\norm{\overline{H}}_{k,0}$ translates, via the inequality in \eqref{E:R2est}, into the requirement
\begin{equation}
\max\big\{\norm{H}_{k,0},\norm{K(0)}_k^{(A)}\big\}<\frac{1}{64C_{\eqref{E:R2est}} A^2},
\end{equation}
in which expression we have re-labelled the constant called $C_0$ in \eqref{E:R2est} to avoid confusion. The other condition amounts to an upper bound on the error map $\be^*$:
\begin{equation}
\abs{\be^*(K(0),\bP^*(\bE(-H),\bE^*(-H),K(0),K(<k)))-1} < \frac{1}{5A^2}.
\end{equation}
Using the inequality $\abs{e^x-1}\le 2\abs{x}$ for $\abs{x}\le 1$, the estimates for $\be^*$ above and \eqref{E:P*calc}, we see that this condition is met if
\begin{equation}
\max\big\{\norm{H}_{k,0},\norm{K(0)}_k^{(A)},\norm{K(<k)}_k^{(A)}\big\}<\frac{1}{140\Acal_{\Bcal} A^2(1+C_k)}.
\end{equation}
We note in passing that, with the above, we have that
\begin{equation}
\abs{\be^*(K(0),\bP^*(\bE(-H),\bE^*(-H),K(0),K(<k)))} \le (5A^2+1)/5A^2 < 2;
\end{equation}
this ensures that the derivatives of $\bP_1$ remain uniformly bounded despite the presence of the additional scalar factor.

Finally, relying again on the estimates for $\bE$ and $\bR_2$, we check that the remaining conditions derived from the domain of $\bP_1$ are met so long as
\begin{equation}
\max\big\{\norm{H}_{k,0},\norm{K(0)}_k^{(A)}\big\}<\frac{1}{32C_{\eqref{E:R2est}}}.
\end{equation}
Recalling the choice $B=A$, we then combine the various constraints identified above into the (non-optimal) joint norm bound
\begin{equation} \label{E:Sdom}
\max\big\{\norm{H_k}_{k,0},\norm{K_k(0)}_k^{(A)},\norm{K_k(<k)}_k^{(A)}\big\}<\frac{1}{140 A^2 C_{\eqref{E:R2est}} \Acal_{\Bcal} (1+C_{\eqref{E:Ckdef}})}\wedge \rho(A),
\end{equation}
where the scale subscript has been re-inserted and (strictly) we read $C_{\eqref{E:R2est}}$ and $\Acal_{\Bcal}$ as floored at one. We have also replaced the scale-dependent integration constant $C_k$ with its uniform upper bound, called $C(L)$ in Proposition \ref{P:wfpsi} and labelled $C_{\eqref{E:Ckdef}}$ in the above expression. It then follows from the preceding steps and the chain rule that, on the domain characterised by \eqref{E:Sdom}, the map $\overline{\bS}_k$ is smooth with derivatives uniformly bounded (depending on $L$ and $A$).

In combination with Theorem \ref{T:smooth}, which provides bounds for all polymers on which $\overline{\bS}_k$ agrees with $\bS_k$, this completes the proof of Theorem \ref{T:EFsmooth} with the domain $\mathfrak{U}_k$ identified in \eqref{E:Sdom}. As noted at the beginning of this section, we further desire that $\overline{\bS}_k$ maps into $\mathfrak{U}_{k+1}$, so that it suffices to check the conditions in \eqref{E:Sdom} once at the smoothness scale $\bar{k}$. This only follows from the argument presented in the next section, but it is easy to see that the desired property of $\overline{\bS}_k$ can be guaranteed to hold at least for any finite number of scales. Indeed, Theorem \ref{T:EFsmooth} guarantees the existence of a constant $C_{\overline{\bS}_k}(L,A)=\bunderline{C}_{1,0,0}+\bunderline{C}_{0,1,0}+\bunderline{C}_{0,0,1}$ such that
\begin{equation}
\norm{K_{k+1}(<k+1)}_k^{(A)}\le C_{\overline{\bS}_k} \max\big\{\norm{H_k}_{k,0},\norm{K_k(0)}_k^{(A)},\norm{K_k(<k)}_k^{(A)}\big\}
\end{equation}
whenever \eqref{E:Sdom} holds. Given our constraints on $L$ and $A$, we have that $C_{\overline{\bS}_k}>1$ in general. However, the functionals $H_k$ and $K_k(0)$ are unchanged from the bulk renormalisation group flow and, according to Theorem \ref{T:mainrep}, satisfy  norm constraints $\norm{H_k}_{k,0},\norm{K_k(0)}_k^{(A)} \le C_{\textup{bulk}} \vartheta^k$ for some $C_{\textup{bulk}}>0$ and $\vartheta \le 2/3$. Moreover, at the coalescence scale $\bar{k}$, the initial perturbation functional equals the unperturbed functional, that is, $K_{\bar{k}}(<\bar{k})=K_{\bar{k}}(0)$. We deduce from these observations a crude upper bound at scale $k>\bar{k}$ as follows:
\begin{equation} \label{E:kfplusMbound}
\max\big\{\norm{H_k}_{k,0},\norm{K_k(0)}_k^{(A)},\norm{K_k(<k)}_k^{(A)}\big\} \le C_{\textup{bulk}} \, \vartheta^{\bar{k}} \, C_{\overline{\bS}_k}^{k-\bar{k}}.
\end{equation}
Since the constant $\mathfrak{M}$ from Proposition \ref{P:wfpsi} does not depend on $\eps$ (nor on $L$ or $N$), we can thus choose $\bar{k}$ large enough ($\eps$ small enough) to ensure that the condition in \eqref{E:Sdom} is met at all scales from $\bar{k}$ to (and including) $\bar{k}+\mathfrak{M}$.

\section{Contraction and error estimate} \label{S:9}

\subsection{Contraction of $\overline{\bS}_k$}

Following the strategy outlined in section \ref{S6:3}, we aim to show next that $\overline{\bS}_k$ is a contraction at all scales save finitely many. We thus concentrate on scales $k\ge \bar{k}+\mathfrak{M}$ below and, pursuant to Lemma \ref{L:K<kprop}(2), it is enough to consider polymers $U\in\Pcal_{k+1}^c$ containing the zero-block $B_{k+1}^0$. We recall the expression in \eqref{E:DSzero} for the linearisation of $\overline{\bS}_k$ about $(0,0,0)$. The formal statement anticipated by our initial outline is the following.
\begin{prop} \label{P:DSkbarbound}
There exists an $L_0>0$ with the following property. Given $L\ge L_0$ there is a choice of triplet $(L,A,h)$ such that the conclusions of Theorem \ref{T:mainrep} and the inequality
\begin{multline} \label{E:DSkbarbound}
A^{\abs{U}_{k+1}}\norm{D\overline{\bS}_k(0,0,0)(H_k,K_k(0),K_k(<k))(U)}_{k+1,U} \le 4A\norm{H_k}_{k,0}+\frac{2}{3}\norm{K_k(0)}_k^{(A)} \\
+\frac{5}{9}\norm{K_k(<k)}_k^{(A)}\qquad\forall\, k\ge k+\mathfrak{M} \textup{ and } B_{k+1}^0\subset U\in\Pcal_{k+1}^c
\end{multline}
hold simultaneously, for all $\bar{k}$ large enough and all $N>\bar{k}+\mathfrak{M}$.
\end{prop}
\begin{proof}
The restriction to sufficiently large $\bar{k}$ is required for the bound $\abs{\boldsymbol{\nabla}_N\psi_k}_{k,X}<1/2$ (see section \ref{S7:1}), which in turn is needed for Theorem \ref{T:EFsmooth}. This ensures, in particular, that references to a Fr\'ech\'et derivative of $\overline{\bS}_k$ are proper. As anticipated, the estimates for the first two sums in \eqref{E:DSzero} mirror those for the bulk renormalisation group flow, and we will not repeat the detailed counting arguments from Lemmas 10.1 to 10.4 of \cite{ABKM}.

We begin with the first term and define the shorthand
\begin{equation}
F(U,v)=\sum_{\substack{X\in\Pcal_k^c\setminus\Bcal_k \\ \pi_k(X)=U}}\int_{\cbO_N}\;K_k(X,<k,v+\psi_k+\eta)\,\nu_{k+1}^{\ssup{\bq}}(\d\eta)
\end{equation}
for the sum over larger polymers. Monotonicity of $\abs{\cdot}_{k,X,v}$ and our weight functions in their polymer-dependence yield the following estimate for the norm at scale $(k+1)$:
\begin{align}
\norm{F(U)}_{k+1,U} &\le \sum_{\substack{X\in\Pcal_k^c\setminus\Scal_k \\ \pi_k(X)=U}} \norm{\bR_{k+1}^{\ssup{\bq}}K_k(X,<k,\cdot+\boldsymbol{\nabla}_N\psi_k)}_{k:k+1,X} \nonumber \\
&\quad\quad + \sum_{\substack{X\in\Scal_k\setminus\Bcal_k \\ \pi_k(X)=U}} \norm{\bR_{k+1}^{\ssup{\bq}}K_k(X,<k,\cdot+\boldsymbol{\nabla}_N]\psi_k)}_{k:k+1,X}.
\end{align}
For each of the norms appearing above, we deploy the bound
\begin{equation}
\norm{\bR_{k+1}^{\ssup{\bq}}K_k(X,<k,\cdot+\boldsymbol{\nabla}_N\psi_k)}_{k:k+1,X} \le C_k \left(\frac{\Acal_{\Pcal}}{2}\right)^{\abs{X}_k}\norm{K_k(X,<k)}_{k,X},
\end{equation}
which follows straightforwardly from Proposition \ref{P:wfpsi}. The same reasoning as in Lemma 10.2 of \cite{ABKM} then leads to the upper bound 
\begin{equation}
A^{\abs{U}_{k+1}} \norm{F(U)}_{k+1,U} \le \frac{C_k}{4}\vartheta\norm{K_k(<k)}_k^{(A)}
\end{equation}
for the first term in \eqref{E:DSzero}. The second term, which we write as
\begin{equation}
G(U,v)=\sum_{\substack{B\in\Bcal_k \\ \overline{B}=U}}\big(1-\Pi_2\big)\int_{\cbO_N}\;K_k(X,0,v+\eta)\,\nu_{k+1}^{\ssup{\bq}}(\d\eta),
\end{equation}
is unaffected by the external field $\psi_k$ and so the inequality
\begin{equation}
A\norm{G(U)}_{k+1,U} \le \frac{1}{2}\vartheta\norm{K_k(0)}_k^{(A)}
\end{equation}
follows immediately from the general case (that is, Lemma 10.4 of \cite{ABKM}, adapted to our vector field setting). We consider next the additional terms that appear only in the case $U=B_{k+1}^0$, denoted by $\bI^0$ as in \eqref{E:DSzero}.

Our approach is similar to that for the single block terms in the bulk renormalisation group flow, in that we aim to use the inequality in \eqref{E:Tayk+1est}. In order to compute the norm $\abs{\cdot}_{k+1,B_k^0,0}$, we examine the Taylor polynomials at $v=0$ of each of the terms in $\bI^0$. By construction, the zero-order Taylor polynomials vanish. Hence, we consider $r\ge 1$ and let $g^{\ssup{r}}\in\mathfrak{V}_N^{\otimes r}$. We obtain the following estimate:
\begin{align}
\bigg|\text{Tay}_0^{\ssup{r}}\int_{\cbO_N}\; K_k(B_k^0,<k, & \cdot+\boldsymbol{\nabla}_N\psi_k+\eta) - K_k(B_k^0,<k,\boldsymbol{\nabla}_N\psi_k+\eta)\,\nu_{k+1}^{\ssup{\bq}}(\d\eta) \big(g^{\ssup{r}}\big)\bigg| \nonumber \\
&\le \int_{\cbO_N}\; \big|\text{Tay}_{\boldsymbol{\nabla}_N\psi_k+\eta}^{\ssup{r}} K_k(B_k^0,<k)\big(g^{\ssup{r}}\big)\big| \,\nu_{k+1}^{\ssup{\bq}}(\d\eta) \nonumber \\
&\le \int_{\cbO_N}\; \abs{K_k(B_k^0,<k)}_{k,B_k^0,\boldsymbol{\nabla}_N\psi_k+\eta}\,\abs{g^{\ssup{r}}}_{k,B_k^0} \,\nu_{k+1}^{\ssup{\bq}}(\d\eta) \nonumber \\
&\le \int_{\cbO_N}\; \norm{K_k(B_k^0,<k)}_{k,B_k^0}\,\abs{g^{\ssup{r}}}_{k,B_k^0} \,\bunderline{w}_k^{B_k^0}(\boldsymbol{\nabla}_N\psi_k+\eta)\,\nu_{k+1}^{\ssup{\bq}}(\d\eta).
\end{align}
We now rely on the bound $\abs{g^{\ssup{r}}}_{k,B_k^0}\le 2L^{-d/2}\abs{g^{\ssup{r}}}_{k+1,B_k^0}$, which holds as $r\ge 1$, to see that
\begin{multline}
\bigg|\int_{\cbO_N}\; K_k(B_k^0,<k,\cdot+\boldsymbol{\nabla}_N\psi_k+\eta) - K_k(B_k^0,<k,\boldsymbol{\nabla}_N\psi_k+\eta)\,\nu_{k+1}^{\ssup{\bq}}(\d\eta)\bigg|_{k+1,B_k^0,0} \\
\le C_k \Acal_{\Bcal} A^{-1} L^{-d/2} \norm{K_k(<k)}_k^{(A)}, 
\end{multline} 
pursuant to Proposition \ref{P:wfpsi}. Arguing along similar lines, we further obtain the estimate
\begin{equation}
\bigg|\int_{\cbO_N}\; K_k(B_k^0,0,\cdot+\eta) - K_k(B_k^0,0,\eta)\,\nu_{k+1}^{\ssup{\bq}}(\d\eta) \bigg|_{k+1,B_k^0,0} \le \Acal_{\Bcal} A^{-1} L^{-d/2} \norm{K_k(0)}_k^{(A)}
\end{equation} 
for the $K_k(0)$ term not depending on $\psi_k$. For the relevant Hamiltonian term involving the external field $\psi_k$, it is convenient to regard the functional as $H_k(B_k^0,\cdot+\boldsymbol{\nabla}_N\psi_k)$ (rather than pushing the $\boldsymbol{\nabla}_N\psi_k$ into the base point of the Taylor polynomial) and rely on the upper bound $\tnorm{H_k(B_k^0,\cdot+\boldsymbol{\nabla}_N\psi_k)}_{k,B_k^0}\le 8 \norm{H_k}_{k,0}$ shown in section \ref{S8:E*}. We thus arrive at the inequalities
\begin{equation}
\bigg|\int_{\cbO_N}\; H_k(B_k^0,\cdot+\boldsymbol{\nabla}_N\psi_k+\eta) - H_k(B_k^0,\boldsymbol{\nabla}_N\psi_k+\eta)\,\nu_{k+1}^{\ssup{\bq}}(\d\eta) \bigg|_{k+1,B_k^0,0} \le 8 \Acal_{\Bcal} L^{-d/2} \norm{H_k}_{k,0}
\end{equation}
and
\begin{equation}
\bigg|\int_{\cbO_N}\; H_k(B_k^0,\cdot+\eta) - H_k(B_k^0,\eta)\,\nu_{k+1}^{\ssup{\bq}}(\d\eta) \bigg|_{k+1,B_k^0,0} \le 4 \Acal_{\Bcal} L^{-d/2} \norm{H_k}_{k,0}.
\end{equation}

The next task involves the $\abs{\cdot}_{k,B_k^0,tv}$-norms for $t\in [0,1]$. Here, a very crude estimate suffices. We simply take the bound
\begin{align}
\bigg|\int_{\cbO_N}\; & K_k(B_k^0,<k,\cdot+\boldsymbol{\nabla}_N\psi_k+\eta) - K_k(B_k^0,<k,\boldsymbol{\nabla}_N\psi_k+\eta)\,\nu_{k+1}^{\ssup{\bq}}(\d\eta)\bigg|_{k,B_k^0,tv} \le \nonumber \\
&\le \int_{\cbO_N}\; \abs{K_k(B_k^0,<k)}_{k,B_k^0,tv+\boldsymbol{\nabla}_N\psi_k+\eta} + \abs{K_k(B_k^0,<k)}_{k,B_k^0,\boldsymbol{\nabla}_N\psi_k+\eta} \,\nu_{k+1}^{\ssup{\bq}}(\d\eta) \nonumber \\
&\le C_k\frac{\Acal_{\Bcal}}{2}\big(\bunderline{w}_{k:k+1}^{B_k^0}(tv)+1\big)\norm{K_k(B_k^0,<k)}_{k,B_k^0} \le C_k\frac{\Acal_{\Bcal}}{A} \bunderline{w}_{k:k+1}^{B_k^0}(tv) \norm{K_k(<k)}_k^{(A)}
\end{align}
for the $K_k(<k)$ term and likewise for the other terms:
\begin{equation}
\bigg|\int_{\cbO_N}\; K_k(B_k^0,<k,\cdot+\eta) - K_k(B_k^0,<k,\eta)\,\nu_{k+1}^{\ssup{\bq}}(\d\eta) \bigg|_{k,B_k^0,tv} \le \frac{\Acal_{\Bcal}}{A} \bunderline{w}_{k:k+1}^{B_k^0}(tv) \norm{K_k(0)}_k^{(A)}, 
\end{equation} 
\begin{multline}
\bigg|\int_{\cbO_N}\; H_k(B_k^0,\cdot+\boldsymbol{\nabla}_N\psi_k+\eta) - H_k(B_k^0,\boldsymbol{\nabla}_N\psi_k+\eta)\,\nu_{k+1}^{\ssup{\bq}}(\d\eta) \bigg|_{k,B_k^0,tv} \\
\le 8 \Acal_{\Bcal} \bunderline{w}_{k:k+1}^{B_k^0}(tv) \norm{H_k}_{k,0}
\end{multline}
and
\begin{equation}
\bigg|\int_{\cbO_N}\; H_k(B_k^0,\cdot+\eta) - H_k(B_k^0,\eta)\,\nu_{k+1}^{\ssup{\bq}}(\d\eta) \bigg|_{k,B_k^0,tv} \le 4 \Acal_{\Bcal} \overline{w}_{k:k+1}^{B_k^0}(tv) \norm{H_k}_{k,0}.
\end{equation}

Since our weight functions are constructed from quadratic forms, taking suprema over $t\in [0,1]$ simply means the $t$ variable drops from the above expressions. Relying further on the monotonicity of our norms and weight functions - $\abs{v}_{k+1,B_k^0}\le \abs{v}_{k+1,B_{k+1}^0}$, $\abs{\cdot}_{k+1,B_{k+1}^0,v}\le \abs{\cdot}_{k+1,B_k^0,v}$ and $\bunderline{w}_{k:k+1}^{B_k^0}\le \bunderline{w}_{k:k+1}^{B_{k+1}^0}$ - we thus arrive at the following estimate:
\begin{multline}
\abs{\bI^0}_{k+1,B_{k+1}^0,v}\le \Acal_{\Bcal}(1+\abs{v}_{k+1,B_{k+1}^0})^3\bigg(L^{-d/2}\big(C_k A^{-1}\norm{K_k(<k)}_k^{(A)}+A^{-1}\norm{K_k(0)}_k^{(A)} + \\
12\norm{H_k}_{k,0}\big) + 16L^{-\frac{3d}{2}}\bunderline{w}_{k:k+1}^{B_{k+1}^0}(v)\big(C_k A^{-1}\norm{K_k(<k)}_k^{(A)}+A^{-1}\norm{K_k(0)}_k^{(A)} + 12\norm{H_k}_{k,0}\big)\bigg).
\end{multline}
Designating by $\Gamma$ the analytic constant $\sup_{x\in\R}(1+x)^3e^{-x^2/2}$, we rely on Theorem 5.2(8) to control the weight function and the factor in $(1+\abs{v}_{k+1,B_{k+1}^0})^3$ through the expression
\begin{equation}
(1+\abs{v}_{k+1,B_{k+1}^0})^3\bunderline{w}_{k:k+1}^{B_{k+1}^0}(v)\le \Gamma\exp\Big(\frac{\abs{v}_{k+1,B_{k+1}^0}^2}{2}\Big)\bunderline{w}_{k:k+1}^{B_{k+1}^0}(v)\le \Gamma\,\bunderline{w}_{k+1}^{B_{k+1}^0}(v).
\end{equation} 
Accordingly, this yields an upper bound on the norm of $\bI^0$ given by
\begin{equation}
A\norm{\bI^0}_{k+1,B_{k+1}^0} \le \Acal_{\Bcal} \Gamma \big(L^{-\frac{d}{2}}+16L^{-\frac{3d}{2}}\big)\big(C_k\norm{K_k(<k)}_k^{(A)}+\norm{K_k(0)}_k^{(A)}+12A\norm{H_k}_{k,0}\big).
\end{equation}

Combining the various estimates above, we have thus shown that, for $B_{k+1}^0\subset U$, the first-order contribution to the norm at scale $(k+1)$ is bounded as
\begin{multline}
A^{\abs{U}_{k+1}}\norm{D\overline{\bS}_k(0,0,0)(H_k,K_k(0),K_k(<k))(U)}_{k+1,U} \le \frac{C_k}{4}\vartheta\norm{K_k(<k)}_k^{(A)} + \frac{1}{2}\vartheta\norm{K_k(0)}_k^{(A)}\\
+ \Acal_{\Bcal} \Gamma \big(L^{-\frac{d}{2}}+16L^{-\frac{3d}{2}}\big)\big(C_k\norm{K_k(<k)}_k^{(A)}+\norm{K_k(0)}_k^{(A)}+12A\norm{H_k}_{k,0}\big).
\end{multline}
Now, $k\ge \bar{k}+\mathfrak{M}$ implies that $C_k\le 4/3$ and $\vartheta\le 2/3$ by assumption. Since neither $\Acal_\Bcal$ nor $\Gamma$ depend on $L$, we may impose a lower bound $L_0$ on $L$ that includes the requirements set out in section \ref{S:5} as well as
\begin{equation} \label{E:Lconst}
\Acal_{\Bcal} \Gamma \big(L^{-\frac{d}{2}}+16L^{-\frac{3d}{2}}\big)C_k \le \frac{1}{3},
\end{equation}
for all $k\ge \bar{k}+\mathfrak{M}$. The inequality claimed in \eqref{E:DSkbarbound} then follows immediately from the above and this completes the proof.
\end{proof}

\subsection{Norm decay of $K_k(<k)$ and error estimate}

The preceding sections have laid the ground work for our next goal, which is to show that the perturbed functionals $K_k(<k)$ vanish for sufficiently large scales as
\begin{equation}
\norm{K_k(<k)}_k^{(A)} \le O\big((2/3)^{k'}),
\end{equation}
where $k'$ designates the relative scale $k'=k-(\bar{k}+\mathfrak{M})$. In the interest of clarity, we begin with a summary of the choices regarding the free constants we made along the way that permit us to achieve this aim.

The only (possibly) additional constraint on $L$ we need is \eqref{E:Lconst}. We then choose $(L,A,h)$ so that, additionally, all of section \ref{S:5} holds and the bulk renormalisation group flow exists for all $N$. With $(L,A,h)$ fixed, section \ref{S7:1} provides a lower bound on $\bar{k}$, depending on $C_f$, such that $\abs{\boldsymbol{\nabla}_N\psi_k}< 1/2$ holds for all $k\ge \bar{k}$. This ensures that Theorem \ref{T:EFsmooth} and Proposition \ref{P:DSkbarbound} are available to us. Let us then emphasise that the constant $\mathfrak{M}$ only depends on $f$ (which we regard as given throughout), but not on $\bar{k}$, $\eps$ or $N$. As outlined previously, the crude bound
\begin{equation}
\norm{K_k(<k)}_k^{(A)} \le C_{\textup{bulk}} \, \vartheta^{\bar{k}} \, C_{\overline{\bS}_k}^{k-\bar{k}}
\end{equation}
thus enables us to pick $\bar{k}$ large enough ($\eps$ small enough) such that the requirements for Theorem \ref{T:EFsmooth} in \eqref{E:Sdom} are satisfied for all $\bar{k}\le k \le \bar{k}+\mathfrak{M}$. Indeed, it follows from the preceding remarks that we can make $\norm{K_k(<k)}_k^{(A)}$ for $k=\bar{k}+\mathfrak{M}$ as small as we like by restricting to suitably small values of $\eps$, and we take advantage of this to show the main result of this section.

\begin{theorem} \label{T:K<kdecay}
On the assumptions of Proposition \ref{P:DSkbarbound} and for a choice of constants consistent with it, there is an integer $\bar{k}_0$ and there are constants $C_{\textup{ext}}=C_{\textup{ext}}(\bar{k})$ for all $\bar{k}\ge\bar{k}_0$ such that
\begin{equation} \label{E:K<kdecay}
\norm{K_k(<k)}_k^{(A)} \le C_{\textup{ext}}(\bar{k}) (2/3)^{k'}
\end{equation}
holds for all $k\ge\bar{k}+\mathfrak{M}$ (where $k'=k-\bar{k}-\mathfrak{M}$).
\end{theorem}
\begin{proof}
We proceed by a second-order Taylor expansion of the map $\overline{\bS}_k$. To that end, it is useful to designate by $\bunderline{C}_{(2)}=\bunderline{C}_{(2)}(L,A)$ the constant such that
\begin{multline}
\norm{D^2\overline{\bS}_k(H_k,K_k(0),K_k(<k))(\dot{H}^2,\dot{K}(0)^2,\dot{K}(<k)^2)}_{k+1}^{(A)} \\
\le \frac{\bunderline{C}_{(2)}(L,A)}{4} \big(\norm{\dot{H}}_{k,0}+\norm{\dot{K}(0)}_k^{(A)}+\norm{\dot{K}(<k)}_k^{(A)}\big)^2
\end{multline}
holds for all $(H_k,K_k(0),K_k(<k))\in \mathfrak{U}_k$, the existence of which follows from Theorem \ref{T:EFsmooth}. We now use that the decay of the bulk renornalisation group flow is controlled by the constant $\vartheta\le 2/3$. This guarantees that there is a $\bar{k}_0$, depending on $(L,A,h)$ and the fixed parameters, such that the following requirements,
\begin{equation} \label{E:kbarreqs}
\begin{cases}
\abs{\boldsymbol{\nabla}_N\psi_k}_{k,B_k^0} < \frac{1}{2} & \text{ for all } k\ge \bar{k}; \\
\eqref{E:Sdom} \text{ holds } & \text{ for all } \bar{k}\le k\le \bar{k}+\mathfrak{M}; \\
\norm{K_k(<k)}_k^{(A)} \le \frac{1}{36\bunderline{C}_{(2)}}\big(\frac{3}{4}\big)^{\bar{k}} & \text{ for } k=\bar{k}+\mathfrak{M}; \text{ and} \\
\big(\frac{2}{3}+4A\big)C_{\textup{bulk}}\vartheta^k \le \frac{1}{36^2\bunderline{C}_{(2)}}\big(\frac{3}{4}\big)^{\bar{k}} & \text{ for } k=\bar{k}+\mathfrak{M},
\end{cases}
\end{equation}
are simultaneously satisfied for any $\bar{k}\ge\bar{k}_0$. We claim that letting $C_{\textup{ext}}(\bar{k})=(36\bunderline{C}_{(2)})^{-1}(3/4)^{\bar{k}}$ fulfils the desired upper bound on $\norm{K_k(<k)}_k^{(A)}$ in \eqref{E:K<kdecay}.

Plainly, we have that $\norm{K_k(<k)}_k^{(A)}\le C_{\textup{ext}}(2/3)^{k'}$ and \eqref{E:Sdom} holds at scale $k=\bar{k}+\mathfrak{M}$ by virtue of \eqref{E:kbarreqs}. We now argue by induction on $k'$. Assume therefore that $\norm{K_k(<k)}_k^{(A)}\le C_{\textup{ext}}(2/3)^{k'}$ and \eqref{E:Sdom} holds at some scale $k'>0$. Observe first that, for $B_{k+1}^0\not\subset U\in\Pcal_{k+1}^c$, we have $K_{k+1}(U,<k+1)=K_{k+1}(U,0)$ and 
\begin{equation}
\norm{K_{k+1}(0)}_{k+1}^{(A)} \le C_{\textup{bulk}}\vartheta^{k+1}\le \frac{\vartheta^{k'+1}}{36^2\bunderline{C}_{(2)}\big(\frac{2}{3}+4A\big)}\big(\frac{3}{4}\big)^{\bar{k}} < C_{\textup{ext}}\big(\frac{2}{3}\big)^{k'+1},
\end{equation}
also following from the assumptions in \eqref{E:kbarreqs}. We therefore concentrate on the case $B_{k+1}^0 \subset U$, for which Theorem \ref{T:EFsmooth} and Proposition \ref{P:DSkbarbound} provide that
\begin{multline}
A^{\abs{U}_{k+1}}\norm{K_{k+1}(U,<k+1)}_{k+1,U} \le 4A\norm{H_k}_{k,0}+\frac{2}{3}\norm{K_k(0)}_k^{(A)} +\frac{5}{9}\norm{K_k(<k)}_k^{(A)} \\
+ \frac{\bunderline{C}_{(2)}}{4} \big(\norm{H_k}_{k,0}+\norm{K_k(0)}_k^{(A)}+\norm{K_k(<k)}_k^{(A)}\big)^2 \le \frac{5}{9}C_{\textup{ext}}\big(\frac{2}{3}\big)^{k'} +\big(\frac{2}{3}+4A\big)C_{\textup{bulk}}\vartheta^{\bar{k}+\mathfrak{M}+k'} \\
+\bunderline{C}_{(2)}\Big(C_{\textup{ext}}^2\big(\frac{2}{3}\big)^{2k'}+C_{\textup{ext}}\big(\frac{2}{3}\big)^{k'}C_{\textup{bulk}}\vartheta^{\bar{k}+\mathfrak{M}+k'} + C_{\textup{bulk}}^2\vartheta^{2\bar{k}+2\mathfrak{M}+2k'}\Big) \\
\le C_{\textup{ext}}\big(\frac{2}{3}\big)^{k'}\Big(\frac{5}{9}+\frac{1}{36}\big(\frac{3}{2}\big)^{k'}\vartheta^{k'}+\frac{1}{36}\big(\frac{2}{3}\big)^{k'}+\frac{1}{36^2}\vartheta^{k'}+\frac{1}{36^4\bunderline{C}_{(2)}}\big(\frac{3}{2}\big)^{k'}\vartheta^{2k'}\Big)\le C_{\textup{ext}}\big(\frac{2}{3}\big)^{k'+1},
\end{multline}
having regard to \eqref{E:kbarreqs} and that $\vartheta\le 2/3$. This also means that \eqref{E:Sdom} holds at scale $k+1$, since all relevant norms satisfy upper bounds that contract from scale $k$ to $k+1$. The desired norm decay for $K_k(<k)$ in \eqref{E:K<kdecay} is thus established.
\end{proof}

It remains to estimate the cumulative error $e_k$. Our strategy anticipates that it should be of the order of the first error at the smoothness scale $\bar{k}$, and that it vanishes as $\bar{k}$ is taken to infinity. This is indeed what we show next.
\begin{lemma} \label{L:eN}
On the assumptions of Theorem \ref{T:K<kdecay} and for a choice of constants consistent with it, the cumulative error coordinate
\begin{equation}
e_N\le\sum_{k=\bar{k}}^{N-1}\abs{e_{k+1}-e_k}
\end{equation}
is bounded uniformly in $N$ and decays as
\begin{equation}
\sum_{k=\bar{k}}^\infty\abs{e_{k+1}-e_k}=O\Big(\big(\frac{3}{4}\big)^{\bar{k}}\Big)
\end{equation}
in the limit $\bar{k}\to\infty$.
\end{lemma}
\begin{proof}
Since $e_{\bar{k}}=0$, we compute the contribution at scale $k\ge \bar{k}$:
\begin{align}
\abs{e_{k+1}-e_k} &\le \int_{\cbO_N}\; \abs{K_k(B_k^0,<k,\eta+\boldsymbol{\nabla}_N\psi_k)} + \Big|e^{-H_k(B_k^0,\eta+\boldsymbol{\nabla}_N\psi_k)}-e^{-H_k(B_k^0,\eta)}\Big| \nonumber \\
&\quad\quad\quad + \abs{K_k(B_k^0,0,\eta)}\,\nu_{k+1}^{\ssup{\bq}}(\d\eta) \nonumber \\
&\le \frac{\Acal_{\Bcal}}{2} \Big(C_k \norm{K_k(B_k^0,<k)}_{k,B_k^0} +  \norm{K_k(B_k^0,0)}_{k,B_k^0} + \tnorm{\bE^*(-H_k)-\bE(-H_k)}_{k,B_k^0}\Big) \nonumber \\
&\le \frac{\Acal_{\Bcal}}{2A} C_k \norm{K_k(<k)}_k^{(A)} + \frac{\Acal_{\Bcal}}{2A} \norm{K_k(0)}_k^{(A)}  + 12 \Acal_{\Bcal} \norm{H_k}_{k,0}.
\end{align}
The norm bounds implied by Theorems \ref{T:mainrep} and \ref{T:K<kdecay} show that, from $k\ge \bar{k}+\mathfrak{M}$ onward, this expression is summable and thus bounded by
\begin{equation}
\sum_{k=\bar{k}+\mathfrak{M}}^\infty \abs{e_{k+1}-e_k} \le \frac{2\Acal_{\Bcal}C_{\textup{ext}}(\bar{k})}{A} + \Acal_{\Bcal}\big(\frac{1}{2A}+12\big)\frac{C_{\textup{bulk}}\vartheta^{\bar{k}+\mathfrak{M}}}{1-\vartheta},
\end{equation}
uniformly in $N$. For the first $\mathfrak{M}$ contributions, we see from the preceding reasoning that $\norm{K_{\bar{k}+\mathfrak{M}}(<\bar{k}+\mathfrak{M})}_{\bar{k}+\mathfrak{M}}^{(A)}$ is the largest $K_k(<k)$-norm that we have to contend with. This follows as otherwise $\overline{\bS}_k$ contracts from $k=\bar{k}$ and we may take $\mathfrak{M}=0$. Accordingly, we obtain an overall upper bound of the form
\begin{equation}
\sum_{k=\bar{k}}^\infty \abs{e_{k+1}-e_k} \le \big(\mathfrak{M} + \frac{2\Acal_{\Bcal}}{A}\big) C_{\textup{ext}}(\bar{k}) + \Acal_{\Bcal}\big(\frac{1}{2A}+12\big)\frac{C_{\textup{bulk}}\vartheta^{\bar{k}}}{1-\vartheta},
\end{equation}
which vanishes in the $\bar{k}\to\infty$ limit as claimed, having regard to the definition of $C_{\textup{ext}}(\bar{k})$. 
\end{proof}

\section{Proofs of the main results} \label{S:10}

\subsection{Identifying the Hessian of surface tension}

We return to our overall strategy as outlined in section \ref{S:3}, with the goal of formalising the asymptotic behaviour of the finite volume surface tension claimed in \eqref{E:Hesslim}. The renormalisation group analysis in sections \ref{S:4} and \ref{S:5}, which adapt \cite{ABKM} to functionals on arbitrary vector fields, allows us to evaluate the perturbed partition function $Z_{N,\beta}(u+\dot{u})$ in the manner we intend. We show the following statement.

\begin{prop} \label{P:Hesslim}
Let $L\ge L_0$, $\beta\ge\beta_0$ and $\delta_0$ be as in Theorem \ref{T:K<kdecay} and Proposition \ref{P:Vass} respectively, and let $V$ satisfy the assumptions in \eqref{E:Assumptions}. There is a positive constant $C=C(L,h,A,\Acal_{\Bcal})$ such that the inequality
\begin{equation} \label{E:HesslimP}
\Big|\frac{\p^2}{\p\dot{u}_i\p\dot{u}_j}\sigma_{N,\beta}(u+\dot{u})\big\rvert_{\dot{u}=0} - (\bQ_V-\bq(N,u))_{i,j} \Big|\le C\vartheta^N
\end{equation}
holds for all $N\in\N$, uniformly in $u\in B_{\delta_0}(0)$.  
\end{prop}
\begin{proof}
We recall that the assumptions of Theorem \ref{T:K<kdecay} incorporate those of Theorem \ref{T:mainrep}, which is the statement of primary importance in this context. The assumptions on the potential in \eqref{E:Assumptions}, via Proposition \ref{P:Vass} and Theorem \ref{T:sigmatilde}, guarantee that the conclusions of Theorem \ref{T:mainrep} hold for the perturbation $\Kcal_{u,\b,V}$ defined in \eqref{E:KubV} for $\beta\ge\beta_0$ and $u\in B_{\delta_0}(0)$. Moreover, the finite volume surface tension $\sigma_{N,\beta}(u)$ is a $C^3$ function of tilt on the domain $B_{\delta_0}(0)$.

Fix $N$ and $u\in B_{\delta_0}(0)$. To simplify the notation, we write $\lambda$ and $\bq$ for the real number $\l_N(\Kcal_{u,\b,V})$ and the symmetric matrix $\bq_N(\Kcal_{u,\b,V})$ given by Theorem \ref{T:mainrep}. By Steps 7 and 8 of section \ref{S:5}, this ensures the existence of a sequence $Z=(H_0,H_1,\ldots, K_N)\in\Zcal$ of functionals such that, in particular, the identity $(H_{k+1},K_{k+1})=\bT_k(H_k,K_k,\bq)$ holds for $0\le k\le N-1$. We thus write the perturbation $\Zcal_{N,\b,u}(\dot{u})$ introduced in \eqref{E:ZNuUqgrad} in terms of the functionals featuring in the renormalisation group flow as follows:
\begin{align}
\Zcal_{N,\beta,u}(\dot{u}) &= \int_{\cbO_N}\;\ex^{-\frac{1}{2}\sum_{x\in\mathbb{T}_N}(eta(x)+\beta^{1/2}\dot{u})\cdot \bq(\eta(x)+\beta^{1/2}\dot{u})}\ex^{-\beta\sum_{x\in\mathbb{T}_N}U(\beta^{-1/2}(\eta+\beta^{1/2}\dot{u}),u)}\,\nu^{\ssup{\bq(u)}}(\d\eta) \nonumber \\
&= \int_{\cbO_N}\;\ex^{-\frac{1}{2}\sum_{x\in\mathbb{T}_N}(\eta(x)+\beta^{1/2}\dot{u})\cdot \bq(\eta(x)+\beta^{1/2}\dot{u})}\sum_{X\subset\mathbb{T}_N}\prod_{x\in X}\Kcal_{u,\b,V}(\eta(x)+\beta^{1/2}\dot{u})\,\nu^{\ssup{\bq}}(\d\eta) \nonumber \\
&= \ex^{L^{Nd}(\lambda+\sum_{i=1}^d a_{i,0}\b^{1/2}\dot{u}_i)}\int_{\cbO_N}\; \big(\ex^{-H_0}\circ K_0\big)(\Lambda_N,\eta+\beta^{1/2}\dot{u})\,\nu^{\ssup{\bq}}(\d\eta),
\end{align}
where $a_{i,0}$ stands for the corresponding linear coefficient of the relevant Hamiltonian $H_0$. No other terms - beyond $\lambda$ and $a_{i,0}$ - appear as all derivatives of $\dot{u}$ vanish. Note further that $\eta+\beta^{1/2}\dot{u}\in\mathfrak{V}_N$ is a valid argument for the functionals $H_0$ and $K_0$ thanks to our construction. Repeated applications of Proposition \ref{P:Tk} then show the representation anticipated in \eqref{E:Zudotrep} (where the additional factors involving $\lambda$ and $a_{i,0}$ had been omitted for simplicity):
\begin{equation} \label{E:Zudotrepactual}
\Zcal_{N,\beta,u}(\dot{u})=\ex^{L^{Nd}(\lambda+\sum_{i=1}^d a_{i,0}\b^{1/2}\dot{u}_i)}\int_{\cbO_N}\;\Big(1+K_N(\Lambda_N,\eta+\beta^{1/2}\dot{u})\Big)\,\nu_{N+1}^{\ssup{\bq}}(\d\eta).
\end{equation}
We now determine the partial derivatives $\p^2/\p\dot{u}_i\p\dot{u}_j$ for $i,j=1,...,d$. A straightforward computation from \eqref{E:ZNuudot} and \eqref{E:Zudotrepactual} yields the desired leading term, $(\bQ_V-\bq)_{i,j}$, and leaves us with the task of controlling the derivatives
\begin{equation} \label{E:deriv1}
\frac{\frac{\p^2}{\p\dot{u}_i\p\dot{u}_j}\bR_{N+1}^{\ssup{\bq}}K_N(\Lambda_N,\b^{1/2}\dot{u})}{\b L^{Nd}(1+\bR_{N+1}^{\ssup{\bq}}K_N(\Lambda_N,\b^{1/2}\dot{u}))}
\end{equation}
and
\begin{equation} \label{E:deriv2}
\frac{\frac{\p}{\p\dot{u}_i}\bR_{N+1}^{\ssup{\bq}}K_N(\Lambda_N,\b^{1/2}\dot{u})\frac{\p}{\p\dot{u}_j}\bR_{N+1}^{\ssup{\bq}}K_N(\Lambda_N,\b^{1/2}\dot{u})}{\b L^{Nd}(1+\bR_{N+1}^{\ssup{\bq}}K_N(\Lambda_N,\b^{1/2}\dot{u}))^2},
\end{equation}
evaluated at $\dot{u}=0$.
Using Lemma \ref{L:Rreg} and the chain rule, we compute
\begin{equation}
\frac{\p^2}{\p\dot{u}_i\p\dot{u}_j}\bR_{N+1}^{\ssup{\bq}}K_N(\Lambda_N,\b^{1/2}\dot{u})=\b\int_{\cbO_N}\;D^2K_N(\Lambda_N,\eta+\b^{1/2}\dot{u})(\boldsymbol{i},\boldsymbol{j})\,\nu_{N+1}^{\ssup{\bq}}(\d\eta),
\end{equation}
where $\boldsymbol{i}$ and $\boldsymbol{j}$ are the constant vector fields whose $i$-th (and respectively $j$-th) component at each site in $\mathbb{T}_N$ is one and all other components are zero. By definition of the dual norm $\abs{\cdot}_{k,X,v}$ and \eqref{E:dualform}, we have the estimate
\begin{align}
\abs{D^2 K_N(\Lambda_N,\eta+\b^{1/2}\dot{u})(\boldsymbol{i},\boldsymbol{j})} &\le 2\abs{K_N(\Lambda_N)}_{N,\Lambda_N,\eta+\b^{1/2}\dot{u}}\abs{\boldsymbol{i}\otimes\boldsymbol{j}}_{N,\Lambda_N} \nonumber \\
&= 2h_N^{-2}L^{Nd}\abs{K_N(\Lambda_N)}_{N,\Lambda_N,\eta+\b^{1/2}\dot{u}} \nonumber \\
&\le 2h_N^{-2}L^{Nd}\bunderline{w}_N^{\Lambda_N}(\eta+\b^{1/2}\dot{u}) \norm{K_N(\Lambda_N)}_{N,\Lambda_N} \nonumber \\
&\le 2 A^{-1} h_N^{-2}L^{Nd}\bunderline{w}_N^{\Lambda_N}(\eta+\b^{1/2}\dot{u}) \norm{K_N}_N^{(A)}.
\end{align}
We now appeal to the bound on $\norm{K_N}_N^{(A)}$ in \eqref{E:KNbound} and use Theorem \ref{T:wf}(10) to control the remaining $\bunderline{w}_N^{\Lambda_N}$ integral. This yields the following inequality for the numerator in \eqref{E:deriv1}:
\begin{equation}
\Big|\frac{\p^2}{\p\dot{u}_i\p\dot{u}_j}\bR_{N+1}^{\ssup{\bq}}K_N(\Lambda_N,\b^{1/2}\dot{u})\Big|\le \b L^{Nd} \frac{\Acal_{\Bcal}}{A h_N^2} \bunderline{w}_{N:N+1}^{\Lambda_N}(\b^{1/2}\dot{u}) C_{\textup{bulk}} \vartheta^N.
\end{equation}
Arguing along similar lines, we also obtain a bound for the numerator in \eqref{E:deriv2}:
\begin{multline}
\Big|\frac{\p}{\p\dot{u}_i}\bR_{N+1}^{\ssup{\bq}}K_N(\Lambda_N,\b^{1/2}\dot{u})\frac{\p}{\p\dot{u}_j}\bR_{N+1}^{\ssup{\bq}}K_N(\Lambda_N,\b^{1/2}\dot{u})\Big| \\ \le \b L^{Nd} \frac{\Acal_{\Bcal}^2}{4A^2h_N^2}\bunderline{w}_{N:N+1}^{\Lambda_N}(\b^{1/2}\dot{u})^2 C_{\textup{bulk}}^2 \vartheta^{2N}.
\end{multline}
Crucially, we see that the volume factors $L^{Nd}$ cancel from both expressions. Moreover, it follows from the estimate for the integrated error term in \eqref{E:KNbound1} of Theorem \ref{T:mainrep} that the denominators of \eqref{E:deriv1} and \eqref{E:deriv2} are bounded away from zero, uniformly in $N$, when evaluated at $\dot{u}=0$. Since $\bunderline{w}_{N:N+1}^{\Lambda_N}(0)=1$ and $h_N^{-2}\le h^{-1}$, the claim in \eqref{E:HesslimP} follows and this completes the proof.
\end{proof}

In combination with Theorem \ref{T:sigmatilde}, this result implies that the Hessian of surface tension that appears in our main results, $\mathscr{H}_\sigma(\beta,u)$, is given by the expression
\begin{equation}
\mathscr{H}_\sigma(\beta,u)=\bQ_V-\bq(u),
\end{equation}
where all limits are taken along the appropriate sub-sequence such that Theorem \ref{T:sigmatilde} holds and $\bq(N,u)\to\bq(u)$. For the torus scaling limit, it thus suffices then to verify that the same object also characterises the covariance of the scaling limit, as it does in \cite{ABKM}. We emphasise that this step, although relatively simple, is nonetheless important. The renormalisation group representation in Theorem \ref{T:mainrep} is, strictly speaking, not the same as in \cite{ABKM}: it holds for a different class of functionals in a different series of Banach spaces defined by different (albeit analogous) constants. Accordingly, it is necessary to state that the scaling limit still follows from it; otherwise our identification of the Hessian of $\sigma_{N,\beta}$ would have no bearing on it.

For the infinite volume scaling limit, the additional task of proving its existence as a Gaussian Free Field arises. The heaviest burden falls of course on the extended renormalisation group analysis constructed in sections \ref{S:6} to \ref{S:9}. The remaining steps are collected in the next sub-section, which closes this article with a proof of Theorems \ref{T:scaling} and \ref{T:scalingIV}.

\subsection{Characterising the scaling limits}

We begin by adapting Theorem 2.10 of \cite{ABKM} into our setting to complete the proof of the torus scaling limit.
\begin{proof} [Proof of Theorem \ref{T:scaling}]
We work at a fixed tilt $u$ and our assumptions guarantee that Theorems \ref{T:mainrep} and \ref{T:sigmatilde} hold. For simplicity, we write $N$ in the following, but we mean in fact the sub-sequence $N_\ell$ given by Theorem \ref{T:sigmatilde} restricted (if necessary) to a further sub-sequence such that the matrices $\bq_N(u)$ converge to a limit $\bq$. As in \cite{ABKM}, we note that we may shift the test function $f_N$ by a constant (which may depend on $N$) without affecting any of other expressions in view of the definition of $\cbX_N$. We therefore assume that $f_N\in\cbX_N$.

We translate the expectation we want to evaluate into our notation. We use the expression
\begin{equation}
\bF_{\beta,u,V}(\phi)=\exp\Big(-\frac{1}{2}\sum_{x\in\mathbb{T}_N}\nabla\phi(x)\cdot\bq\nabla\phi(x)-\beta U(\beta^{-1/2}\nabla\phi(x),u)\Big)
\end{equation}
as a convenient shorthand. In this somewhat lighter notation, we obtain the identities
\begin{align} \label{E:fnexpcalc}
\mathbb{E}_{\gamma_{N,\beta}^{\ssup{u}}}\big[\ex^{(f_N,\phi)}\big] &= \frac{\int_{\cbX_N}\;\ex^{(\beta^{-1/2}f_N,\phi)}\bF_{\beta,u,V}(\phi)\,\mu^{\ssup{\bq}}(\d\phi)}{\int_{\cbX_N}\;\bF_{\beta,u,V}(\phi)\,\mu^{\ssup{\bq}}(\d\phi)} \nonumber \\
&= \ex^{\frac{1}{2\beta}(f_N,\mathscr{C}^{\ssup{\bq_N}}f_N)}\frac{\int_{\cbX_N}\;\bF_{\beta,u,V}(\phi+\mathscr{C}^{\ssup{\bq_N}}f_N)\,\mu^{\ssup{\bq}}(\d\phi)}{\int_{\cbX_N}\;\bF_{\beta,u,V}(\phi)\,\mu^{\ssup{\bq}}(\d\phi)} \nonumber \\
&= \ex^{\frac{1}{2\beta}(f_N,\mathscr{C}^{\ssup{\bq_N}}f_N)}\frac{\int_{\cbO_N}\;\bF_{\beta,u,V}\circ\boldsymbol{\nabla}_N^{-1}(\eta+\boldsymbol{\nabla}_N\mathscr{C}^{\ssup{\bq_N}}f_N)\,\nu^{\ssup{\bq}}(\d\eta)}{\int_{\cbO_N}\;\bF_{\beta,u,V}\circ\boldsymbol{\nabla}_N^{-1}(\eta)\,\nu^{\ssup{\bq}}(\d\eta)} \nonumber \\
&= \ex^{\frac{1}{2\beta}(f_N,\mathscr{C}^{\ssup{\bq_N}}f_N)}\frac{\int_{\cbO_N}\;1+K_N(\Kcal_{u,\beta,V})(\mathbb{T}_N,\eta+\boldsymbol{\nabla}_N\mathscr{C}^{\ssup{\bq_N}}f_N)\,\nu_{N+1}^{\ssup{\bq}}(\d\eta)}{\int_{\cbO_N}\;1+K_N(\Kcal_{u,\beta,V})(\mathbb{T}_N,\eta)\,\nu_{N+1}^{\ssup{\bq}}(\d\eta)},
\end{align}
using the representation provided by \ref{T:mainrep} and noting that all other pre-factors cancel. Using the bound on $\norm{K_N}_N^{(A)}$ in \eqref{E:KNbound} and Theorem \ref{T:wf}(10), we obtain the inequality
\begin{multline} \label{E:wfNest}
\Big|\int_{\cbO_N}\;1+K_N(\Kcal_{u,\beta,V})(\mathbb{T}_N,\eta+\boldsymbol{\nabla}_N\mathscr{C}^{\ssup{\bq_N}}f_N)\,\nu_{N+1}^{\ssup{\bq}}(\d\eta)\Big| \\
\le \frac{C_{\textup{bulk}}\Acal_\Bcal}{2A}\vartheta^N\bunderline{w}_{N:N+1}^{\mathbb{T}_N}(\boldsymbol{\nabla}_N\mathscr{C}^{\ssup{\bq_N}}f_N)
\end{multline}
for the numerator in \eqref{E:fnexpcalc}. Now we rely on Theorem \ref{T:wf}(2) to estimate the weight function. Since $\boldsymbol{\nabla}_N\mathscr{C}^{\ssup{\bq_N}}f_N\in\cbO_N$, we see that
\begin{align}
\bunderline{w}_{N:N+1}^{\mathbb{T}_N}(\boldsymbol{\nabla}_N\mathscr{C}^{\ssup{\bq_N}}f_N) &\le \exp\Big(\frac{\big(\boldsymbol{\nabla}_N\mathscr{C}^{\ssup{\bq_N}}f_N,\bunderline{M}_N\boldsymbol{\nabla}_N\mathscr{C}^{\ssup{\bq_N}}f_N\big)_{\mathfrak{V}_N}}{2\lambda}\Big) \nonumber \\
&= \exp\Big(\frac{\big(\boldsymbol{\nabla}_N\mathscr{C}^{\ssup{\bq_N}}f_N,\mathscr{M}_{N,\nabla}\boldsymbol{\nabla}_N\mathscr{C}^{\ssup{\bq_N}}f_N\big)_{\cbO_N}}{2\lambda}\Big),
\end{align}
and our construction in Appendix \ref{A:C} ensures that the inner product on the second line above corresponds exactly to the $\cbX_N$ inner product that appears in the proof of Theorem 2.10 of \cite{ABKM}. We do not repeat the detailed calculation here and conclude that the weight function in \eqref{E:wfNest} is bounded uniformly in $N$, depending on $f$. This in turn implies that the numerator in \eqref{E:fnexpcalc} converges to one as $N\to\infty$ and the same holds for the denominator. The convergence of the discrete inner product, $\big(f_N,\mathscr{C}^{\ssup{\bq_N}}f_N\big)_{\cbX_N}$, to the continuous inner product on the continuum torus $\mathbb{T}^d$, $\big((f,\mathscr{C}_{\mathbb{T}^d}^{\ssup{u}}f\big)_{L^2(\mathbb{T}^d)}$, where $\mathscr{C}_{\mathbb{T}^d}^{\ssup{u}}$ is the inverse of the operator
\begin{equation}
\mathscr{A}_{\mathbb{T}^d}^{\ssup{u}}=-\sum_{i,j=1}^d\big(\bQ_V-\bq\big)_{i,j}\partial_i\partial_j,
\end{equation}
is exactly as in \cite{ABKM}, so we skip a detailed proof. It then follows from Proposition \ref{P:Hesslim} and following remarks that the coefficients of $\mathscr{A}_{\mathbb{T}^d}^{\ssup{u}}$ indeed correspond to the Hessian of surface tension $\mathscr{H}_\sigma(\beta,u)$ as asserted. This completes the proof of the torus scaling limit.
\end{proof}

We now proceed to the infinite volume scaling limit. As anticipated, we focus on the expectation in finite volume, which we represent in our formalism as follows:
\begin{align} \label{E:IVscalinglimit}
\mathbb{E}_{\gamma_{N,\beta}^{\ssup{u}}}\Big[ & \exp\Big(\sum_{x\in\mathbb{T}_N}\nabla\phi(x)\cdot \nabla g_\eps(x)\Big)\Big] = \frac{\int_{\cbX_N}\;\ex^{\beta^{-1/2}\sum_{x\in\mathbb{T}_N} \nabla g_\eps(x)\cdot\nabla\phi(x)}\bF_{\beta,u,V}(\phi)\,\mu^{\ssup{\bq}}(\d\phi)}{\int_{\cbX_N}\;\bF_{\beta,u,V}(\phi)\,\mu^{\ssup{\bq}}(\d\phi)} \nonumber \\
&= \frac{\int_{\cbO_N}\;\ex^{\beta^{-1/2}(\eta,\boldsymbol{\nabla}_N g_\eps)}\bF_{\beta,u,V}\circ\boldsymbol{\nabla}_N^{-1}(\eta)\,\nu^{\ssup{\bq}}(\d\eta)}{\int_{\cbO_N}\;\bF_{\beta,u,V}\circ\boldsymbol{\nabla}_N^{-1}(\eta)\,\nu^{\ssup{\bq}}(\d\eta)} \nonumber \\
&= \ex^{\frac{1}{2\beta}(\boldsymbol{\nabla}_N g_\eps,\mathscr{C}^{\ssup{\bq_N},\nabla}\boldsymbol{\nabla}_N g_\eps)}\frac{\int_{\cbO_N}\;\bF_{\beta,u,V}\circ\boldsymbol{\nabla}_N^{-1}(\eta+\mathscr{C}^{\ssup{\bq_N},\nabla}\boldsymbol{\nabla}_N g_\eps)\,\nu^{\ssup{\bq}}(\d\eta)}{\int_{\cbO_N}\;\bF_{\beta,u,V}\circ\boldsymbol{\nabla}_N^{-1}(\eta)\,\nu^{\ssup{\bq}}(\d\eta)} \nonumber \\
&= \ex^{\frac{1}{2\beta}(\boldsymbol{\nabla}_N g_\eps,\mathscr{C}^{\ssup{\bq_N},\nabla}\boldsymbol{\nabla}_N g_\eps)}\frac{\int_{\cbO_N}\;\Big(\ex^{-H_0}\circ K_0(0)\Big)(\mathbb{T}_N,\eta+\mathscr{C}^{\ssup{\bq_N},\nabla}\boldsymbol{\nabla}_N g_\eps)\,\nu^{\ssup{\bq}}(\d\eta)}{\int_{\cbO_N}\;1+K_N(\mathbb{T}_N,0,\eta)\,\nu_{N+1}^{\ssup{\bq}}(\d\eta)}.
\end{align}
In the above expression, $H_0$, $K_0(0)$ and $K_N(0)$ refer to the appropriate functionals in the bulk renormalisation group flow provided by Theorem \ref{T:mainrep}. As before, we take advantage of cancellations of various pre-factors and the fact that only the constant and quadratic terms of $\sum_{x\in\mathbb{T}_N}H_0(\{x\})$ acting on a gradient vector field are non-vanishing. Our goal is to show that the inner product involving the test function converges to the desired continuum object and the ratio of partition functions converges to one, in each case, as $N\to\infty$ followed by $\eps\to 0$. In view of Lemma \ref{L:tightness}, in the following we take the $N\to\infty$ limit along a sub-sequence such that the measures $\gamma_{N,\beta}^{\ssup{u}}$ converge weakly to a limit $\gamma_{\infty,\beta}^{\ssup{u}}$. We begin with the inner product in $g_\eps$. The next result is essentially Lemma 2.4 of \cite{BPR22b}, but we include a proof for clarity.

\begin{lemma} \label{L:innerprodscaling}
Let $f=\Delta_{\R^d}g\in C_c^{\infty}(\R^d)$ and $f_\eps$, $g_\eps$ be as described in section \ref{S6:1}. Then
\begin{equation}
\limsup_{\eps\to 0}\limsup_{N\to\infty} \big(\boldsymbol{\nabla}_N g_\eps,\mathscr{C}^{\ssup{\bq_N},\nabla}\boldsymbol{\nabla}_N g_\eps\big)_{\cbO_N} < \infty,
\end{equation}
and, if in addition $\bq_N\to\bq$, the limit exists and equals
\begin{equation}
\lim_{\eps\to 0}\lim_{N\to\infty} \big(\boldsymbol{\nabla}_N g_\eps,\mathscr{C}^{\ssup{\bq_N},\nabla}\boldsymbol{\nabla}_N g_\eps\big)_{\cbO_N} = \big(f,\mathscr{C}_{\R^d}^{\ssup{u}}f\big)_{L^2(\R^d)}
\end{equation}
where $\mathscr{C}_{\R^d}^{\ssup{u}}$ stands for the inverse of the operator
\begin{equation}
\mathscr{A}_{\R^d}^{\ssup{u}}=-\sum_{i,j=1}^d\big(\bQ_V-\bq\big)_{i,j}\partial_i\partial_j.
\end{equation}
\end{lemma}
\begin{proof}
We work in Fourier space: see section \ref{S:7} above and, more generally, \cite{ABKM} and \cite{B18}. Accordingly, we begin by expressing the inner product using Plancherel's identity:
\begin{align} \label{E:discinnerprod}
\big(\boldsymbol{\nabla}_N g_\eps,\mathscr{C}^{\ssup{\bq_N},\nabla}\boldsymbol{\nabla}_N g_\eps\big)_{\cbO_N} &= \big(f_\eps,\mathscr{C}^{\ssup{\bq_N}} f_\eps\big)_{\cbX_N} = L^{-dN} \sum_{p\in\widehat{T}_N\setminus\{0\}} \widehat{\Ccal^{\ssup{\bq_N}}}(p)\abs{\widehat{\Delta g_\eps}(p)}^2 \nonumber \\
&= L^{-dN} \sum_{p\in\widehat{T}_N\setminus\{0\}} \widehat{\Ccal^{\ssup{\bq_N}}}(p)\abs{\widehat{\Delta}(p)}^2\abs{\widehat{g_\eps}(p)}^2.
\end{align}
Our assumptions on the potential in \eqref{E:Assumptions} and on $\norm{\bq_N}$ via the requirements for Theorem \ref{T:wf} provide an upper bound on $\widehat{\Ccal^{\ssup{\bq_N}}}(p)$. Indeed, we see that
\begin{align}
\widehat{(\Ccal^{\ssup{\bq_N}})^{-1}}(p) &= \sum_{j,\ell=1}^d\big(\bQ_V-\bq_N\big)_{j,l}\widehat{\nabla_j^*\nabla_\ell}(p)=\sum_{j,\ell=1}^d\big(\bQ_V-\bq_N\big)_{j,l}\overline{q_j(p)}q_\ell(p) \nonumber \\
&= \big(q(p),(\bQ_V-\bq_N)q(p)\big)_{\C^d}\ge \frac{\omega_0}{2}\abs{q(p)}^2.
\end{align}
Next, the simple estimates, $\abs{\widehat{\Delta}(p)}^2\le \abs{p}^4$ and $\abs{q(p)}^2=2d-2\sum_{j=1}^d \cos(p_j)\ge\frac{4}{\pi^2}\abs{p}^2$, confirm that the expression in \eqref{E:discinnerprod} is bounded, uniformly in $N$. The inequality
\begin{equation} \label{E:FTconvN}
\limsup_{N\to\infty}\big(\boldsymbol{\nabla}_N g_\eps,\mathscr{C}^{\ssup{\bq_N},\nabla}\boldsymbol{\nabla}_N g_\eps\big)_{\cbO_N} \le \frac{2}{\omega_0}\int_{[-\pi,\pi]^d}\;\frac{\abs{p}^4\abs{g_\eps(p)}^2}{2d-2\sum_{j=1}^d \cos(p_j)}\,\d p
\end{equation}
then follows by dominated convergence. After a change of variables $p\mapsto \eps p$, we can bound the resulting integrand by the function
\begin{align}
\1_{[-\eps^{-1}\pi,\eps^{-1}\pi]}(p)\frac{\eps^{d+4}\abs{p}^4\abs{g_\eps(\eps p)}^2}{2d-2\sum_{j=1}^d \cos(\eps p_j)} &\le \frac{\pi^2}{4}\eps^{d+2}\abs{p}^2\abs{g_\eps(\eps p)}^2 = \frac{\pi^2\abs{p}^2}{4}\Big|\sum_{x\in\Z^d} \eps^{\frac{d+2}{2}}e^{-ix\cdot \eps p} g_\eps(x)\Big|^2 \nonumber \\
&= \frac{\pi^2\abs{p}^2}{4(1+\abs{p}^2)^{d+1}}\Big|\sum_{x\in\eps\Z^d} \eps^{d}(1+\abs{p}^2)^{\frac{d+1}{2}}e^{-ix\cdot p} g(x)\Big|^2 \nonumber \\
&\le \frac{\pi^2\abs{p}^2}{4(1+\abs{p}^2)^{d+1}}C(g),
\end{align}
which is integrable over $\R^d$. Arguing again by dominated convergence and using that $g\in C_c^\infty(\R^d)$, we thus arrive at the upper bound
\begin{equation}
\limsup_{\eps\to 0}\limsup_{N\to\infty} \big(\boldsymbol{\nabla}_N g_\eps,\mathscr{C}^{\ssup{\bq_N},\nabla}\boldsymbol{\nabla}_N g_\eps\big)_{\cbO_N} \le \frac{\pi^2}{2\omega_0}\int_{\R^d}\;\abs{k}^2\abs{\widehat{g}(k)}^2\,\d k < \infty.
\end{equation}
If the sequence of matrices $\bq_N$ converges, the limit in \eqref{E:FTconvN} exists and we obtain that
\begin{equation}
\lim_{N\to\infty}\big(\boldsymbol{\nabla}_N g_\eps,\mathscr{C}^{\ssup{\bq_N},\nabla}\boldsymbol{\nabla}_N g_\eps\big)_{\cbO_N} = \int_{[-\pi,\pi]^d}\;\frac{\Big|\sum_{j=1}^d\abs{q_j(p)}^2\Big|^2}{\sum_{j,\ell=1}^d\big(\bQ_V-\bq\big)_{j,l}\overline{q_j(p)}q_\ell(p)}\abs{g_\eps(p)}^2\,\d p.
\end{equation}
Finally, the same reasoning as above then shows the equality
\begin{align}
\lim_{\eps\to 0}\lim_{N\to\infty} \big(\boldsymbol{\nabla}_N g_\eps,\mathscr{C}^{\ssup{\bq_N},\nabla}\boldsymbol{\nabla}_N g_\eps\big)_{\cbO_N} &= \int_{\R^d}\;\frac{\abs{k}^4\abs{\widehat{g}(k)}^2}{\sum_{j,\ell=1}^d\big(\bQ_V-\bq\big)_{j,l}k_jk_l}\,\d k \nonumber \\
&= \int_{\R^d}\;\frac{\abs{\widehat{\Delta_{\R^d}g}(k)}^2}{\sum_{j,\ell=1}^d\big(\bQ_V-\bq\big)_{j,l}k_jk_l}\,\d k,
\end{align}
which characterises the operator $\mathscr{C}_{\R^d}^{\ssup{u}}$ in our statement through its Fourier transform.
\end{proof}

The denominator of the fraction in \eqref{E:IVscalinglimit} converges to one, by virtue of Theorem \ref{T:mainrep}, and so it remains to estimate the numerator. That is the purpose of our next statement. 
\begin{lemma} \label{L:IVnumerator}
Let $f,g$ be as described in section \ref{S6:1}. On the assumptions of Theorem \ref{T:K<kdecay} and for a choice of constants consistent with it, it holds that
\begin{equation} \label{E:IVnumerator}
\Big|\int_{\cbO_N}\;\Big(\ex^{-H_0}\circ K_0(0)\Big)(\mathbb{T}_N,\eta+\mathscr{C}^{\ssup{\bq_N},\nabla}\boldsymbol{\nabla}_N g_\eps)\,\nu^{\ssup{\bq}}(\d\eta)-1\Big| \le O\big((3/4)^{\bar{k}}\big)
\end{equation}
for all $\bar{k}\ge\bar{k}_0$, uniformly in $N$ for $N\ge\bar{k}+\mathfrak{N}$.
\end{lemma}
\begin{proof}
Fix the test function (i.e., $f$ and $g$) and determine $\bar{k}_0$ in accordance with Theorem \ref{T:K<kdecay}. For $\bar{k}\ge \bar{k}_0$ and $N\ge\bar{k}+\mathfrak{N}$, repeated applications of Proposition \ref{P:Tk} and Lemma \ref{L:K<kprop}(4) yield the following representation in terms of the extended renormalisation group flow:
\begin{multline}
\int_{\cbO_N}\;\Big(\ex^{-H_0}\circ K_0(0)\Big)(\mathbb{T}_N,\eta+\mathscr{C}^{\ssup{\bq_N},\nabla}\boldsymbol{\nabla}_N g_\eps)\,\nu^{\ssup{\bq}}(\d\eta) \\
= \int_{\cbO_N}\;\ex^{e_N}\Big(1+ K_N(\Lambda_N,<N,\eta+\boldsymbol{\nabla}_N\psi_N)\Big)\,\nu_{N+1}^{\ssup{\bq}}(\d\eta).
\end{multline}
With the help of Theorems \ref{T:K<kdecay}, \ref{T:wf}(2) and \ref{T:wf}(10), we estimate the residual integral,
\begin{align}
\Big|\int_{\cbO_N}\;K_N(\Lambda_N,<N,\eta+\boldsymbol{\nabla}_N\psi_N)\,\nu_{N+1}^{\ssup{\bq}}(\d\eta)\Big| &\le A^{-1}\norm{K_N}_N^{(A)}\int_{\cbO_N}\;\bunderline{w}_N^{\Lambda_N}(\eta+\boldsymbol{\nabla}_N\psi_N)\Big)\,\nu_{N+1}^{\ssup{\bq}}(\d\eta) \nonumber \\
&\le \frac{\Acal_\Bcal}{2A}\norm{K_N}_N^{(A)}\bunderline{w}_{N:N+1}^{\Lambda_N}(\boldsymbol{\nabla}_N\psi_N) \nonumber \\
&\le \frac{\Acal_\Bcal}{2A}\norm{K_N}_N^{(A)} \exp\Big(\frac{\big(\boldsymbol{\nabla}_N\psi_N,\bunderline{M}_N \boldsymbol{\nabla}_N\psi_N\big)}{2\lambda}\Big) \nonumber \\
&\le \frac{\Acal_\Bcal}{2A}C_{\textup{ext}}(\bar{k})\big(\frac{2}{3}\big)^{N-(\bar{k}+\mathfrak{M})} C(d,f,L),
\end{align}
where the upper bound in \eqref{E:psiNest} was used for the quadratic form in $\boldsymbol{\nabla}_N\psi_N$ on the penultimate line. Finally, the inequality $\ex^x-1\le 2x$ for $x\in (0,1)$, combined with Lemma \ref{L:eN}, confirms the asymptotic behaviour of the above expression as claimed in \eqref{E:IVnumerator}.
\end{proof}

We are now in a position to conclude with a proof of the infinite volume scaling limit.
\begin{proof} [Proof of Theorem \ref{T:scalingIV}]
We study the expectation with respect to the infinite volume state,
\begin{equation} \label{E:IVscalingexp}
\mathbb{E}_{\gamma_{\infty,\beta}^{\ssup{u}}}\Big[\exp\Big({\sum_{x\in\Z^d}\nabla\phi(x)\cdot \nabla g_\eps(x)}\Big)\Big],
\end{equation}
in the $\eps\to 0$ limit. To make use of our analysis of the finite volume expectations, it suffices to show that the random variables $\exp\Big({\sum_{x\in\Z^d}\nabla\phi(x)\cdot \nabla g_\eps(x)}\Big)$, where $\phi$ is distributed under $\gamma_{N,\beta}^{\ssup{u}}$, are uniformly integrable. But this follows from our preceding results applied to, say, the test function $2f$. Indeed, for $\eps$ fixed (and small enough for Lemma \ref{L:IVnumerator} to hold) and all $N$ sufficiently large, $2f_\eps$ is supported in $\Lambda_N$ and then Lemmas \ref{L:innerprodscaling} and \ref{L:IVnumerator} imply that
\begin{equation}
\sup_{N\in\N} \mathbb{E}_{\gamma_{N,\beta}^{\ssup{u}}}\Big[\exp\Big({\sum_{x\in\Lambda_N}\nabla\phi(x)\cdot \nabla g_\eps(x)}\Big)^2\Big]<\infty,
\end{equation}
from which we infer the desired uniform integrability. Combined with the weak convergence of the measures $\gamma_{N,\beta}^{\ssup{u}}$ (as random gradient vector fields), which follows from Lemma \ref{L:tightness} as discussed, this implies the limit
\begin{equation}
\mathbb{E}_{\gamma_{\infty,\beta}^{\ssup{u}}}\Big[\exp\Big({\sum_{x\in\Z^d}\nabla\phi(x)\cdot \nabla g_\eps(x)}\Big)\Big]=\lim_{N\to\infty} \mathbb{E}_{\gamma_{N,\beta}^{\ssup{u}}}\Big[\exp\Big(\sum_{x\in\Lambda_N}\nabla\phi(x)\cdot \nabla g_\eps(x)\Big)\Big].
\end{equation}
Lemma \ref{L:IVnumerator} then ensures that the ratio of partition functions that appears in \eqref{E:IVscalinglimit} converges to one as $\eps\to 0$ (that is, $\bar{k}\to\infty$). Moreover,  since $\bq_N\to\bq$ at least along a sub-sequence, Lemma \ref{L:innerprodscaling} actually shows that the limit of \eqref{E:IVscalingexp} admits the representation
\begin{equation}
\lim_{\eps\to 0} \mathbb{E}_{\gamma_{\infty,\beta}^{\ssup{u}}}\Big[\exp\Big({\sum_{x\in\Z^d}\nabla\phi(x)\cdot \nabla g_\eps(x)}\Big)\Big]=\exp\big(\frac{1}{2\beta}\big(f,\mathscr{C}_{\R^d}^{\ssup{u}}f\big)_{L^2(\R^d)}\Big),
\end{equation}
where $\mathscr{C}_{\R^d}^{\ssup{u}}$ stands for the operator given in Lemma \ref{L:innerprodscaling}. Finally, the identification of the matrix $\bQ_V-\bq$ with the Hessian of surface tension follows from Proposition \ref{P:Hesslim} as above and this completes the proof.
\end{proof}

\section*{Acknowledgements and Dedication}

This article is dedicated to the memory of Stefan Adams, the second author's late PhD adviser, who sadly passed away before its completion. We thank Roland Bauerschmidt warmly for his support and encouragement during the preparation of this article.

Andreas Koller is supported by the Warwick Mathematics Institute Centre for Doctoral Training, and gratefully acknowledges funding from the University of Warwick and the UK Engineering and Physical Sciences Research Council (EPSRC) (Grant number: EP/T51794X/1).

\pagebreak

\section*{Appendices}
\begin{appendices}

\section{Index of notational changes from \cite{ABKM}} \label{A:A}

\subsection{Global adjustments}

In general, the setting considered in this article corresponds to a particular subset of the range of models covered by \cite{ABKM}. The limitation to scalar fields corresponds to setting $m=1$ (in the notation of \cite{ABKM}) and gradient-only interactions reflect the choice of $\mathcal{G}=\R^{\mathcal{I}}$ with $\mathcal{I}=\{e_i\in\R^d:1\le i\le d\}$ as the space of field derivatives. We therefore do not use the notation $\mathcal{G}$ and simply consider gradients as elements of $\R^d$. The fixed parameters listed in the beginning of section \ref{S:5} generally play the same roles as in \cite{ABKM}.

\subsection{Specific changes in notation}

The following table collects the symbols used in this article and lists the corresponding symbol or object in the notation of \cite{ABKM}. Notation and conventions that are identical are omitted from this inventory.

\begin{center}
\begin{tabular}{ |l|l| }
\hline
Symbol & Corresponding symbol in \cite{ABKM} \\
\hline
$\mathbb{T}_N$, the discrete torus of side length $L^{dN}$ & $T_N$ \\
$V:\R^d\to\R$ & $U:(\R^m)^{A}\to\R$ \\
$\mathcal{Q}_V$, quadratic form on $\R^d$ & $\mathcal{Q}_{\Ucal}$, quadratic form on $\Gcal$ \\ 
$u\in\R^d$ & $\overline{F}\in\Gcal^{\nabla}$ \\
$U:\R^d\times\R^d\to\R$ & $\overline{\Ucal}:\Gcal\times\R^{d\times m}\to\R$ \\
$\Kcal_{u,\beta,V}:\R^d\to\R$ & $\Kcal_{F,\beta,\Ucal}:\Gcal\to\R$ \\
$\s_{N,\beta}(u)$, $\s_\beta(u)$ & $W_{N,\beta}(F)$, $W_{\beta}(F)$ \\
$\l\in\R$, the constant relevant monomial & $a_{\emptyset}\in\R$ \\
$m\in\mathfrak{m}$, the linear relevant monomial indices & $(i,\a)\in\mathfrak{v}_1$ \\
$\bd\in\R_{\text{sym}}^{d\times d}$, the quadratic relevant monomial indices & $a_{(i,\a),(j,\b)}$, $(i,\a),(j,\b)\in\mathfrak{v}_2$ \\
$\big\langle \text{Tay}_vF,g\big\rangle_{r_0}$, the dual pairing with $\bigoplus_{r=0}^{r_0}\mathfrak{V}_N^{\otimes r}$ & $\big\langle F,g\big\rangle_\phi=\big\langle \text{Tay}_\phi F,g\big\rangle$ \\
$\abs{\cdot}_{j,X,v}$, the norm on Taylor polynomials & $\abs{\cdot}_{j,X,T_{\phi}}$ \\
$\bunderline{w}_k^X$, $\bunderline{w}_{k:k+1}^X$ and $\bunderline{W}_k^X$, the weight functions & $w_k^X$, $w_{k:k+1}^X$ and $W_k^X$ \\
$\vartheta$, the contraction parameter & $\eta$ \\
$Z:B_{\eps_1}(0)\times B_{\eps2}(0)\to B_{\eps}(0)$, the fixed point map & $\widehat{Z}:B_{\eps_1}(0)\times B_{\eps2}(0)\to B_{\eps}(0)$ \\
in Theorem \ref{T:Tfp} &  \\
$Z_{H_0}$, the projection onto the $H_0$ component of $Z(\Kcal,\Hcal)$ & $\Pi_{H_0}\widehat{Z}(\Kcal,\Hcal)$ \\
$K_N:B_{\varrho}(0)\to\bM(\Pcal_N^c,\mathfrak{V}_N)$, the final component & $\widehat{K}_N:B_{\varrho}(0)\to\bM_N^{(A)}$ \\
of the fixed point map & \\
\hline
\end{tabular}
\end{center}

\section{Existence of the infinite volume limit} \label{A:B}

We show below that the sequence of measures $\big(\gamma_{N,\beta}^{\ssup{u},\nabla}\big)_{N\in\N}$ is tight. For this purpose, we regard them as measures on gradient vector fields on $\Z^d$ by periodic extension.

\begin{lemma} \label{L:tightness}
Suppose that $V$, $\beta$ and $u$ satisfy the assumptions of Theorem \ref{T:scalingIV}. Then there exist positive reals $\iota>0$ and $\mathfrak{K}>0$ such that
\begin{equation}\label{E:expbound}
\frac{1}{\beta L^{dN}}\log\gamma_{N,\beta}^{\ssup{u},\nabla}\Big(\frac{1}{L^{dN}}\sum_{x\in \mathbb{T}_N}\;\abs{\eta(x)}^2\ge \mathfrak{L}\Big)\le -\iota \mathfrak{L} + \mathfrak{K}
\end{equation}
for all $\mathfrak{L}>0$, uniformly in $N$. In particular, $(\gamma_{N,\beta}^{\ssup{u},\nabla})_{N\in\N}$ is tight.
\end{lemma}
\begin{proof}
Our assumptions ensure that the main result of \cite{ABKM} for generalised gradient models, Theorem 2.2, holds in the setting we consider. This implies the existence (along a sub-sequence) of the infinite volume surface tension, which we need here for the lower bound
\begin{equation}
\inf_{N\in\N}\frac{1}{\beta L^{dN}}\log Z_{N,\beta}(u)>-\infty.
\end{equation}
(Alternatively, any quadratic upper bound on $V$ would easily yield the same.) We now use the quadratic lower bound in \eqref{E:Assumptions} to deduce that, for any $0<\iota<\omega$, 
\begin{align}
-H_N^{\ssup{u}}(\eta)+\beta \iota\sum_{x\in\mathbb{T}_N} & \abs{\eta(x)}^2 \le -\beta\sum_{x\in\mathbb{T}_N} (\omega-\iota)\abs{\eta(x)+u}^2 +DV(0)(\eta(x)+u)+V(0) \nonumber \\
&\le -\beta L^{dN}\big(DV(0)(u)+V(0)+\omega\abs{u}^2\big) -\beta(\omega-\iota)\sum_{x\in\mathbb{T}_N} \abs{\eta(x)}^2.
\end{align}
Combined with standard results from Gaussian calculus, this provides a bound on the logarithmic rate of the following expectation,
\begin{multline}
\frac{1}{\beta L^{dN}}\log\mathbb{E}_{\gamma_{N,\beta}^{\ssup{u},\nabla}}\Big[\exp\Big(\beta \iota\sum_{x\in\mathbb{T}_N}\abs{\eta(x)}^2\Big)\Big] \le -DV(0)(u)-V(0)-\omega\abs{u}^2 \\
+ \sup_{N\in\N}\frac{1}{\beta L^{dN}}\log \int_{\cbO_N}\;\ex^{-\beta(\omega-\iota)\sum_{x\in\mathbb{T}_N} \abs{\eta(x)}^2}\,\omega_N(\d\eta)-\inf_{N\in\N}\frac{1}{\beta L^{dN}}\log Z_{N,\beta}(u) \\
\le C(\iota)<\infty,
\end{multline}
for some finite constant $C$, depending on $\iota$. Chebyshev's inequality then implies the estimate
\begin{equation}
\gamma_{N,\beta}^{\ssup{u},\nabla}\Big(\frac{1}{L^{dN}}\sum_{x\in\mathbb{T}_N}\abs{\eta(x)}^2 \ge \mathfrak{L}\Big) \le \ex^{-\beta\iota\mathfrak{L}L^{dN}}\mathbb{E}_{\gamma_{N,\beta}^{\ssup{u},\nabla}}\Big[\exp\Big(\beta \iota\sum_{x\in\mathbb{T}_N}\abs{\eta(x)}^2\Big)\Big],
\end{equation}
and hence
\begin{equation}
\frac{1}{\beta L^{dN}}\log\gamma_{N,\beta}^{\ssup{u},\nabla}\Big(\frac{1}{L^{dN}}\sum_{x\in\mathbb{T}_N}\abs{\eta(x)}^2 \ge \mathfrak{L}\Big) \le -\iota\mathfrak{L}+C(\iota).
\end{equation}
This confirms the claim in \eqref{E:expbound}. To see that tightness follows from it, we use translation invariance to obtain, for any $x\in\mathbb{T}_N$ and $\ell>0$, the identity
\begin{equation}
\gamma_{N,\beta}^{\ssup{u},\nabla}\big(\abs{\eta(x)}>\ell\big)=\mathbb{E}_{\gamma_{N,\beta}^{\ssup{u},\nabla}}\big[\1_{\abs{\eta(x)}>\ell}\big]=\mathbb{E}_{\gamma_{N,\beta}^{\ssup{u},\nabla}}\Big[\frac{1}{L^{dN}}\sum_{x\in\mathbb{T}_N}\1_{\abs{\eta(x)}>\ell}\Big].
\end{equation}
Given $\eps>0$, we then deduce that
\begin{align}
\mathbb{E}_{\gamma_{N,\beta}^{\ssup{u},\nabla}}\Big[\frac{1}{L^{dN}}\sum_{x\in\mathbb{T}_N}\1_{\abs{\eta(x)}>\ell}\Big] &\le \eps + \gamma_{N,\beta}^{\ssup{u},\nabla}\Big(\frac{1}{L^{dN}}\sum_{x\in\mathbb{T}_N}\1_{\abs{\eta(x)}>\ell} \ge \eps\Big) \nonumber \\
&\le \eps + \gamma_{N,\beta}^{\ssup{u},\nabla}\Big(\frac{1}{L^{dN}}\sum_{x\in\mathbb{T}_N}\abs{\eta(x)}^2 \ge \eps\ell^2\Big) \nonumber \\
&\le 2\eps,
\end{align}
provided that $\ell^2\ge \mathfrak{K}+\frac{\log \eps^{-1}}{\eps\beta\iota}$ in view of \eqref{E:expbound}. This confirms that the probability
\begin{equation}
\gamma_{N,\beta}^{\ssup{u},\nabla}\big(\abs{\eta(x)}>\ell\big)
\end{equation}
is arbitrarily small for suitable $\ell$, uniformly in $N$, and this suffices to establish tightness.
\end{proof}

\section{Extension of weight functions to arbitrary vector fields}  \label{A:C}

Below, we set out an explicit construction of the weight functions $\bunderline{w}_k^X$, $\bunderline{w}_{k:k+1}^X$ and $\bunderline{W}_k^X:\mathfrak{V}_N\to \R^+$ that appear in the norms defined in section \ref{S:4}. We then give a proof of the various properties of these functions stated in Theorem \ref{T:wf}. The weight functions introduced in Chapter 7 of \cite{ABKM} constitute our starting point for this effort and most of the properties we show in the following are a straightforward consequence or extension of the results set out in that work. 

\subsection{Definitions}

As in \cite{ABKM}, the weight functions are built with the help of suitable families of symmetric operators. In our case, their domain is the space of arbitrary vector fields $\mathfrak{V}_N$. Our fundamental tool is the the direct sum decomposition $\mathfrak{V}_N = \cbO_N \oplus \cbO_N^\perp$, where $\cbO_N^\perp$ is the orthogonal complement of $\cbO_N$ with respect to the standard (Euclidean) inner product given in \eqref{E:innerproduct}. We designate by $P_{\cbO}$ and $P_\perp$ the orthogonal projections onto these two subspaces.

By way of rough outline, our strategy is to define operators that act on $\cbO_N$ effectively as the corresponding operator in \cite{ABKM} and on $\cbO_N^\perp$ as simple discrete difference operators that can be related to our field norms $\abs{\cdot}_{j,X}$ in the desired manner. To implement this approach, we first define several (symmetric) reference operators on $\mathfrak{V}_N$. For any multi-index $\alpha\in\N_0^d$, we designate by $\boldsymbol{m}(\a)$ the diagonal $d\times d$-matrix whose entries are given by
\begin{equation}
\big(\boldsymbol{m}(\a)\big)_{i,i}=\abs{\{j:\alpha_j+\delta_{i,j}>0\}}^{-1}.
\end{equation}
For convenience, we recall the block-neighbourhood counting functions $\chi_X:\mathbb{T}_N\to\N$ defined in Chapter 7.1 of \cite{ABKM}:
\begin{equation}
\chi_X(x)=\abs{\{B\in\Bcal_k(X):x\in B^+\}}.
\end{equation}
We then define:
\begin{align}
\mathscr{M}_{k,\nabla}^X &= \sum_{0\le \abs{\alpha}\le M-1}L^{2k\abs{\alpha}}(\nabla^*)^\alpha\chi_X\boldsymbol{m}(\alpha)\nabla^\alpha; \\
\mathscr{M}_k^X &= \sum_{0\le \abs{\alpha}\le M-1}L^{2k\abs{\alpha}}(\nabla^*)^\alpha\chi_X\nabla^\alpha; \\
\mathscr{G}_{k,\nabla}^X &= \frac{1}{h_k^2}\sum_{0\le \abs{\alpha}\le \floor{d/2}}L^{2k\abs{\alpha}}(\nabla^*)^\alpha\1_X\boldsymbol{m}(\alpha)\nabla^\alpha; \text{and} \\
\mathscr{G}_k^X &= \frac{1}{h_k^2}\sum_{0\le \abs{\alpha}\le \floor{d/2}}L^{2k\abs{\alpha}}(\nabla^*)^\alpha\1_X\nabla^\alpha.
\end{align}
Using these reference operators, we construct symmetric operators that are block-diagonal with respect to the above orthogonal decomposition of $\mathfrak{V}_N$. We let
\begin{align}
\bunderline{M}_k^X &= P_{\cbO} \mathscr{M}_{k,\nabla}^X P_{\cbO} \oplus P_\perp \mathscr{M}_k^X P_\perp \\
&= \bunderline{M}_{k,\nabla}^X \oplus \bunderline{M}_{k,\perp}^X; \nonumber
\end{align}
and
\begin{align}
\bunderline{G}_k^X &= P_{\cbO} \mathscr{G}_{k,\nabla}^X P_{\cbO} \oplus P_\perp \mathscr{G}_k^X P_\perp \\
&= \bunderline{G}_{k,\nabla}^X \oplus \bunderline{G}_{k,\perp}^X. \nonumber
\end{align}
Similar notation without $k$-polymer specified in the superscript is used for operators on the whole torus $\mathbb{T}_N$ where the $\chi$-factor is ignored. That is, we write
\begin{equation}
\mathscr{M}_{k,\nabla}=\sum_{0\le \abs{\alpha}\le M-1}L^{2k\abs{\alpha}}(\nabla^*)^\alpha\boldsymbol{m}(\alpha)\nabla^\alpha
\end{equation} 
and likewise for the other operators. We further let $\bunderline{M}_k=P_{\cbO} \mathscr{M}_{k,\nabla} P_{\cbO} \oplus P_\perp \mathscr{M}_k P_\perp$ and we borrow the shorthand symbols $\Xi_k=\abs{B^+}_k$ and $\Xi_{\text{max}}=(2R+1)^d$ from \cite{ABKM}.

We pursue the same idea in relation to the main operators $\bunderline{A}_k^X$ and $\bunderline{A}_{k:k+1}^X$, with which the weak norm weight functions are defined. Accordingly, each of these operators is the sum of a gradient component (i.e., an operator that restricts to an endomorphism of $\cbO_N$) and an orthogonal component (i.e., an operator that restricts to an endomorphism of $\cbO_N^\perp$). In symbols, we write
\[\bunderline{A}_k^X=\bunderline{A}_{k,\nabla}^X \oplus \bunderline{A}_{k,\perp}^X\,\text{ and }\,\bunderline{A}_{k:k+1}^X=\bunderline{A}_{k:k+1,\nabla}^X \oplus \bunderline{A}_{k:k+1,\perp}^X.\]
For our iterative construction, we first give an explicit definition of the operators at the $k=0$ scale. Recall that the Gaussian measures $\mu_k^{\ssup{\bq}}$ (and hence $\nu_k^{\ssup{\bq}}$) depend on the $\R^d$-quadratic form $\Qcal_V$ given in terms of the potential $V$. We write $\bQ_V\in\R_{\text{sym}}^{d\times d}$ for the matrix representing that form. We extend it to an operator on $\mathfrak{V}_N$ by letting
\begin{equation}
\big(\mathscr{A}_{\bQ_V}^Xv\big)(x)=\1_X(x)\bQ_V v(x)
\end{equation}
for any vector field $v\in\mathfrak{V}_N$. We recall further the constant $\zeta\in (0,1)$, which appears in the definition of the $\norm{\cdot}_{\zeta,\Qcal_V}$-norm used to control the initial perturbation. We need to be able to relate the exponential weight $\exp\big(-\frac{1}{2}(1-\zeta)\Qcal_V(z)\big)$ in that norm to the weight functions at the $k=0$ scale. Our choice of block-diagonal operators imposes an inefficiency here: we need to rely on the bound
\begin{equation}
\frac{1}{2}(1-\zeta)\sum_{x\in X}\Qcal_V(v(x))\le \frac{1}{2}(1-\zeta)\sum_{x\in X}\Big((1+\eps)\Qcal_V(P_{\cbO}v(x))+(1+\eps^{-1})\Qcal_V(P_{\perp}v(x))\Big),
\end{equation}
for a suitably small $0<\eps<1$. Choosing $\eps\le\frac{\zeta}{2(1-\zeta)}$, we ensure that $(1+\eps)(1-\zeta)$ is strictly less than one and we write this number as $(1-\zeta')$ with $\zeta'\in (0,1)$. An inspection of Chapter 7 of \cite{ABKM} reveals that this is sufficient for the construction of the weight functions used in that work. Indeed, Theorem 7.1 of \cite{ABKM} requires the parameter $\bar{\zeta}$, corresponding to $\zeta'/4$ in our setting, to satisfy $\bar{\zeta}\in (0,1/4)$. We also introduce the shorthand $\mathfrak{r}=\frac{1+\eps^{-1}}{1+\eps}$. With this notation in place, we then define $\bunderline{A}_0^X$ by setting
\begin{equation}
\bunderline{A}_{0,\nabla}^X= (1-\zeta') P_{\cbO} \mathscr{A}_{\bQ_V}^X P_{\cbO} + \delta\bunderline{M}_{0,\nabla}^X;
\end{equation}
and
\begin{equation}
\bunderline{A}_{0,\perp}^X= \mathfrak{r}(1-\zeta') P_\perp \mathscr{A}_{\bQ_V}^X P_\perp + \delta\bunderline{M}_{0,\perp}^X.
\end{equation}
The operators at higher scales are then constructed as follows. We define $\bunderline{A}_{k:k+1}^X$ by letting
\begin{equation}
\bunderline{A}_{k:k+1,\nabla}^X= (\boldsymbol{\nabla}_N^{-1})^* A_{k:k+1}^X \,\boldsymbol{\nabla}_N^{-1} \,\text{ and }\, \bunderline{A}_{k:k+1,\perp}^X= \bunderline{A}_{k,\perp}^X\,\text{ for }0\le k\le N,
\end{equation}
where $A_{k:k+1}^X$ stands for the operator on $\cbX_N$ as defined in Chapter 7 of \cite{ABKM}. The adjoint $(\boldsymbol{\nabla}_N^{-1})^*$ is determined here by the condition
\begin{equation}
(\phi,\boldsymbol{\nabla}_N^{-1}\eta)_{\cbX_N}=((\boldsymbol{\nabla}_N^{-1})^*\phi,\eta)_{\cbO_N}
\end{equation}
for all $\phi\in\cbX_N$ and $\eta\in\cbO_N$. The operator $\bunderline{A}_{k+1}^X$ is then in turn given by
\begin{equation}
\bunderline{A}_{k+1}^X= \bunderline{A}_{k:k+1}^{X^*} + \delta_{k+1}\bunderline{M}_{k+1}^X\,\text{ for }0\le k\le N-1.
\end{equation}

Finally, we define the weak and strong norm weight functions for arbitrary vector fields $v\in\mathfrak{V}_N$ using the quadratic forms induced by the symmetric operators constructed above. We define the strong norm weight function as
\begin{equation}
\bunderline{W}_k^X(v)=e^{\frac{1}{2}(v,\bunderline{G}_k^X v)};
\end{equation}
and the weak norm weight functions as
\begin{equation}
\bunderline{w}_k^X(v)=e^{\frac{1}{2}(v,\bunderline{A}_k^X v)}\,\text{ and }\,\bunderline{w}_{k:k+1}^X(v)=\ex^{\frac{1}{2}(v,\bunderline{A}_{k:k+1}^X v)}.
\end{equation}

We observe that the definitions stated above fall into two categories. Most of them are explicit expressions involving simple discrete difference operations. The remaining definitions are given in terms of objects that are shown to exist as well-defined operators on $\cbX_N$ in pursuant to Lemma 7.5 of \cite{ABKM}. To see that the functions $\overline{w}_k^X$, $\overline{w}_{k:k+1}^X$ and $\overline{W}_k^X$ are well-defined, it thus suffices to check that our choice of constants is consistent with Lemmas 7.3 and 7.5 of \cite{ABKM}. This forms part of the proof in the next sub-section. Instead of tying the operator $\bunderline{A}_{k:k+1}^X$ directly to the $\cbX_N$-operator $A_{k:k+1}^X$, we could replicate the iterative scheme underlying the construction of $A_{k:k+1}^X$ in Chapter 7 of \cite{ABKM}. For completeness, we verify explicitly that these approaches are equivalent. Accordingly, before turning to the numbered claims of Theorem \ref{T:wf}, we record two basic features of our operators that follow directly from their definitions. The first simply states the direct relationship between $\bunderline{A}_{k,\nabla}^X$ and $A_k^X$ at all scales $k$ as outlined initially. The second feature exhibits how the iterative construction of $A_{k:k+1}^X$ (including the need to restrict to an appropriate subspace to make sense of the inverse $(A_k^X)^{-1}$) shapes the structure of the operator $\bunderline{A}_{k:k+1,\nabla}^X$.

\begin{lemma} \label{L:wfdef}
With the definitions as set out above and in the previous sub-section, the operators $\bunderline{A}_k^X$ and $\bunderline{A}_{k:k+1}^X$ satisfy the following:
\begin{enumerate}
\item $\bunderline{A}_{k,\nabla}^X=(\nabla_N^{-1})^* A_k^X \,\nabla_N^{-1}$ for all $k=0,\ldots,N$; and
\item $\bunderline{A}_{k:k+1,\nabla}^X=\begin{cases}
\big((\bunderline{A}_{k,\nabla}^X)^{-1}-(1+\zeta'/4)P_{(\ker \bunderline{A}_{k,\nabla}^X)^\perp}\nabla_N \mathscr{C}_{k+1} \nabla_N^* P_{(\ker \bunderline{A}_{k,\nabla}^X)^\perp})^{-1} &\textup{ on }(\ker \bunderline{A}_{k,\nabla}^X)^\perp \\
0&\textup{ on } \ker \bunderline{A}_{k,\nabla}^X
\end{cases}$ \par
for $k=0,\ldots,N$;
\end{enumerate}
where $P_{(\ker \bunderline{A}_{k,\nabla}^X)^\perp}$ is the orthogonal projection onto $(\ker \bunderline{A}_{k,\nabla}^X)^\perp$ and $(\bunderline{A}_{k,\nabla}^X)^{-1}$ refers to the inverse of the restriction to $(\ker \bunderline{A}_{k,\nabla}^X)^\perp$ of the operator $\bunderline{A}_{k,\nabla}^X$.
\end{lemma}
\begin{proof}
We begin by verifying the claim in (1). For $k=0$, let $\phi$ be any scalar field in $\cbX_N$ and we compute
\begin{align}
(\nabla_N\phi,\bunderline{A}_{0,\nabla}^X\nabla_N\phi)_{\cbO_N} =& (1-\zeta')(\nabla_N\phi,\mathscr{A}_{\bQ_V}^X\nabla_N\phi)_{\cbO_N}+\delta(\nabla_N\phi,\mathscr{M}_{0,\nabla}^X\nabla_N\phi)_{\cbO_N} \nonumber \\
=& (1-\zeta')\sum_{x\in X}\sum_{i,j=1}^d \bQ_{Vi,j}\nabla_i\phi(x)\nabla_j\phi(x) \nonumber \\
&+ \delta \sum_{x\in\mathbb{T}_N}\sum_{0\le \abs{\alpha}\le M-1}L^{2k\abs{\alpha}}\chi_X(x)\big(\nabla^\alpha\nabla\phi(x)\big)\cdot\big(\boldsymbol{m}(\alpha)\nabla^\alpha\nabla\phi(x)\big) \nonumber \\
=& (\phi,A_0^X\phi)_{\cbX_N}.
\end{align}
Here we used the fact that the term $\abs{\nabla_1^{\alpha_1}...\nabla_d^{\alpha_d}\phi(x)}^2$ appears once for every non-zero $\alpha_j$, which is exactly the over-counting for which the diagonal matrix $\boldsymbol{m}$ compensates. This confirms (1) at the initial scale. Moreover, the same argument shows that $\bunderline{M}_{k,\nabla}^X=(\nabla_N^{-1})^* M_k^X \,\nabla_N^{-1}$ for all scales $k$. This observation then allows us to see that (1) holds for $k\ge 1$ simply by expanding the definitions:
\begin{multline}
\bunderline{A}_{k+1,\nabla}^X=\bunderline{A}_{k:k+1,\nabla}^{X^*}+\delta_{k+1}\bunderline{M}_{k+1,\nabla}^X=(\nabla_N^{-1})^* A_{k:k+1}^{X^*} \nabla_N^{-1} + \delta_{k+1}(\nabla_N^{-1})^* M_{k+1}^X \nabla_N^{-1} \\
=(\nabla_N^{-1})^* A_{k+1}^X \nabla_N^{-1}.
\end{multline}

Now, we show that (2) follows from (1). Note first that $\bunderline{A}_{k,\nabla}^X=(\nabla_N^{-1})^* A_k^X \,\nabla_N^{-1}$ implies $\ker \bunderline{A}_{k,\nabla}^X=\nabla_N(\ker A_k^X)$, as well as $\nabla_N^*((\ker \bunderline{A}_{k,\nabla}^X)^\perp)=(\ker A_k^X)^\perp$. The first equality is straightforward and the second follows from it immediately by observing that $k\in\ker A_k^X$ and $\eta\in (\ker \bunderline{A}_{k,\nabla}^X)^\perp$ implies
\begin{equation}
(k,\nabla^*\eta)_{\cbX_N}=(\nabla_N k,\eta)_{\cbO_N}=0
\end{equation}
and that the dimensions must coincide. This is sufficient to conclude that the restriction of $\bunderline{A}_{k,\nabla}^X$ to $(\ker \bunderline{A}_{k,\nabla}^X)^\perp$ is an invertible endomorphism of that sub-space. To give an explicit representation of its inverse, we introduce the projection $\slashed{P}= P_{(\ker \bunderline{A}_{k,\nabla}^X)^\perp}\rvert_{\nabla_N((\ker A_k^X)^\perp)}$. The above relationship between the kernels of $A_k^X$ and $\bunderline{A}_{k,\nabla}^X$ ensures that $\slashed{P}$ is itself invertible. We claim that $\bunderline{A}_{k,\nabla}^X$ acts on $(\ker \bunderline{A}_{k,\nabla}^X)^\perp$ as the operator $(\nabla_N^{-1})^* A_k^X\rvert \,\nabla_N^{-1}\slashed{P}^{-1}$, where the symbol $A_k^X\rvert$ designates the explicit restriction of $A_k^X$ to $(\ker A_k^X)^\perp$. To see this, we note that, for $\eta\in(\ker \bunderline{A}_{k,\nabla}^X)^\perp$, we must have that $\slashed{P}^{-1}(\eta)=\eta+k$ for some $k\in\ker \bunderline{A}_{k,\nabla}^X$. Since $\nabla_N^{-1}(k)\in\ker A_k^X$, this then implies that
\begin{equation}
(\nabla_N^{-1})^* A_k^X\rvert \,\nabla_N^{-1}\slashed{P}^{-1}(\eta)=(\nabla_N^{-1})^* A_k^X\rvert (\nabla_N^{-1}(\eta)+\nabla_N^{-1}(k))=(\nabla_N^{-1})^* A_k^X \,\nabla_N^{-1}(\eta),
\end{equation}
which is $\bunderline{A}_{k,\nabla}^X(\eta)$ by (1). We can therefore write the inverse of $\bunderline{A}_{k,\nabla}^X\rvert_{(\ker \bunderline{A}_{k,\nabla}^X)^\perp}$ as $\slashed{P}\nabla_N(A_k^X\rvert)^{-1}\nabla_N^*$.

We now apply a similar reasoning to $\bunderline{A}_{k:k+1,\nabla}^X=(\nabla_N^{-1})^* A_{k:k+1}^X \,\nabla_N^{-1}$. Pursuant to the construction in Lemma 7.2 of \cite{ABKM}, the operator $A_{k:k+1}^X$ has kernel equal to $\ker A_k^X$ and acts on $(\ker A_k^X)^\perp$ as the invertible operator $\big((A_k^X\rvert)^{-1}-(1+\zeta'/4)P_{(\ker A_k^X)^\perp}\mathscr{C}_{k+1}^{\ssup{0}}P_{(\ker A_k^X)^\perp}\big)^{-1}$. It follows that $\bunderline{A}_{k:k+1,\nabla}^X$ may be represented by the orthogonal decomposition
\begin{equation*}
\bunderline{A}_{k:k+1,\nabla}^X=\begin{cases}
(\nabla_N^{-1})^* \big((A_k^X\rvert)^{-1}-(1+\zeta'/4)P_{(\ker A_k^X)^\perp}\mathscr{C}_{k+1}^{\ssup{0}}P_{(\ker A_k^X)^\perp}\big)^{-1} \nabla_N^{-1}\slashed{P}^{-1} &\textup{ on }(\ker \bunderline{A}_{k,\nabla}^X)^\perp \\
0&\textup{ on } \ker \bunderline{A}_{k,\nabla}^X.
\end{cases}
\end{equation*}
By rewriting the non-trivial component of this representation,
\begin{align}
(\nabla_N^{-1})^* & \big((A_k^X\rvert)^{-1} - (1+\zeta'/4)P_{(\ker A_k^X)^{\perp}}\mathscr{C}_{k+1}^{\ssup{0}}P_{(\ker A_k^X)^{\perp}}\big)^{-1} \nabla_N^{-1}\slashed{P}^{-1} \nonumber \\
&= \big(\slashed{P}\nabla_N(A_k^X\rvert)^{-1}\nabla_N^* - (1+\zeta'/4)\slashed{P}\nabla_N P_{(\ker A_k^X)^{\perp}} \mathscr{C}_{k+1}^{\ssup{0}}P_{(\ker A_k^X)^{\perp}} \nabla_N^*\big)^{-1} \nonumber \\
&= \big((\nabla_N^{-1})^* A_k^X\rvert \,\nabla_N^{-1}\slashed{P}^{-1})^{-1} - (1+\zeta'/4)P_{(\ker \bunderline{A}_{k,\nabla}^X)^\perp} \nabla_N \mathscr{C}_{k+1}^{\ssup{0}}\nabla_N^* P_{(\ker \bunderline{A}_{k,\nabla}^X)^\perp} \big)^{-1},
\end{align}
we obtain the form asserted in (2) and this completes the proof. 
\end{proof}

\subsection{Proof of Theorem \ref{T:wf}}

The purpose of the remainder of this section is to prove the various claims comprising Theorem \ref{T:wf}. In general, we rely on Lemmas 7.3 to 7.8 of \cite{ABKM}, in combination with Lemma \ref{L:wfdef}(1), for the curl-free part of our operators. The main effort of the following paragraphs is therefore directed at showing the appropriate statements for the orthogonal complements.

\subsubsection*{Proof. Property (1)}

We show the (equivalent) claim that $\bunderline{A}_{k,\nabla}^Y \le \bunderline{A}_{k,\nabla}^X$ and $\bunderline{A}_{k,\perp}^Y \le \bunderline{A}_{k,\perp}^X$ for $Y\subset X\in\Pcal_k$ and likewise for the $(k:k+1)$ operators. Observe first that our reference operators, $\mathscr{M}_{k,\nabla}^X$, $\mathscr{M}_k^X$ and $\mathscr{A}_{\bQ_V}^X$, are monotonic in their $k$-polymer argument $X$. This immediately confirms the desired inequalities at the initial scale $Y\subset X\in\Pcal_0$. The claim now follows by induction. We have $\bunderline{A}_{k:k+1,\nabla}^Y\le \bunderline{A}_{k:k+1,\nabla}^X$ as a consequence of the corresponding result in \cite{ABKM} and
\begin{equation}
\bunderline{A}_{k:k+1,\perp}^Y=\bunderline{A}_{k,\perp}^Y\le\bunderline{A}_{k,\perp}^X = \bunderline{A}_{k:k+1,\perp}^X
\end{equation}
from the induction hypothesis. This result and the inclusion $Y^*\subset X^*$ then imply
\begin{equation}
\bunderline{A}_{k+1,\nabla}^Y=\bunderline{A}_{k:k+1,\nabla}^{Y^*}+\bunderline{M}_{k+1,\nabla}^Y\le\bunderline{A}_{k:k+1,\nabla}^{X^*}+\bunderline{M}_{k+1,\nabla}^X=\bunderline{A}_{k+1,\nabla}^X,
\end{equation}
where we again used the monotonicity of the reference operators. The inequality $\bunderline{A}_{k+1,\perp}^Y \le \bunderline{A}_{k+1,\perp}^X$ is deduced similarly.

\subsubsection*{Property (2)}

In light of the identity in Lemma \ref{L:wfdef}(1), the bounds for the gradient vector field components $\bunderline{A}_{k,\nabla}^X\le \lambda^{-1}\bunderline{M}_{k,\nabla}$ (and similarly for the $(k:k+1)$ operators) follow directly from \cite{ABKM}. We have to show the corresponding result for the orthogonal complements. At the initial scale $k=0$, we note that a smaller choice of the constant $\lambda$ may be required, essentially due to the additional $\mathfrak{r}$ factor and because we use different families of reference operators $\mathscr{M}_{k,\nabla}$ and $\mathscr{M}_k$. To see that a valid choice of $\lambda$ exists at all, it is enough to realise the straightforward bound
\begin{equation}
(v,\bunderline{A}_{0,\perp}^X v)\le \big(\Xi_{\text{max}}\delta+\mathfrak{r}(1-\zeta')\omega_0^{-1}\big)(v,\bunderline{M}_{0,\perp} v)
\end{equation}
implied by the upper bound on the quadratic form $\Qcal_V(z)\le \o_0^{-1}\abs{z}^2$. The dependence on $\delta$ (which in \cite{ABKM} is chosen as a function of $\l$) is only apparent: inspecting the proof of Lemma 7.5 of \cite{ABKM}, we see that $\delta$ is in any event bounded by $\frac{\zeta'}{4\Xi_{\text{max}}\O}$, where $\O$ depends only on the quadratic form $\Qcal_V$ and fixed parameters. We may therefore certainly write $\bunderline{A}_{0,\perp}^X \le \lambda_0^{-1}\bunderline{M}_{0,\perp}$ for all $X\in\Pcal_0$ and for some $\l_0\le (\zeta'/4\O+\mathfrak{r}(1-\zeta')\o_0^{-1})^{-1}$. There is nothing separate to prove for $\bunderline{A}_{k:k+1,\perp}^X$, since that operator is simply $\bunderline{A}_{k,\perp}^X$ by definition. Using the simple estimates, $\bunderline{M}_{k,\perp}^X\le \Xi_{\text{max}}\bunderline{M}_{k,\perp}$ for all $X\in\Pcal_k$ and $\bunderline{M}_{\ell,\perp}\le \bunderline{M}_{k,\perp}$ for all $\ell\le k$, we iterate the definition of $\bunderline{A}_{k+1,\perp}^X$ to obtain (for $k\ge 1$)
\begin{align} \label{E:lambdaconstr}
\bunderline{A}_{k,\perp}^X &\le \bunderline{A}_{0,\perp}^{\mathbb{T}_N}+\Xi_{\text{max}}\delta\sum_{\ell=1}^k 4^{-\ell}\bunderline{M}_{\ell,\perp} \le \frac{1}{\lambda_0}\bunderline{M}_{0,\perp}+\frac{\zeta'}{4\O} \bunderline{M}_{k,\perp} \sum_{\ell=1}^k 4^{-\ell} \nonumber \\
&\le \Big(\frac{1}{\lambda_0}+\frac{\zeta'}{12\O}\Big)\bunderline{M}_{k,\perp}.
\end{align}
It follows that a sufficiently small choice of $\lambda$ exists to ensure that $(v,\bunderline{A}_{k,\perp}^X v)\le \lambda^{-1} (v,\bunderline{M}_{k,\perp} v)$ holds for all $0\le k\le N$. Returning to the proof of Lemma 7.5 of \cite{ABKM}, a straightforward check confirms that the argument still holds with a (possibly) smaller choice of $\lambda$ albeit with (possibly) different constants $\mu$ and $\delta$. Combined with the bound for the gradient vector field components, this gives the desired inequalities:
\begin{equation}
\bunderline{w}_k^X(v)\le \ex^{\frac{1}{2\lambda}(v,\bunderline{M}_k v)}\,\text{ and }\,\bunderline{w}_{k:k+1}^X(v)\le \ex^{\frac{1}{2\lambda}(v,\bunderline{M}_k v)}\,\text{ for all }v\in\mathfrak{V}_N.
\end{equation}

\subsubsection*{Properties (3)-(5)}

These items cover various factorisation properties of the weight functions, which all derive from the feature that our reference operators have an additive decomposition for disjoint polymers $X\cap Y=\emptyset$. Indeed, while indicator functions trivially satisfy this feature, the neighbourhood counting function $\chi_X$ is defined as a sum over constituent $k$-blocks and hence also decomposes additively. As before, it suffices to verify the claims for the orthogonal complement components of the relevant operators $\bunderline{A}_{k,\perp}^X$, $\bunderline{A}_{k:k+1,\perp}^X$ and $\bunderline{G}_{k,\perp}^X$. 

We have seen above that $\bunderline{A}_{k,\perp}^X$ is a weighted sum of the initial operator $\bunderline{A}_{0,\perp}^Z$ and the operators $\bunderline{M}_{\ell,\perp}^Z$ for $\ell\le k$, where $Z$ runs over iterated small set neighbourhoods of $X$. We thus need to verify that these iterated neighbourhoods remain disjoint, provided that $X$ and $Y$ are sufficiently distant. By definition of the $(\cdot)^*$ operation on polymers, we have that $\text{dist}_\infty(x,X)\le (2^d+R)L^{k-1}$ for $X\in\Pcal_k$ and $x\in X^*$. We deduce that the maximum distance from $X$ or $Y$ reached by iterating $(\cdot)^*$ from a $k$-polymer to a $0$-polymer is
\begin{equation}
\sum_{\ell=0}^{k-1} (2^d+R)L^{\ell} = (2^d+R)\frac{L^k-1}{L-1}\le \frac{2^d+R}{L-1}L^k.
\end{equation}
We observe that $L$ is chosen large enough (at least $L\ge 2^{d+3}+16R$) so that the pre-factor to $L^k$ appearing above is strictly less than $\frac{1}{4}$. Since $X$ and $Y$ strictly disjoint implies $\text{dist}_\infty(X,Y)>L^k$, it follows that for such polymers $X$ and $Y$ all their iterated small set neighbourhoods remain disjoint and hence we conclude that $\bunderline{A}_{k,\perp}^{X\cup Y}=\bunderline{A}_{k,\perp}^X+\bunderline{A}_{k,\perp}^Y$. This gives property (3) and property (4) follows immediately since $\bunderline{A}_{k:k+1,\perp}^X=\bunderline{A}_{k,\perp}^X$ and the constraint $\text{dist}_\infty(X,Y)\ge \frac{3}{4}L^{k+1}$ is stronger than mere strict disjointedness. Finally, property (5) is a straightforward consequence of our previous observation that the reference operators (here $\mathscr{G}_{k,\nabla}^X$ and $\mathscr{G}_k^X$) are additive on disjoint $k$-polymers.

\subsubsection*{Properties (6) and (7)}

Note first that, in \cite{ABKM}, $h$ is chosen large enough so that $h^{-2}\le \delta$. Together with the definition of $M$ (which controls the number of derivatives in the weak norm weight functions), this implies that $\mathscr{G}_k^X\le \delta_k\mathscr{M}_k^X$ for all $k$ and all $k$-polymers $X$. We then see that
\begin{equation}
\bunderline{A}_{0,\perp}^{X\cup Y}=\bunderline{A}_{0,\perp}^X+\bunderline{A}_{0,\perp}^Y \ge \bunderline{A}_{0,\perp}^X + \delta \bunderline{M}_{0,\perp}^Y \ge \bunderline{A}_{0,\perp}^X + \bunderline{G}_{0,\perp}^Y
\end{equation}
holds at the initial scale for $X$ and $Y$ disjoint. For $k\ge 1$ we argue by induction:
\begin{align}
\bunderline{A}_{k,\perp}^{X\cup Y} &= \bunderline{A}_{k-1,\perp}^{(X\cup Y)^*}+\delta_k\bunderline{M}_{k,\perp}^{X\cup Y} = \bunderline{A}_{k-1,\perp}^{(X\cup Y)^*}+\delta_k\bunderline{M}_{k,\perp}^X + \delta_k\bunderline{M}_{k,\perp}^Y \nonumber \\
&\ge \bunderline{A}_{k-1,\perp}^{X^*}+\bunderline{G}_{k-1,\perp}^{Y^*\setminus X^*} + \delta_k\bunderline{M}_{k,\perp}^X + \delta_k\bunderline{M}_{k,\perp}^Y = \bunderline{A}_{k,\perp}^{X} + \bunderline{G}_{k-1,\perp}^{Y^*\setminus X^*} + \delta_k\bunderline{M}_{k,\perp}^Y \nonumber \\
&\ge \bunderline{A}_{k,\perp}^{X} + \bunderline{G}_{k,\perp}^Y,
\end{align}
where the induction hypothesis was used for the first inequality. This gives property (6) for the orthogonal complement components of the relevant operators. To see property (7), we note that, for $U=\pi_k(X)$, we have $X\subset U^*$ and hence the estimate
\begin{equation}
\bunderline{A}_{k+1,\perp}^U=\bunderline{A}_{k:k+1,\perp}^{U^*}+\delta_{k+1}\bunderline{M}_{k+1,\perp}^U=\bunderline{A}_{k,\perp}^{U^*}+\delta_{k+1}\bunderline{M}_{k+1,\perp}^U\ge\bunderline{A}_{k,\perp}^X+\delta_{k+1}\bunderline{M}_{k+1,\perp}^U
\end{equation}
holds. Now we use that $\chi_U\ge\1_{U^+}$ and that \cite{ABKM} actually gives the stronger bound $h^{-2}\le \frac{\delta}{8}$ to deduce that $\delta_{k+1}\bunderline{M}_{k+1,\perp}^U\ge 2\bunderline{G}_{k,\perp}^{U^+}$, which yields the desired result. We observe in passing that choosing $h$ larger than in \cite{ABKM} (as we do further below) only makes the inequalities needed here stronger.

\subsubsection*{Property (8)}

Given an arbitrary vector field $v\in\mathfrak{V}_N$, let $v=\eta+w$ be its orthogonal decomposition. We get the straightforward bound $\frac{\abs{v}_{k+1,X}^2}{2}\le \abs{\eta}_{k+1,X}^2+\abs{w}_{k+1,X}^2$. Observe that $\abs{\eta}_{k+1,X}=\abs{\boldsymbol{\nabla}_N^{-1}\eta}_{k+1,X}$, where the second norm is taken as defined in \cite{ABKM}. Note that there is a repetition of terms, for example both $\abs{\nabla_1\nabla_2\phi(x)}$ and $\abs{\nabla_2\nabla_1\phi(x)}$ appear separately (where we write $\phi=\boldsymbol{\nabla}_N^{-1}\eta$), but in the context of supremum norms this is irrelevant. Lemma 7.8 of \cite{ABKM} thus implies the bound
\begin{equation}
\abs{\eta}_{k+1,X}^2\le \delta_{k+1}(\eta,\bunderline{M}_{k+1,\nabla}^X \eta)_{\cbO_N},
\end{equation}
which is too weak by a factor of $2$ for our purposes. However, by inspecting the proof of this result in \cite{ABKM}, we see that choosing the constant $h$ larger by a factor of $\sqrt{2}$ gives us the sharper inequality we require.

We now turn to $\abs{w}_{k+1,X}$ with a view to applying the same Sobolev inequality that is at the core of Lemma 7.8 of \cite{ABKM}. Arguing similarly, we compute (for $B\in\Bcal_{k+1}(X)$)
\begin{align}
\sup_{x\in B^*}\abs{\nabla^\alpha w_i(x)}^2 &\le (\floor{d/2}+1)S(d)L^{-d(k+1)}\sum_{0\le \abs{\beta}\le \floor{\frac{d}{2}}+1}(3L^{k+1})^{2\abs{\beta}}\sum_{x\in B^*}\abs{\nabla^\beta\nabla^\alpha w_i(x)}^2 \nonumber \\
&\le (\floor{d/2}+1)3^{2(\floor{\frac{d}{2}}+1)}S(d)L^{-d(k+1)}\sum_{0\le \abs{\gamma}\le M-1}L^{2(\abs{\gamma}-\abs{\alpha})(k+1)}\sum_{x\in B^*}\abs{\nabla^\gamma w_i(x)}^2 \nonumber \\
&\le (\floor{d/2}+1)3^{2(\floor{\frac{d}{2}}+1)}S(d)L^{-(d+2\abs{\alpha})(k+1)}\sum_{0\le \abs{\gamma}\le M-1}L^{2\abs{\gamma}(k+1)}(\nabla^\gamma w,\1_{B^*}\nabla^\gamma w) \nonumber \\
& \le h_{k+1}^{-2}\mathfrak{w}_{k+1}(\alpha)^2(\floor{d/2}+1)3^{2(\floor{\frac{d}{2}}+1)}S(d)(w,\bunderline{M}_{k+1,\perp}^B w),
\end{align}
from which we deduce that the same (greater by $\sqrt{2}$ than in \cite{ABKM}) choice of $h$ ensures that $\abs{w}_{k+1,X}^2\le \frac{\delta_{k+1}}{2}(w,\bunderline{M}_{k+1,\perp}^X w)$. Combining these inequalities, we thus obtain the desired bound:
\begin{multline}
\frac{\abs{v}_{k+1,X}^2}{2}\le \frac{\delta_{k+1}}{2}(\eta,\bunderline{M}_{k+1,\nabla}^X \eta) + \frac{\delta_{k+1}}{2}(w,\bunderline{M}_{k+1,\perp}^X w)= \\
\frac{\delta_{k+1}}{2}(v,\bunderline{M}_{k+1}^X v)=\frac{\bunderline{A}_{k+1}^X-\bunderline{A}_{k:k+1}^{X^*}}{2}\le \frac{\bunderline{A}_{k+1}^X-\bunderline{A}_{k:k+1}^X}{2}.
\end{multline}

\subsubsection*{Properties (9) and (10)}

Our block-diagonal approach to the operators $\bunderline{A}_k^X$ makes the consistency of the weak norm weight functions under the integration map $\bR_{k+1}^{\ssup{\bq}}$ a relatively straightforward corollary of the corresponding result in \cite{ABKM}. Given any $v\in\mathfrak{V}_N$ and any polymer $X\in\Pcal_k$, we compute as follows
\begin{align}
\bigg(\int_{\cbO_N} &\; \big(\bunderline{w}_k^X(v+\eta)\big)^{1+\bunderline{p}}\,\nu_{k+1}^{\ssup{\bq}}(\d\eta)\bigg)^{\frac{1}{1+\bunderline{p}}} \nonumber \\
&= \bigg(\int_{\cbO_N}\;\big(\exp\big(\frac{1}{2}(v^\perp,\bunderline{A}_{k,\perp}^X v^\perp) + \frac{1}{2}(v^{\cbO}+\eta,\bunderline{A}_{k,\nabla}^X(v^{\cbO}+\eta))\big)\big)^{1+\bunderline{p}}\,\nu_{k+1}^{\ssup{\bq}}(\d\eta)\bigg)^{\frac{1}{1+\bunderline{p}}} \nonumber \\
&= \ex^{\frac{1}{2}(v^\perp,\bunderline{A}_{k,\perp}^X v^\perp)}\bigg(\int_{\cbO_N}\;\big(\exp\big(\frac{1}{2}(\boldsymbol{\nabla}_N^{-1}(v^{\cbO}+\eta),A_k^X\boldsymbol{\nabla}_N^{-1}(v^{\cbO}+\eta))\big)\big)^{1+\bunderline{p}}\,\nu_{k+1}^{\ssup{\bq}}(\d\eta)\bigg)^{\frac{1}{1+\bunderline{p}}} \nonumber \\
&= \ex^{\frac{1}{2}(v^\perp,\bunderline{A}_{k:k+1,\perp}^X v^\perp)}\bigg(\int_{\cbX_N}\;\big(\exp\big(\frac{1}{2}(\boldsymbol{\nabla}_N^{-1}v^{\cbO}+\phi,A_k^X(\boldsymbol{\nabla}_N^{-1}v^{\cbO}+\phi))\big)\big)^{1+\bunderline{p}}\,\mu_{k+1}^{\ssup{\bq}}(\d\phi)\bigg)^{\frac{1}{1+\bunderline{p}}} \nonumber \\
&\le \ex^{\frac{1}{2}(v^\perp,\bunderline{A}_{k:k+1,\perp}^X v^\perp)}\Big(\frac{A_{\Pcal}}{2}\Big)^{\abs{X}_k}w_{k:k+1}^X(\boldsymbol{\nabla}_N^{-1}v^{\cbO}) \\
&= \Big(\frac{A_{\Pcal}}{2}\Big)^{\abs{X}_k}\exp\big(\frac{1}{2}(v^\perp,\bunderline{A}_{k:k+1,\perp}^X v^\perp)+\frac{1}{2}(\boldsymbol{\nabla}_N^{-1}v^{\cbO},A_{k:k+1}^X \boldsymbol{\nabla}_N^{-1}v^{\cbO})\Big) = \Big(\frac{A_{\Pcal}}{2}\Big)^{\abs{X}_k}\bunderline{w}_{k:k+1}^X(v), \nonumber
\end{align}
where we used the notation $v=v^{\cbO}+v^\perp$ for the orthogonal decomposition of the vector field $v$. This gives the bound in (9) and, by specialising to the case $X=B\in\Bcal_k$, we immediately see that we get the $L$-independent constant $\frac{A_{\Bcal}}{2}$ from the corresponding inequality in \cite{ABKM}. This confirms property (10).

\subsubsection*{Choice of constants}

The final step in the proof of Theorem \ref{T:wf} involves reconciling the various constraints on the constants $\lambda$, $\mu$, $h$ etc. imposed by the preceding arguments with our stated assumptions. The constant $\zeta$, which controls the norm $\norm{\cdot}_{\Qcal_V,\zeta}$, is considered fixed and two further parameters are immediately determined by it. The choice $\zeta'=\zeta/2$ - corresponding to $\eps=\zeta/2(1-\zeta)$ - is consistent with our construction above and $\rho=(1+\zeta'/4)^{1/3}-1$ is then given as an explicit function of $\zeta'$.

The choice of free parameters begins with $L$, from which the remaining decisions cascade in a logically consistent manner. Assuming that $L$ has been fixed such that $L\ge 2^{d+3}+16R$, we then pick $\lambda$ and $\kappa$. Lemma 7.5 of \cite{ABKM} requires $\lambda\le\min\{\omega_0\zeta'/4,1/4\}$ and we see from \eqref{E:lambdaconstr} in combination with the definition of $\lambda_0$ that the constraint
\begin{equation}
\lambda\le \Big(\frac{\zeta'}{3\Omega}+\mathfrak{r}(1-\zeta')\omega_0^{-1}\Big)^{-1}
\end{equation}
is sufficient to obtain property (2) for all scales $0\le k\le N$. We thus conclude with the choice
\begin{equation}
\lambda=\min\Big\{\frac{\o_0\zeta'}{4},\Big(\frac{\zeta'}{3\O}+\mathfrak{r}(1-\zeta')\o_0^{-1}\Big)^{-1},\frac{1}{4}\Big\}
\end{equation} 
stated in Theorem \ref{T:wf}. Next, $\kappa$ is determined by Lemma 7.7 of \cite{ABKM} and the proof of this result shows that $\kappa\le \frac{\rho L^{-4(d+\widetilde{n})-2}}{K_1}$ ensures that the desired claim holds. $K_1$ is the constant from the bound on the  derivatives of the kernels of the finite range decomposition in Theorem 6.1 of \cite{ABKM}. Together with the basic requirement from \eqref{E:Bkappa}, we thus obtain the choice of $\kappa$ stated in Theorem \ref{T:wf}. The constant $\mu$ is then given by Lemma 7.3 of \cite{ABKM} as a function of $\lambda$ and $L$. With the value of $\mu$ so determined, a valid choice of $\delta$ is guaranteed by Lemma 7.5 of \cite{ABKM}. Finally, Lemma 7.6 of \cite{ABKM} imposes the contraint $h_0\ge \sqrt{8}\delta^{-1/2}$ and claim (8), as we show above, requires $h_0\ge \sqrt{2}c_d\delta^{-1/2}$, where $c_d$ is shorthand for the expression
\begin{equation}
\Big((\floor{d/2}+1)3^{2(\floor{d/2}+1)} S(d)\Big)^{1/2}
\end{equation}
that appears in the estimate above (without the square root). Combining these requirements yields $h_0=\delta^{-1/2}\max\{\sqrt{8},\sqrt{2}c_d\}$ and thereby completes the proof of Theorem \ref{T:wf}.
\ProofEnde

\section{Construction of the operator $\Mcal$} \label{A:D}

We give below an explicit construction of the operator $\Mcal$ used in the proof of Lemma \ref{L:k=kbarcase}. All required properties are part of the next statement.

\begin{lemma} \label{L:Fmult}
Given a positive integer $a<L^{\bar{k}-1}$ and $L\ge 2^{d+3}+16R$, there is an invertible Fourier multiplier $\Mcal: \boldsymbol{\Vcal}_N\to\boldsymbol{\Vcal}_N$, with eigenvalues bounded from below as
\begin{equation}
\widehat{\Mcal}(p) \ge \begin{cases}
\frac{1}{2} & \text{ if } \;\abs{p}\le L^{-\bar{k}} \\
C_a L^{2a(\bar{k}-j-1)} & \text{ if }\; L^{-j-1}<\abs{p}\le L^{-j} \text{ for } j<\bar{k}
\end{cases},
\end{equation}
such that $\Mcal f_\eps$ is supported in $B_{\bar{k}}^0$ and the inequality
\begin{equation} \label{E:Mfepsinner}
\big(\Mcal f_\eps,\Mcal f_\eps\big)_{\cbX_N} \le C(d,f,a) L^{d+2+4a} L^{-2\bar{k}} 
\end{equation}
holds for all $\bar{k}$ large enough (uniformly in $N$ and depending only on the fixed parameters) and for a constant $C(d,f,a)$ independent of $L$.
\end{lemma}
\begin{proof}
We construct $\Mcal$ as linear combination of two Fourier multipliers with range no greater than $L^{\bar{k}-1}$. By our assumptions on $f_\eps$ in \eqref{E:fassump}, this ensures that the support of $\Mcal f_\eps$ extends at most an $\abs{\cdot}_\infty$-distance of $\floor{L^{\bar{k}}/4}+L^{\bar{k}-1}<(L^{\bar{k}}-1)/2$ from the origin in $\Lambda_N$. Specifically, we let 
\begin{equation}
\Mcal = M +\tau \Delta^a,
\end{equation}
where $\Delta = \sum_{i=1}^d \nabla_i^*\nabla_i$ is the discrete Laplacian, $\tau>0$ is determined below and $M$ is an operator constructed in the sequel. Observe that $\Delta^a$ has range $a$ and satisfies the following upper and lower bounds in Fourier space:
\begin{equation}
2C_a \abs{p}^{2a} \le \widehat{\Delta^a}(p) \le \abs{p}^{2a},
\end{equation}
for some constant $0<C_a<1/2$ depending on $a$. The operator $M$ is defined with the help of a kernel $F:\Lambda_N\to\R$, that is, we require that
\begin{equation}
M\phi(x)=\sum_{y\in\Lambda_N} F(y-x)\phi(y).
\end{equation}
The kernel has range $L^{\bar{k}-1}$ and is defined by the expression
\begin{equation}
F(x) = \begin{cases}
(2\pi\theta_0)^{-d/2}L^{-d(\bar{k}-1)}e^{-\frac{\abs{L^{-(\bar{k}-1)}x}^2}{2\theta_0}} & \text{ if }\; \abs{x}_\infty < L^{\bar{k}-1} \\
0 & \text{ if }\; \abs{x}_\infty \ge L^{\bar{k}-1}.
\end{cases}
\end{equation}
The variance of the Gaussian density is determined by the requirements imposed subsequently on the discrete Fourier transform $\widehat{F}$. Let $\mathscr{D}=L^{-(\bar{k}-1)}\Z^d\cap [-1,1)^d$ denote the discretisation of the domain $[-1,1)^d$ with lattice spacing $L^{-(\bar{k}-1)}$ and write $\mathfrak{d}$ for the (continuous) interval $[0,L^{-(\bar{k}-1)})^d$. We now represent the Fourier transform of $F$ as
\begin{align} \label{E:stepint}
\widehat{F}(p) &= \sum_{x\in\Lambda_N}e^{-i x\cdot p}F(x)=\sum_{x\in \mathscr{D}}e^{-i L^{\bar{k}-1} x\cdot p} F(L^{\bar{k}-1}x) \nonumber \\
&= \int\displaylimits_{[-1,1)^d}\; \sum_{y\in \mathscr{D}}e^{-i L^{\bar{k}-1} y\cdot p} F(L^{\bar{k}-1}y)L^{d(\bar{k}-1)}\1_{y+\mathfrak{d}}(x)\,\d x \nonumber \\
&= (2\pi\theta_0)^{-d/2}\int\displaylimits_{[-1,1)^d}\; \sum_{y\in \mathscr{D}}e^{-i L^{\bar{k}-1} y\cdot p} e^{-\frac{\abs{y}^2}{2\theta_0}} (1-\1_{\abs{y}_\infty = 1}) \1_{y+\mathfrak{d}}(x)\,\d x.
\end{align} 
We compare the step function in \eqref{E:stepint} with the continuous integrand $\exp(-i L^{\bar{k}-1} x\cdot p -\abs{x}^2/\theta_0)$. Save when $x_j<L^{-(\bar{k}-1)}-1$ for some $j=1,\ldots,d$, it holds that
\begin{multline}
\big|e^{-i L^{\bar{k}-1} x\cdot p} e^{-\frac{\abs{x}^2}{2\theta_0}}-\sum_{y\in \mathscr{D}}e^{-i L^{\bar{k}-1} y\cdot p} e^{-\frac{\abs{y}^2}{2\theta_0}} (1-\1_{\abs{y}_\infty = 1}) \1_{y+\mathfrak{d}}(x)\big| \\
= \big|e^{-i L^{\bar{k}-1} x\cdot p} e^{-\frac{\abs{x}^2}{2\theta_0}}-e^{-i L^{\bar{k}-1} y\cdot p} e^{-\frac{\abs{y}^2}{2\theta_0}}\big|,
\end{multline}
for some $y\in\mathscr{D}$ such that $\abs{y-x}_\infty < L^{-(\bar{k}-1)}$. We then obtain the following estimate:
\begin{align}
\big| &e^{-i L^{\bar{k}-1} x\cdot p} e^{-\frac{\abs{x}^2}{2\theta_0}} - e^{-i L^{\bar{k}-1} y\cdot p} e^{-\frac{\abs{y}^2}{2\theta_0}}\big| \nonumber \\
&\quad\le \frac{1}{2}\big|e^{-i L^{\bar{k}-1} x\cdot p} - e^{-i L^{\bar{k}-1} y\cdot p} \big|\big|e^{-\frac{\abs{x}^2}{2\theta_0}} + e^{-\frac{\abs{y}^2}{2\theta_0}}\big| + \frac{1}{2}\big|e^{-i L^{\bar{k}-1} x\cdot p} + e^{-i L^{\bar{k}-1} y\cdot p} \big|\big|e^{-\frac{\abs{x}^2}{2\theta_0}} - e^{-\frac{\abs{y}^2}{2\theta_0}}\big| \nonumber \\
&\quad\le \big|e^{-i L^{\bar{k}-1} x\cdot p} - e^{-i L^{\bar{k}-1} y\cdot p} \big| + \big|e^{-\frac{\abs{x}^2}{2\theta_0}} - e^{-\frac{\abs{y}^2}{2\theta_0}}\big| \nonumber \\
&\quad\le L^{\bar{k}-1}\abs{y-x}\abs{p} + \Big|\int_0^1\; \frac{(x+t(y-x))\cdot (y-x)}{\theta_0}e^{-\frac{\abs{x+t(y-x)}^2}{2\theta_0}}\,\d t\Big| \nonumber \\
&\quad\le \sqrt{d}\abs{p} + \frac{d}{\theta_0}L^{-(\bar{k}-1)}.
\end{align}
The volume of the set of values of $x$ excluded from the above is $2^d-(2-L^{-(\bar{k}-1)})^d\le d2^{d-1}L^{-(\bar{k}-1)}$, and so we arrive at the upper bound
\begin{multline} \label{E:discerror}
\Big|\widehat{F}(p)-\int_{[-1,1]^d}\;(2\pi\theta_0)^{-d/2} e^{-i L^{\bar{k}-1} x\cdot p} e^{-\frac{\abs{x}^2}{2\theta_0}}\Big| \\
\le (2\pi\theta_0)^{-d/2} \big(d2^{d-1}L^{-(\bar{k}-1)} + \frac{d2^d}{\theta_0}L^{-(\bar{k}-1)} + 2^d\sqrt{d}\abs{p}\big).
\end{multline}
The next step is to consider the error made by truncating the continuous Fourier transform. We obtain the following inequality,
\begin{align} \label{E:FTerror}
\Big| \int_{\R^d\setminus [-1,1]^d}\;e^{-i L^{\bar{k}-1} x\cdot p} (2\pi\theta_0)^{-d/2} e^{-\frac{\abs{x}^2}{2\theta_0}}\,\d x \Big| &\le \int_{\abs{x}_\infty>1}\;(2\pi\theta_0)^{-d/2}e^{-\frac{\abs{x}^2}{2\theta_0}}\,\d x \le 2de^{-\frac{1}{2\theta_0}},
\end{align}
using simple Gaussian tail estimates. Finally, we have the standard result
\begin{equation} \label{E:FT}
\int_{\R^d}\;(2\pi\theta_0)^{-d/2} e^{-i L^{\bar{k}-1} x\cdot p} e^{-\frac{\abs{x}^2}{2\theta_0}}\,\d x = e^{-\frac{\theta_0 \abs{L^{\bar{k}-1}p}^2}{2}}.
\end{equation}
Combining these estimates yields the following constraint on $\theta_0$. We see that picking $\theta_0<\min\{(2\log 8 d)^{-1},2\log 4/3\}$ ensures that the Fourier transform in \eqref{E:FT} is bounded below by $3/4$ for $\abs{p}\le L^{-\bar{k}+1}$ and the gap in \eqref{E:FTerror} is less than $1/8$. With $\theta_0$ fixed in this way, we then infer that the bound in \eqref{E:discerror} is equally less than $1/8$ for all $\bar{k}$ sufficiently large. The expression in \eqref{E:discerror} is monotone in $L$, so that a lower bound on $\bar{k}$ can be given independently of $L$. We thus conclude that $\widehat{M}(p)\ge 1/2$ for all $\abs{p}\le L^{-\bar{k}+1}$.

Since $\widehat{\Delta^a}(p)$ is non-negative for all $p$, the above immediately implies the desired lower bound $\widehat{\Mcal}(p)\ge 1/2$ for $\abs{p}\le L^{-\bar{k}}$. For the large $\abs{p}$ eigenvalues, we see that a choice $\tau=L^{2 a \bar{k}}$ renders the lower bound
\begin{equation}
\tau\widehat{\Delta^d}(p)\ge 2 C_a \tau\abs{p}^{2a}
\end{equation}
attractive for $\abs{p}\ge L^{-\bar{k}+1}$. For $L^{-\bar{k}}<\abs{p}\le L^{-\bar{k}+1}$ we get $\widehat{\Mcal}(p)\ge 1/2+2C_a>C_a>0$, which is enough to guarantee invertibility. A crude upper bound for $\widehat{M}(p)$ is given by
\begin{equation}
\abs{\widehat{F}(p)} \le 2^d(2\pi\theta_0)^{-d/2} \le C_a L^{2a},
\end{equation}
for all $L$ large enough depending on $d$. A direct calculation shows that the assumption $L\ge 2^{d+3}+16R$, which is in any event needed for our analysis, is (more than) sufficient for this purpose. This confirms the inequality $\widehat{\Mcal}(p) \ge C_a L^{2a(\bar{k}-j-1)}$ for the remaining values of $p$ such that  $L^{-j-1}<\abs{p}\le L^{-j}$ with $j<\bar{k}-1$. Finally, we compute the inner product in \eqref{E:Mfepsinner} to complete the proof:
\begin{align}
\big(\Mcal f_\eps,\Mcal f_\eps\big)_{\cbX_N} &\le 2\big(M f_\eps,M f_\eps\big)_{\cbX_N}+2\tau^2\big(\Delta^a f_\eps,\Delta^a f_\eps\big)_{\cbX_N} \nonumber \\
&\le 2\sum_{x\in B^0} \abs{f_\eps}_\infty^2\Big(\sum_{\abs{y-x}_\infty\le L^{\bar{k}-1}}F(y-x)\Big)^2 + 2\tau^2\sum_{x\in B^0} \abs{\Delta^a f_\eps}_\infty^2 \nonumber \\
&\le 2^{4d+5}C_f^{d+4}(2\pi\theta_0)^{-d}L^{d+2}L^{-2\bar{k}} + 2^{2d+5+8a}d^{2a}C_f^{d+2+4a}L^{d+2+4a}L^{-2\bar{k}} \nonumber \\
&\le C(d,f,a)L^{d+2+4a}L^{-2\bar{k}}.
\end{align}
\end{proof}

\end{appendices}


\begin{thebibliography}{WWW98} 

\bibitem[AKM16]{AKM16} \textsc{S.~Adams, R.~Koteck\'{y}} and \textsc{S.~M\"uller},
\newblock Strict convexity of the surface tension for non-convex potentials, arXiv:1606.09541 (2016).

\smallskip
 
\bibitem[ABKM]{ABKM} \textsc{S.~Adams, S.~Buchholz, R.~Koteck\'{y}} and \textsc{S.~M\"uller}, Cauchy-Born Rule from Microscopic Models with Non-convex Potentials, arXiv:1910.13564v2 (2024).
 
\smallskip

\bibitem[AW22]{AW}\textsc{S.~Armstrong} and \textsc{W.~Wu}, \newblock $C^2$-regularity of the surface tension for the $ \nabla\phi$ interface model, \newblock Commun. pure appl. math. \textbf{75(2)}, 349--421 (2022).

\smallskip

\bibitem[BBS19]{BBS18}
\textsc{R. Bauerschmidt}, \textsc{D.C. Brydges}, and \textsc{G. Slade}, Introduction to a renormalisation group method, Lect. Notes Math. {\bf 2242}, xii + 281, Springer (2019). 

\smallskip

\bibitem[BPR22a]{BPR22a} \textsc{R.~Bauerschmidt, J.~Park} and \textsc{P.-F.~Rodriguez}, \newblock The Discrete Gaussian model, I. Renormalisation group flow at high temperature, \newblock Ann. Probab. \textbf{52(4)}, 1253--1359 (2024).

\smallskip

\bibitem[BPR22b]{BPR22b} \textsc{R.~Bauerschmidt, J.~Park} and \textsc{P.-F.~Rodriguez}, \newblock The Discrete Gaussian model, II. Infinite-volume scaling limit at high temperature, \newblock Ann. Probab. \textbf{52(4)}, 1360--1398 (2024).

\smallskip

\bibitem[BK07]{BK07}
\textsc{M.~Biskup} and  \textsc{R.~Koteck{\'y}},
\newblock Phase coexistence of gradient Gibbs states,
\newblock   Probab. Theory Relat. Fields
\textbf{139},  1--39 (2007).

\smallskip

\bibitem[BS11]{BS11} 
\textsc{M. Biskup} and \textsc{H. Spohn}, 
Scaling limit for a class of gradient fields with nonconvex potentials, 
Ann. Probab. \textbf{39}, 224--251 (2011). 
  
\smallskip
  
\bibitem[Bry09]{Bry09} 
\textsc{D.C. Brydges}, Lectures on the renormalisation group, 
Statistical  mechanics, IAS/Park City Math. Ser., vol.~16, 7--93, Amer. Math. Soc. (2009)

\smallskip

\bibitem[BY90]{BY90} \textsc{D.C.~Brydges} and \textsc{H.T.~Yau},
\newblock Grad $ \varphi $ Perturbations of Massless Gaussian Fields,
\newblock Commun. Math. Phys. \textbf{129}, 351--392 (1990).

\smallskip
\bibitem[Buc18]{B18} \textsc{S.~Buchholz}, Finite Range Decomposition of Gaussian Measures with Improved Regularity, J. Funct. Anal. \textbf{275}, 1674--1711 (2018).

\smallskip

\bibitem[CDM09]{CDM} \textsc{C.~Cotar, J.-D.~Deuschel} and \textsc{S.~M\"uller}, \newblock Strict Convexity of the Free Energy for a Class of Non-Convex Gradient Models,
\newblock Commun. Math. Phys. \textbf{286}, 359--376 (2009).

\smallskip

\bibitem[CD12]{CD09}\textsc{C.~Cotar} and \textsc{J.-D.~Deuschel},
\newblock  Decay of covariances, uniqueness of ergodic component and scaling limit for a class of $\nabla\phi $ systems with non-convex potential, \newblock Ann. Inst. H. Poincar\'{e} Probab. Statist. \textbf{48(3)}, 819--853 (2012). 

\smallskip

\bibitem[DGI00]{DGI00} \textsc{J.D.~Deuschel, G.~Giacomin} and \textsc{D.~Ioffe},
\newblock Large deviations and concentration properties for $\nabla\varphi$ interface models,
\newblock Probab. Theory Relat. Fields \textbf{117}, 49--111 (2000).

\smallskip 

\bibitem[Fun05]{Fun05} \textsc{T.~Funaki},
\newblock Stochastic Interface Models,
\newblock In: Lectures on Probability Theory and Statistics, Ecole d'Et\'{e} de Probabilit'{e}s de Saint-Flour XXXIII - 2003 (ed. J. Picard), Lect. Notes Math. \textbf{1869}, 103--274, Springer (2005).

\smallskip
\bibitem[FS97]{FS97} \textsc {T.~Funaki} and \textsc{H.~Spohn},
\newblock Motion by Mean Curvature from the Ginzburg-Landau $
\nabla \varphi $ Interface Model, 
\newblock Commun. Math. Phys. \textbf{185}, 1--36 (1997).

\smallskip

\bibitem[GOS]{GOS} \textsc{G.~Giacomin, S.~Olla} and \textsc{H.~Spohn}, \newblock Equilibrium fluctuations for $ \nabla\phi$ interface models, \newblock Ann. Probab. \textbf{29(3)}, 1138--1172 (2001).

\smallskip

\bibitem[Hil16]{Hil16}  {\sc S.~Hilger},
\newblock Scaling limit and convergence of smoothed covariance for
gradient models with non-convex potential,
\newblock arXiv:1603.04703v1 (2016).

\smallskip

\bibitem[Hil20a]{Hil20(1)}  {\sc S.~Hilger},
\newblock Decay of covariance for gradient models with non-convex potential,     \newblock arXiv:2007.10869v1 (2020).

\smallskip

\bibitem[Hil20b]{Hil20(2)}  {\sc S.~Hilger},
\newblock Scaling limit and strict convexity of free energy for gradient models with non-convex potential,
\newblock arXiv:2005.12973v3 (2020).

\smallskip

\bibitem[NS97]{NS97} \textsc{A.~Naddaf} and \textsc{T.~Spencer}, \newblock On Homogenization and Scaling Limit of Some Gradient Perturbations of a Massless Free Field, \newblock Commun. Math. Phys. \textbf{183}, 55--84 (1997).

\smallskip

\bibitem[Mil11]{Miller} \textsc{J.~Miller}, \newblock Fluctuations for the Ginzburg-Landau $\nabla \phi $  Interface Model on a Bounded Domain, \newblock Commun. Math. Phys. \textbf{308}, 591--639 (2011).
\smallskip

\bibitem[She05]{She05} \textsc{S.~Sheffield}, \newblock Random surfaces, \newblock Asterisque \textbf{304} vi+175 (2005).

\smallskip

\bibitem[Vel06]{Velenik} \textsc{Y.~Velenik}, \newblock Localization and delocalization of random interfaces, \newblock Probab. Surv. \textbf{3}, 112--169 (2006).

\end{thebibliography}
\end{document}